\documentclass{amsart}
\usepackage{amsmath, amssymb, graphicx, epsfig, verbatim, psfrag, color,amscd,stmaryrd,rsfso,mathabx} 
\usepackage{accents}
\usepackage{dutchcal}
\usepackage{tikz}
\usepackage{amsthm}
\usepackage{enumitem}
\usetikzlibrary{arrows}
\usepackage[T1]{fontenc}
\usepackage[utf8]{inputenc}

\newcommand\Cidemp{\mathfrak j}
\newcommand\west{\mathbf w}
\newcommand\east{\mathbf e}

\usepackage{mathrsfs}
\newcommand\Uweight{\gamma}
\newcommand\Vweight{\delta}
\newcommand\HH{\mathbb H}
\newcommand\cald{\mathcal D}
\newcommand\BigMod{P}
\newcommand\Real{\mathrm{Re}}

\newcommand\Zdown{\Zin}

\newcommand\NearModMatched[2]{{\mathcal{MM}}^{#2}_{\flat;{#1}}}
\newcommand\BigModMatched{{\vardbtilde{\mathcal{MM}}}}
    
\newcommand\ModMor[1]{{\mathcal{MM}}_{\star;#1}}

\newcommand\dInt[1]{\partial^{(#1)}}
\newcommand\ModInt[1]{\mathcal{MM}_{\natural;\ell}}

\newcommand\ModMatchedChanged{{\mathcal{MM}}_{\sharp}}
\newcommand\ModMatchedX{{\mathcal{MM}}_{\flat}}
\newcommand\ModMatched{{\mathcal{MM}}_{\natural}}
\newcommand\SModMatched{{\mathcal{SM}}_{\natural}}

\newcommand\UnparModMatched{\widehat{\mathcal{MM}}}

\newcommand\ModMatchedTSIC{{\mathcal{MM}}_{\flat;\mathrm{\star}}}
\newcommand\dChanged{\partial^\sharp}
\newcommand\DChanged{\delta^1_\sharp}
\newcommand\dChangedT{\partial^{T;\sharp}}

\newcommand{\vardbtilde}[1]{\widetilde{\raisebox{0pt}[0.85\height]{$\widetilde{#1}$}}}

\newcommand\wpt{\mathrm{w}}
\newcommand\zpt{\mathrm{z}}
\newcommand\Ta{{\mathbb T}_{\alpha}}
\newcommand\dbar{\overline\partial}

\newcommand\EastIn{\widecheck{\East}}

\newcommand\AllPunctIn{\widecheck{\AllPunct}}

\newcommand\oin{\widecheck{o}}
\newcommand\cin{\widecheck{c}}
\newcommand\win{\widecheck{\weight}}

\newcommand\ModEast{\mathcal E}
\newcommand\ModWest{\mathcal N}
\newcommand\UnparModEast{\widehat\ModEast}
\newcommand\UnparModWest{\widehat\ModWest}

\newcommand\chiEmb{\chi_{\mathrm{emb}}}
\newcommand\Tb{{\mathbb T}_{\beta}}
\newcommand\UnparModDeg{\widehat{\mathcal N}}
\newcommand\ModDeg{\mathcal N}
\newcommand\Wall{\mathcal W}
\newcommand\Chamber{\mathcal C}

\newcommand\Brs{{{\mathcal B}_{\{r,s\}}}}
\newcommand\Bjk{{{\mathcal B}_{\{j,k\}}}}

\newcommand\DiagUp{\widehat{\Diag}}

\newcommand\gr{\mathbf{gr}}
\newcommand\gen{K}
\newcommand\ind{\mathrm{ind}}
\newcommand\longchord{\upsilon}

\newcommand\orK{\vec{K}}
\newcommand\Ring{\mathcal{R}}
\newcommand\Hwz{\mathcal{J}}
\newcommand\DmodAlg{\Dmod^{alg}}

\newcommand\cS{\overline S}
\newcommand\cSigma{\overline \Sigma}
\newcommand\Shadow{\mathbf{S}}

\newcommand\Sym{\mathrm{Sym}}
\newcommand\DA{DA}
\newcommand\MatchIn{\widecheck{\Matching}}
\newcommand\Win{\widecheck{W}}
\newcommand\Mout{\widehat{\Matching}}
\newcommand\Zin{\widecheck{Z}}
\newcommand\Zmid{Z^{\parallel}}
\newcommand\Zout{\widehat{Z}}
\newcommand\Mmid{M^\parallel}
\newcommand\Wmid{W^{\parallel}}
\newcommand\Hmid{{\mathcal H}^\parallel}
\newcommand\HmidHalfExt{{\mathcal H}^{\parallel,\frac{x}{2}}}
\newcommand\HmidExt{{\mathcal H}^{\parallel,x}}
\newcommand\alphain{\widecheck{\alpha}}

\newcommand\alphaout{\widehat{\alpha}}
\newcommand\Clgin{\widecheck{\Clg}}
\newcommand\Clgout{\widehat{\Clg}}
\newcommand\Blgin{\widecheck{\Blg}}
\newcommand\nDuAlgin{\widecheck{\nDuAlg}}
\newcommand\Blgout{\widehat{\Blg}}
\newcommand\cBlgout{\Blgout}
\newcommand\cBlgin{\Blgin}
\newcommand\DAmodExt{\DAmod^x}
\newcommand\DAmodP{\DAmod'}
\newcommand\DAmodHalfExt{\DAmod^{\frac{x}{2}}}

\newcommand\jSource{j_\flat}
\newcommand\bSource{S_\flat}
\newcommand\wSource{R_\flat}

\newcommand\compactev{\overline{ev}}
\newcommand\EastSource{\mathcal{T}}
\newcommand\WestSource{\mathcal{R}}
\newcommand\ModFlow{\mathcal M}

\newcommand\ev{\mathrm{ev}}
\newcommand\evB{\mathrm{ev}^{\betas}}
\newcommand\rhos{\boldsymbol{\rho}}
\newcommand\sigmas{\boldsymbol{\sigma}}
\newcommand\UnparModFlow{\widehat{\mathcal M}}
\newcommand\alphas{\boldsymbol{\alpha}}
\newcommand\betas{\boldsymbol{\beta}}
\newcommand\Source{\mathscr{S}}
\newcommand\orbits{\mathrm{orb}}
\newcommand\chords{\mathrm{cho}}
\newcommand\CDisk{\mathbb{D}}

\newcommand\orb{o}
\newcommand\goesto{\mapsto}
\newcommand{\Ground}{R}
\newcommand\bOut{\widehat{b}}
\newcommand\bIn{\widecheck{b}}
\newcommand\Hup{\mathcal{H}^{\wedge}}
\newcommand\Hdown{\mathcal{H}^{\vee}}

\usetikzlibrary{matrix,arrows,positioning}
\tikzset{cdlabel/.style={above,sloped,
    execute at begin node=$\scriptstyle,execute at end node=$}}
\tikzset{algarrow/.style={->, thick}}
\tikzset{blgarrow/.style={->, thick}}
\tikzset{clgarrow/.style={->, thick}}
\tikzset{tensoralgarrow/.style={double, double equal sign distance, -implies}}
\tikzset{tensorblgarrow/.style={double, double equal sign distance, -implies}}
\tikzset{tensorclgarrow/.style={double, double equal sign distance, -implies}}
\tikzset{modarrow/.style={->, dashed}}
\tikzset{othmodarrow/.style={->, thick}}
\tikzset{Amodar/.style={->, dashed}}
\tikzset{Dmodar/.style={->, dashed}}
\newcommand\op{\mathrm{op}}
\newcommand\opp{\op}

\newcommand\HD{\mathcal H}

\newcommand\TerMin{t\GenMin}

\newcommand\Min{\GenMin}
\newcommand\GenMin{\mho}
\newcommand\KHm{H^-}

\newcommand\Diag{\mathbf{D}}
\newcommand\States{\mathfrak{S}}
\newcommand\HFKm{\mathrm{HFK}^-}

\def\endproof{\relax\ifmmode\expandafter\endproofmath\else
  \unskip\nobreak\hfil\penalty50\hskip.75em\hbox{}\nobreak\hfil\bull
  {\parfillskip=0pt \finalhyphendemerits=0 \bigbreak}\fi}
\def\endproofmath$${\eqno\bull$$\bigbreak}
\def\bull{\vbox{\hrule\hbox{\vrule\kern3pt\vbox{\kern6pt}\kern3pt\vrule}\hrule}}

\def\mathcenter#1{%
  \vcenter{\hbox{$#1$}}%
}
\newcommand\CanonDD{\mathcal K}
\newcommand\North{\mathbf N}
\newcommand\South{\mathbf S}
\newcommand\East{\mathbf E}
\newcommand\West{\mathbf W}

\newcommand\IntPunct{\mathbf{\Omega}}
\newcommand\IntPunctEv{\mathbf{\Omega}_+}
\newcommand\IntPunctOdd{\mathbf{\Omega}_-}
\newcommand\AllPunct{\mathbf{P}}

\newcommand\Pos{\mathcal P}
\newcommand\Neg{\mathcal N}

\newcommand\HFKa{\widehat{\mathrm{HFK}}}

\newtheorem{thm}{Theorem}[section]

\newtheorem{lem}[thm]{Lemma}
\newtheorem{lemma}[thm]{Lemma}
\newtheorem{prop}[thm]{Proposition}
\newtheorem{defn}[thm]{Definition}

\newtheorem{example}[thm]{Example}
\newtheorem{rem}[thm]{Remark}
\newtheorem{remark}[thm]{Remark}

\numberwithin{equation}{section}

\newcommand\OneHalf{\frac{1}{2}}
\newcommand\One{\mathbf{1}}

\newcounter{bean}

\newcommand\IdempRing{I}
\newcommand\RestrictIdempRing{I_0}
\newcommand\Idemp[1]{\mathbf{I}_{#1}}
\newcommand\DT{\boxtimes}
\newcommand\x{\mathbf x}
\newcommand\w{\mathbf w}
\newcommand\y{\mathbf y}

\newcommand\lsup[2]{^{#1}{#2}}
\newcommand\lsub[2]{{}_{#1}{#2}}

\newcommand\cBlg{\Blg^\star}

\newcommand{\Mor}{\mathrm{Mor}}

\newcommand\z{\mathbf z}

\renewcommand{\u}{\mathbf u}

\newcommand{\C}{\mathbb C} \newcommand{\Z}{\mathbb Z}  \newcommand{\Q}{\mathbb Q} \newcommand{\R}{\mathbb R}

\newcommand\HFKsimp{{H}_{\Ring}}
\newcommand\CFKsimp{{C}_{\Ring}}

\newcommand\XX{\mathbf X}
\newcommand\YY{\mathbf Y}
\newcommand\ZZ{\mathbf Z}
\newcommand\Max{\Omega}

\newcommand\Field{\mathbb F}

\DeclareMathOperator{\Id}{Id}

\newcommand\Alg{\mathcal A}
\newcommand\Blg{\mathcal B}
\newcommand\BlgZ{{\mathcal B}_0}
\newcommand\ClgZ{\Clg_0}
\newcommand\Clg{\mathcal C}

\newcommand\Ainf{{\mathcal A}_{\infty}}
\newcommand\Ainfty\Ainf
\newcommand\Zmod[1]{{\mathbb Z}/{#1}{\mathbb Z}}

\newcommand\Opposite{o}

\newcommand\Amod{Q}
\newcommand\Dmod{R}

\newcommand\DAmod{RQ}

\newcommand\Mdown{M^{\vee}}
\newcommand\Mup{M^{\wedge}}

\newcommand\doms{\mathcal D}

\newcommand\MGradingSet{S}
\newcommand\Mgr{\mathbf{m}}
\newcommand\Agr{\mathbf{A}}
\newcommand\AlexGr{\mathbf A}
\newcommand\weight{\mathbcal{w}}

\newcommand\Partition{P}
\newcommand\DuAlg{{\mathcal A}'}

\newcommand\Matching{M}

\newcommand\PartInv{\mathcal Q}
\newcommand\Cwz{\mathcal{C}}
\newcommand\cClg{\Clg^\star}
\newcommand\nDuAlg{\Alg''}
\newcommand\Hpos[1]{{\mathcal H}^+_{#1}}
\newcommand\Hneg[1]{{\mathcal H}^-_{#1}}
\newcommand\Hmin[1]{{\mathcal H}^\cup_{#1}}
\newcommand\Hmax[1]{{\mathcal H}^\cap_{#1}}
\newcommand\Hid{{\mathcal H}^{\natural}}

\newcommand\Iup{\widehat{\mathbf I}}
\newcommand\Idown{\widecheck{\mathbf I}}

\makeatletter
\renewenvironment{proof}[1][\proofname]{\par
\pushQED{\qed}%
\normalfont \topsep6\p@\@plus6\p@\relax
\trivlist
\item\relax
{\bf#1\@addpunct{.}}\hspace\labelsep\ignorespaces
}{%
\popQED\endtrivlist\@endpefalse
}
\makeatother

\begin{document}
\title{Algebras with matchings and knot Floer homology}

\author[Peter S. Ozsv\'ath]{Peter Ozsv\'ath}
\thanks {PSO was supported by NSF grant number DMS-1405114 and DMS-1708284.}
\address {Department of Mathematics, Princeton University\\ Princeton, New Jersey 08544} 
\email {petero@math.princeton.edu}

\author[Zolt{\'a}n Szab{\'o}]{Zolt{\'a}n Szab{\'o}}
\thanks{ZSz was supported by NSF grant numbers DMS-1606571 and DMS-1904628.}
\address{Department of Mathematics, Princeton University\\ Princeton, New Jersey 08544}
\email {szabo@math.princeton.edu}

\begin{abstract}
  Knot Floer homology is a knot invariant defined using holomorphic
  curves.  In more recent work, taking cues from bordered Floer
  homology, the authors described another knot invariant, called
  ``bordered knot Floer homology'', which has an explicit algebraic
  and combinatorial construction.  In the present paper, we extend the
  holomorphic theory to bordered Heegaard diagrams for partial knot
  projections, and establish a pairing result for gluing such
  diagrams, in the spirit of the pairing theorem of bordered Floer
  homology.  After making some model calculations, we obtain an
  identification of a variant of knot Floer homology with its
  algebraically defined relative.  These results give a fast algorithm
  for computing knot Floer homology.
\end{abstract}

\maketitle

\section{Introduction}
\label{sec:Intro}

The aim of the present paper is to identify the knot invariant
from~\cite{Bordered2} with a certain specialization of the knot Floer
homology from~\cite{Knots} and~\cite{RasmussenThesis}. 

One version of knot Floer homology associates to a knot $K$, a
bigraded module over $\Ring=\Field[U,V]/UV=0$ (where $\Field=\Zmod{2}$
is the field with two elements), which we denote here by
$\HFKsimp(K)$.  That module in turn is the homology of a chain
complex, denoted here $\CFKsimp(\HD)$, associated to a doubly-pointed
genus $g$ Heegaard diagram $\HD$, whose differential counts
pseudo-holomorphic disks $u$ in $\Sym^g(\Sigma)$, weighted with the
monomial $U^{n_w(u)} V^{n_z(u)}$. Specializing further to $V=0$ (and
then taking homology), we obtain the knot invariant denoted $\HFKm(K)$
in~\cite{Knots}; and setting $U=V=0$, we obtain the knot invariant
$\HFKa(K)$.  (Note that throughout the present paper, we are using
knot Floer homology groups with coefficients mod $2$, though we
suppress this from the notation.)  In~\cite{AltKnots}, we described a
Heegaard diagram associated to a knot projection, where the generators
correspond to Kauffman states for the knot diagram; but the
differentials still remained elusive counts of pseudo-holomorphic
disks.

Taking algebraic clues from the bordered Floer homology
of~\cite{InvPair}, in~\cite{BorderedKnots,Bordered2}, we
defined chain complexes whose generators correspond to Kauffman
states and whose differentials are determined by certain explicit
algebraic constructions; we then verified that their homology groups are knot
invariants. Specifically, the constructions from~\cite{Bordered2}
give an oriented knot invariant $\Hwz(\orK)$, that is a bigraded module
over $\Ring$, whose $V=0$ specialization gives the knot invariant $\KHm(-\orK)$
from~\cite{BorderedKnots}. 

The aim of the present paper is to prove the following:

\begin{thm}
  \label{thm:MainTheorem}
  Bigraded knot Floer homology $\HFKsimp(K)$ (with coefficients mod
  $2$) is identified with the bordered knot invariant $\Hwz(\orK)$
  (again, with coefficients mod $2$) from~\cite{Bordered2}; the
  bigraded knot Floer homology $\HFKm(K)$ is identified with the
  bordered knot invariant $\KHm(\orK)$ from~\cite{BorderedKnots}; the
  bigraded knot Floer homology $\HFKa(K)$ is identified with the
  $U=V=0$ specialization of $\Hwz(\orK)$.
\end{thm}

(As the notation suggests, knot Floer homology is independent of the
choice of orientation on $K$; in view of the above theorem,
$\Hwz(\orK)$ is independent of this choice, as well.)

The bordered knot invariants are defined in terms of a decorated knot
projection. We start from the projection, cut it up into elementary
pieces, associate bimodules to those pieces, and then define the
invariant to be a (suitable) tensor product of the bimodules;
compare~\cite{Bimodules, HFa} for the three-dimensional analogue.

In its original formulation, knot Floer homology can be constructed
using any Heegaard diagram representing a knot. We work here with the
``standard Heegaard diagram'' associated to a decorated knot
projection, as used in~\cite{ClassKnots}.  After degenerating this
Heegaard diagram, we obtain partial Heegaard diagrams, formalized in
Section~\ref{sec:Heegs}, which come in two kinds: the ``lower
diagrams'' and the ``upper diagrams''.  Suitable counts of
pseudo-holomorphic curves in upper diagrams leads to the definition of
the ``type $D$ structure'' associated to an upper diagram
(Section~\ref{sec:TypeD}), and counting pseudo-holomorphic curves in
lower diagrams leads to the definition of the ``type $A$ structure''
associated to the lower diagram (Section~\ref{sec:TypeA}).  The
underlying algebras for these modules are closely related to the ones
defined in~\cite{BorderedKnots,Bordered2}. Specifically, we will be
working with curved modules over the algebra $\Blg(2n,n)$
from~\cite{BorderedKnots} (which we abbreviate here by $\Blg(n)$), with a curvature specified by a matching
$\Matching$ (as described in Section~\ref{sec:Algebra}), which in turn
is  very similar to working over the algebra
$\Alg(n,\Matching)$ from~\cite{Bordered2}.

A suitable tensor product of the type $A$ and type $D$ structures
computes knot Floer homology, by an analogue of the pairing theorem
from~\cite{InvPair}, proved in Section~\ref{sec:Pairing}.  The theory
is generalized to bimodules in Section~\ref{sec:Bimodules}.

With the bordered theory in place, the verification of
Theorem~\ref{thm:MainTheorem}, which is completed in
Section~\ref{sec:Comparison}, rests on some model computations
(see Section~\ref{sec:ComputeDDmods}).

The pseudo-holomorphic curve counting here has a few complications
beyond the material present in~\cite{InvPair}. As in~\cite{TorusMod},
we must account for ``Reeb orbits'' rather than merely Reeb chords.
The present analysis is simpler, though, because our hypotheses on the
partial diagrams exclude boundary degenerations on the
``$\alpha$-side''. Rather, the boundary degenerations that occur here
all happen on the ``$\beta$-side'', and indeed they happen in a
controlled manner that can be accounted for neatly in the algebra.
With this understanding, the proof of the pairing theorem
from~\cite{InvPair} adapts readily.

Theorem~\ref{thm:MainTheorem} can be viewed as providing explicitly
computable models for variants of knot Floer homology. Grid diagrams
also give explicit combinatorial models for the necessary holomorphic
curve counts; see~\cite{MOS,MOST}. While the more algebraic bordered
constructions studied here (and in~\cite{BorderedKnots,Bordered2}) at
present describe a specialized version of these invariants, their
computations are much more efficient; see~\cite{Program}. 

\subsection{Organization}
This paper is organized as follows. In Section~\ref{sec:Heegs}, we
formalize the partial Heegaard diagrams used in the bordered
pseudo-holomophic construction.  Section~\ref{sec:Algebra} contains
the algebraic background, with a formalization of the framework of
curved algebras and the modules over them, which will be used
throughout. Next we recall bordered algebras defined
in~\cite{BorderedKnots}, and show how they can be fit into the curved
framework. The next few sections deal with upper
diagrams, with the ultimate aim of defining their associated type $D$
structures. In Section~\ref{sec:Shadows}, we formalize the shadows of
domains that connect upper Heegaard states, and use these to define a
grading on the upper Heegaard states. In Section~\ref{sec:CurvesD}, we
formulate the pseudo-holomorphic curves that represent these shadows,
and lay out properties of their moduli spaces.  In Section~\ref{sec:TypeD} we define the
type $D$ structure associated to an upper diagram, and verify that it
satisfies the required structural identities, using the results from
Section~\ref{sec:CurvesD}.  In Section~\ref{sec:CurvesA}, we
generalize the material from Section~\ref{sec:CurvesD} to the curves
relevant for the type $A$ modules, which are defined in
Section~\ref{sec:TypeA}. In Section~\ref{sec:Pairing}, we prove a
pairing theorem for the modules associated to upper diagrams with
lower diagrams, expressing knot Floer homology in terms of the modules
associated to the two pieces.  In Section~\ref{sec:Bimodules}, the
material is adapted to the case of type $DA$ bimodules associated to
``middle diagrams''.  A corresponding pairing theorem is established
(Theorem~\ref{thm:PairDAwithD}). In fact, the above material is set up
over a subalgebra of $\Blg(n)$, corresponding to restricting the
idempotent ring. In Section~\ref{sec:ExtendDA}, we extend the $DA$
bimodules over all of $\Blg(n)$.

We would like to understand these bimodules for certain standard
middle diagrams.  For this description, we adapt the
algebraically-defined bimodules from~\cite{Bordered2}; to the curved
context. This adaptation is given in 
Section~\ref{sec:AlgDA}. Bimodules derived from these are identified
with the bimodules associated to certain basic middle diagrams in
Section~\ref{sec:ComputeDDmods}. Combining these model computations
with the pairing theorem for bimodules, we obtain an algebraic
description of the type $D$ modules associated to Heegaard diagrams
for upper knot diagrams, in Section~\ref{sec:computeD}. This is
readily adapted to a proof of Theorem~\ref{thm:MainTheorem} in
Section~\ref{sec:Comparison}.

\subsection{Acknowledgements}

The authors wish to express their gratitude to Robert Lipshitz and
Dylan Thurston. Many of the ideas here were inspired by bordered Floer
homology as described in~\cite{InvPair}; and moreover the use of Reeb
orbits also figured heavily in joint work with the first author which
will appear soon~\cite{TorusMod}.

We would also  like to thank Andy Manion for sharing with us
his work~\cite{Manion1,Manion2}; and  to thank Nate Dowlin, Ian Zemke,
and Rumen Zarev for interesting conversations,

\newcommand\SigmaUp{\widehat{\Sigma}_0}
\newcommand\SigmaDown{\widecheck{\Sigma}_0}

\section{Heegaard diagrams}
\label{sec:Heegs}

\subsection{Upper diagrams}

An {\em upper Heegaard diagram} is the following data:
\begin{itemize}
\item a surface $\Sigma_0$ of genus $g$ and
$2n$ boundary components, labelled $Z_1,\dots,Z_{2n}$, 
\item a collection
of disjoint, embedded arcs $\{\alpha_i\}_{i=1}^{2n-1}$, so that
$\alpha_i$ connects $Z_i$ to $Z_{i+1}$,
\item a collection of
disjoint embedded closed curves $\{\alpha^c_i\}_{i=1}^{g}$
(which are also disjoint from $\alpha_1,\dots,\alpha_{2n-1}$),
\item 
another collection of embedded, mutually disjoint closed curves
$\{\beta_i\}_{i=1}^{g+n-1}$.  
\end{itemize}
We require this data to also satisfy the following properties:
\begin{enumerate}[label=(UD-\arabic*),ref=(UD-\arabic*)]
\item For each $i\in\{1,\dots,2n-1\}$, $j\in \{1,\dots,g\}$, and $k\in\{1,\dots,g+n-1\}$,
$\alpha_i$ and $\alpha_j^c$ curves are transverse to $\beta_k$.
\item
Both sets of $\alpha$-and the
$\beta$-circles consist of homologically linearly
independent curves (in $H_1(\Sigma_0)$).
\item 
\label{UD:TwoBoundariesApiece}
The surface
obtained by cutting $\Sigma_0$ along $\beta_1,\dots,\beta_{g+n-1}$,
which has $n$ connected components, is required to contain exactly
two boundary circles in each component. 
\item
  \label{UD:NoPerDom}
  The subspace of $H_1(\Sigma_0;\Z)$ spanned by the curves
  $\{\alpha_i^c\}_{i=1}^g$ meets transversely the space spanned by
  $\{\beta_i\}_{i=1}^{g+n-1}$.
\item 
  \label{UD:NoHone}
  Letting $\cSigma$ denote the closed surface obtained by filling
  in $\partial \Sigma_0$ with $2n$ disks, 
  the space $H_1(\cSigma;\Z)$ is spanned by the images of the curves
  $\{\alpha_i^c\}_{i=1}^g$ and
  $\{\beta_i\}_{i=1}^{g+n-1}$.
\end{enumerate}

\begin{rem}
  An upper diagram specifies a three-manifold $Y$ whose boundary is a
  sphere, with an embedded collection of $n$ arcs.
  Condition~\ref{UD:NoPerDom} ensures that $H_2(Y;\Z)=0$.
  (Compare~\cite[Proposition~2.15]{HolDisk}.)
  When $g=0$, the three-manifold $Y$ is a three-ball.
  Also, when $g=0$, Condition~\ref{UD:NoHone} is automatically satisfied.
\end{rem}

Condition~\ref{UD:TwoBoundariesApiece} gives a
matching $\Matching$ on $\{1,\dots,2n\}$ (a partition into two-element subsets), where $\{i,j\}\in \Matching$ if $Z_i$ and $Z_j$ 
can be connected by a path that does not cross any $\beta_k$. 

We sometimes abbreviate the data
\[\Hup=(\Sigma_0,Z_1,\dots,Z_{2n},\{\alpha_1,\dots,\alpha_{2n-1}\},\{\alpha^c_1,\dots,\alpha^c_{g}\},
\{\beta_1,\dots,\beta_{g+n-1}\}),\]
and let $\Matching(\Hup)$ be the induced matching.

 \begin{figure}[h]
 \centering
 \input{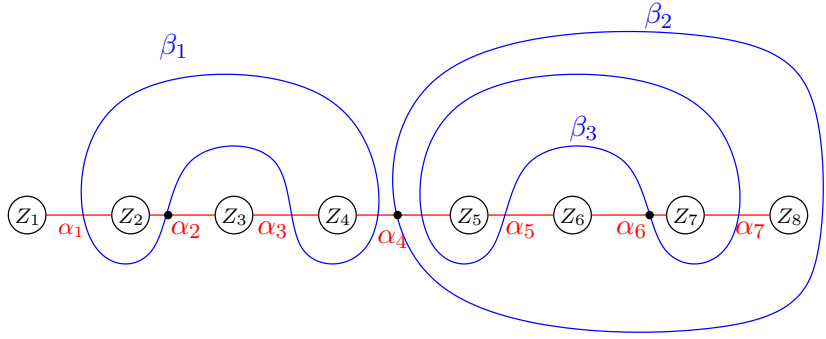}
 \caption{{\bf An upper Heegaard diagram.} The black dots represent an
   upper Heegaard state $\x$ with $\alpha(\x)=\{2,4,6\}$}
 \label{fig:UpperHeeg}
 \end{figure}

\begin{defn}
  \label{def:UpperState}
  An {\em upper Heegaard state} is a subset $\x$ of $\Sigma_0$
  consisting of $g+n-1$ points in the intersection of the various
  $\alpha$-and $\beta$-curves, distributed so that each $\beta$-circle
  contains exactly one point in $\x$, each $\alpha$-circle contains
  exactly one point in $\x$, and no more than one point lies on any
  given $\alpha$-arc.  Each Heegaard state $\x$ determines a subset
  $\alpha(\x)\subset \{1,\dots,2n-1\}$ with cardinality $n-1$
  consisting of those $i\in\{1,\dots,2n-1\}$ with $\x\cap \alpha_i\neq
  \emptyset$.
\end{defn}

\subsection{Lower diagrams}

A {\em lower diagram} $\Hdown$ is an upper diagram, equipped with an
extra pair of basepoints $w$ and $z$ and one
additional $\beta$-circle $\beta_{g+n}$ 
so that exactly two of the  $n+1$ components in $\Sigma_0\setminus\betas$
contains one boundary component of $\Sigma_0$ apiece, and these are marked
by the basepoints $w$ and $z$ (and the remaining $n-1$ components meet
two boundary components of $\Sigma_0$).

Note that a lower diagram also determines an equivalence relation
$\Mdown$ on the boundary circles $Z_1,\dots,Z_{2n}$, where $\{i,j\}\in
\Mdown$ if $Z_i$ and $Z_j$ can be connected by a path that does not
cross any $\beta_k$.

A {\em lower Heegaard state} $\x$ is a set of $g+n$ points, where one
lies on each $\beta$-circle, one lies on each $\alpha$-circle, and no
more than one lies on each $\alpha$-arc. Each lower Heegaard state $\x$
determines a subset $\alpha(\x)$ of $\{1,\dots,2n-1\}$ with
cardinality $n$, once again, consisting of those $i=1,\dots,2n-1$ with
$\x\cap \alpha_i\neq \emptyset$.

 \begin{figure}[h]
 \centering
 \input{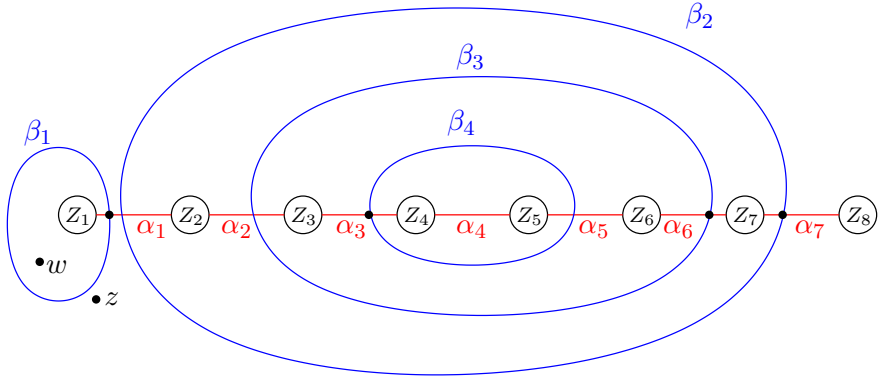}
 \caption{{\bf A lower Heegaard diagram.}  The basepoints $w$ and $z$
   are marked; furthermore, a lower state $\x$ with
   $\alpha(\x)=\{1,3,6,7\}$ is indicated by black dots.}
 \label{fig:LowerHeeg}
 \end{figure}

\subsection{Gluing upper and lower diagrams}
\label{subsec:Gluing}

\begin{defn}
  \label{def:CompatibleMatching}
  Let $M_1$ and $M_2$ be two matchings on $\{1,\dots,2n\}$.  We say
  that $M_1$ and $M_2$ are {\em compatible} if together they generate
  an equivalence relation with a single equivalence class.
\end{defn}

It is convenient to phrase this geometrically in the following terms.

\begin{defn}
\label{def:AssociatedW}
Let $M$ be a matching on a finite set $\{1,\dots,2k\}$. There is an
associated one-manifold $W(M)$, whose boundary components correspond
to the points in $\{1,\dots,2k\}$, and whose components correspond to
$\{i,j\}\in M$; that component connects the boundary component $i$ to
the boundary component $j$. 
\end{defn}

Suppose $\{1,\dots,2k\}$ is given with a matching $M_1$ and another
matching $M_2$ on a subset $S\subset\{1,\dots,2k\}$.  Those two
matchings together induce an equivalence relation $M_3$ on
$\{1,\dots,2k\}\setminus S$. The associated spaces are related by
\[ W(M_3)=W(M_2)\cup_{S} W(M_1).\]

In this language, the compatibility of $M_1$ and $M_2$ is
equivalent to the condition that $W(M_1)\cup_S W(M_2)$ has no closed
components.  See Figure~\ref{fig:AssociateW}.

\begin{defn}
  \label{def:Compatible}
  Suppose that $\Hup$ and $\Hdown$ are upper and lower diagrams with the
  same number $2n$ of boundary circles, and genera $g_1$ and $g_2$.
  We say that
  $\Hup$ and $\Hdown$ are {\em compatible} if the corresponding
  matchings $\Matching(\Hup)$ and $\Matching(\Hdown)$ are compatible,
  in the sense of Definition~\ref{def:CompatibleMatching}
\end{defn}

\begin{example}
  The upper Heegaard diagram $\Hup$ from Figure~\ref{fig:UpperHeeg} determines
  the matching on $\{1,\dots,8\}$, $\Mup=\{\{1,3\},\{2,4\},\{5,7\},\{6,8\}\}$;
  and the lower diagram $\Hdown$ from Figure~\ref{fig:LowerHeeg}
  determines the matching 
  $\Mdown=\{\{2,7\},\{3,6\},\{4,5\}\}$. together, they determine an 
  equivalence relation with two equivalence classes,
  $\{2,4,5,7\}$ and $\{1,3,6,8\}$;
  thus, $\Hup$ and $\Hdown$ are  not
   compatible, in the sense of Definition~\ref{def:Compatible}. 
 \begin{figure}[h]
 \centering
 \input{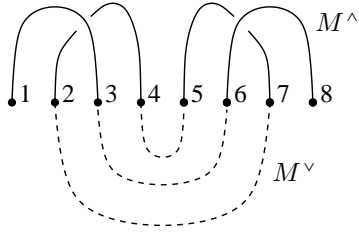}
 \caption{{\bf One-manifold associated to pairs of matchings.}}
 \label{fig:AssociateW}
 \end{figure}
\end{example}

\begin{prop}
  \label{prop:ConstructDiagram}
  Suppose that  $\Hup$ and $\Hdown$ are compatible upper and lower diagrams
  (with underlying surfaces $\SigmaUp$ and $\SigmaDown$ respectively).
  Glue $\SigmaUp$ and $\SigmaDown$ together along their boundaries
  to get a closed surface $\Sigma$ of genus $g=g_1+g_2+2n-1$. Collect
  the $\beta$-circles for $\Hup$ and $\Hdown$  into a $g$-tuple,
  thought of now as circles in $\Sigma$; and form a $g$-tuple
  of $\alpha$-circles consisting of the $\alpha$-circles in $\Hup$ and $\Hdown$
  and further $\alpha$-circles formed by gluing 
  $\alpha$-arcs in $\Hup$ to $\alpha$-arcs in $\Hdown$.
  The result is a doubly-pointed Heegaard diagram.
\end{prop}

\begin{proof}
  It is straightforward to see that $\Sigma$ has the stated genus, and
  that the $\alpha$-circles are homologically linearly independent.
  The compatibility condition guarantees that the complement of the
  $\beta$-circles is connected in $\Sigma$, and hence these circles
  are also linearly independent.
\end{proof}

We denote the doubly-pointed Heegaard diagram constructed in Proposition~\ref{prop:ConstructDiagram}
$\Hup\#\Hdown$.

\subsection{Examples}
\label{subsec:Examples}

The {\em standard upper diagram} is the diagram $\Hup(n)$ equipped with a
planar surface,  $\alpha$-arcs as in Figure~\ref{fig:StandardUpperDiagram}
(with $n=4$),
$\beta$-circles arranged as
in the figure
(after deleting $\beta_4$),
and no $\alpha$-circles.
(That figure also contains two basepoints $\wpt$ and $\zpt$ which should be
disregarded.)
Note that the standard upper diagram
has a single state $\x$ with $\alpha(\x)=\{2,\dots,2n-2\}$.
 \begin{figure}[h]
 \centering
 \input{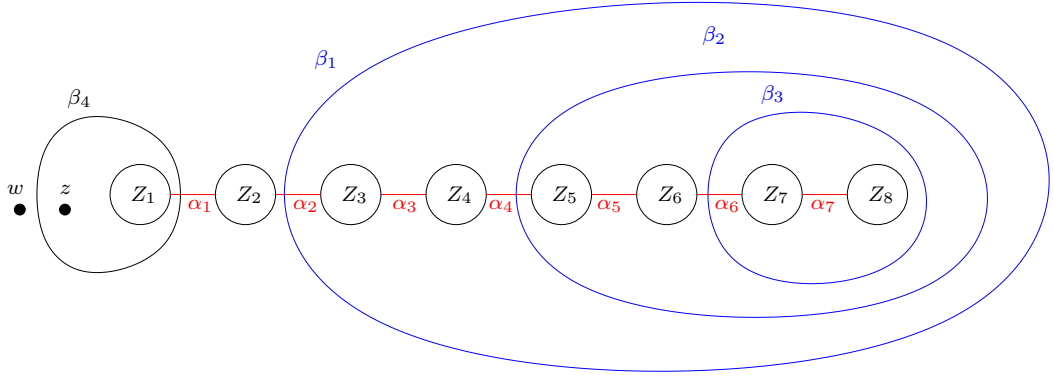}
 \caption{{\bf Standard Heegaard diagrams $n=4$.}  This is  the
   standard lower diagram; the standard upper diagram is obtained by
   removing the basepoints $w$ and $z$ and the circle $\beta_4$.}
 \label{fig:StandardUpperDiagram}
 \end{figure}

 The {\em standard lower diagram} is obtained from the standard upper
 diagram by adding an extra $\beta$-circle around $Z_1$, and two adjacent 
 basepoints $w$ and $z$ as shown
 in Figure~\ref{fig:StandardUpperDiagram}.

 Let $\Hup$ be any upper diagram drawn with Heegaard surface
 $\Sigma_0$, and fix an orientation preserving diffeomorphism
 $\phi\colon \Sigma_0 \to \Sigma_0$. There is a new upper Heegaard
 diagram
 \[ \phi(\Hup)=(\Sigma_0,\alphas,\phi(\betas))\cong (\Sigma_0,\phi^{-1}(\alphas),\betas).\]  We have
 illustrated this new upper diagram in the case where $\phi$ is a half
 twist switching $Z_2$ and $Z_3$ in
 Figure~\ref{fig:BraidGroupAction}.
 \begin{figure}[h]
 \centering
 \input{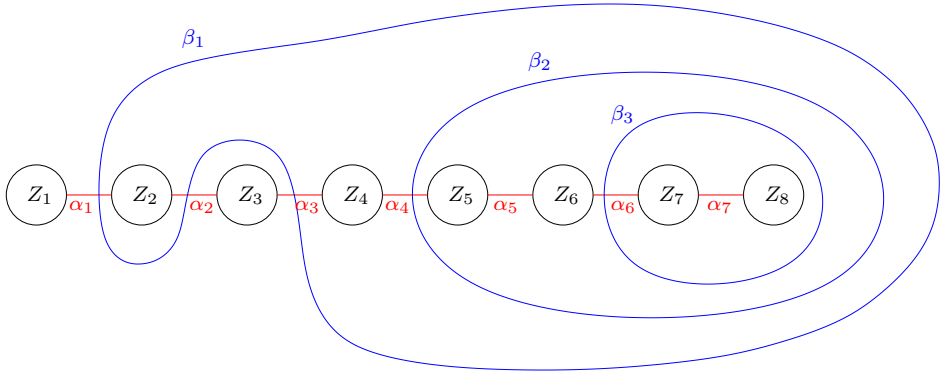}
 \caption{{\bf Acting on the standard upper diagram by a half twist switching
   $Z_2$ and $Z_3$}.
 \label{fig:BraidGroupAction}}
 \end{figure}

 Let $\Hup$ be an upper diagram on a surface $\Sigma_0$ of genus zero.
 Every mapping class of $\C$ punctured at points
 $\{p_1,\dots,p_{2n}\}$ can be represented by a diffeomorphism
 $\phi_0$ of $\Sigma_0$ to itself. Letting $\phi_0$ act on $\Hup$, we
 obtain an action of the mapping class group of
 $\C\setminus\{p_1,\dots,p_{2n}\}$ (which in turn is identified with
 the quotient of the braid group on $2n$ elements by the cyclic
 subgroup generated by a full twist; see~\cite{Birman}) on genus zero
 upper diagrams modulo isotopy.  The braid group is generated by
 consecutive switches $\sigma_i$; their corresponding mapping classes
 are represented by half twists: suppose that $\gamma_i$ is a curve
 that encloses exactly two puncture points $p_i$ and $p_{i+1}$, the
 half twist along $\gamma_i$ is the mapping class $\tau_{i,+}$
 supported in the compact component of $\C\setminus\gamma_i$ which
 switches $p_i$ and $p_{i+1}$, and whose square is a Dehn twist along
 $\gamma_i$.

 Gluing the standard upper diagram to $\phi$ applied to the standard
 lower diagram induces a Heegaard diagram representing some knot in
 $S^3$, provided that the permutation $\sigma$ induced by $\phi$ has
 the property that the product of permutations 
 \[ \sigma\cdot
 ((1,2),(3,4),\dots,(2n-1,2n))\cdot \sigma^{-1}\cdot
 ((1,2),(3,4),\dots,(2n-1,2n))\]
 has order $n$.
 
 Consider a decorated knot projection, equipped with a height
 function, so that the marked edge contains the global minimum for the height
 function.  Consider the associated Heegaard diagram
 (following~\cite{AltKnots}).  Slicing the knot projection along a
 generic horizontal slice corresponds to decomposing the Heegaard
 diagram as the gluing of some upper and lower Heegaard diagrams; see
 Figure~\ref{fig:SliceProjection}.

 As usual, the Heegaard surface $\Sigma$ is oriented as the boundary
 of the $\alpha$-handlebody.

 \begin{figure}[h]
 \centering
 \input{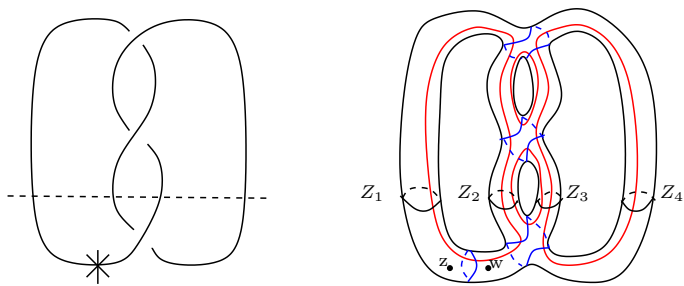}
 \caption{{\bf Slicing the Heegaard diagram corresponding
     to a decorated knot projection.}
   Cutting the decorated projection on the left along the dotted line
   corresponds to slicing the Heegaard diagram along the right along the
   circles $Z_1\cup Z_2\cup Z_3\cup Z_4$. }
 \label{fig:SliceProjection}
 \end{figure}

\subsection{Middle diagrams}

In this paper, we will define algebraic objects  using
holomorphic curve counts associated to upper and lower diagrams. Knot
Floer homology then will be expressed as a pairing between the
invariants of the upper and lower diagrams. To compute these
individual objects, it will help to be able to decompose the
knot diagram into further pieces, the {\em middle diagrams} we introduce presently.

\begin{defn}
  \label{def:MiddleDiagram}
A {\em middle diagram} is a Heegaard diagram obtained from an upper
diagram with $2m+2n$ boundary components, by deleting the arc
$\alpha_{2m}$ which would have connected $Z_{2m}$ with $Z_{2m+1}$.
Thus, the $\alpha$-arcs are labelled
$\alpha_i$ for 
\[ i\in\{1,\dots,2m-1,2m+1,\dots,2m+2n-1\}.\]
Think of some boundary components of $\Sigma_0$
as ``incoming'' boundary and other components  as ``outgoing'' boundary. 
The incoming boundary components are labelled as 
$\Zin_i=Z_i$ for $i=1,\dots,2m$; and the outgoing boundary components 
$\Zout_i=Z_{2m+i}$
for $i=1,\dots,2n$; $\alphain_i=\alpha_i$ for $i=1,\dots,2m-1$ and
$\alphaout_i=\alpha_{2m+i}$ for $i=1,\dots,2n-1$.  The $\beta$-circles
are labelled $\{\beta_i\}_{i=1}^{g+m+n-1}$.  We abbreviate the
resulting data
\begin{align*} \Hmid=(\Sigma_0,(\Zin_1,&\dots,\Zin_{2m}),(\Zout_1,\dots,\Zout_{2n}),
\{\alphain_1,\dots,\alphain_{2m-1}\},
\{\alphaout_1,\dots,\alphaout_{2n-1}\},\\
&\{\alpha^c_1,\dots,\alpha^c_{g}\},
\{\beta_1,\dots,\beta_{g+m+n-1}\}).
\end{align*}
\end{defn}

We will be primarily interested in five model middle diagrams: the
{\em identity diagram} $\Hid=\Hid(n)$ (Figure~\ref{fig:IdDiag}), the
{\em middle diagram of a local maximum} $\Hmax{c,n}$
(Figure~\ref{fig:MaximumHeeg}), the {\em middle diagram of a local
  mimum} $\Hmin{c,n}$ (Figure~\ref{fig:MinimumHeeg}), the {\em middle
  diagram of a positive crossing} $\Hpos{i,n}$, and the {\em middle
  diagram of a negative crossing} $\Hneg{i,n}$ (both in
Figure~\ref{fig:CrossDiag}).

We explain the interpretations of these Heegaard diagrams in terms of
partial knot projection; though this connection will become important
for us only much later (in Section~\ref{sec:ComputeDDmods}).  For
$\Hmax{c,n}$, where $2n$ denotes the number of input strands (so there are
$2n+2$ outputs), and $c\in 1,\dots,2n+1$ is chosen so that strands $c$ and
$c+1$ come out of the local maximum.  By contrast, for $\Hmin{c,n}$,
$2n$ is the number of output strands and, $c$ is choosen so that
the output strands $c$ and $c+1$ connect to the local maximum.  The diagrams for
$\Hpos{i,n}$ and $\Hneg{i,n}$ both represent $2n$ strands, so that the
$i$ and $i+1^{st}$ ones cross. Orienting each strand upwards, the
crossing is positive for $\Hpos{}$ and negative for $\Hneg{}$.  (In
practice, the strands in a crossing region will not all be oriented
pointing upward.)

\begin{figure}[h]
 \centering
 \input{IdDiagram.pstex_t}
 \caption{{\bf Middle diagram of the identity $\Hid(n)$.} 
 We have labelled a middle Heegaard state, with $|\alphain(\x)|=2$.}
 \label{fig:IdDiag}
 \end{figure}

 \begin{figure}[h]
 \centering
 \input{Maximum.pstex_t}
 \caption{{\bf Middle diagram of a local maximum $\Hmax{c,n}$.}  We have drawn here the case when $n=1$ and $c=2$.
   \label{fig:MaximumHeeg}}
 \end{figure}

 \begin{figure}[h]
 \centering
 \input{Minimum.pstex_t}
 \caption{{\bf Middle diagram of a local minimum $\Hmin{c,n}$.}  We have drawn here the case where $n=2$ and $c=1$.
 \label{fig:MinimumHeeg}}
 \end{figure}

 \begin{figure}[h]
 \centering
 \input{CrossingDiagram.pstex_t}
 \caption{{\bf Middle diagram of crossings $\Hpos{i,n}$ and $\Hneg{i,n}$.} 
  $\Hneg{2,2}$ is above and $\Hpos{1,2}$ is below.}
 \label{fig:CrossDiag}
 \end{figure}

A {\em middle Heegaard state} is a set of $g+m+n-1$ points of $\Sigma_0$ in
the locus where the various $\alpha$- and $\beta$-curves intersect,
distributed so that one lies on each $\beta$-circle, one lies on each
$\alpha$-circle, and no more than one lies on each $\alpha$-arc. For
each middle Heegaard state, let $\alphain(\x)\subset \{1,\dots,2m-1\}$
consists of those $i$ for which $\x\cap\alphain_i\neq
\emptyset$; and $\alphaout(\x)\subset\{1,\dots,2n-1\}$ consist of
those $j$ for which $\x\cap \alphaout_{j}\neq \emptyset$.
Obviously,
\begin{equation}
  \label{eq:alpaInOut}
  |\alphain(\x)|+|\alphaout(\x)|=m+n-1.
\end{equation}
We will be primarily interested in middle states with $|\alphain(\x)|=m$.

An upper diagram can be thought of as a middle diagram with $m=0$.

The $\beta$-circles induce an equivalence relation $\Mmid$  on 
the components of $\partial\Sigma_0$.

We should say a word about orientation conventions in pictures such
as Figure~\ref{fig:CrossDiag}. According to our
conventions from~\cite{HolDisk}, the Heegaard surface is oriented as
the boundary of the $\alpha$-handlebody. Thus, the Heegaard diagram in
picture in Figure~\ref{fig:CrossDiag} is oriented opposite to the
orientation it inherits as the boundary of a neighborhood of the
crossing apparent in the figure.

\subsection{Gluing middle diagrams to upper diagrams}

Let $\Hup_1$ be an upper diagram with $2m$ outgoing boundary
components, and $\Hmid$ a middle diagram with $2m$ incoming boundary
components and $2n$ outgoing ones.  Combine the equivalence relation
$\Matching(\Hup_1)$ with the equivalence relation $\Matching(\Hmid)$
to obtain an equivalence relation on all the boundary components of
$\Hmid$. If every equivalence class contains an outgoing boundary
component (of $\Hmid$), then we say that $\Hup_1$ and $\Hmid$ are {\em
  compatible}.

If $\Hup_1$ and $\Hmid$ are compatible, we can glue the boundary of
$\Hup_1$ to the incoming boundary of $\Hmid$ to form a new upper diagram
$\Hup_2=\Hup_1\# \Hmid$. Note that $\Matching(\Hup_2)$ is the matching
on $\{1,\dots,2n\}$, thought of as labelling the outgoing boundary components of $\Hmid$
induced by restricting $\Matching(\Hup_1)\cup\Matching(\Hmid)$ to the outgoing boundary of $\Hmid$.
Each upper state on $\Hup_2$ can be
restricted to give an upper state on $\Hup_1$ and a middle state in
$\Hmid$.  By Equation~\eqref{eq:alpaInOut}, for any middle state
obtained in this manner, $|\alphain(\x)|=m$.

We say that two upper diagrams $\Hup_1$ and $\Hup_2$ are {\em
  equivalent} if we can obtain $\Hup_2$ from $\Hup_1$ by isotopies,
handle slides, and stabilizations/destabilizations. Handleslides here
involve two $\beta$ circles, two $\alpha$ circles, or an $\alpha$-arc
slid over an $\alpha$-circle.

Observe that if $\Hid$ is the standard diagram for the identity with
$2n$ outgoing boundary components, then for $\Hup_2=\Hup_1\#\Hid$, we
can slide each new outgoing $\alpha$-arc over its corresponding
newly-formed $\alpha$-circle, and then destabilize that $2n-1$ new
$\beta$-circles with their newly-formed $\alpha$-circles, to get back
the original Heegaard diagram $\Hup_1$; i.e. $\Hup_1$ and $\Hup_2$ are
equivalent. 

\subsection{The Heegaard diagram of a knot projection}

Fix the sphere with four boundary circles, and four $\alpha$-arcs
connecting them in a circular configuration. Equip with a further
$\beta$-circle. When the $\beta$-circle meets each of the 
four arcs exactly once, we can think of the configuration as specifying a crossing,
where the type of the crossing depends on how the $\beta$-circle meets
the $\alpha$-arcs. We denote the two diagrams $H_+$ and $H_-$; see Figure~\ref{fig:CrossingPieces}.  
(As in the case of $\Hpos{}$ and $\Hneg{}$, the sign of the crossing using the usual conventions from knot theory agrees with
the sign of $H_{\pm}$ when both strands are oriented upwards.)

 \begin{figure}[h]
 \centering \input{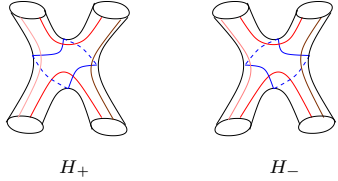}
 \caption{{\bf Crossing pieces.} 
 \label{fig:CrossingPieces}}
 \end{figure}

 We can glue $H_+$ or $H_-$to two consecutive output circles $Z_i$ and
 $Z_{i+1}$ in an upper diagram $\Hup_1$ to get new upper diagrams
 $\Hup_1\cup_{i} H_+$ or $\Hup_1\cup_{i}H_-$.  These diagrams be
 easily seen to be equivalent to the diagram obtained by gluing
 $\Hup_1$ to the middle diagrams $\Hmid_{+,i}$ and $\Hmid_{-,i}$.

We can relate this construction with the action of the mapping class
group on the upper diagram discussed in Section~\ref{subsec:Examples}.

Let $\tau_{i,+}$ be the half twist whose square is a positive Dehn
twist around a curve that encircles $Z_i$ and $Z_{i+1}$.  We have seen
that $\Hup\cup_{i} H_{+}$ is equivalent to $\Hup\cup\Hmid_{+,i}$ We
claim these two diagrams are in turn equivalent to
$\tau_{i,+}(\Hup)$. This can be seen by applying three handleslides to
$\Hup\cup_i H_+$, so that the newly-formed $\alpha$-circle and the
$\beta$-circle supported in $H_+$ meet in a single
point. Destabilizing the resulting diagram, we obtain a new diagram
which is diffeomorphic to $\tau_{i,+}(\Hup)$; see
Figure~\ref{fig:Handleslides}. An analogous computation shows the
equivalence between $\Hup\#\Hneg{i,n}$ and $\tau_{i,-}(\Hup)$.

 \begin{figure}[h]
 \centering
 \input{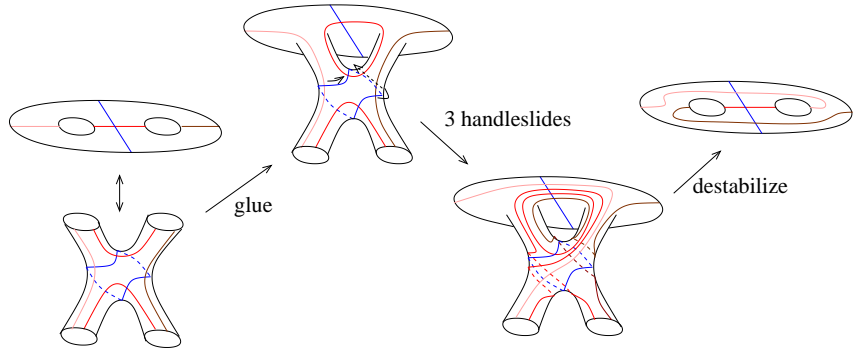}
 \caption{{\bf Attaching $H_+$ is equivalent to acting on a diagram by
     a half twist.}  At the left, we have a portion of an upper
   diagram (with two consecutive boundary components separated by an
   arc; in general, they are separated by a collection of parallel
   arcs), which is glued to $H_+$. Performing three handleslides (the
   first two of which are indicated) and destabilizing is
   diffeomorphic to the original upper diagram and acting on its
   $\alpha$-curves by a half twist.}
 \label{fig:Handleslides}
 \end{figure}

 Suppose that $K$ is a knot projection all of whose local maxima are
 global maxima, and whose local minima are global minima. The
 crossings correspond to the factorization of some mapping class
 $\phi$ as a product of half twists. From the above discussion, it is
 clear that the Heegaard diagram of the knot projection is
 equivalent to the Heegaard diagram obtained by gluing the standard
 upper and lower diagrams $\Hup(n)$ and $\Hdown(n)$, twisted by
 $\phi$; i.e.  $\phi(\Hup(n))\# \Hdown(n)$.

\section{Algebra}
\label{sec:Algebra}

We recall the algebraic preliminaries used in this paper, starting in
Section~\ref{subsec:Framework} with an algebraic framework familiar
from bordered Floer homology~\cite{InvPair}, applied to the case of
algebras with curvature.  In Subsection~\ref{subsec:BorderedAlgebras},
we recall the bordered algebras from~\cite{BorderedKnots}
and~\cite{Bordered2}, and explain how to fit them into this framework.

\subsection{Algebraic framework}
\label{subsec:Framework}

For the purposes of the present paper, we find it convenient to work
with smaller algebras than those from~\cite{Bordered2}, but equipped
with further algebraic structure, in the form of a {\em curvature}
element.  Such curved algebras have long played a role in Floer
homology; see for example~\cite{FOOO}. They have also played a role in
categorified knot invariants; see for example~\cite{KhovanovRozansky}.
Using curved algebras allows us to work with somewhat smaller
algebras, and the curve counting involves fewer types of
pseudo-holomorphic curves; see also~\cite{TorusMod}.

The following is a straightforward adaptation of the algebraic
material familiar in bordered Floer homology (see
especially~\cite[Chapter~2]{InvPair} and~\cite[Chapter~2]{Bimodules})
to the curved setting.

Fix a ground ring $\Ground$. A {\em curved algebra} is a graded
$\Ground$-bimodule $B$, equipped with an associative multiplication
\[ \mu_2\colon B\otimes_{\Ground}B \to B \]
(with unit)
and a preferred central element
$\mu_0\in B$. We denote the triple $(B,\mu_0,\mu_2)$ by $\cBlg$.

A {\em curved module} over $\cBlg$ is a right $\Ground$-module $M$,
equipped with right module maps $m_1\colon M \to M$
and $m_2\colon M\otimes_{\Ground} B \to M$ 
satisfying the structure relations
\begin{align}
  \label{eq:CurvedDiff}
  m_1\circ m_1 + m_2\circ (\Id_M \otimes \mu_0) &= 0 \\
  m_1\circ m_2 + m_2\circ (m_1\otimes \Id_{B})&= 0 \\
  m_2\circ (m_2\otimes \Id_{B})=m_2\circ (\Id_M\otimes \mu_2)
\end{align}

This has a generalization to the $\Ainf$ context, as follows.
A (curved) $\Ainf$ module $M_{\cBlg}$ consists of a right $\Ground$-module 
$M$, equipped with right $\Ground$-module homomorphisms
\[ m_i\colon M\otimes_\Ground \overbrace{B\otimes_R\dots \otimes_R
  B}^{i-1}\to B,\] satisfying a sequence of $\Ainf$ relations, indexed
by an integer $k\geq 0$, an element $x\in M$,
and a (possibly empty) sequence $b_1,\dots,b_k$ of algebra
elements:
\begin{align*}
  0 &=\sum_{j=0}^k m_{k-j+1}(m_{j+1}(x,b_1,\dots,b_j),b_{j+1},\dots,b_k) \\
&+ \sum_{i=1}^{k-1} m_{k}(x,b_1,\dots,b_{i-1},\mu_2(b_i,b_{i+1}),b_{i+2},\dots,b_k) \\
&+ \sum_{i=1}^{k+1}
  m_{k+2}(x,b_1,\dots,b_{i-1},\mu_0,b_{i},\dots,b_{k}).
\end{align*}
For example, when $k=0$, the above relation gives Equation~\eqref{eq:CurvedDiff}.

A {\em (curved) type $D$ structure} is a left $\Ground$-module $X$ equipped
with a left $\Ground$-module homomorphism
$\delta^1\colon X \to B\otimes_{\Ground} X$,
satisfying the structure relation
\[ 0 = \mu_0\otimes \Id_{X} + (\mu_2\otimes \Id_X)\circ(\Id_B\otimes
\delta^1)\circ \delta^1,\] thought of as maps $X\to B\otimes_{\Ground}
X$.  We will denote this data by $\lsup{\cBlg}X$.  (This notion is
equivalent to the notion of the {\em matrix factorizations} considered
in~\cite{KhovanovRozansky}.)

The definition of the tensor product $M_{\cBlg}\DT \lsup{\cBlg}X$ is
the same as in the uncurved context. The verification that this is a
chain complex is the same as in the uncurved case (see for
example~\cite[Section~2.4]{InvPair}), except that the curved type $D$
structure relations give rise to extra terms which are cancelled by
the extra terms in the curved $\Ainf$ relation.

Bimodules also can be generalized to the curved case in a
straightforward way.  For example, let $(B_1,\mu^{B_1}_0)$ and
$(B_2,\mu^{B_2}_0)$ be two curved algebras over ground rings $\Ground_1$ and $\Ground_2$ respectively.  A {\em curved type $DA$
  bimodule} is a $\Ground_2$-$\Ground_1$-bimodule $X$, equipped
with operations
\[ \delta^1_{\ell+1}\colon X \otimes_{\Ground_1} \overbrace{B_1\otimes_{\Ground_1}\dots\otimes_{\Ground_1}\otimes B_1}^{\ell}\to B_2\otimes_{\Ground_2} X,\]
satisfying the structural relations for any $x\in X$,
\begin{equation}
  \label{eq:CurvedDAbimoduleRelation1}
0 =(\mu^{B_2}_2\otimes\Id_X)\circ (\Id_{B_2}\otimes \delta^1_1)\circ \delta^1_1(x) \\
+ \delta^1_{2}(x,\mu_0^{B_1})
+ \mu_0^{B_2} \otimes x;
\end{equation}
and, for each sequence $b_1,\dots,b_k$ in $B_1$ (with $k\geq 1$),
\begin{align}
  0 &=\sum_{j=0}^k \mu_2^{B_2}\circ(\Id\otimes \delta^1_{k-j+1})\circ
(\delta^1_{j+1}(x,b_1,\dots,b_j),b_{j+1},\dots,b_k) 
   \label{eq:CurvedDAbimoduleRelation2}
\\
&+ \sum_{i=1}^{k-1} \delta^1_{k}(x,b_1,\dots,b_{i-1},\mu_2^{B_1}(b_i,b_{i+1}),b_{i+2},\dots,b_k) \nonumber \\
&+ \sum_{i=1}^{k+1}
  \delta^1_{k+2}(x,b_1,\dots,b_{i-1},\mu_0^{B_1},b_{i},\dots,b_{k}). \nonumber
\end{align}
(Compare~\cite{Bimodules} in the uncurved case.)
Note that the curvature on $B_2$ appears in Equation~\eqref{eq:CurvedDAbimoduleRelation1}, but not Equation~\eqref{eq:CurvedDAbimoduleRelation2}.

\begin{example}
  Fix curved algebras $(B_1,\mu_0^{B_1})$ and $(B_2,\mu_0^{B_2})$ over
  $\Ground$, and suppose that $\phi\colon B_1\to B_2$ is an
  $\Ground$-algebra homomorphism with the property that
  $\phi(\mu_0^{B_1})=\mu_0^{B_2}$. Then $\phi$ induces a type $DA$
  bimodule $\lsup{B_2}[\phi]_{B_1}$, whose underlying
  $\Ground$-bimodule is isomorphic to
  $\lsub{\Ground}\Ground_{\Ground}$, and whose
  operations are specified by
  $\delta^1_2(x,b)=\phi(b)\otimes x$ (where $x$ corresponds to
  $\One\in\lsub{\Ground}\Ground_{\Ground}$) and $\delta^1_k\equiv 0$
  for $k\neq 2$.
\end{example}

The above notions have obvious generalizations to the case where $B$
is a differential graded algebra, i.e. it is equipped with a
differential $\mu_1$ (which satisfies a Leibniz rule) so that
$\mu_1(\mu_0)=0$.

In practice, we will often have another DGA $\nDuAlg$ over $\Ground$,
and will consider various bimodules over $\cBlg$ and $\nDuAlg$.  For
example, thinking of $\nDuAlg\otimes\cBlg$ as curved, with curvature
$\One\otimes \mu_0$, we define a type $DD$ bimodule $X$ over $\nDuAlg$ and
$\cBlg$ to be a curved type $D$ structure over $\nDuAlg\otimes\cBlg$.

\subsection{The bordered algebras}
\label{subsec:BorderedAlgebras}
In~\cite{BorderedKnots}, to each pair of integers $(m,k)$ with
$0\leq k\leq m+1$, 
we associated an algebra
$\Blg(m,k)$. We recall the construction presently.

$\Blg(m,k)$ is constructed as the quotient of a larger algebra,
$\BlgZ(m,k)$, associated to $(m,k)$. The base ring of $\BlgZ(m,k)$ is the polynomial
algebra $\Field[U_1,\dots,U_m]$. Idempotents correspond to $k$-element
subsets $\x$ of $\{0,\dots,m\}$ called {\em idempotent states}. We think of
these as generators of a ring of idempotents $\IdempRing(m,k)$.

Given idempotents states $\x,\y$, the $\Field[U_1,\dots,U_m]$-module
$\Idemp{\x}\cdot\BlgZ(m,k)\cdot\Idemp{\y}$ is identified with
$\Field[U_1,\dots,U_m]$, given with a preferred generator $\gamma_{\x,\y}$.
To specify the product, we proceed as follows. Each idempotent state
$\x$ has a {\em weight vector} $v^\x\in \Z^m$, with components $i=1,\dots,m$
given by
$v^\x_i=\#\{x\in \x\big|x\geq i\}$.
The multiplication is specified by 
\[ \gamma_{\x,\y}\cdot \gamma_{\y,\z}=U_1^{n_1}\cdots U_m^{n_m} \cdot \gamma_{\x,\z},\] where
\[ n_i = \frac{1}{2}(|v_i^\x-v_i^\y|+|v_i^\y-v_i^\z|-|v_i^\x-v_i^\z|).\]

For $i=1,\dots,m$, let $L_i$ be the sum of $\gamma_{\x,\y}$, taken
over all pairs of idempotent states $\x, \y$ so that there is some integer $s$
with $x_s= i$ and $y_s=i-1$ and $x_t=y_t$ for all $t\neq
s$. Similarly, let $R_i$ denote the sum of all the $\gamma_{\y,\x}$
taken over the same pairs of idempotent states as
above. 

Under the identification $\Idemp{\x}\cdot \BlgZ(m,k)\cdot
\Idemp{\y}\cong \Field[U_1,\dots,U_m]$, the elements that correspond to monomials
in the $U_1,\dots,U_m$ are called {\em pure algebra elements in $\BlgZ(m,k)$}.
These elements are  specified by their idempotents $\x$ and $\y$, and their
{\em relative weight vector} $w(b)\in
\Q^m$, which in turn is uniquely characterized by
\[
   \weight_i(\gamma_{\x,\y})=\OneHalf|v_i^\x-v_i^\y| \qquad
   \weight_i(U_j\cdot b) = \weight_i(b)+\left\{\begin{array}{ll}
      0 &{\text{if $i\neq j$}} \\
      1 &{\text{if $i=j$.}} 
      \end{array}\right.
\]

Let ${\mathcal J}\subset \BlgZ(m,k)$ be the two-sided ideal generated by
$L_{i+1}\cdot L_i$, $R_{i}\cdot R_{i+1}$ and, for all choices of
$\x=\{x_1,...,x_k\}$ with $\x\cap \{j-1,j\}=\emptyset$,
the element 
$\Idemp{\x}\cdot U_j$.
Then, 
\[ \Blg(m,k)=\BlgZ(m,k)/{\mathcal J}.\]
The pure algebra elements in $\Blg(m,k)$ are those elements that
are images of pure algebra elements in $\BlgZ(m,k)$ under the above quotient map.

Idempotents $\x=x_1<\dots<x_k$ and $\y=y_1<\dots<y_k$ are said to be {\em too far} 
if for some $t\in \{1,\dots,k\}$, $|x_t-y_t|\geq 2$. If $\x$ and $\y$ are too far,
then $\Idemp{\x}\cdot\BlgZ(m,k)\cdot\Idemp{\y}\in {\mathcal J}$; i.e.
$\Idemp{\x}\cdot\Blg(m,k)\cdot\Idemp{\y}=0$.

We restate here the concrete description of the ideal ${\mathcal J}$ given in~\cite{BorderedKnots}:

\begin{prop}\cite[Proposition~3.7]{BorderedKnots}
  \label{prop:Ideal}
  Suppose that $b=\Idemp{\x}\cdot b\cdot \Idemp{\y}$ is a pure algebra element in 
  ${\mathcal J}$. Then, either $\x$ and $\y$ are too far, or there is 
  a pair of integers $i<j$ so that
  \begin{itemize}
  \item $i,j\in\{0,\dots,m\}\setminus \x\cap\y$
  \item for all $i<t<j$, $t\in\x\cap\y$
  \item $\weight_t(b)\geq 1$ for all $t=i+1,\dots,j$
  \item $\#(x\in \x\big| x\leq i)=\#(y\in \y\big| y\leq i)$.
  \end{itemize}
\end{prop}

We specialize to the case of $\Blg(n)=\Blg(2n,n)$, 
which we think of as an algebra over 
the idempotent ring $\IdempRing(n)=\IdempRing(2n,n)$.

We will typically consider a subalgebra
$\Clg(n)\subset \Blg(2n,n)$
given by
\[ \Clg(n)= \left(\sum_{\{\x\mid \x\cap \{0,
    2n\}=\emptyset\}} \Idemp{\x}\right)\cdot \Blg(2n,n)\cdot 
\left(\sum_{\{\x\mid \x\cap \{0, 2n\}=\emptyset\}}
  \Idemp{\x}\right).\] In particular, the elements $L_1$, $R_1$,
$L_{2n}$, $R_{2n}$ are not in this subalgebra; but $U_1$ and $U_{2n}$ are.
Let $\RestrictIdempRing(n)\subset \Clg(n)$ denote the subring
spanned by the idempotents $\Idemp{\x}$ where $\x\cap \{0,2n\}=\emptyset$.
Thus, 
\[ \Clg(n)=\RestrictIdempRing(n)\cdot\Blg(n)\cdot\RestrictIdempRing(n).\]
Sometimes, we also consider
\[ \ClgZ(n)=\RestrictIdempRing(n)\cdot\BlgZ(n)\cdot\RestrictIdempRing(n).\]

A {\em matching} is a partition of $\{1,\dots,2n\}$ into $2$-element subsets.
The matching $\Matching$ specifies a central algebra element
in $\Blg(n)$,
\begin{equation}
  \label{eq:CurvatureOfB}
  \mu_0^{\Matching}=\sum_{\{i,j\}\in\Matching} U_i U_j,
\end{equation}
which we
think of as specifying a curvature for
$\cBlg(n)=(\Blg(n),\mu^{\Matching}_0)$ or for 
\begin{equation}
  \label{eq:DefcClg}
  \cClg(n)=(\Clg(n),\mu^{\Matching}_0).
\end{equation}

We 
compare this with the algebraic set-up from~\cite{Bordered2}. In that paper, we
defined an algebra $\Alg(n,\Matching)$ containing $\Blg(n)$, with
new variables $C_{i,j}$ for each $\{i,j\}\in\Matching$ satisfying
$d C_{i,j}=U_i U_j$ and
$C_{i,j}^2=0$. 

\begin{defn}
  \label{def:Transformer}
  The {\em $\Blg$-to-$\Alg$ transformer} is the 
  type $DA$ bimodule $\lsup{\Alg}T_{\cBlg}$
  which, as a bimodule over $\IdempRing(2n,n)$, is identified with
  $\IdempRing(2n,n)$, and with operations specified by
  \begin{align*}
    \delta^1_1(\One)&=\sum_{(i,j)\in\Matching} C_{i,j}\otimes \One \\
    \delta^1_2(\One,b)&= b\otimes \One \\
    \delta^1_\ell(\One,b_1,\dots,b_{\ell-1})&= 0 \quad\text{for $\ell>2$}.
  \end{align*}
\end{defn}

Thus, a curved type $D$ structure over $\cBlg$
naturally gives rise to a type $D$ structure over $\Alg(n,\Matching)$,
$\lsup{\cBlg}X\to \lsup{\Alg}T_{\cBlg}\DT~\lsup{\cBlg}X$.

\subsection{Gradings}
\label{subsec:Gradings}

Our algebras are equipped with two types of gradings: an Alexander
grading, with values in some Abelian group, which is preserved by the
algebra operations; and a homological grading, with values in $\Z$, so
that $\mu_i$ shifts by $i-2$ (and in particular, the element $\mu_0$ has
Alexander grading zero and homological grading $-2$).

The weight function induces a grading on the algebra $\Blg(n)$ with
values in $(\OneHalf\Z)^{2n}\subset \Q^{2n}$. 
Choose for each $\{i,j\}\in\Matching$ a preferred ordering $(i,j)$ of the integers $i$ and $j$.
There is an induced {\em Alexander vector}
$\AlexGr\colon \Matching\to \Q$ 
defined by
\begin{equation}
  \label{eq:DefAgrAlg}
  {\mathbf A}_{\{i,j\}}(a)=\weight_i(a)-\weight_j(a), 
\end{equation}
where $(i,j)$ is the ordering on $\{i,j\}$.
Of course, this can be thought of as a grading with values in $\Q^n$.
Since $\AlexGr(\mu_0)=0$, the Alexander function induces a
well-defined (Alexander-type) $\Q^n$-grading on the curved algebra $\cBlg$.
Furthermore, there is an induced $\Q$-valued Alexander grading specified by
the function on homogeneous algebra elements
$A=\sum_{\{i,j\}\in\Matching} \AlexGr_{\{i,j\}}$.

Sometimes, when we wish to distinguish this from Alexander gradings on modules,
we write $\Agr^{\Matching}$.

More abstractly, we can think of the matching as giving rise to a
one-manifold $W=W(\Matching)$, consisting of $n$ arcs and boundary the
points $Y$ in $\{1,\dots,2n\}$ (i.e. each pair $\{i,j\}\in\Matching$
determines an arc connecting $i$ and $j$).  The weight of a given
algebra element gives an element of $H^0(Y)$; and the Alexander
grading can be thought of as an element of the cokernel $H^0(W)\to
H^0(Y)$, which is identified with $H^1(W,\partial W)\cong \Q^n$.  A
choice of isomorphism above is equivalent to an orientation on $W$.

We have also a homological $\Delta$-grading, determined by
\begin{equation} 
  \label{eq:DefDeltaAlg}
  \Delta(a)=-\sum_{i} \weight_i(a) 
\end{equation}
if $a\in \Blg$.  Note that $\Delta(\mu_0)=-2$, as required.

\subsection{Adapted bimodules}

We follow the algebraic set-up from~\cite[Section~2]{Bordered2} with
slight modifications. 

We can think of $\Blg(n)$ as an algebra associated to the
zero-manifold $Y_2$, which consists of $2n$ points. A matching on
$Y_2$, which we think of as a one-manifold $W_2$ with $\partial
W_2=Y_2$, specifies a curvature $\mu_0\in\Blg(Y_2)$. There is an
induced grading on $\Agr_{W_2}$ on $\Blg(Y_2)$ by
$H^1(W_2,Y_2)$, for which $\Agr_{W_2}(\mu_0)=0$.
Thus, we think of $\cBlg(Y_2,W_2)$ as graded by $\Agr_{W_2}$.

Fix cobordism $W_1$ from $Y_2$ to $Y_1$, and let $W=W_1\cup_{Y_2}
W_2$. If $X$ is a $H^1(W_1,\partial W_1)$-graded vector space, then
$X\otimes \cBlg(Y_2,W_2)$ inherits a grading by $H^1(W,\partial W)$,
using the natural map 
\begin{equation}
  \label{eq:MayerVietoris}
   H^1(W_1,\partial W_1)\oplus H^1(W_2,\partial
W_2)\to H^1(W,\partial W).
\end{equation}

\begin{defn}
  Suppose that $\cBlg_2$ is an algebra graded by $H^1(W_2,Y_2)$.
  Fix a cobordism $W_1\colon Y_2\to Y_1$.   A curved type $DA$
  bimodule $\lsup{\cBlg_1}X_{\cBlg_2}$ is called {\em adapted to $W_1$} if it
  is equipped with the following  additional data:
  \begin{itemize}
  \item a grading of $X$ by $H^1(W_1,\partial W_1)$, satisfying the
    following compatibility condition: if
    $a_1,\dots,a_{\ell-1}$ are $H^2(W_2,\partial W_2)$-homogenous elements,
    and $\x$ is an $H^1(W_1,\partial W_1)$-homogenous element, then
    $\delta^1_{\ell}(\x,a_1,\dots,a_{\ell-1})$ is 
    $H^1(W,\partial W)$-homogeneous, where $W=W_1\cup W_2$, 
    with 
    grading given by 
    \[\gr(x) + \Agr(a_1)+\dots + \Agr(a_\ell),\]
    viewed as an
    element of $H^1(W,\partial W)$ using the Mayer-Vietoris maps
    from Equation~\eqref{eq:MayerVietoris}).
  \item 
    a grading of $X$ by $\Q$, so that if $\x$, $a_1,\dots,a_{\ell-1}$
    are homogeneous, then  $\delta^1_{\ell}(\x,a_1,\dots,a_{\ell-1})$ 
    is homogeneous of degree
    \[ \Delta_X(\x) + \Delta(a_1)+\dots+\Delta(a_{\ell-1})-\ell+2.\]
  \item $X$ is a finite-dimensional $\Field$-vector space.
  \end{itemize}
\end{defn}

By contrast, recall that the adapted bimodules in the uncurved case
(\cite[Section~2]{Bordered2}) were graded by $H^1(W_1,\partial)$,
rather than $H^1(W_1\cup W_2,\partial)$. This causes no additional
difficulties. In particular, we have the following straightforward
modification
of~\cite[Proposition~\ref{BK1:prop:AdaptedTensorProducts}]{BorderedKnots}:

\begin{prop}
  Let $W_3$ be an oriented one-manifold with $Y_3=\partial W_3$. Fix
  also $W_2\colon Y_3\to Y_2$ and $W_1\colon Y_2\to Y_1$. If
  $\lsup{\cBlg_2}Y_{\cBlg_3}$ is adapted to $W_2$ and
  $\lsup{\cBlg_1}X_{\cBlg_2}$ is adapted to $W_1$, 
  and $W_1\cup W_2$ has no closed components; then we can form
  their tensor product $\lsup{\cBlg_1}X_{\cBlg_2}\DT \lsup{\cBlg_2}
  Y_{\cBlg_3}$ to get a curved DA bimodule $\lsup{\cBlg_1}(X\DT
  Y)_{\cBlg_3}$ adapted to $W_1\cup W_2$. \qed
\end{prop}

\section{Gradings on upper Heegaard states}
\label{sec:Shadows}

Throughout this section, we will fix an  upper diagram $\Hup$ with $2n$ boundary circles
\[\Hup=(\Sigma_0,Z_1,\dots,Z_{2n},\{\alpha_1,\dots,\alpha_{2n-1}\},\{\alpha^c_1,\dots,\alpha^c_{g}\},
\{\beta_1,\dots,\beta_{g+n-1}\})\] throughout. Let $\Matching$ be the
matching on $\{1,\dots,2n\}$ induced by $\Hup$.

In Section~\ref{sec:TypeD}, we will explain how to associate a type
$D$ structure to $\Hup$, over the curved algebra $\cClg$ from
Equation~\eqref{eq:DefcClg}. This structure has a differential, which
is defined by counting pseudo-holomorphic curves.  Here, we explain
the data needed to specify gradings on these structures.

\subsection{Preliminaries: filling in the Heegaard surface}
\label{sec:FillSurface}

Before proceeding to the main material in this section, we introduce
some notation which will be used throughout the paper.

Recall that the Heegaard surface  $\Sigma_0$ for $\Hup$ is an  oriented, connected
two-manifold of genus $g$ with boundary
$\partial{\Sigma_0}=Z=Z_1\cup\dots\cup Z_{2n}$.  By attaching
infinite cylinders $Z\times [0,\infty)$ to $\Sigma_0$, we obtain an oriented
two-manifold $\Sigma$ with punctures $p_1,\dots,p_{2n}$. Filling in
these punctures, we obtain a compact surface, denoted $\cSigma$. 

Extend $\alphas$ in $\Sigma$, by attaching two rays in each
cyclinder $Z_i\times [0,\infty)$ for $i=2,\dots,2n-1$ and a single ray
  inside each of $Z_1\times [0,\infty)$ and $Z_{2n}\times[0,\infty)$.
      In the filled surface $\cSigma$
      the union of $\alpha$-arcs  completes to
form a single closed interval. Let  ${\overline\alphas}\subset \overline\Sigma$
denote the subspace which is the union of the above defined interval and the union of curves $\{\alpha_i^c\}_{i=1}^g$.

\subsection{Gradings}
To each upper Heegaard state $\x$
for $\Hup$, there is an associated idempotent in
$\RestrictIdempRing(n)$ (the ring generated by the idempotent states in $\Clg(n)\subset
\Blg(n)$), defined by the formula
\[ \Iup(\x)=\Idemp{\{1,\dots,2n-1\}\setminus\alpha(\x)},\]
where $\alpha(\x)$ is defined as in Definition~\ref{def:UpperState}.

The complement of ${\overline\alphas}\cup\betas$ inside $\cSigma$ can be written
as a disjoint union of connected open sets called {\em elementary
  domains}. 

\begin{defn}
  \label{def:TwoChainFromXtoY}
Given upper states $\x$ and $\y$, a {\em two-chain from $\x$ to $\y$}
is a formal integral combination $\phi$ of the elementary domains in
$\cSigma$, with the following property. If $\partial_{\alpha}(\phi)$
resp.  $\partial_\beta(\phi)$ denotes the portion of the boundary of
$\phi$ contained in $\alphas$ resp. $\betas$, we require that
$\partial (\partial_{\alpha}(\phi))=\y-\x$ (and hence $\partial
(\partial_{\beta}(\phi))=\x-\y$).  Let $\doms(\x,\y)$ denote the space
of two-chains from $\x$ to $\y$.
\end{defn}

Given $\phi\in\doms(\x,\y)$ and $\psi\in\doms(\y,\z)$,
their sum can be viewed as  an element $\phi*\psi\in\doms(\x,\z)$.

\begin{defn}
  Let $\phi\in\doms(\x,\y)$.  Define $b_0(\phi)$ to be the element
  $b\in \BlgZ(2n,n)$ characterized by the following two properties that
  \begin{itemize}
    \item 
      $\Iup(\x)\cdot b\cdot \Iup(\y) = b$; and
    \item for all $i=1,\dots,2n$,
      $\weight_i(b)$ is the average of the local multiplicities of $\phi$ in the
      two elementary domains adjacent to $Z_i$.
  \end{itemize}
  We will also write $\weight(b)=\sum_{i=1}^{2n}\weight_i(b)$.
\end{defn}

\begin{lemma}
  \label{lem:MultUnderJuxtaposition}
  $b_0(\phi*\psi)=b_0(\phi)\cdot b_0(\psi)$,
  where the right hand side is multiplication in $\BlgZ(2n,n)$.
\end{lemma}

\begin{proof}
  This is clear from the additivity of the local multiplicities 
  under juxtaposition.
\end{proof}

Each elementary domain ${\mathcal D}_i$ in $\cSigma$ has an {\em Euler
  measure}, which is the integral of $1/2\pi$ times the curvature of a
metric for which the boundary consists of geodesics meeting at
$90^\circ$ angles along the corners.

For $\phi\in\doms(\x,\y)$, we define its {\em point measure} $P(\phi)$
to be the sum $n_{\x}(\phi)+n_{\y}(\phi)$, so that each elementary
domain ${\mathcal D}$ contributes $1/4$ times the number of components
of $\x$ and $\y$ contained as corners of ${\mathcal D}$.   The
{\em Maslov grading of $\phi$} is defined by the formula:
\begin{equation}
  \label{eq:DefineMgr}
\Mgr(\phi)=e(\phi)+P(\phi).
\end{equation}

\begin{lemma}
  \label{lem:GradingsWellDefined}
  If $\doms(\x,\y)$ is non-empty, then for any $\phi\in\doms(\x,\y)$,
  the integers $\Mgr(\phi)-\weight(b_0(\phi))$ and
  $\weight_i(\phi)-\weight_j(\phi)$ (with $\{i,j\}\in \Matching$) are independent of the
  choice of $\phi$.
\end{lemma}

\begin{proof}
  Before verifying the independence of the choice of $\phi$, we start
  by verifying that $\Mgr(\phi)$, which is evidently a rational
  number, is in fact an integer. This could be seen either by the
  interpretation of $\Mgr(\phi)$ as a Maslov index, but instead we
  recall here a more elementary argument. Since
  $\Mgr(\phi+\Sigma)=\Mgr(\phi)+2$, it suffices to verify the
  integrality of $\Mgr(\phi)$ for positive $\phi$. If $\phi$ is
  positive, it is elementary to construct a surface with corners $F$
  (cf.~\cite[Lemma~2.17]{HolDisk}) at $\x$ and $\y$; this surface is
  equipped with a branched covering to $\Sigma$ with branching at the
  intersection points of the $\alpha$- and the
  $\beta$-circles. Suppose for simplicity that each elementary domain
  is topologically a disk, so that the $\alpha$ and $\beta$-arcs and
  circles give $\Sigma$ the structure of a $CW$ complex. The surface
  $F$ has a $CW$ complex structure, obtained by pulling back this $CW$
  complex structure. Consider the function $f$ on subcomplexes of $F$
  that associates to each $2$-cell $1$, to each edge $-1/2$, and to
  each vertex $1/4$. Clearly, $f(\cald)=e(\cald)$ for each elementary
  domain. Moreover, since each interior edge is contained in two
  domains, and each interior vertex is contained in four elementary
  domains, it follows that
  \[ e(F)=\chi(F)+ \OneHalf \#(e\subset \partial F)+ \sum_{p}
  \left(\frac{N_p}{4}-1\right)=\chi(F)+ \sum_{p}
  \left(\frac{N_p-2}{4}\right),\] where $N_p$ denotes the number of
  elementary domains that meet at a $0$-cell in $F$.  The integrality of
  $\Mgr(\phi)$ follows from the observation that at each corner point
  $p$, $\frac{N_p-2}{4}+n_{p\cap \x}(\phi)+n_{p\cap \y}(\phi)$ is an integer.
  This argument can be easily adapted also to the case where
  the elementary domains are not disks.

  If $\phi,\phi'\in\doms(\x,\y)$ then $\phi-\phi'$ can be written as a
  formal sum of components ${\mathcal D}$ of
  $\cSigma\setminus\betas$.  This is true since
  $\alpha_1^c,\dots,\alpha_g^c$ are linearly independent in
  $H_1(\cSigma)$ and the intersection of their span with the
  span of $\beta_1,\dots,\beta_{g+n-1}$ is trivial (by
  Condition~\ref{UD:NoPerDom}). Each of those latter components has
  $e({\mathcal D})=1$ and $P({\mathcal D})=1$, contributing
  $2$ to $\Mgr(\phi)$; similarly, the addition of ${\mathcal D}$
  contributes $2$ to $\weight(b_0(\phi))$.  To complete the lemma, note
  that if $i$ and $j$ are matched, then $\weight_i({\mathcal
    D})=\weight_j({\mathcal D})$ for any ${\mathcal D}$ of
  $\cSigma\setminus\betas$.
\end{proof}

Given $\phi\in\doms(\x,\y)$, let $\bOut(\phi)\in\Idemp{\x}\cdot
\Blg\cdot \Idemp{\y}$ denote the image of $b_0(\phi)$ under the
quotient map $\BlgZ(n)\to\Blg(n)$. Clearly,
$\bOut(\phi)\in\Clg(n)\subset\Blg(n)$.

\begin{prop}
  \label{prop:DefineBigradingD}
  There is a function $\Mgr\colon \States(\Hup)\to \Z$
  that is uniquely characterized up to an overall constant by the property that
  \begin{equation}
    \label{eq:DefMgr}
    \Mgr(\x)-\Mgr(\y)=\Mgr(\phi)-\weight(b_0(\phi)),
  \end{equation}
  for any $\phi\in\doms(\x,\y)$.  Similarly, given an orientation for
  the one-manifold $W$ specified by the matching $\Matching(\Hup)$,
  there is another function $\Agr\colon \States(\Hup)\to \OneHalf\Z^n$
  with components $\Agr_{\{i,j\}}$ cooresponding to each
  $\{i,j\}\in\Matching$, characterized by the following condition,
  uniquely up to overall translation by some vector in
  $\OneHalf\Z^n\subset\Q^n$:
\begin{equation}
  \label{eq:DefAgr}
  \Agr_{\{i,j\}}(\x)-\Agr_{\{i,j\}}(\y)=
  \weight_i(b_0(\phi))-\weight_j(b_0(\phi)),
\end{equation}
for any choice of $\phi\in\doms(\x,\y)$;
i.e. $\Agr(\x)-\Agr(\y)=\Agr(b_0(\phi))$, where the right-hand-side uses the $\Q^n$-valued Alexander vector grading on the algebra.
\end{prop}

\begin{proof}
  Fix $\x$ and $\y$.  Condition~\ref{UD:NoHone} ensures that for any
  two upper states $\x$ and $\y$, there is some $\phi\in\doms(\x,\y)$.
  Thus, according to Lemma~\ref{lem:GradingsWellDefined}, given $\x$
  and $\y$, the right-hand-side of Equation~\eqref{eq:DefMgr} is
  well-defined. To see that it can be written as $\Mgr(\x)-\Mgr(\y)$,
  it suffices to observe that the right hand side of
  Equation~\eqref{eq:DefMgr} is additive under juxtaposition. This is
  mostly straightforward; see~\cite[Theorem 3.3]{SarkarWhitney} for an
  elementary proof of the additivity of $\Mgr$ under juxtapositions.

  The corresponding statement for $\Agr$ follows similarly.
\end{proof}

\begin{rem}
  \label{rem:OurDomains}
  We could have chosen to work instead with elementary domains
  ${\mathcal D}_i^0$, which are the closures of the components of
  $\Sigma_0\setminus\alphas\cup\betas$. 
  Clearly each elementary domain
  ${\mathcal D}_i$ in $\cSigma$ is obtained
  from some elementary domain ${\mathcal D}^0_i$ in $\Sigma_0$ by
  attaching half disks to each boundary component of ${\mathcal D}_i$
  obtained as $Z_j\cap {\mathcal D}_i$ with $j\neq 1$ or $2n$; and
  attaching disks along the boundary components $Z_1\cap{\mathcal
    D}_i$ and $Z_{2n}\cap{\mathcal D}_i$.  
  Thus, the Euler measure of
  each elementary doman in ${\overline \Sigma}$ equals the Euler
  measure of the corresponding domain in $\Sigma_0$ plus $1/2$ for
  each boundary component induced by $Z_j$ with $j\neq 1$ or $2n$ and
  $1$ for the boundary components coming from $Z_1$ or $Z_{2n}$.
  We could work with  $\doms_0(\x,\y)$, which are domains in $\Sigma_0$.
  The Euler measure on elementary domains can be extended linearly to obtain an Euler measure
  of any $\phi_0\in\doms_0(\x,\y)$.
  If $\phi_0\in\doms_0(\x,\y)$, and $\phi\in \doms(\x,\y)$ is 
  the corresponding domain in $\cSigma$, then
  \[ e(\phi)=e(\phi_0)+\weight(b_0(\phi)).\]
  With these conventions, then, 
  Lemma~\ref{lem:GradingsWellDefined} says that
  and 
  \[ \Mgr(\phi)-\weight(b_0(\phi))=e(\phi_0)+P(\phi_0)\]
  is independent of the choice of $\phi_0\in\doms_0(\x,\y)$.
\end{rem}

\section{Holomorphic curves used for type $D$ structures}
\label{sec:CurvesD}

We will describe now the holomorphic curves that go into the
construction of the type $D$ structure associated to an upper diagram,
returning to the type $A$ structure in Section~\ref{sec:CurvesA}. The
curves counted in the present work are similar to the curves counted in~\cite{InvPair}.  Since our context here is slightly
different, we recall material from~\cite{InvPair}, with an emphasis on
the differences.

Fix some upper
diagram \[\Hup=(\Sigma_0,Z_1,\dots,Z_{2n},\{\alpha_1,\dots,\alpha_{2n-1}\},\{\alpha^c_1,\dots,\alpha^c_{g}\},
\{\beta_1,\dots,\beta_{g+n-1}\}).\] 
Filling in the boundary of $\Sigma_0$ as explained in Section~\ref{sec:FillSurface},
we get $\Sigma_0\subset \Sigma\subset \cSigma$.

We will use  Lipshitz's reformulation
of Heegaard Floer homology~\cite{LipshitzCyl}, where the
pseudo-holomorphic curve counting takes place in
$\Sigma\times[0,1]\times \R$. To this end, we will use the class of
almost-complex structures appearing there (see
also~\cite[Definition~5.1]{InvPair}), which we recall presently.

There are two projection maps
\[
\pi_{\Sigma}\colon \Sigma\times [0,1]\times \R \to \Sigma\qquad\text{and}\qquad
\pi_{\CDisk}\colon \Sigma\times [0,1]\times \R \to [0,1]\times \R.
\]
The last projection map $\pi_{\CDisk}$ can be further decomposed into
its components
\[ s\colon \Sigma\times [0,1]\times \R \to [0,1] \qquad{\text{and}}\qquad
t\colon \Sigma\times[0,1]\times \R\to \R. \]

\begin{defn}
  \label{def:AdmissibleAlmostCx}
  An almost complex structure $J$ on ${\overline\Sigma}\times [0,1]\times \R$ is
  called {\em admissible} if 
  \begin{itemize}
  \item The projection $\pi_\CDisk$ is
  $J$-holomorphic.
  \item 
    $J$ preserves the 
    subspace $\ker(d_p \pi_\Sigma)\subset T_p(\Sigma\times [0,1]\times \R)$.
  \item The $\R$-action is $J$-holomorphic.
  \item The
    complex structure is split in some $\R$-invariant neighborhood of
    \[\{p_1,\dots,p_{2n}\}\times [0,1]\times \R,\] where the $p_i$ are the punctures.
    \end{itemize}
\end{defn}

We will consider $J$-holomorphic curves
\[ u\colon (S,\partial S)\to (\Sigma\times[0,1]\times \R,
(\alphas\times\{1\}\times \R)\cup(\betas\cup\{0\}\times \R)),\] with
certain asymptotic conditions. To state those, we view $\Sigma\times
[0,1]\times \R$ as having three kinds of infinities, $\Sigma\times
[0,1]\times\{+\infty\}$, $\Sigma\times [0,1]\times\{-\infty\}$, and
$\{p_i\}\times[0,1]\times \R$; the first two of these are
referred to as $+\infty$ and $-\infty$ respectively.  
Let 
\begin{equation}
  \label{eq:DefineMultUpper}
  d=g+n-1.
\end{equation}
The asymptotics
of the holomorphic curves we consider are as follows:
\begin{itemize}
  \item At $\pm\infty$, $u$ is asymptotic to a $d$-tuple of
    chords of the form $x_i\times[0,1]\times \{\pm \infty\}$,
    where $\x=\{x_i\}_{i=1}^{d}$ is an upper Heegaard state.
  \item For boundary punctures $p_i$,
    at $\{p_i\}\times [0,1]\times \R$, the curve $u$ is
    asymptotic to a collection of Reeb chords
    $\rho_i\times 1 \times t_i$; where $\rho_i$ is a Reeb chord in
    $\partial\Sigma_0=Z_1\cup\dots\cup Z_{2n}$ with endpoints on
    ${\mathbf a}=\alphas\cap Z$. These ends are called {\em east
      infinity} boundaries of $u$, and $t_i$ is called its {\em
      height}.
  \item For interior punctures $p_i$, 
    at $\{p_i\}\times [0,1]\times \R$, the curve $u$ is
    asymptotic to a collection of Reeb orbits $\{\orb_i\}\times
    s_i\times t_i$ for $0<s_i<1$, where $\orb_i$ is the simple Reeb orbit
    corresponding to the puncture $p_i$.  These ends are called {\em
      middle infinity}, and the values $t_i$ are also called their
    {\em height}.
\end{itemize}

We give the details  in this section.

\subsection{Naming Reeb chords}
\label{subsec:NamingChords}

A Reeb chord is an arc $\rho$ in some boundary component $Z_i$, with
endpoints on the intersection points between $Z_i$ and
$\alpha_{i-1}\cup\alpha_i$. As such, it has an initial point $\rho^-$
and a terminal point $\rho^+$.

When describing Reeb chords, we will use the following notation.  Each
circle boundary component $Z_i$ with $i=2,\dots,2n-1$ meets two
$\alpha$-curves $\alpha_{i-1}$ and $\alpha_{i}$, and so the boundary
circle is divided into two Reeb chords.  Label $L_i$ the chord that
goes from $\alpha_{i-1}$ to $\alpha_{i}$ with respect to the boundary
orientation of the circle, and $R_i$ the one which goes from
$\alpha_{i}$ to $\alpha_{i-1}$; see Figure~\ref{fig:ChordNames}. All
Reeb chords on $Z_i$ can thus be represented as words in the $L_i$ and
$R_i$ that alternate between the two letters. In particular, the two
Reeb chords that cover the circle once can be written as $L_i R_i$ and
$R_i L_i$; moreover, $L_i R_i$ starts and ends at $\alpha_{i-1}$ and
$R_i L_i$ starts and ends at $\alpha_{i}$. 

The boundary component $Z_1$ meets only one $\alpha$-arc, $\alpha_1$;
and hence all Reeb chords are multiples of the same Reeb chord from
$\alpha_1$ to itself.  For consistency with the above, we label this
basic Reeb chord, that covers $Z_1$ once, $R_1 L_1$ (although
independently, $R_1$ and $L_1$ do not make sense); similarly, we label
the Reeb chord that covers $Z_{2n}$ once $L_{2n} R_{2n}$.

 \begin{figure}[h]
 \centering
 \input{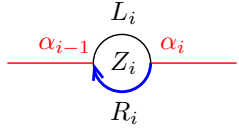}
 \caption{{\bf Names of Reeb chords.}  
   The Reeb chord $R_i$ is indicated by the oriented half circle.}
 \label{fig:ChordNames}
 \end{figure}

\subsection{Pre-flowlines}
We start with a more precise formulation of the asymptotic conditions for our holomorphic curves.

\begin{defn}
  \label{def:DecoratedSource}
  A {\em decorated source} $\Source$ is the following collection of data:
\begin{enumerate}
\item a smooth oriented surface $S$ with boundary and punctures
  (some of which can be on the boundary)
\item a labeling of each boundary puncture of $S$ by one of $+$, $-$, or $\east$
\item a labeling of each $\east$ puncture on $S$ by a Reeb chord
\item a labeling of each interior puncture by a Reeb orbit.
\end{enumerate}
Let $\East(\Source)\subset S$ denote the set of punctures marked $\east$;
$\IntPunct(\Source)\subset S$ denote the set of interior punctures,
and $\AllPunct(\Source)=\East(\Source)\cup\IntPunct(\Source)$.
\end{defn}

\begin{remark}
  In Section~\ref{subsec:CurvesAtWest}, our decorated sources will
  also include boundary punctures marked by $w$. Unlike the $\east$
  punctures, such punctures are not labelled by Reeb chords.
\end{remark}

We recall what it means for a map to be asymptotic to a given Reeb
chord $\rho$ at $q$.  Suppose that $S$ is a decorated source with a
puncture $q$ on its boundary, thought of as a point in $\cS$.
$f\colon S \to \Sigma$ is a smooth map with a continuous extension
${\overline f}\colon \cS\to \cSigma$, so that $f(q)=z_i$.  Let
\[ \CDisk^+=\{z=x+iy\in\C\mid |z|\leq 1, y\geq 0\}\] and let
$\phi\colon \CDisk^+ \to {\overline\Source}$ be a holomorphic
parameterization around the puncture $q$, i.e.  so that
$\phi(0)=q$. If the chord $\rho$ is supported in $Z_i$, consider the
identification of the corresponding puncture in $\Sigma$,
\[ \psi\colon S^1 \times [0,\infty) \to \Sigma, \] with image
$Z_i\times [0,\infty)\subset \Sigma$.
We say that {\em $u$ is asymptotic to $\rho$ at the
puncture $q$} if the family of function $[0,1]\to S^1$ indexed
by $r\in (0,1]$ specified by 
\[ \theta\mapsto \pi_{S^1} \circ \psi^{-1}\circ \pi_{\Sigma} \circ u\circ
\phi(r e^{\pi i\theta}) \] converges to $\rho$ as
${\mathcal C}^{\infty}$ functions from $[0,1]$ to $S^1$, as $r\goesto 0$.

This definition has a straightforward adaptation to 
interior punctures $q$ in $\Sigma$, where $\orb$ is some Reeb orbit.
In that case, we choose a parameterization around $q$ by the punctured disk
$\{z\in \C\setminus\{0\}\big||z|\leq 1\}$ about $q$, and 
we require that the ${\mathcal C}^{\infty}$ functions
from $S^1\to S^1$  indexed by $r\in (0,1]$ defined by
\[ \theta\mapsto \pi_{S^1} \circ \psi^{-1}\circ \pi_{\Sigma} \circ u\circ
\phi(r e^{2 \pi i\theta}) \]
converge to the given  Reeb orbit as $r\goesto 0$.

A {\em generalized upper Heegaard state} is a $d$-element subset of
points $\x=\{x_i\}_{i=1}^d$ in $\Sigma_0$, each of which is contained
in the intersection of the various $\alpha$-and $\beta$-curves,
distributed so that each $\beta$-circle contains exactly one point in
$\x$, each $\alpha$-circle contains exactly one some point in $\x$,
and no more than one point lies on any given $\alpha$-circle. Note
that a generalized upper Heegaard state can have more than one element
on the same $\alpha$-arc.

Analogous to~\cite[Section~5.2]{InvPair}, 
given a decorated source $\Source$, we consider maps as follows:
\begin{defn}
  \label{def:GenFlow}
  A {\em pre-flowline} is a map
\[ u\colon (\Source,\partial\Source)\to
(\Sigma\times[0,1]\times\R,(\alphas\times\{1\}\times \R)\cup(\betas\times\{0\}\times\R))\]
subject to the constraints:
\begin{enumerate}[label=(${\mathcal M}$-\arabic*),ref=(${\mathcal M}$-\arabic*)]
  \item 
    \label{property:First}
    $u\colon \Source\to \Sigma\times[0,1]\times \R$ is proper.
  \item 
    \label{property:ProperTwoa}
    $u$ extends to a proper map ${\overline u}\colon
    {\bSource}' \to {\overline{\Sigma}}\times[0,1]\times \R$,
    where $\bSource'$ is obtained from $\Source$ by filling in the interior
    and the $\east$
    punctures
    (so $\bSource\subset \bSource'\subset \cS$)
  \item
    \label{prop:BrCover}
    $\pi_{\CDisk}\circ u$ is a $d$-fold branched cover
    (with $d$ as in Equation~\eqref{eq:DefineMultUpper}).
  \item At each $-$-puncture $q$ of $\Source$,
    $\lim_{z\goesto q}(t\circ u)(z)=-\infty$.
  \item At each $+$-puncture $q$ of $\Source$,
    $\lim_{z\goesto q}(t\circ u)(z)=+\infty$.
  \item 
    At each $\east$-puncture $q$ of $\Source$,
    $\lim_{z\to q}(\pi_{\Sigma}\circ u)(z)$ is the Reeb
    chord $\rho$ labeling $q$. 
    The same holds for
    middle infinity punctures $q$, with limits
    to the corresponding Reeb orbit.
  \item \label{prop:FiniteEnergy}
    There are a generalized upper Heegaard states $\x$ and $\y$
    with the property that as $t\goesto - \infty$,
    $\pi_{\Sigma}\circ u$ is asympotic to $\x$  and
    as $t\goesto +\infty$, $\pi_{\Sigma}\circ u$ is asymptotic
    to $\y$.
    \setcounter{bean}{\value{enumi}}
  \item 
    \label{prop:WeakBoundaryMonotone}
    For each $t\in \R$ and $i=1,\dots,d$, 
    $u^{-1}(\beta_i\times\{0\}\times\{t\})$ consists of exactly one point.
    Similarly, for each $t\in \R$ and each $i=1,\dots,g$,
    $u^{-1}(\alpha_i^c\times\{1\}\times\{t\})$ consists of exactly one point.
\end{enumerate}
\end{defn}

By the definition of branched covers of manifolds with boundary, Condition~\ref{prop:BrCover}
ensures that $\pi_{\CDisk}\circ u$ (and indeed $t\circ u$) has no critical points over
$\partial [0,1]\times \R$.

\begin{defn}
  Let $\x$ and $\y$ be generalized upper Heegaard states, and $u$ a pre-flowline
  that connects them.
  The {\em Reeb asymptotics} of $u$ is  the ordered partition of Reeb chords
  $\vec{P}=(P_1,\dots,P_{\ell})$ appearing in the asymptotics of $u$,
  ordered by the value of $t\circ u$.
\end{defn}

Property~\ref{prop:WeakBoundaryMonotone} is called {\em weak boundary
  monotonicity}. We will often consider curves satisfying the
following additional condition, called {\em strong boundary
  monotonicity}:
\begin{enumerate}[label=(${\mathcal M}$-\arabic*s),ref=(${\mathcal M}$-\arabic*s)]
\setcounter{enumi}{\value{bean}}
\item
  \label{property:StronglyBoundaryMonotone}
  For each $t\in \R$ and $i=1,\dots 2n-1$,
  $u^{-1}(\alpha_i\times\{1\}\times \{t\})$ consists of at most
  one point.
\setcounter{bean}{\value{enumi}}  
\end{enumerate}

If $u$ satisfies this stronger condition, then $u$ is asymptotic to 
upper Heegaard states $\x$ and $\y$ over $-\infty$ and $+\infty$ respectively.

\begin{rem}
  \label{rem:FlowsAsDisks}
  Pre-flows can be thought of as Whitney disks in
  $\Sym^{g+n-1}(\cSigma)$, mapping one of the boundary arcs into the
  smooth torus $\beta_1\times\dots\times \beta_{g+n-1}$, and the other
  boundary arc $a$ into the singular space
  $\alpha^c_1\times\dots\times \alpha^c_g\times \Sym^{n-1}(I)$, where
  $I={\overline\alpha}_1\cup\dots\cup{\overline\alpha}_{2n-1}$.  The
  strong monotonicity condition guarantees that the arc $a$ in fact
  maps into a smooth part of
  $\alpha_1^c\times\dots\times\alpha_g^c\times \Sym^{n-1}(I)$.
\end{rem}

\subsection{On boundary monotonicity}
\label{subsec:BoundaryMonotone}

If $\rhos=\{\rho_1,\dots,\rho_m, \orb_1,\dots,\orb_k\}$ is a set of
Reeb chords and orbits, let $\rhos^-=\{\rho_1^-,\dots,\rho_m^-\}$ be
the multi-set (i.e. set with repeated entries) of initial points of
the Reeb chords, and $\rhos^+=\{\rho_1^+,\dots,\rho_m^+\}$ be the
multi-set of terminal points.

Strong boundary monotonicity can be
formulated in terms of the initial generalized Heegaard state $\x$ and
the Reeb asymptotics.

\begin{defn}
  \label{def:StronglyBoundaryMonotone}
  Let $\x$ be a generalized upper Heegaard state and
  $(\rhos_1,\dots,\rhos_\ell)$, a sequence of sets of Reeb chords and
  orbits.  We formulate {\em strong boundary monotonicity} of
  $(\x,\rhos_1,\dots,\rhos_\ell)$ inductively in $\ell$; at the same
  time, we also define the {\em terminal $\alpha$-set} of a strongly
  boundary monotone sequence
  $\alpha(\x,\rhos_1,\dots,\rhos_\ell)\subset \{1,\dots,2n-1\}$, as
  follows. When $\ell=0$, $(\x)$ is called strongly boundary monotone
  if $\x$ is an upper Heegaard state; and its terminal $\alpha$-set is
  defined to be the set of $i=1,\dots,2n-1$ so that $\x\cap
  \alpha_i\neq \emptyset$. (Note that this definition of $\alpha(\x)$
  coincides with the earlier definition given in
  Definition~\ref{def:UpperState}.)  For $\ell\geq 1$, we say that
  $(\x,\rhos_1,\dots,\rhos_{\ell})$ is strongly boundary monotone if
  all of the following conditions hold:
  \begin{itemize}
  \item The sequence $(\x,\rhos_1,\dots,\rhos_{\ell-1})$ is strongly boundary monotone.
  \item No two points in $\rhos_{\ell}^-$
  lies on the same $\alpha$-arc, and no two points in $\rhos_{\ell}^+$
  lies on the same $\alpha$-arc.
\item Letting $A_-$ resp. $A_+\subset \{1,\dots,2n-1\}$
  consist of all $i$ so that $\rhos_\ell^-\cap \alpha_i\neq \emptyset$ resp. 
   $\rhos_\ell^+\cap \alpha_i\neq \emptyset$, we require that
  \[ A_-\subset \alpha(\x,\rhos_1,\dots,\rhos_{\ell-1}).\]
\item   
  The set
  \[\alpha(\x,\rhos_1,\dots,\rhos_{\ell})=A_+\cup\Big(\alpha(\x,\rhos_1,\dots,\rhos_{\ell-1})\setminus A_-\Big)\]
  consists of $n-1$ elements.
\end{itemize}
\end{defn}

Note that Condition~\ref{prop:BrCover} for a pseudo-holomorphic
flowline follows from the other conditions, as follows. It is clear
that $\pi_{\CDisk}\circ u$ is a pseudo-holomorphic map from $\Source$
to $[0,1]\times \R$.  Since $\Source$ has positive and negative
punctures, 
$t\circ u$ is not constant, so $\pi_{\CDisk}\circ u$ is a branched cover.
The degree of the branching is determined by Property~\ref{prop:FiniteEnergy}.

The following is a variant of~\cite[Lemma~5.53]{InvPair}:

\begin{lemma}
  \label{lem:SBB}
  Suppose that $u$ is a weakly boundary monotone
  flowline from $\x$ to $\y$ with asymptotics specified by $\vec{\rhos}$.
  Then, $(\x,\vec{\rhos})$ is
  strongly boundary monotone if and only if $u$ is strongly boundary monotone.
\end{lemma}

\begin{proof}
  Fix $\tau\in\R$, so $(t\circ u)^{-1}(\tau)$ contains none of the punctures of $\Source$, and
  let
  \[ \alpha(u,\tau)=\{i\in\{1,\dots,2n-1\}\big| u^{-1}(\alpha_i\times
  \{1\}\times \{\tau\})\neq \emptyset\}.\] Let $q$ be some puncture on
  $\Source$ labelled by $\rho$, a Reeb chord with $\rho^-$ on
  $\alpha_i$ and $\rho^+$ on $\alpha_j$.  Since $t\circ u$ is strictly
  monotone on the arc through $q$ (in view of
  Property~\ref{prop:BrCover}), it follows that for all sufficiently
  small $\epsilon>0$, $i\in\alpha(u,t(q)-\epsilon)$ and
  $j\in\alpha(u,t(q)+\epsilon)$.  In fact, by continuity (and
  induction on $\ell$), we see that
  $\alpha(\x,\rhos_1,\dots,\rhos_{\ell-1})=\alpha(u,\tau)$ for all
  $\tau$ with $t_{\ell-1}<\tau<t_\ell$, where $t_i$ denotes the
  $t$-value of the punctures labelled by $\rhos_i$, and $t_0=-\infty$. 
  It follows easily that the two 
  formulations of boundary monotonicity coincide:
  strong boundary monotonicity on $u$ is a condition on 
  $\alpha(u,\tau)$ and strong boundary monotonicity of $(\x,\vec{\rhos})$
  is the corresponding condition on the $\alpha(\x,\rhos_1,\dots,\rhos_\ell)$.
\end{proof}

The following result will allow us to restrict attention to moduli
spaces containing only strongly boundary monotone sequences:

\begin{prop}
  \label{prop:SBD}
  Suppose that $\x$ and $\y$ are upper Heegaard states. If $u$ is a weakly,
  but not strongly boundary monotone pre-flowline representing
  $\phi\in\doms(\x,\y)$, then $b_0(\phi)$ is in the
  ideal ${\mathcal J}\subset \BlgZ$; i.e. its image in
  $\Blg$ vanishes.
\end{prop}

\begin{proof}
    Let $X\subset \R$ denote the set of points $\tau\in\R$ for which
    $u^{-1}(\alpha_i\times\{1\}\times \{\tau\})$ consists of more than one
    point for some $i$.  The set $X$ is bounded below since $u$ is asymptotic to
    an upper Heegaard state $\x$ as $t\goesto -\infty$.  Thus, it has an infimum
    $\tau_0$. There must be some puncture $p\in
    \partial\cS$ asymptotic to a Reeb chord $\rho$ that ends on
    $\alpha_i$, with $t(u(p))=\tau_0$, and another point
    $q\in\partial{\cS}$ with
    $\pi_{\Sigma}(u(q))\in{\overline\alpha}_i$ and $t(u(q))=\tau_0$.  The
    initial point of $\rho$ cannot be on $\alpha_i$, for that would
    violate boundary monotonicity for the portion of the curve 
    in
    values $<\tau_0$. 

    Given a pre-flowline $u$ and generic $\tau<\tau_0$, we construct certain pure
    algebra elements $b_\tau,c_\tau\in \BlgZ$ with $\Iup(\x)\cdot
    b_{\tau} =b_{\tau}$, and
    $b_0(\phi) = b_\tau\cdot c_\tau$.

    The algebra element $b_\tau$ for any $\tau<\tau_0$ is specified by
    its initial idempotent $\Iup(\x)$, and its weight, which is given
    by the sum of the weights of all the Reeb chords and orbits in
    $(t\circ u)^{-1}((-\infty,\tau))$. Note that $b_\tau =
    \Iup(\x)\cdot b_\tau\cdot\Idemp{\x_\tau}$, where
    \[ \x_\tau= \{1,\dots,2n-1\}\setminus \{i\mid \pi_{\Sigma}(u(1,\tau))\cap \alpha_i\neq \emptyset \}, \] 
    and $\Idemp{\x_\tau}$ denotes its corresponding idempotent.

    Assume that the initial point of $\rho$ is on $\alpha_{i-1}$, so
    that $\rho$ is of the form $L_i (R_i L_i)^k$. (The case where the
    initial point of $\rho$ is on $\alpha_{i+1}$ will follow
    similarly.) We could write
    $b_0(\phi)=b_{\tau_0-\epsilon}\cdot c_{\tau_0-\epsilon}$, where $\Iup(\x)\cdot
    b_{\tau_0-\epsilon} \cdot \Idemp{\x_{\tau_0-\epsilon}}=b_{\tau_0-\epsilon}$ and
    $i-1,i\not\in\Idemp{\x_\tau}$. 
    
    There are two cases. Either $\weight_i(c_{\tau_0-\epsilon})\geq 1$, in
    which case $c_{\tau_0-\epsilon}=U_i\cdot c'$ for some algebra
    element $c'$; so $c_{\tau_0-\epsilon}\in{\mathcal J}$.  If
    $\weight_i(c_{\tau_0-\epsilon})=1/2$, then $c_{\tau_0-\epsilon}$ moves one
    of its coordinates from $\geq i+1$ to $\leq i-1$, so once again
    $c_{\tau_0-\epsilon}\in{\mathcal J}$.
\end{proof}

\subsection{Pseudo-holomorphic flows}

\begin{defn}
  \label{def:HolFlow}
  A {\em pseudo-holomorphic flowline} is a
  pre-flow satisfying the following further hypothesis:

\begin{enumerate}[label=(${\mathcal M}$-\arabic*h),ref=(${\mathcal M}$-\arabic*h)]
\setcounter{enumi}{\value{bean}}
\item
  \label{property:Holomorphic}
  The map $u$ is $(j,J)$-holomorphic with respect to some fixed admissible
  almost-complex structure $J$ (Definition~\ref{def:AdmissibleAlmostCx})
  and complex structure $j$ on $\Source$.
\end{enumerate}
\end{defn}

Recall that $\cSigma$ is equipped with $2n$ points $z_1, \dots,
z_{2n}$.  If $u$ is a pseudo-holomorphic flow, then
$f=\pi_{\Sigma}\circ u$ is a local branched cover over $z_i$, with
branching specified by the Reeb chords.

Generalized pseudo-holomorphic flowlines can be collected into
homology classes.  Specifically, if $u$ is a pre-flowline from $\x$ to
$\y$, then the projection to $\Sigma$ induces a two-chain from $\x$ to
$\y$, in the sense of Definition~\ref{def:TwoChainFromXtoY}, obtained
from assembling the local multiplicities of $\pi_{\Sigma}\circ u$. We
call the two-chain so obtained $\Shadow(u)$, the {\em shadow} of $u$.

Fix an admissible almost-complex structure $J$.
We will consider moduli spaces $\ModFlow^B(\x,\y;\Source;\vec{P})$ of
curves from a decorated source 
asymptotic to $\x$ and $\y$ at
$-\infty$ and $+\infty$ respectively, with given shadow
$B\in \doms(\x,\y)$, and respecting the partition ${\vec{P}}$.
We will typically take
the quotient of these moduli spaces by the natural $\R$ action, to get
moduli spaces
\[  \UnparModFlow^B(\x,\y;\Source;\vec{P})=\ModFlow^B(\x,\y;\Source;\vec{P})/\R.\]

\begin{example}
  Consider the top picture in Figure~\ref{fig:AlgebraRelations},
  showing a shaded domain $B$ connecting
  upper states $\x$ and $\y$. (We have illustrated only two components of each
  Heegaard state; assume for all $i>2$, $x_i=y_i$.) The two holomorphic
  disks crossing $L_i$ and $L_{i+1}$ can be translated relative to one other
  to obtain a one-parameter family of holomorphic curves in 
  $\ModFlow^B(\x,\y,(\{L_i\},\{L_{i+1}\}))$ (which are not boundary monotone) and 
  $\ModFlow^B(\x,\y,(\{L_{i+1}\},\{L_i\}))$ (which are boundary monotone),
  and a single curve in $\ModFlow^B(\x,\y,(\{L_i,L_{i+1}\}))$.
  Note that $b_0(B)=\Iup(\x)\cdot L_{i+1} L_{i}\in{\mathcal J}$.
\end{example}

\begin{example}
  Consider the second line in Figure~\ref{fig:AlgebraRelations}. Here,
  the shaded domain supports holomorphic curves in six different
  moduli spaces, $\ModFlow^B(\x,\y,\vec{\rhos})$,
  with partitions $\vec{\rhos}=(\{L_i\},\{R_i\})$, 
  $(\{R_i\},\{L_i\}))$, $(\{L_i\cdot R_i\})$, 
  $(\{R_i\cdot L_i\})$, $(\orb_i)$, and $(\{R_i,L_i\})$.
  The first two are not boundary
  monotone, and the four are. (The first five are two-dimensional
  moduli spaces and the last one is one-dimensional.) Now, 
  $b_0(B)=\Iup(\x)\cdot U_i \cdot \Iup(\y)$ (noting that $\Iup(\x)=\Iup(\y)$,
  and $i-1,i\not\in\{1,\dots,2n-1\}\setminus \alphas(\x)$).
\end{example}

\begin{figure}[h]
 \centering
 \input{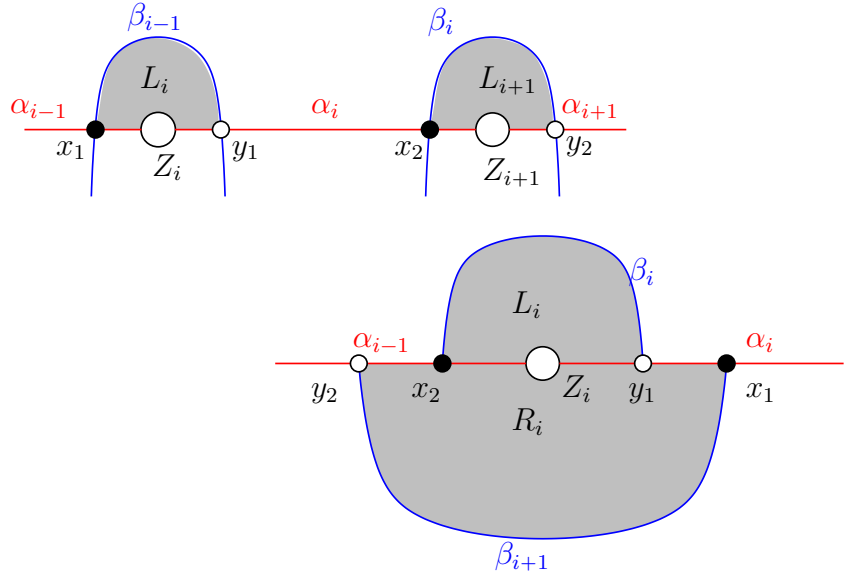}
 \caption{{\bf Some moduli spaces.}}
 \label{fig:AlgebraRelations}
\end{figure}

\subsection{The expected dimension of the moduli spaces}

\begin{defn}
  A {\em Reeb sequence} $\vec{\rho}=(\rho_1,\dots,\rho_{\ell})$ is an ordered
  sequence of Reeb orbits and chords. 
\end{defn}

A Reeb sequence $\vec{\rho}=(\rho_1,\dots,\rho_{\ell})$ gives rise to
a partition of Reeb chords, where each term consists of one
element sets, $\vec{\rhos}=(\{\rho_1\},\dots,\{\rho_{\ell}\})$. We
call such a partition {\em simple}. We will use interchangeably a Reeb
sequence with its associated simple partition; e.g.  we say that
$(\x,\vec{\rho})$ is boundary monotone if $\x$, together with the
simple partition associated to $\vec{\rho}$ is.  Similarly, given a
Reeb sequence $\vec{\rho}$, when we write
$\ModFlow^B(\x,\y,\Source,\vec{\rho})$, we mean the moduli space with
associated simple partition.

Fix $(B,{\vec\rho})$ with $B\in\doms(\x,\y)$, and $\vec\rho$ is a
sequence of Reeb chords and orbits. We say that
$(B,{\vec\rho})$ is {\em compatible} if the sum of the weights of ${\vec\rho}$
agree with the local multiplicities of $B$ around the boundary, and
$(B,{\vec\rho})$ is strongly boundary monotone.

\begin{defn}Let $|\orb(\vec\rho)|$ be the number of 
Reeb orbits appearing in 
$\vec\rho$, and $|\chords(\vec\rho)|$ be the number of chords.
If $(B,{\vec\rho})$ is compatible, we can define the {\em embedded
  Euler characteristic}, the {\em embedded index}, and the {\em
  embedded moduli space}:
\begin{align}
  \chiEmb(B)&= d + e(B)-n_\x(B)-n_\y(B) \label{eq:ChiEmb} \\
  \ind(B,\x,\y;\vec{\rho})&=e(B)+n_\x(B)+n_\y(B)
  + 2|\orb(\vec\rho)|+|\chords(\vec\rho)|
  -2 \weight_\partial(B) \label{eq:IndEmb}\\
  \ModFlow^B(\x,\y,\vec{\rho})&=
  \bigcup_{\chi(\Source)=\chiEmb(B)}
  \ModFlow^B(\x,\y,\Source;\vec{\rho}),
  \label{eq:EmbedMod}
\end{align}
where $\weight_\partial(B)$ is the total weight of $B$ at the boundary; i.e.
\[\weight_\partial(B)=\sum_{\rho_i}\weight(\rho_i).\]
\end{defn}

Note also that $\ind(B,\x,\y;\vec\rho)=\Mgr(B)$, in cases where each Reeb
chord in $\vec\rho$ has length $1/2$ and each orbit has length $1$.

\begin{rem}
  \label{rem:EulerMeasures}
  When comparing the above formulas with, for
  example,~\cite[Section~5.7.1]{InvPair}, bear in mind that
  there, the Euler measure of $B$ is defined in terms of
  the Heegaard surface with boundary ($\Sigma_0$); whereas here we
  think of it as the Euler measure in $\overline\Sigma$ instead.
\end{rem}

The following is a straightforward adaptation 
of~\cite[Proposition~5.29]{InvPair}:
\begin{prop}
  \label{prop:ExpectedDimension}
  If $\ModFlow^B(\x,\y,\Source;\vec\rho)$ is represented by some
  pseudo-holomorphic $u$, then $\chi(\Source)=\chiEmb(B)$ if and only
  if $u$ is embedded. In this case, the expected dimension of the
  moduli space is computed by $\ind(B,\x,\y;\vec\rho)$.  Moreover, if a
  strongly monotone moduli space $\ModFlow^B(\x,\y,\Source;\vec\rho)$
  has a non-embedded holomorphic representative, then its expected
  dimension is $\leq \ind(B,\x,\y;\vec\rho)-2$.
\end{prop}

\begin{proof}
  The proof is as in~\cite[Proposition~5.29]{InvPair}, which in turn follows~\cite{LipshitzCyl}.

  Suppose that $u$ is a weakly boundary monotone pre-flow. 
  Let $b_\Sigma$ be the
  ramification number of $\pi_\Sigma\circ u$, defined so that each
  interior branch point contributes $1$; each boundary branched points
  contribute $1/2$.  We think of $\cS$ as a manifold with corners, one
  for each $\pm\infty$ puncture (but the Reeb orbit punctures fill in
  to give ordinary boundary). Let $e(\cS)$ denote the corresponding
  Euler measure.

  By the
  Riemann-Hurwitz formula,
  \begin{equation}
    \label{eq:RiemannHurwitz}
    \chi(\cS)=e(\cS)+\frac{d}{2}=
    e(B)+\frac{d}{2}-b_\Sigma.
  \end{equation} Let $\tau_R(u)$ denote a copy
  of $u$ translated by $R$ units in the $\R$-direction.  
  Since $u$ is
  embedded, for small $\epsilon$, $u$ and $\tau_\epsilon(u)$ intersect
  only near branch points of $f=\pi_\Sigma\circ u$;
  and since both are
  pseudo-holomorphic, their algebraic intersection number is precisely
  $b_\Sigma$.  When $R$ is large, 
  \[u\cdot
  \tau_R(u)=n_\x(B)+n_\y(B)-\frac{d}{2}.\]  As in~\cite{LipshitzCyl},
  the intersection number $u\cap\tau_t(u)$ is independent of $t$, so
  in particular
  \begin{equation}
    \label{eq:InParticular}
  u \cdot \tau_R(u)=u\cdot \tau_\epsilon(u).
  \end{equation}
  (In~\cite{InvPair}, the intersection number is not independent of
  $t$; rather, there are possible correction terms when Reeb chords
  are slid past one another. This contribution takes the form of a
  linking number near the boundary which, in the present context
  vanishes.)  
  Thus, Equation~\eqref{eq:InParticular} shows that
  \[ b_\Sigma = n_\x(B)+n_\y(B)-\frac{d}{2}.\] Substituting this back
  into Equation~\eqref{eq:RiemannHurwitz} shows that
  \[ \chi(\cS)=d+e(B)-n_\x(B)-n_\y(B) = \chiEmb(B),\]
  in the case where $u$ is embedded.  

  When $u$ is pseudo-holomorphic,
  but not embedded, it has $s>0$ (positive) double points
 (and no negative double-points). By boundary monotonicity, do
  not occur on the boundary.  In this case
  $u\cdot\tau_\epsilon(u)=b_\Sigma + 2 s$,
  Equation~\eqref{eq:InParticular} shows that 
  $b_\Sigma+ 2s = n_\x(B)+n_\y(B)-\frac{d}{2}$, and so
  \[ \chi(\cS)=\chiEmb(B)+2s>\chiEmb(\cS).\]

  Suppose once again that $u$ is embedded. Thinking of 
  of $\cS$ as a branched cover of the disk with branching $b_\CDisk$,
  we have that
  \[ b_{\CDisk}=d-\chi(\cS)=n_\x(u)+n_\y(u)-e(B).\]
  From the point of view of the symmetric product, since $u$ is embedded,
  $b_{\CDisk}$ is the
  intersection number of the disk corresponding to $u$ with the diagonal locus
  in the symmetric product.
  In the case where the sequence $\vec{\rho}$ is empty, $\ind(u)$ is
  computed by a Maslov index, which, according a result of
  Rasmussen~\cite{RasmussenThesis}, equals $2e(\Shadow(u))+b_{\CDisk}$.
  Thus, in this case where $\ell=0$,
  \[ \ind(u)= e(B)+n_\x(B)+n_\y(B).\]
  
  In general, each Reeb chord and orbit gives a correction to the
  above formula for the index. If a Reeb chord has weight $w$,
  its correction is $1-2w$. This can be seen by looking, for example,
  at a model computation as shown on the right in
  Figure~\ref{fig:IndexModel}.  In this example, we have arranged for
  the source to be a $\Source$ is a disk with a single Reeb chord with
  weight $w$; $B$ has $e(B)=2w$; and the moduli space of
  pseudo-holomorphic representatives (modulo $\R$) is rigid, and hence
  has index $1$. For the Reeb orbit with weight $w$, a similar rigid
  solution can be found with $e=\frac{w+1}{2}$, $n_\x(B)+n_\y(B)=\frac{w-1}{2}$.
 \begin{figure}[h]
 \centering
 \input{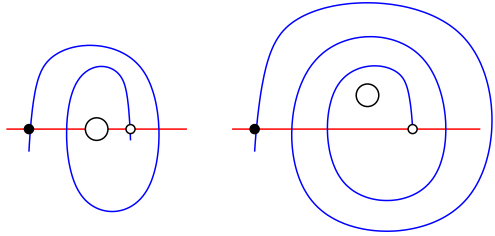}
 \caption{{\bf Model computations for the index.} }
 \label{fig:IndexModel}
 \end{figure}
  
 Finally, when $u$ is not embedded, and has $s$ double points, the
 intersection number with the diagonal corresponds to $b_\CDisk+2s$, so
 $\ind(u)=\ind(B,\x,\y;\rhos)-2s$.
\end{proof}

\subsection{Boundary degenerations}

We formalize the notion of $\beta$-boundary degenerations.

\begin{defn}
  \label{def:BoundaryDegenerations}
  Let $\HH\subset \C$ (the ``lower half plane'') consist of
  $\{x+iy\big| x,y\in \R, y\leq 0\}$, so $\partial\HH=\R$.  A {\em
    boundary degeneration} consists of the following data.
  \begin{itemize}
  \item a smooth, oriented surface $\wSource$ with boundary and punctures,
    exactly $d$ of which are on the boundary.
  \item a labelling of the interior punctures of $\wSource$ by Reeb orbits
  \item a complex structure structure $\jSource$ on $\wSource$. 
  \item 
    a smooth map                       
    $w\colon (\wSource,\partial \wSource)\to (\Sigma\times \HH,\betas\times\R\times\{0\})$.
  \item a constant $\tau\in\R$.
      \end{itemize}
  These data are required to satisfy the following conditions:
  \begin{enumerate}
  \item The map $w\colon \wSource\to\Sigma\times\HH$ is proper, and it extends to
    a proper map 
    ${\overline w}\colon
    {\wSource}' \to {\overline{\Sigma}}\times\HH$,
    where $\wSource'$ is obtained from $\wSource$ by filling in the $\east$ punctures
    (so $\wSource\subset \wSource'\subset \overline{\wSource}$).
  \item The map $\pi_\HH\circ w$ is a $d$-fold branched cover.
  \item At each interior puncture of $\wSource$, $\pi_\Sigma\circ w$ is asymptotic
    to the Reeb orbits which labels the puncture.
  \item For each $t\in \R=\partial\HH$ and $i=1,\dots,d$, $w^{-1}(\beta_i\times
    \{t\})$ consists
    of exactly one point.
  \item The map $w$ is holomorphic, with respect to $\jSource$ on 
    the domain and the split complex structure
    $j\times j_\HH$ on the range.
  \end{enumerate}
\end{defn}

The map $w\colon \wSource\to \Sigma\times \HH$
extends to a map
\[ {\overline{w}}\colon {\overline\wSource}\to \cSigma\times {\overline\HH}.\]

\begin{defn}
  \label{def:evB}
  Think of the boundary punctures in $\wSource$ as a $d$-tuple of points
  in ${\overline\wSource}$.  The images of these points under
  ${\overline{\pi_\HH\circ w}}$ gives a point, denoted $\evB(w)$, in
  $\beta_1\times\dots\times\beta_d=\Tb$.
\end{defn}

The boundary degeneration induces a map $u\colon
(\wSource,\partial\wSource)\to (\Sigma\times [0,1]\times
\R,\betas\times \R)$, so that $\pi_\Sigma\circ u=\pi_\Sigma \circ w$,
$s\circ u = 0$, and $t\circ u\equiv \tau$.

Such a boundary degeneration $w$ has a shadow which is a two-chain $B$
which is a formal linear combination of the components of
$\Sigma\setminus\betas$.

For $\{r,s\}\in\Matching$, there is a two-chain $\Brs$ corresponding
to the component of $\Sigma\setminus \betas$ containing $z_r$ and
$z_s$.  Let $\UnparModDeg_j^{\Brs}$ denote the moduli space of
boundary degenerations with shadow $\Brs$ as above modulo the (real
two-dimensional) group of automorphisms of $\HH$.

\begin{prop}
  \label{prop:SmoothBoundaryDeg}
  For  a generic complex structure $j$ on $\Sigma$, the moduli
  space $\ModDeg_j^{\Brs}$ of boundary degeneration
  is a smooth manifold of dimension $d=g+n-1$.
\end{prop}
\begin{proof}
  This follows from the fact that the corresponding moduli space is
  somewhere injective near the boundary; see~\cite[Proposition~3.9 and
  Lemma~3.10]{LipshitzCyl}; see also~\cite[Proposition~3.14]{HolDisk}
  and~\cite{OhBoundary}.
\end{proof}

There is an evaluation map $\evB\colon\ModDeg_j^{{\mathcal
    B}_{\{r,s\}}} \to \Tb$. Given $\x\in\Tb$, let
\[ \ModDeg_j^{\Brs}(\x)=(\evB)^{-1}(\x).\]

The following result will  be important:

\begin{lem}
  \label{lem:BoundaryDegenerationsDegree1}
  The evaluation map $\evB\colon \ModDeg_j^{{\mathcal B}_{\{r,s\}}}\to \Tb$ has odd degree.
\end{lem} 
\begin{proof}
  Consider first the case where $g=0$. In this case, $\evB$ is clearly a
  homeomorphism: the boundary degeneration consist of $d-1$ constant
  disks, and one disk that maps to $\Brs$ with degree one. We can
  move the constants around
  on at $d-1$ dimensional portion of $\Tb$ and reparameterize the remaining component
  to obtain the claimed homeomorphism. In cases where $g>0$, gluing spheres gives the desired
  degree statement~\cite[Section~12]{LipshitzCyl}; see also~\cite[Section~10]{HolDisk}.
\end{proof}

\subsection{Regular moduli spaces of embedded curves}

\begin{defn}
  \label{def:Typical}
  A Reeb sequence $(\rho_1,\dots,\rho_\ell)$ is called {\em{typical}}
  if each chord $\rho_i$ appearing in the sequence covers half of some
  boundary circle (i.e. it is one of $L_i$ or $R_i$), and each Reeb
  orbit covers some boundary circle exactly once.
\end{defn}

\begin{thm}
  \label{thm:GeneralPosition}
  Choose a generic $J$.  Let $B$ be a shadow with $\Mgr(B)\leq 2$ and
  $\vec{\rho}$ is a typical Reeb sequence. Then, the moduli spaces
  $\ModFlow^B(\x,\y,\Source;\vec{\rho})$ is a smooth manifold of
  dimension given by $\Mgr(B)$.  Moreover, if $\Mgr(B)\leq 1$ and
  $\vec{\rho}$ is not typical, then
  $\ModFlow^B(\x,\y,\Source;\vec{\rho})$ is empty.
\end{thm}

\begin{proof}
  Consider first the moduli space $\ModFlow^B(\x,\y,\Source)$, where
  the order of the punctures is left unspecified. Standard arguments
  show that, for generic $J$, the corresponding moduli space is a
  manifold transversely cut out by the $\dbar$ operator;
  see~\cite[Proposition~5.6]{InvPair}. Moreover, the evaluation map
  at the punctures gives a map from the moduli space to $\R^{E}$,
  where here $E=E(\Source)$ denotes the number of east punctures of
  $\Source$. There is a dense set of $J$ for which the evaluation map
  is transverse to the various diagonals in $\R^{E}$, so that their
  preimages give submanifolds of $\ModFlow^B(\x,\y,\Source)$; the
  moduli spaces $\ModFlow^B(\x,\y,\Source;\vec{\rho})$ are the
  complement of these submanifolds.

  This transversality argument shows that
  $\ModFlow^B(\x,\y,\Source;\vec\rho)$ is a smooth manifold with
  dimension (as computed in
  Proposition~\ref{prop:ExpectedDimension}) given by
  \begin{align*}
    \ind(B,\x,\y;\vec{\rho})&=e(B)+n_\x(B)+n_\y(B)
  + 2|\orb(\vec\rho)|+|\chords(\vec\rho)|
  -2 \sum_{\rho_i}\weight(\rho_i)\\
  &= \Mgr(B)-\sum_{\orb\in \orb(\vec\rho)} (2\weight(\orb)-2) -
  \sum_{\rho\in \chords(\vec\rho)} (2\weight(\rho)-1) \\
  & \leq \Mgr(B), 
  \end{align*}
  with equality exactly when each orbit has length one and each chord has length $1/2$.
\end{proof}

\begin{rem}
  The above general position can be seen from the point of view of the
  symmetric product as follows.  As in Remark~\ref{rem:FlowsAsDisks},
  we think of our pseudo-holomorphic curves as giving
  pseudo-holomorphic disks in $\Sym^d(\overline\Sigma)$.  The subspace
  $\alpha_1^c\times\dots\times \alpha_g^c\times \Sym^{n-1}(I)$ is
  equipped with codimension one walls of the form $\{p_i\}\times\Sym^{n-2}(I)$,
  with  $p_i\in\cSigma$.
  Typical sequences arise for holomorphic disks that are transverse to
  these walls; chords or orbits with larger weight occur when
  the curves have higher order contact with the submanifolds.
\end{rem}

\subsection{Ends of one-dimensional moduli spaces}

We will consider ends of one-dimensional moduli spaces
$\UnparModFlow^B(\x,\y,\Source;\vec{P})$. These will include ends that
consist of two-story buildings. Another kind of end consists of the
formation of an ``orbit curve'' at east infinity.  This occurs when
some Reeb orbit constraint $\orb_i$ in $\vec{P}$ slides off to $s=1$.
(See Figure~\ref{fig:OrbitEnd}.)  Other ends occur when two consecutive
parts in $\vec{P}$ collide. We call these {\em collision ends}.

In formulating these ends, we use the following terminology from~\cite{InvPair}:

\begin{defn}
  \label{def:ComposableChords}
  An ordered pair of Reeb chords $\rho$ and $\sigma$ are called {\em
    weakly composable} if $\rho^+$ and $\sigma^-$ are contained on the same $\alpha$-arc.  Moreover, if $\rho$ and $\sigma$ weakly
  composable, they are further called {\em strongly composable} if
  $\rho^+=\sigma^-$. If $\rho$ and $\sigma$
  are strongly composable we can join them to get a new Reeb chord
  $\rho\uplus\sigma$.
\end{defn}

Thus, $L_i$ and $L_{i+1}$ are weakly but not strongly composible,
while $L_i$ and $R_i$ are strongly composable.

\begin{defn}
  A collision end is called {\em invisible} if $\rho_i$ and
  $\rho_{i+1}$ are the same Reeb orbit. Otherwise, it is called {\em visible}.
\end{defn}

\begin{thm}
  \label{thm:DEnds}
  Let $\Hup$ be an upper diagram and $\Matching$ its associated
  matching.  Fix upper Heegaard states $\x$ and $\y$ and a typical
  Reeb sequence $\vec{\rho}$.  Suppose moreover that $(\x,\vec{\rho})$
  is strongly boundary monotone, and
  $B\in\doms(\x,\y)$ has vanishing local multiplicity somewhere.
  Fix $\Source$ and $\vec{\rho}$ so that
  $\ind(B,\x,\y;\Source,\vec{\rho})=2$. Let
  $\UnparModFlow=\UnparModFlow^B(\x,\y;\Source;{\vec{\rho}})$. The total
  number of ends of $\UnparModFlow$ of the following types are even in
  number:
  \begin{enumerate}[label=(DE-\arabic*),ref=(DE-\arabic*)]
  \item
    \label{typeDE:2StoryEnd} 
    Two-story ends, which are of the form
    \[ \UnparModFlow^{B_1}(\x,\w;\Source_1;\rho_1,\dots,\rho_i)\times 
    \UnparModFlow^{B_2}(\w,\y;\Source_2;\rho_{i+1},\dots,\rho_{\ell}),\]
    taken over all upper Heegaard states $\w$ and choices of $\Source_1$ and $\Source_2$ so that
    $\Source_1\natural \Source_2=\Source$, and $B_1\natural B_2=B$.
  \item 
    \label{typeDE:OrbitEnd}Orbit curve ends,
    of the form $\UnparModFlow^B(\x,\y,\Source;\rho_1,\dots,\rho_{i-1},\longchord_j,\rho_{i+1},\dots,\rho_{\ell})$,
    where some Reeb orbit component $\rho_i=\orb_j$ slides off to $s=1$
    and is replaced by a Reeb chord $\longchord_j$ that
    covers $Z_j$ with multiplicity $1$. (When $j\neq 1$ or $2n$, 
    there are two possible choices:  $\longchord_j=R_j L_j$ or $L_j R_j$.)
  \item
    \label{typeDE:CollisionWithOrbit} Visible collision ends where at least one of $\rho_i$ or $\rho_{i+1}$ is a Reeb orbit.
  \item 
    \label{typeDE:ChordChord}
    Collision ends where $\rho_i$ and $\rho_{i+1}$ are Reeb chords,
    where one of the two conditions are satisfied:
    $\rho_i$ and $\rho_{i+1}$ are not weakly composable, or they are strongly composable.
  \item 
    \label{typeDE:BoundaryDegeneration}
    An end consisting of a boundary degeneration that meets a constant
    flowline.  In this special case, $\vec{\rho}$ consists of exactly two
    constraints, which are matched Reeb orbits, and $\x=\y$. Moreover,
    if $\{r,s\}\in\Matching$, then the number of boundary degeneration
    ends of the union
    \[ \ModFlow(\x,\x,\{e_r\},\{e_s\})\cup\ModFlow(\x,\x,\{e_s\},\{e_r\})\]
    is odd.
  \end{enumerate}
\end{thm}

\begin{figure}[h]
 \centering
 \input{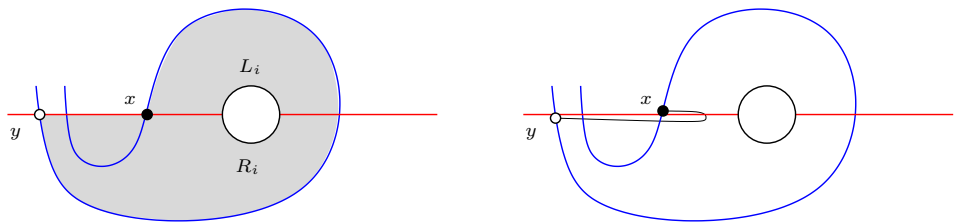}
 \caption{{\bf Orbit curve end.}  Consider the moduli space
   $\UnparModFlow^B(x,y,\{\orb_i\})$, where $B$ is shaded on the left.
   This one-dimensional moduli space has an end which is a two-story
   building, and another which is an orbit curve end with $\longchord_i=L_i R_i$.}
 \label{fig:OrbitEnd}
\end{figure}

The above result is proved in Section~\ref{sec:CurvesA}.

\section{Type $D$ modules}
\label{sec:TypeD}

Let $\Hup$ be an upper diagram, and let $\Matching$ be its
corresponding matching. Let $W$ be the one-manifold associated to
$\Matching$ (as in Definition~\ref{def:AssociatedW}), and fix an
orientation on $W$; i.e. a choice of preferred $i$ for each
$\{i,j\}\in\Matching$. (This latter data is needed to define the
$\Q^n$-valued Alexander grading.)

Choose an admissible almost-complex structure $J$ over
$\Hup$ (as in Definition~\ref{def:AdmissibleAlmostCx}). 
Let $\Dmod(\Hup)$ be the vector
space generated over $\Field$ by the upper states. 
For $\alpha(\x)$ as in Definition~\ref{def:UpperState}, let 
\begin{equation}
  \label{eq:IdempOfUpper}
  \Iup(\x)=\Idemp{\{1,\dots,2n-1\}\setminus \alpha(\x)}
\end{equation}
thought of as an idempotent in the algebra $\Clg(n)$.  The left action of
the idempotent subalgebra of $\Clg(n)$  is specified by the condition that
$\Iup(\x)\cdot \x = \x$.

Consider the functions $\Mgr\colon \States(\Hup)\to \Z$ and
$\Agr\colon \States(\Hup)\to \OneHalf \Z^n$ defined in
Equations~\eqref{eq:DefMgr} and~\ref{eq:DefAgr} above. These endow
$\Dmod(\Hup)$ with a $\Z$-valued $\Delta$-grading and a $\OneHalf
\Z^n$-valued Alexander grading, denoted $\Agr$.

Let $\x,\y$ be upper Heegaard states, and $B\in\doms(\x,\y)$.  Let
$\ModFlow^B(\x,\y)$ be the union of $\ModFlow^B(\x,\y;\vec\rho)$ (as
in Equation~\eqref{eq:EmbedMod}), taken over all typical sequences
$\vec\rho$ of Reeb chords and orbits that are compatible with $B$.
Recall that by Theorem~\ref{thm:GeneralPosition}, $\ModFlow^B(\x,\y)$
has expected dimension $\Mgr(B)$ (independent of the typical sequence
$\vec\rho$).

Define the operation
\[ \delta^1\colon \Dmod(\Hup)\to \Clg(n)\otimes \Dmod(\Hup)  \]
by
\begin{equation}
  \label{eq:TypeDOperation}
  \delta^1(\x) = \sum_{\{\y\in \States, B\in\doms(\x,\y)\big|
  \Mgr(B)=1\}} \#\UnparModFlow^B(\x,\y) \cdot \bOut(B)\otimes  \y,
\end{equation}
 where
$\bOut(B)\in\Clg(n)\subset\Blg(2n,n)$ is as in
Section~\ref{sec:Shadows}. 
\begin{prop}
  The sum appearing on the right of Equation~\eqref{eq:TypeDOperation}
  is finite. 
\end{prop}

\begin{proof}
  The non-zero terms arise from $B\in\doms(\x,\y)$ with $\Mgr(B)=1$,
  which have pseudo-holomoprhic representatitves. We must show that there are only finitely many
  such $B$. To this end,
  fix some $B_0\in\doms(\x,\y)$.
  Our hypothesis on upper diagrams (Property~\ref{UD:NoPerDom})
  ensures that for any other $B\in\doms(\x,\y)$, 
  we can find integers $n_{\{r,s\}}$ so that
  \[ B=B_0 + \sum_{\{r,s\}\in\Matching} n_{\{r,s\}} \cdot \Brs.\] If $B$ has  a holomorphic representative, then
  all of its local multiplicities must be non-negative, giving a lower bound on each
  $n_{\{r,s\}}$. Since
  \[ \Mgr(B)=\Mgr(B_0)+ 2 \sum_{\{r,s\}\in\Matching} n_{\{r,s\}},\]
  condition that $\Mgr(B)=1$ also places an upper bound on all the $n_{\{r,s\}}$, proving the desired
  finiteness statement.
\end{proof}

\begin{prop}
  The map $\delta^1$ respects (relative) gradings in the following sense:
  if $b\otimes \y$ appears with non-zero multiplicity in $\delta^1(\x)$,
  then 
  \begin{align*}
    \Mgr(\x)-1=\Delta(b)+\Mgr(\y) \\
    \Agr(\x)=\Agr(b)+\Agr(\y).
  \end{align*}
\end{prop}

\begin{proof}
  The above equations follow at once from
  Equations~\eqref{eq:DefMgr},~\eqref{eq:DefAgr}, and the definitions
  of the gradings on the algebra, Equations~\eqref{eq:DefDeltaAlg}
  and~\eqref{eq:DefAgrAlg}.
\end{proof}

\begin{prop}
  \label{prop:CurvedTypeD}
  The map $\delta^1$ defined above
  satisfies a curved type $D$ structure relation,
  with curvature $\mu_0=\sum_{\{r,s\}\in\Matching} U_r\cdot U_s$.
\end{prop}

\begin{proof}
  Choose $\x$ and $\z$ so that there is some
  $B\in\doms(\x,\y)$ with $\Mgr(B)=2$.
  Consider the ends of the moduli spaces
  $\UnparModFlow^B(\x,\z,\Source;\vec{P})$, where we take
  the union over all choices of typical Reeb sequences
  $[\vec{P}]=(\rho_1,\dots,\rho_\ell)$ and all choices of source
  $\Source$.  We can assume without loss of generality that the homology
  class $B$ of the curve does not cover all of $\Sigma$; for otherwise,
  the corresponding term in 
  $(\mu_2\otimes \Id_{\Dmod(\Hup)})\circ (\Id_{\Clg}\otimes \delta^1)\circ\delta^1$ vanishes;
  there are no non-zero algebra elements with positive weight everywhere.

  These moduli spaces are one-dimensional according to
  Theorem~\ref{thm:GeneralPosition}.  Next we appeal to
  Theorem~\ref{thm:DEnds}, observing cancellations of the counts of
  various ends cancel.  The various collision ends where at least one
  of $\rho_i$ or $\rho_{i+1}$ is a Reeb orbit cancel with one
  another. Specifically, consider a typical Reeb sequence
  $(\rho_1,\dots,\rho_i,\rho_{i+1},\dots,\rho_\ell)$ with a visible
  collision end for $\rho_i$ and $\rho_{i+1}$, where at least one of
  the two is a Reeb orbit (i.e. in the terminology of
  Theorem~\ref{thm:DEnds}, this is an end of
  type~\ref{typeDE:CollisionWithOrbit}). These ends correspond to the
  corresponding collision end of the moduli space where the order of
  $\rho_i$ and $\rho_{i+1}$ are permuted.  (This is a different moduli
  space, since the collision is visible.)  Similarly, if $\rho_i$ and
  $\rho_{i+1}$ are two Reeb chords that are not weakly composable (a
  subcase of~\ref{typeDE:ChordChord}), we can permute them to get
  another moduli space with a corresponding end.  When $\rho_i$ and
  $\rho_{i+1}$ are strongly composable chords, the corresponding ends
  cancel against orbit ends (Type~\ref{typeDE:OrbitEnd}).

  The two types of ends left unaccounted for are the two-story ends
  (Type~\ref{typeDE:2StoryEnd}) and the boundary degeneration ends
  (Type~\ref{typeDE:BoundaryDegeneration}). The fact there is an even
  number of remaining ends gives the type $D$ structure relation.
\end{proof}

The above three propositions can be summarized as follows: the vector
space $\Dmod(\Hup)$, with differential as in
Equation~\eqref{eq:TypeDOperation} is a curved type $D$ structure,
with a homological grading induced by $\Mgr$ and Alexander gradings
$\Agr$.

The following invariance property of this curved type $D$ structure
will be important for us:

\begin{prop}
  If $J_0$ and $J_1$ are any two generic almost-complex structures,
  there is a type $D$ structure quasi-isomorphism of graded type $D$ structures
  \[\Dmod(\Hup,J_0)\simeq \Dmod(\Hup,J_1).\]
\end{prop}

\begin{proof}
  As usual, one must show that a path $\{J_t\}_{t\in[0,1]}$ induces a
  type $D$ morphism.  This is done via the straightforward
  modification of Theorem~\ref{thm:DEnds} to varying $\{J_t\}$, and
  considering moduli spaces between generators $\x$ and $\y$ where the
  with index $1$.
  Specifically, fix a generic one-parameter family
  $\{J_t\}_{t\in[0,1]}$ of almost-complex structures, and consider the
  moduli space $\ModFlow^B(\x,\y;\{J_t\})$ of $J_t$-holomorphic curves,
  where $t$ is the second parameter in the projection to $[0,1]\times
  \R$ (parameterized by pairs $(s,t)$).  For such moduli spaces, the
  ends of Type~\ref{typeDE:BoundaryDegeneration} do not exist (such
  moduli spaces connect a generator to itself; and indeed they count
  curves with index $2$). With this remark in place, the above proof
  of Proposition~\ref{prop:CurvedTypeD} shows that there is an even
  number of two-story ends. This immediately shows that the map
  \[ h^1(\x)=\sum_{\y\in\States,B\in\doms(\x,\y)}
  \bOut(B)\otimes\#\ModFlow^B(\x,\y;\{J_t\})\cdot \y.\] gives a type
  $D$ morphism
  \[ h^1\colon \Dmod(\Hup,J_0)\to \Clg(n)\otimes \Dmod(\Hup,J_1). \]

  Homotopies of paths of complex structures induce homotopies of type
  $D$ morphisms and the identity path induces the identity map as
  usual; so it follows that $h^1$ is a type $D$ quasi-isomorphism.

  Verifying that the maps are graded is straightforward.
\end{proof}

\begin{remark}
  \label{rem:NoInvariance}
  The above proposition shows that the quasi-isomorphism type of the
  type $D$ structure of an upper Heegaard diagram is independent of
  the analytical choices (of almost-complex structures) made. One
  could aim for more invariance.  One could think of an upper Heegaard
  diagram as in fact representing an upper knot diagram, and then try
  to prove dependence of the type $D$ structure only on the upper knot
  diagram; and indeed one could try to show that it is an invariant of
  the tangle represented by the diagram. We do not pursue this route,
  in the interest of minimizing the road to
  Theorem~\ref{thm:MainTheorem}.
\end{remark}

\subsection{Examples}
\label{subsec:ExampleTypeDs}

\begin{figure}[h] 
  \centering \input{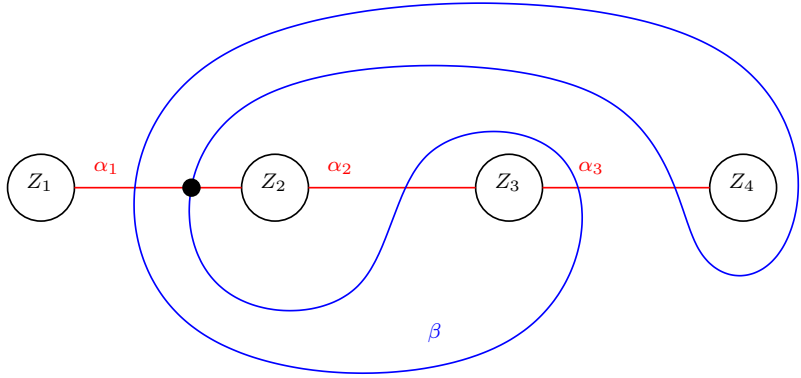} 
  \caption{{\bf Upper Heegaard
      diagram.} This diagram has five Heegaard states, corresponding
    to the five intersection points of $\beta$ with
    $\alpha_1\cup\alpha_2\cup\alpha_3$. One of these ($x_2$ in the text)
    is indicated in
    black.} \label{fig:UpperHeegTref} \end{figure}

We start with a simple example. Consider the upper Heegaard diagram
from Figure~\ref{fig:UpperHeegTref}.  This has five Heegaard states, which
we label from left to right, $x_1$, $x_2$, $t$, $y_1$, $y_2$. Since we
have $d=1$, the curve counting is straightforward, and we find that
the type $D$ structure is as indicated in the following diagram:

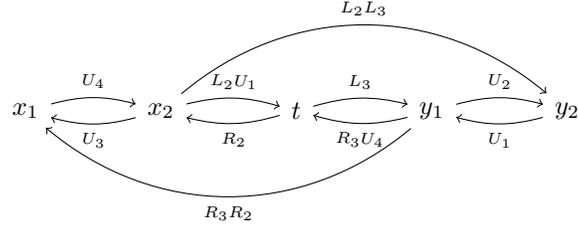
\begin{equation}
  \label{eq:Dmod}
  \begin{tikzpicture}[scale=1.8]
  \node at (0,0) (x1) {$x_1$}   ;
  \node at (1,0) (x2) {$x_2$} ;
  \node at (2,0) (t) {$t$} ;
  \node at (3,0) (y1) {$y_1$} ;
  \node at (4,0) (y2){$y_2$} ;
    \draw[->] (x1) [bend left=15] to node[above,sloped] {\tiny{$U_4$}}  (x2)  ;
    \draw[->] (x2) [bend left=15] to node[below,sloped] {\tiny{$U_3$}}  (x1)  ;
    \draw[->] (x2) [bend left=15] to node[above,sloped] {\tiny{$L_2 U_1$}}  (t)  ;
    \draw[->] (t) [bend left=15] to node[below,sloped] {\tiny{$R_2$}}  (x2)  ;
    \draw[->] (t) [bend left=15] to node[above,sloped] {\tiny{$L_3$}}  (y1)  ;
    \draw[->] (y1) [bend left=15] to node[below,sloped] {\tiny{$R_3 U_4$}}  (t)  ;
    \draw[->] (y1) [bend left=15] to node[above,sloped] {\tiny{$U_2$}}  (y2)  ;
    \draw[->] (y2) [bend left=15] to node[below,sloped] {\tiny{$U_1$}}  (y1)  ;
    \draw[->] (x2) [bend left=40] to node[above,sloped] {\tiny{$L_2 L_3$}} (y2) ;
    \draw[->] (y1) [bend left=40] to node[below,sloped] {\tiny{$R_3 R_2$}} (x1) ;
  \end{tikzpicture}
\end{equation}

We give a (very simple) family of examples which will play a
fundamental role in our future computations.  

  For $n>1$, consider the
  matching $M=\{\{1,2\},\{3,4\},\dots\{2i-1,2i\},\dots,\{2n-1,2n\}\}$ on
  $\{1,\dots,2n\}$ where $2i-1$ is matched with $2i$ for
  $i=1,\dots,n$. Let ${\mathbf s}=\{2i-1\}_{i=1}^{n}$, and
  $\mu_0=\sum_{i=1}^n U_{2i-1} U_{2i}$.  The type $d$ structure
  with a single generator $\x$ satisfying
  $\Idemp{\mathbf s}\cdot \x = \x$ and $\delta^1(\x)=0$ can be viewed
  as a curved module over $\Clg(n)$ since $\Idemp{\mathbf s}\cdot
  \mu_0=0$. We write this type $D$ structure $\lsup{\cClg}k$,

\begin{lemma}
  \label{lem:StandardTypeD}
For any $n>1$, let $\Hup$ denote the standard upper diagram
with $2n$ local maxima (pictured in
Figure~\ref{fig:StandardUpperDiagram}; after deleting 
$\beta_4$ and the basepoints $\wpt$ and $\zpt$).
There is an identification of type $D$ structures
\[ \Dmod(\Hup)\cong~\lsup{\cClg}k. \]
\end{lemma}

\begin{proof}
  The diagram has exactly one Heegaard state $\x$; it has
  $\Idemp{{\mathbf s}}\cdot \x = \x$; and there are no holomorphic
  curves that can induce $\delta^1$-actions.
\end{proof}

\newcommand\ClosedModFlow{\overline\ModFlow}
\newcommand\ConstrPack{\mathbf R}
\newcommand\Packets{\mathcal P}
\newcommand\inv{\mathrm{inv}}
\section{More holomorphic curves}
\label{sec:CurvesA}

We describe here the pseudo-holomorphic curves used in the
constructions of the type $A$ modules for a given lower diagram
$\Hdown$.  The moduli spaces we consider here are similar to the ones
considered in Section~\ref{sec:CurvesD}, except that now the limiting
values as $t\goesto \pm \infty$ are lower (rather than upper) Heegaard
states; moreover, the partitions we consider here
$\Partition=(\rhos_1,\dots,\rhos_\ell)$ of the east and interior
punctures need not be simple
(c.f. Definition~\ref{def:DecoratedSource}): each term in the
partition $\rhos_i$ is some (non-empty) subset of chords and orbits.
More formally:

\begin{defn}
  \label{def:ConstraintPacket}
  a {\em constraint packet} $\rhos$ is a pair consisting
  of a set of Reeb orbits, denoted $\orbits(\rhos)$; and a set of Reeb
  chords, denoted $\chords(\rhos)$.  
\end{defn}

Thus, in our curve counting, for
the partition $(\rhos_1,\dots,\rhos_{\ell})$, each term $\rhos_i$ is a
constraint packet.  Sometimes, we find it convenient to generalize
this slightly: a {\em generlized constraint packet}, consisting of
{\em multi-sets} of Reeb orbits and chords (i.e. the same chord or
orbit can occur with positive multiplicity).

The differential and, indeed, all the module actions (of sequences of
algebra elements on the module) will count pseudo-holomorphic curves
with constraint packets specified by algebra elements; and these
constraint packets will have a rather special form
(cf. Definition~\ref{def:CompatiblePacket}).  In particular, each
constraint packet in the algebra action definition will contain at
most one Reeb orbit.

We will prove an $\Ainfty$ relation for modules, which will involve
analyzing ends of one-dimensional moduli spaces. Two story building
degenerations correspond to terms in the $\Ainfty$ relation involving
two applications of the module actions.  In~\cite{InvPair}, these
moduli spaces have another kind of end, called {\em join curve ends},
which correspond to the differential in the bordered algebra.  By
contrast, in the present case, join curve ends of the various moduli
spaces which we consider cancel in pairs. (One could construct a
larger algebra than the one considered here, equipped with a
differential, so that the join curve ends of moduli spaces correspond
to terms in the differential of the algebra. The present algebra can
be thought of as the homology of this larger algebra.)  See
Figure~\ref{fig:JoinEnd} for an illustration.  As~\cite{InvPair},
there are also {\em collision ends}, which correspond to
multiplication in the algebra. There is another new kind of end,
corresponding to the formation of $\beta$-boundary degenerations. Some
of these cancel against orbit curve ends; the remaining orbit ends are
accounted for by the curvature of our algebra.

Our goal here is to formulate the moduli spaces we consider precisely,
and to describe the ends of one-dimensional moduli spaces. The
algebraic interpretations of the counts of these ends described above
will be given in detail in Section~\ref{sec:TypeA}.

\subsection{Pseudo-holomorphic flows in lower diagrams}

We will need to name Reeb chords for lower diagrams as shown in Figure~\ref{fig:ChordNamesA}; i.e. switched from the conventions from
the type $D$ side; c.f. Figure~\ref{fig:ChordNames}). This will be
useful for the gluing of diagrams. As a point of comparison: the chord
labelled $L_i$ on the $D$ side has initial point on $\alpha_{i-1}$ and
terminal point on $\alpha_{i}$; while the chord labelled $L_i$ on the
$A$ side has initial point on $\alpha_i$ and terminal point on
$\alpha_{i-1}$. 

\begin{figure}[h]
 \centering
 \input{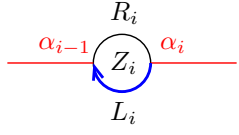}
 \caption{{\bf Names of Reeb chords for lower diagrams.}  The Reeb
   chord $L_i$ is indicated by the oriented half circle.}
 \label{fig:ChordNamesA}
\end{figure}

Strong boundary monotonicity is a closed condition, in the following sense:

\begin{lemma}
  \label{lem:StrongMonotoneClosed}
  Suppose that $u$ is a pseudo-holomorphic flowline which appears in
  the Gromov limit of a sequence of strongly boundary boundary
  monotone curves. Then $u$, is strongly boundary monotone, as well.
\end{lemma}

\begin{proof}
  This follows exactly as in~\cite[Lemma~5.55]{InvPair}: strong
  boundary monotonicity can be phrased in terms of the monotonicity of
  the function $t\circ u$ restricted to each boundary component.
  Monotone functions can limit to constant functions, but in that
  case, there is an $\alpha$-boundary degeneration component, which in
  turn is ruled out by Condition~\ref{prop:CurvesNotToCoverA}.
\end{proof}

Given a lower diagram $(\Sigma,\alphas,\betas,w,z)$,
the operations on the type $A$ module are defined by counting
$J$-holomorphic curves
\[ u\colon (S,\partial S)\to (\Sigma\times[0,1]\times \R,
(\alphas\times\{1\}\times \R)\cup(\betas\cup\{0\}\times \R),\] 
subject to the constraints~\ref{property:First}-\ref{prop:FiniteEnergy}
and~\ref{prop:WeakBoundaryMonotone}, with the understanding that presently, we set
\begin{equation}
  \label{eq:dUp}
  d=g+n. 
\end{equation}

We will make the following
further hypothesis:
\begin{enumerate}[label=(A${\mathcal M}$-\arabic*),ref=(A${\mathcal M}$-\arabic*)]
\setcounter{enumi}{\value{bean}}
\item \label{prop:CurvesNotToCoverA}
  At least one of $n_\wpt(u)$ or $n_\zpt(u)$ vanishes.
\end{enumerate}

The strong boundary monotonicity condition on such a map $u$ looks
exactly as it did earlier
(c.f. Condition~\ref{property:StronglyBoundaryMonotone}).  As in
Section~\ref{sec:CurvesD}, this condition can be formulated in terms
of combinatorial data. Specifically, given a lower generator $\x$ and
a sequence of sets of Reeb chords and orbits
$(\rhos_1,\dots,\rhos_\ell)$, define the terminal $\alpha$-set and
strong boundary monotonicity of $(\x,\rhos_1,\dots,\rhos_\ell)$ as in
Definition~\ref{def:StronglyBoundaryMonotone}, with the understanding
that now $\x$ is a lower, rather than upper, Heegaard state.

Given lower Heegaard states $\x,\y$, a marked source $\Source$ a
strongly boundary monotone sequence
$\rhos_1,\dots,\rhos_\ell=\vec{\rhos}$, we can form moduli spaces of
such pseudo-holomorphic flows, denoted
$\UnparModFlow(\x,\y,\Source,\rhos_1,\dots,\rhos_\ell)$
or simply $\UnparModFlow(\x,\y,\Source;\vec{\rhos})$.

Lemma~\ref{lem:SBB} has the following straightforward adaptation:

\begin{lemma}
  \label{lem:SBA}
  If $(\x,\vec{\rhos})$ is strongly boundary monotone and $u\in
  \ModFlow^B(\x,\y,\Source,\vec{\rhos})$, then $u$ is strongly boundary
  monotone. Conversely, if $u\in\ModFlow^B(\x,\y,\Source,\vec{\rhos})$
  is strongly boundary monotone, then $(\x,\vec{\rhos})$ is strongly
  boundary monotone.\qed
\end{lemma}

Let $\rhos$ be a constraint packet consisting entirely of chords (i.e.
it contains no orbits. Think of each $\rho\in \rhos$ as a path in
$[0,1]\times Z$.  Let $\inv(\rhos)$ denote the minimal number of
crossings between the various chords.  Given any point $p\in Z$, let
$m([\rhos],p)$ denote the multiplicity with which the Reeb chords in
$\rhos$ cover $p$, with the convention that if $p=\alpha_i\cap Z_i$ or
$\alpha_{i-1}\cap Z_i$, then $m([\rhos],p)$ is the average of
$m([\rhos],p')$ over the two nearby choices $p'\in Z$ on both sides of
$p$.  For example, if $p=\alpha_i\cap Z_i$ or $\alpha_{i-1}\cap Z_i$,
then $m(L_i,p)=1/2$.

Let 
\begin{equation}
  \label{eq:DefIota}
  \iota(\rhos)=\inv(\rhos)-m([\rhos],\rhos^-).
\end{equation}
(Note that $m([\rhos],\rhos^-)=m([\rhos],\rhos^+)$.)

\begin{example}
  \label{eq:IotaOfChord}
  If $\rho$ is a single Reeb chord, then 
  $\iota(\{\rho\})=-\weight(\rho)$.
\end{example}

\begin{example}
  \label{ex:JoinCurveEnd}
  Fix integers $a,b\geq 0$, $i\in \{2,\dots,2n-1\}$, and 
  let $\rho_1=(L_i R_i)^\alpha L_i$, 
  $\rho_2=(R_i L_i)^\beta R_i$, and $\rhos=\{\{\rho_1\},\{\rho_2\}\}$.
  Then,
  \begin{align*}
    \inv(\rho_1,\rho_2)&=|\alpha-\beta| \\
    m([\{\rho_1,\rho_2\}],\{\rho_1,\rho_2\}^-)&=
    2(\alpha+\beta+1).
  \end{align*}
  Note that $\iota(\{(L_i R_i)^{\alpha+\beta+1}\})=-\alpha-\beta-1$, so
  \[ \iota(\rhos)-\iota(\{(L_i R_i)^{\alpha+\beta+1}\})
  = |\alpha-\beta|-\alpha-\beta-1\leq -1 ,\]
  with equality iff $a=0$ or $b=0$.
\end{example}

\begin{example}
  \label{ex:CollisionEnd}
  Fix integers $a,b\geq 1$, $i\in\{2,\dots,2n-1\}$;
  let $\rho_1=(L_i R_i)^a$, $\rho_2=(R_i L_i)^b$,
  and $\rhos=\{\rho_1,\rho_2\}$.
  Then
  \begin{align*}
    \inv(\{\rho_1\},\{\rho_2\})&= |a-b| \\
    m([\{\rho_1,\rho_2\}],\{\rho_1,\rho_2\}^-)&= 2a+2b.
  \end{align*}
  Note that  $\iota(\rho_1)=-a$, $\iota(\rho_2)=-b$; so
  \[ \iota(\{\rho_1,\rho_2\})
  = |a-b|-2a-2b <-a-b
  =\iota(\{\rho_1\})+\iota(\{\rho_2\}).\]
\end{example}

\subsection{The chamber structure on $\Tb$}
We will be interested in some further structure on $\Tb$ induced by
the boundary degenerations, as defined  in
Definition~\ref{def:BoundaryDegenerations} (using $d$ now
as in Equation~\eqref{eq:dUp}). 

For each $\{j,k\}\in\Mdown$, there is a corresponding component of
$\Bjk$ of $\Sigma_0\setminus \betas$, which contains the two
boundary components $Z_j$ and $Z_k$; or equivalently, component
$\Bjk\subset \Sigma\setminus\betas$ which contains
the two punctures corresponding to $Z_j$ and $Z_k$.

Let $t\colon {\mathbb H}\cong [0,\infty)\times \R \to \R$ denote the
projection to the second factor.  Given
$w\in\ModDeg_J^\Bjk$, we have two punctures $q_1$ and $q_2$, labelled
by orbits $j$ and $k$ respectively.  Let $\delta(w)=t\circ
w(q_1)-t\circ w(q_2)$. Note that $\delta(w)$ is not invariant under
conformal automorphisms of ${\mathbb H}$; but the sign of $\delta(w)$
is.

\begin{lemma}
  \label{lem:Walls}
  For generic $J$, the subspace of $d$-dimensional torus $\Tb$ 
  \[ \Wall^{\orb_j=\orb_k}=\{\x\in\Tb\big| \exists v\in\UnparModDeg^{\Bjk}(\x)~\text{so that}~\delta(v)=0\}\]
  is the image under a smooth map of a smooth manifold of dimension $d-1$.
\end{lemma}
\begin{proof}
  The map $\delta\colon \ModDeg^{\Bjk}(\x)\to \R$ is a smooth
  map. By transversality arguments, for generic $J$, $0$ is a regular
  value, so $\delta^{-1}(0)$ is a submanifold. Similarly, if we take the quotient
  by the automorphism of group of ${\mathbb H}$, the quotient of $\delta^{-1}(0)$
  is a codimension one submanifold of $\UnparModDeg^{\Bjk}(\x)$. 
  Now, $\Wall^{\orb_j=\orb_k}$ is the image of this submanifold under the evaluation map $\evB$ from Definition~\ref{def:evB}.
\end{proof}

We have the following analogue of 
Lemma~\ref{lem:BoundaryDegenerationsDegree1} for lower diagrams:

\begin{lem}
  \label{lem:BoundaryDegenerationsDegree1A}
  The evaluation map $\evB\colon \ModDeg_J^{{\mathcal B}_{\{r,s\}}}\to \Tb$ has odd degree.
\end{lem} 

It follows that the complement of $\Wall^{\orb_j=\orb_k}$ in $\Tb$ 
consists of two (disjoint) chambers:
\begin{align*}
  \Chamber^{\orb_j>\orb_k}&=\{\x\in\Tb\big| \#(w\in\ModDeg^{\Bjk}(\x)~\text{so that}~\delta(w)>0)\equiv 1 \pmod{2}\} \\
  \Chamber^{\orb_j<\orb_k}&=\{\x\in\Tb\big| \#(w\in\ModDeg^{\Bjk}(\x)~\text{so that}~\delta(w)<0)\equiv 1 \pmod{2}\}
\end{align*}

\subsection{Smooth moduli spaces}

\begin{defn}
  \label{def:IndexTypeA}
  Let $\x$ and $\y$ be lower states, 
  suppose that $(B,\rhos_1,\dots,\rhos_\ell)=(B,\vec{\rhos})$ is strongly boundary monotone.
  Define
  \begin{align}
    \chiEmb(B,\vec{\rhos})&= d+ e(B) - n_\x(B)-n_\y(B) -\sum_{i=1}^{\ell} \Big(\iota(\chords(\rhos_i))+\weight(\chords(\rhos_i))\Big)
    \label{eq:ChiEmbA} \\
    \ind(B,\x,\y;\vec{\rhos}) &= e(B)+n_\x(B)+n_\y(B)+\ell   \label{eq:EmbIndA}
\\
  & \qquad -\weight(\vec{\rhos})+\iota(\chords(\vec{\rhos}))+\sum_{o\in\orb(\vec\rhos)} (1-\weight(o)), \nonumber
  \end{align}
  where
  \[ \iota(\vec{\rhos})=\sum_{i=1}^{\ell} \iota(\chords(\rhos_i))
  \qquad{\text{and}} \qquad
  \weight(\rhos)=\sum_{i=1}^{\ell} \weight(\rhos_i),
  \] 

\end{defn}

\begin{rem}
Note that in the special case where each packet in $\vec{\rhos}$
contains a single chord (so we write $\vec{\rhos}=\vec{\rho}$),
Example~\ref{eq:IotaOfChord} shows that $\iota(\vec{\rho})=-\weight(\vec{\rho})$;
so the above definition of the embedded index is consistent with
Equation~\eqref{eq:ChiEmb}.
\end{rem}

We have the following analogue of Proposition~\ref{prop:ExpectedDimension}

\begin{prop}
  \label{prop:ExpectedDimensionA}
  Suppose that $(\x,\vec\rhos)$ is strongly boundary monotone.
  If $\ModFlow^B(\x,\y,\Source;\vec\rhos)$ is represented by some
  pseudo-holomorphic $u$, then $\chi(\Source)=\chiEmb(B)$ if and only
  if $u$ is embedded. In this case, the expected dimension of the
  moduli space is computed by $\ind(B,\x,\y;\vec\rhos)$.  Moreover, if a
  strongly monotone moduli space $\ModFlow^B(\x,\y,\Source;\vec\rhos)$
  has a non-embedded holomorphic representative, then its expected
  dimension $\leq \ind(B,\x,\y;\vec\rhos)-2$.
\end{prop}

\begin{proof}
  To deduce Equation~\eqref{eq:ChiEmbA}, we apply the proof of
  Proposition~\ref{prop:ExpectedDimension}. As in that proof, we
  compare the intersection number $u\cap
  \tau_R(u)=n_\x(B)+n_\y-\frac{d}{2}$ with $u\cap
  \tau_\epsilon(u)$. In that argument, we used the fact that
  $\tau_\epsilon(u)=b_\Sigma$. In the present case, however, there are additional
  intersection points from $u\cap \tau_\epsilon(u)$ that come from the
  double points at the boundary (arising from the constraint packets).
  Thus, 
  \[u \cap \tau_\epsilon(u)=b_\Sigma + \sum_{i=1}^\ell N(\chords(\rhos_i)),\] where
  $N(\rhos_i)$ is the  number of intersection points of
  $u\cap \tau_\epsilon(u)$ that come from the constraint packet $\rhos_i$.
  We will show that 
  \begin{equation}
    \label{eq:ComputeCorrection}
    N(\chords(\rhos))=-\iota(\chords(\rhos))-\weight(\chords(\rhos)).
  \end{equation}

  To see this, note that contributions arise only for pairs of chords
  in the packet that are contained in some fixed boundary component
  $\Zout_j$. Suppose then that there are exactly two chords $\rho_1$
  and $\rho_2$ in $\rhos$ that are contained in $\Zout_j$. Suppose
  that the length of $\rho_1$ is $a$ and the length of $\rho_2$ is
  $b$. (Here, we normalize so that the whole boundary has length $1$;
  so in particular $L_i$ has length $\OneHalf$.) By boundary
  monotonicity, $a+b$ is an integer. In our local model, the surface
  has a component where $f$ is modelled on $\tau\mapsto \tau^{2a}$ and another
  modeled on $\tau\mapsto c\cdot \tau^{2b}$, for $\tau\in\C$ with
  $\Real(\tau)\geq 0$, and some $c\in \R^{<0}$. To count
  double points, we can halve the number of double points on the maps
  $\tau\mapsto \tau^{2a}$ and $\tau\mapsto c \cdot \tau^{2b}$ for
  $\tau\in \C$. Counting double points there is equivalent to counting
  the intersection number of the quadratic function $(z-\tau^{2a})(z-c
  \tau^{2b})$ with the discriminant locus, which in turn is equivalent
  to the order of vanishing of the function
  \[ (\tau^{2a}+c \tau^{2b})^2-4c \tau^{2(a+b)}=(\tau^{2a}-c  \tau^{2b})^2=\min(4a,4b),\] which, by Examples~\ref{ex:JoinCurveEnd}
  and~\ref{ex:CollisionEnd}, verifies
  Equation~\eqref{eq:ComputeCorrection}.  (Note that the when $a$ has
  fractional length, we are using Example~\ref{ex:JoinCurveEnd} with
  $a=\alpha+\OneHalf$, $b=\beta+\OneHalf$.) 
  
  Having verified Equation~\eqref{eq:ComputeCorrection},
  Equation~\eqref{eq:ChiEmbA} follows at once.

  Deducing the index from the Euler characteristic as in
  proof of Proposition~\ref{prop:ExpectedDimension}, noting that 
  \[   \ind=2e+b_{\CDisk}-\sum_{i=1}^{\ell}(|\rhos_i|-1)+\sum_{\rho\in\chords(\rhos_i)}(1-2\weight(\rho))+\sum_{o\in\orbits(\rhos_i)}(2-2\weight(o)).
  \]
\end{proof}

\begin{example}
  Consider the shadow in Figure~\ref{fig:DoublepointFree}.  This
  shadow occurs for four different boundary monotone moduli spaces:
  $\UnparModFlow^B(\x,\y,(\{\orb_i\};\Source_1))$, where $\Source_1$ is a
  disk with four boundary punctures, and a single orbit puncture;
  $\UnparModFlow^B(\x,\y,(\{L_iR_i\};\Source_2))$, where $\Source_2$ is a
  disk with five boundary punctures (one of which is an East infinity
  boundary puncture, labelled by $L_iR_i$);
  $\UnparModFlow^B(\x,\y,(\{R_iL_i\};\Source_3))$, similar to the above
  moduli space, and finally,
  $\UnparModFlow^B(\x,\y,(\{L_i,R_i\};\Source_4))$, where $\Source_4$ is a
  disjoint union of two three-punctured disks (one of which has an
  East infinity puncture labelled by $L_i$ and the other of which has
  a puncture labelled by $R_i$).  The dimensions of these moduli
  spaces are $2$, $1$, $1$, and $0$ respectively.
  \begin{figure}[h]
    \centering
    \input{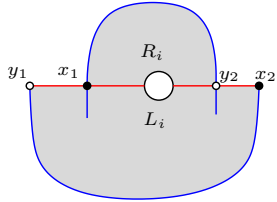}
    \caption{{\bf Moduli spaces with given shadow.}
      \label{fig:DoublepointFree}}
  \end{figure}
\end{example}

There are evaluation maps $\evB_i\colon
\ModFlow(\x,\y,\rhos_1,\dots,\rhos_\ell)\to \Tb$, obtained as
follows. Suppose that the punctures at level $\rhos_i$ are mapped via
$t$ to $\tau\in \R$; then 
\[ \evB_i(v)=\pi_{\Sigma}^{-1}(\{(0,\tau)\}).\]
These maps descend to give maps
\begin{equation}
  \label{eq:EvaluationAtStage}
  \evB_i\colon
  \UnparModFlow(\x,\y,\rhos_1,\dots,\rhos_\ell)\to \Tb.
\end{equation}

The following is a mild elaboration on Theorem~\ref{thm:GeneralPosition}
(see~\cite[Proposition~5.6]{InvPair}):
\begin{thm}
  \label{thm:GeneralPositionA}
  Choose a generic $\{J_t\}$. Suppose that $(\x,\vec{\rhos})$ is
  boundary monotone.  If $\ind(B,\x,\y;\vec{\rhos})\leq 2$, then the
  moduli space $\ModFlow^B(\x,\y,\Source;\vec{\rhos})$ is a smooth
  manifold of dimension given by $\ind(B,\x,\y;\vec{\rhos})$.  Moreover,
  $\evB_i$ are transverse to all of the walls
  $\Wall^{\orb_j=\orb_k}$ for all $\{r,s\}\in\Mdown$; in particular, all the
  disks $v$ appearing in the zero-dimensional moduli spaces 
  $\UnparModFlow^B(\x,\y,\Source,\vec{\rhos})$ with
  $\ind(B,\x,\y;\vec{\rhos})=1$ have
  $\evB_i\in \Chamber^{\orb_j>\orb_k}$ or $\Chamber^{\orb_j<\orb_k}$.
\end{thm}

\begin{proof}
  For generic $\{J_t\}$, the moduli spaces $\ModFlow^B(\x,\y;\Source)$
  are transversely cut out by the $\dbar$ operator.  The $t$
  evaluations of the various punctures give a map
  from this moduli
  space to $\R^{P}$, where the set $P$ corresponds with the
  interior and $\east$-punctures of $\Source$.
  This map $\ev_P\colon \ModFlow^B(\x,\y;\Source)\to \R^P$
  is a submersion.
  The moduli space $\ModFlow^B(\x,\y;\Source;\vec\rhos)$
  can be thought of as the preimage under this evaluation map of a
  suitable diagonal $\Delta_P$ in $\R^P$. (For example, if
  some subset $\{p_1,\dots,p_k\}$
  of punctures are assigned to the same constraint packet,
  then the corresponding diagonal in $\R^P$ consists of
  those $P$-tuples whose components at $p_1,\dots,p_k$ coincide.)
  It follows readily that
  $\ModFlow^B(\x,\y,\Source;\vec{\rhos})$ is a smooth
  manifold of dimension given by $\ind(B,\x,\y;\vec{\rhos})$.

  At each puncture $p$, we also have a corresponding evaluation map
  $\ev_p^\beta\colon \ModFlow^B(\x,\y;\Source)$ defined by
  $\ev_p^\beta(v)=\pi_\Sigma^{-1}(\{(0,t(v(p)))\})$;
  taking the product over each puncture gives a submersion
  \[ \ev_P^\beta\colon \ModFlow^B(\x,\y;\Source)\to (\Tb)^{P}.\]
  It follows that for generic $\{J_t\}$, the evaluations
  are transverse to the diagonals in $\R^P$ and the walls
  $\Wall^{\orb_j=\orb_k}$ from Lemma~\ref{lem:Walls};
  in particular, for generic $\{J_t\}$,
  $\ev^\beta_i$ on
  $\UnparModFlow(\x,\y,\rhos_1,\dots,\rhos_\ell)$
  (which is obtained by evaluating $\ev^\beta_p$ at
  any puncture $p$ belonging to the $i^{th}$ packet)
  is transverse to the codimension one walls
  $\Wall^{\orb_j=\orb_k}$.
  (Compare~\cite[Proposition~5.6]{InvPair};
  see also~\cite[Section~3.4]{McDuffSalamon}.)
\end{proof}

If a holomorphic curve $v$ represents a point in
$\UnparModFlow^B(\x,\y;\Source;\vec{\rhos})$ with
$\evB_i\in \Chamber^{\orb_j>\orb_k}$ or
$\Chamber^{\orb_j<\orb_k}$, we write $v\in \Chamber^{\orb_j>\orb_k}_i$ or
$\Chamber^{\orb_j<\orb_k}_i$.

\subsection{Ends of one-dimensional moduli spaces}

We set up some preliminaries used in the description of the ends of one-dimensional moduli spaces.

\begin{defn}
  \label{def:OrbitMarking}
  Observe that there are two {\em special} Reeb orbits which are not
  matched with any other Reeb orbit. If $\orb_j$ is one of these
  orbits, recall that the corresponding component of
  $\Sigma\setminus\betas$, which we denote ${\mathcal B}_{\{j\}}$,
  contains one of the two basepoints $\wpt$ or $\zpt$.  We call the other
  orbits {\em non-special}.  An {\em orbit marking} is a partition of
  the orbits into two types, the {\em even} ones and the
  {\em odd} ones, so that the following conditions hold:
  \begin{itemize}
    \item each even one is matched with an odd one in
      $\Mdown$.
    \item there is exactly one even special orbit and one odd one.
  \end{itemize}
\end{defn}

\begin{defn}
  \label{def:Allowed}
  Fix an orbit marking. A constraint packet is called {\em allowed by
    the orbit marking}, or simply {\em allowed}, if it satisfies the
  following properties:
  \begin{enumerate}[label=($\rhos$-\arabic*),ref=($\rhos$-\arabic*)]
  \item Each of the chords appearing in $\rhos$ are disjoint from one another.
  \item \label{rhos:OneOrbit} 
    It contains at most one orbit, and that orbit is
    simple. If it contains no orbits, it is called {\em orbitless}.
  \item If $\rhos$ contains an even type orbit, then it contains exactly one
    Reeb chord, as well; and the chord is disjoint from the orbit.
  \item If $\rhos$ contains an odd type orbit, then it contains
    no Reeb chords.
  \end{enumerate}
\end{defn}

Given two constraint packets $\rhos_1$ and
$\rhos_2$, a {\em contained collision} is a new (possibly generalized)
constraint packet
$\sigmas$, where
$\orbits(\sigmas)=\orbits(\rhos_1)\cup\orbits(\rhos_2)$ as multi-sets
(i.e. it might consist of the same orbit with multiplicity $2$), and
$\chords(\sigmas)$ is the union of the following three sets:
\begin{itemize}
  \item those  Reeb chords in $\rhos_1$ that cannot be prepended onto any Reeb chord in $\rhos_2$
  \item those  Reeb chords in $\rhos_2$ that cannot be appended to any Reeb chord in $\rhos_1$
  \item the joins $\rho_1\uplus\rho_2$ 
    of all possible pairs of joinable (i.e. 
    ``strongly composable'') Reeb chords 
    $\rho_1\in \chords(\rhos_1)$ and $\rho_2\in \chords(\rhos_2)$.
\end{itemize}
The contained collision is called {\em visible} if no orbit in
$\orbits(\sigmas)$ is contained in both $\orbits(\rhos_1)$ and
$\orbits(\rhos_2)$.  The collision is called {\em strongly composable}
if whenever the chords $\rho_1\in \chords(\rhos_1)$ and $\rho_2\in
\chords(\rhos_2)$ are weakly composable (as in
Definition~\ref{def:ComposableChords}), they are in fact strongly
composable.

\begin{rem}
  A collision between two constraint packets $\rhos_1$ and $\rhos_2$
  might be merely a a generalized constraint packet. For example, both
  $\rhos_1$ and $\rhos_2$ may contain the same orbit, and their
  collision can contain the same orbit with multiplicity two. When the
  collision is contained and visible this does not occur: the
  multi-set of orbits in the collision is in fact a set. (Indeed, if
  the collision occurs in a boundary monotone sequence, it is also
  easy to see that the multi-set of chords is also a set.)
\end{rem}

\begin{defn}
  \label{def:BoundaryDegCollision}
  Suppose that $\rhos_1$ and $\rhos_2$ are allowed constraint packets
(in the sense of Definition~\ref{def:Allowed}),
which also have the property that there are $\{j,k\}\in\Mdown$ so that
$\orb_j\in\orbits(\rhos_1)$ and $\orb_k\in\orbits(\rhos_2)$. In this
case, $\chords(\rhos_1)\cup\chords(\rhos_2)$ consists of a single
chord, which we denote $\sigma$. We say that the constraint packet is
$\{\sigma\}$ (i.e. with the two orbits removed) is the {\em boundary
  degeneration collision} of $\rhos_1$ and $\rhos_2$.
\end{defn}

With these remarks in place, we state the following analogue
of~\cite[Theorem~5.61]{InvPair}, which will be used in
Section~\ref{sec:TypeA} in the verification of the $\Ainf$ relation:

\begin{thm}
  \label{thm:AEnds}
  Let $\Hdown$ be a
  lower diagram and $\Mdown$ the induced relation among
  $\{1,\dots,2n\}$. Choose also an orbit marking
  as in Definition~\ref{def:OrbitMarking}.
  Fix a lower Heegaard state $\x$ and a sequence of 
  constraint packets $\vec{\rhos}$ with the following properties:
  \begin{itemize}
  \item $(\x,\vec{\rhos})$ is strongly
    boundary monotone.
  \item Each constraint packet $\rhos_i$ is allowed,
    in the sense of Definition~\ref{def:Allowed}
  \end{itemize}
  Let $\y$ be a
  lower Heegaard state, and $B\in\pi_2(\x,\y)$, whose local multiplicity
  vanishes either at $\wpt$ or $\zpt$ (or both). Choose $\Source$ and
  $\vec{P}$ so that $[\vec{P}]=(\rhos_1,\dots,\rhos_\ell)$ and
  so that $\chi(\Source)=\chiEmb(B)$;
  and suppose that $\ind(B,\x,\y;\vec{\rhos})=2$, and abbreviate
  $\UnparModFlow=\UnparModFlow^B(\x,\y;\Source;{\vec{\rhos}})$. The total
  number of ends of $\UnparModFlow$ of the following types are even in
  number:
  \begin{enumerate}[label=(AE-\arabic*),ref=(AE-\arabic*)]
  \item \label{endA:2Story}
    Two-story ends, which are of the form
    \[ \UnparModFlow(\x,\w;\Source_1;\rhos_1,\dots,\rhos_i)\times 
    \UnparModFlow(\w,\y;\Source_2;\rhos_{i+1},\dots,\rhos_{\ell}),\]
    taken over all lower Heegaard states $\w$ and choices of $\Source_1$ and $\Source_2$ so that
    $\Source_1\natural \Source_2=\Source$, and $B_1\natural B_2=B$.
  \item 
    \label{endA:Orbit}
    Orbit curve ends,
    of the form $\UnparModFlow^B(\x,\y,\Source';\rhos_1,\dots,\rhos_{i-1},\sigmas,\rhos_{i+1},\dots,\rhos_{\ell})$,
    where $\orbits(\sigmas)=\orbits(\rhos_i)\setminus\{\orb_r\}$, $\chords(\sigmas)=\chords(\rhos_i)\cup\{\longchord_r\}$
    where $\longchord_r$ is a Reeb chord that covers the  boundary component $Z_r$ with multiplicity $1$.
  \item 
    \label{endA:ContainedCollisions}
    Contained collision ends for two consecutive packets $\rhos_i$ and
    $\rhos_{i+1}$, which correspond to points in
    $\UnparModFlow^{B'}(\x,\y,\Source;\rhos_1,\dots,\rhos_{i-1},\sigmas,\rhos_{i+2},\dots,\dots,\rhos_\ell)$
    with the following properties:
    \begin{enumerate}[label=(C-\arabic*),ref=(C-\arabic*)]
    \item The collision is visible.
    \item The packets $\rhos_i$ and $\rhos_{i+1}$ are strongly composable.
    \item The packet $\sigmas$ is a contained collision of $\rhos_i$
      and $\rhos_{i+1}$
    \item \label{c:Disjoint}
      The chords in $\sigmas$ are disjoint from one
      another.
    \end{enumerate}
  \item
    \label{endA:Join}
    Join ends, of the form
    $\UnparModFlow^B(\x,\y,\Source';\rhos_1,\dots,\rhos_{i-1},\sigmas,\rhos_{i+1},\dots,\rhos_{\ell})$,
    $\orbits(\sigmas)=\orbits(\rhos_i)$, and the following conditions hold:
    \begin{enumerate}[label=(J-\arabic*),ref=(J-\arabic*)]
      \item $(\x,\rhos_1,\dots,\rhos_{i-1},\sigmas,\rhos_{i+1},\dots,\rhos_{\ell})$
        is strongly boundary monotone.
      \item There is some $\rho\in \chords(\rhos_i)$ with the property
        that $\rho=\rho_1\uplus \rho_2$, and
        $\chords(\sigmas)=(\chords(\rhos_i)\setminus \{\rho\})\cup
        \{\rho_1,\rho_2\}$.
      \item \label{OneIsShort} In the above decomposition, at least one of $\rho_1$ and 
        $\rho_2$ covers only half of a boundary component.
      \end{enumerate}
  \item \label{endA:BoundaryDegeneration}
    Boundary degeneration collisions $\sigmas$
    between two consecutive packets $\rhos_i$ and $\rhos_{i+1}$;
    when
    $\orb_j\in\orbits(\rhos_i)$, $\orb_k\in\orbits(\rhos_{i+1})$
    and $\{j,k\}\in\Mdown$; these correspond to points in
    $\ModFlow^{B'}(\x,\y,\Source';\rhos_1,\dots,
    \rhos_{i-1},\sigmas,\rhos_{i+1},\dots,\rhos_\ell)$
    in the chamber $\Chamber^{\orb_j<\orb_k}$.
    The homology class $B'$ 
    is obtained from $B$ by removing 
    a copy of $\Brs$.
  \item \label{endA:SpecialBoundaryDegeneration}
    Special boundary degeneration ends,
    when $\rhos_i$ contains a special Reeb orbit $\orb_k$.
    When $\sigmas=\rhos_i\setminus\{\orb_k\}$ is non-empty,
    these are identified with
    \[ \UnparModFlow^{B'}(\x,\y,\Source';\rhos_1,\dots,\rhos_{i-1},\sigmas,\rhos_{i+2},\dots,
    \rhos_\ell)\] for    $B=B'+{\mathcal B}_{\{k\}}$;
    when $\rhos_i=\{\orb_k\}$, then $\ell=1$, $\x=\y$, $B={\mathcal B}_{\{k\}}$, and the end is unique.
  \end{enumerate}
\end{thm}

\begin{rem}
  In the above statement, some of the sources $\Source'$ are different
  from the original source $\Source$. We have not spelled out the
  precise relationship between $\Source'$ and $\Source$; it is clear
  from the context.
\end{rem}

\begin{rem}
  The packets $\sigmas$ that appear in the join curve ends are not
  {\em allowed} in the sense of Definition~\ref{def:Allowed}; moreover,
  packets appearing in contained collision ends need not be allowed.
\end{rem}

Let
\[ \UnparModFlow^B(\x,\y,\vec{P})=
\bigcup_{\Source}\UnparModFlow^B(\x,\y,\Source;\vec{P}).\]

See Figure~\ref{fig:JoinEnd} for a picture of a join curve end.

 \begin{figure}[h]
 \centering
 \input{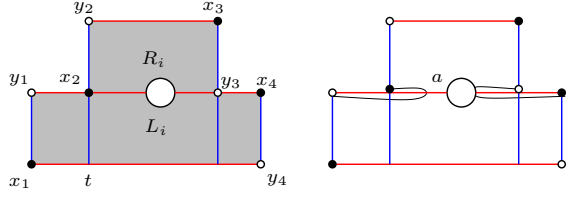}
 \caption{{\bf Join curve end.} 
 The moduli space corresponding to the shaded homotopy class, from $\x=\{x_1,x_2,x_3,x_4\}$
 to $\y=\{y_1,y_2,y_3,y_4\}$ with the Reeb chord $L_i R_i$ has a join curve end as the cut parameter $a\mapsto 0$.
 (There is another end which is a two-story building, corresponding to a flow from $\{x_1,x_2,x_3,x_4\}$ to
 $\{y_1,t,x_3,x_4\}$, followed by the flow from $\{y_1,t,x_3,x_4\}$ to $\{y_1,y_2,y_3,y_4\}$ that crosses
 $L_i R_i$).)}
 \label{fig:JoinEnd}
 \end{figure}

\begin{example}
  Suppose that $\x$ and $\y$ are generators. Let $c$ denote the number of
  points in $\UnparModFlow^B(\x,\y,(\{L_i,  R_i\}))$.
  Let $a$ be the number of two-story ends of
  $\UnparModFlow^B(\x,\y,(\{L_i  R_i\}))$; and $b$ be the number of
  two-story ends of $\UnparModFlow^B(\x,\y,(\{R_i  L_i\}))$.
  The above theorem applied to $\UnparModFlow^B(\x,\y,(\{L_i  R_i\}))$
  implies that $a+c\equiv 0\pmod{2}$; and applied to
  $\UnparModFlow^B(\x,\y,(\{R_i  L_i\}))$ gives $b+c\equiv 0\pmod{2}$.
\end{example}

\begin{example}
  Suppose that $\x$ and $\y$ are generators. 
  The above theorem implies that the number of two-story ends
  of $\UnparModFlow^B(\x,\y,\{L_i R_i L_i\})$ is even; in particular,
  there are no join curve ends because the constraint packet
  $\{L_i, R_i L_i\}$ is not
  part of a boundary monotone sequence.
\end{example}

\begin{example}
  Suppose that $\x$ and $\y$ are generators. 
  Let $c_1$ denote the number of points in 
  $\UnparModFlow(\x,\y,\{L_i,R_i L_i R_i\})$, 
  $c_2$ denote the number of points in 
  $\UnparModFlow(\x,\y,\{R_i,L_i R_i L_i\})$, 
  $a$ denote the number of two-story ends of 
  $\UnparModFlow(\x,\y,\{L_i R_i L_i R_i\})$
  and $b$ the number of two-story ends of 
  $\UnparModFlow(\x,\y,\{R_i L_i R_i L_i \})$.
  The above theorem implies that 
  \begin{align*}
    a+c_1+c_2& \equiv 0\pmod{2} \\
    b+c_1+c_2& \equiv 0\pmod{2}.
  \end{align*}
\end{example}

\begin{example}
  The above theorem shows that there is an even number of two-story
  ends of $\UnparModFlow(\x,\y,(\{L_{i+1}\},\{L_i\}))$, since the collision
  between $L_{i+1}$ and $L_{i}$ is weakly, but not strongly, composable.
  (Note that there are three types of such two-story ends.)
\end{example}

\begin{example}
  Consider the moduli space
  $\UnparModFlow(\x,\y,\{L_i R_i\},\{ L_i R_i\})$. Let $a$ denote the number of two-story ends
  (again, of three possible types); $b$ denote the number of join curve ends
  (of two types, corresponding to the sequence
  $(\{L_i,R_i\},\{L_i R_i\})$ or the sequence
  $(\{L_i R_i\},\{L_i, R_i\})$);
  and $c$ be the number of points in 
  $\UnparModFlow(\x,\y,\{L_i R_i L_i R_i\})$.
  Then, $a+b+c\equiv 0\pmod{2}$.
\end{example}

\begin{example}
  Let $a$ be the number of two-story ends of
  the moduli space $\UnparModFlow^B(\x,\y,\{\orb_i,R_{i+1}\})$,
  and $b_1$ denote the number of points in $\UnparModFlow^B(\x,\y,\{L_i R_i, R_{i+1}\})$,
  and $b_2$ denote the number of points in $\UnparModFlow^B(\x,\y,\{R_i L_i, R_{i+1}\})$.
  Then,
  $a+b_1+b_2\equiv 0\pmod{2}$.
\end{example}

\subsection{Curves at East infinity}

We recall (with very minor adaptation) the material
from~\cite[Section~5.3]{InvPair}.  Let $Z=\bigcup_{i=1}^{2n} Z_i$ be
the boundary of $\Sigma_0$.  Let ${\mathbf a}=Z\cap
\bigcup_{i=1}^{2n-1}\alpha_i$. We consider moduli spaces of
holomorphic curves in $\R \times Z \times [0,1] \times \R$. The ends
of the first $\R$ factor are called east and west infinity; the ends
of the second $\R$ factor are called $\pm\infty$.  There is an
$\R\times \R$-action on $\R \times Z \times [0,1] \times \R$,
projection maps $\pi_{\R\times Z}$ (onto the first two factors), $s$
(to $[0,1]$) and $t$ (to the last $\R$ factor).  Fix a split complex
structure $J$ on $\R\times Z\times [0,1]\times \R$.

\begin{defn}
  An {\em east source} $\EastSource$ is:
  \begin{itemize}
  \item a smooth two-manifold $T$ with boundary
    and punctures
  \item a labeling of each puncture of $T$ by $\east$ or $\west$
  \item a labeling of each $\west$ or $\east$ puncture $q$ by a Reeb orbit, if
    the $q$ is in the interior of $T$, and a labeling of $q$ by a
    chord in $(Z,{\mathbf a})$ if the puncture is on the boundary of
    $T$.
  \end{itemize}
\end{defn}

Given a east source $\EastSource$, we consider maps:
\[ v\colon (T,\partial T)\to (\R\times Z \times [0,1]\times \R,
\R \times {\mathbf a}\times \{1\}\times \R) \]
satisfying:
\begin{enumerate}[label=(E-\arabic*),ref=(E-\arabic*)]
\item \label{EC1} $v$ is $(j,J)$-holomorphic with respect to some almost-complex
  structure $j$ on $T$.
\item $v$ is proper.
\item $(s,t)\colon v \to [0,1]\times \R$ is constant.
\item At each $\west$ puncture $q$ of $T$ labeled by $\rho$ (a chord or orbit), 
  $\lim_{z\goesto q}\pi_\Sigma\circ v(z)$ is 
  $\rho\subset \{-\infty\}\times Z$.
\item\label{ECn} At each $\east$ puncture $q$ of $T$ labeled by $\rho$ (a chord or orbit), 
  $\lim_{z\goesto q}\pi_\Sigma\circ v(z)$ is 
  $\rho\subset \{+\infty\}\times Z$.
\end{enumerate}

Note that if $T$ has non-empty boundary, then $s\circ v=1$.

\begin{defn}
  Let $\ModEast(\EastSource)$ denote the moduli space of 
  holomorphic maps from $T$ satisfying Properties~\ref{EC1}-\ref{ECn} above.
\end{defn}

For each puncture $q$ in $\EastSource$, there is a corresponding
evaluation $\ev_q\colon \ModEast(\EastSource)\to \R$ which computes
the $t\circ v$ on the component of $T$ containing $q$. There are
evaluation maps 
\begin{align*}\ev_\west&=\prod_{q\in\West(\EastSource)} \ev_q \colon
\ModEast(\EastSource)\to ([0,1]\times \R)^{|\West(\EastSource)|} \\
  \ev_\east&=\prod_{q\in\East(\EastSource)} \ev_q \colon
  \ModEast(\EastSource)\to ([0,1]\times \R)^{|\East(\EastSource)|}. 
\end{align*}

Consider the $t$-projection of the evaluation map; e.g.  $t\circ
\ev_{\west}\to \R^{|\West(\EastSource)|}$ Given an east source
$\EastSource$, and ordered partitions $P_\west$ and $P_\east$, we let
$\ModEast(\EastSource;P_\west,P_\east)\subset \ModEast(\EastSource)$
be the subspaces obtained by cutting down by the partial diagonals
associated
to $t\circ \ev_\west$ and $t\circ \ev_\east$.

The following is~\cite[Proposition~5.14]{InvPair}; compare~\cite[Section~3.3]{McDuffSalamon}:
\begin{prop}
  If $\EastSource$ has the property that all of the components of $T$ are topological disks,
  then $\ModEast(\EastSource)$ is transversely cut out by the $\dbar$ equation for any
  split complex structure on $\R\times Z\times [0,1]\times \R$.
\end{prop}

Let $\UnparModEast(\EastSource)=\ModEast(\EastSource)/\R\times \R$.

The following particular components are illustrated in Figure~\ref{fig:EastInfinity}.

A {\em trivial component} is a component of $\EastSource$ with exactly
two punctures, one $\east$ and one $\west$, both labelled by the same Reeb chord.
The holomorphic map to $\R\times Z$ is invariant under translation by $\R$.

A {\em join component} is a component of $\EastSource$ which is a
topological disk with two west boundary punctures and one east
boundary puncture. Labeling, in counterclockwise order $(e,\rho_e)$,
$(w,\rho_1)$, and $(w,\rho_2)$.  There is a holomorphic map from such
a component if and only if $\rho_e=\rho_2\uplus\rho_1$; if it exists,
it is unique up to translation. The puncture $(w,\rho_1)$ is called
the {\em top puncture} and $(w,\rho_2)$ is called the {\em bottom
  puncture}.  A {\em join curve} is a curve that consists of one join
component and a collection of trivial components.

An {\em split component} is defined similarly, only now there the
punctures in counterclockwise order, are $(w,\rho_w)$, $(e,\rho_1)$
and $(e,\rho_2)$. Again, there is a holomorphic map if and only if
$\rho_w=\rho_1\uplus\rho_2$. The puncture $(e,\rho_2)$ is called the
{\em top puncture} and $(e,\rho_1)$ is called the {\em bottom
  puncture}.  

\begin{remark}
  Note that there are join and split curves that cover the cylinder with
  arbitrarily large multiplicity; on the left of
  Figure~\ref{fig:EastInfinity}, we have illustrated a split curve
  that covers the cylinder with multiplicity one, and these are the
  split curves that will occur in our considerations for type $D$
  structures.  When considering type $A$ modules, though, we will be
  forced to consider join and split curve ends that cover the boundary
  cylinder with higher multiplicity.
\end{remark}

An {\em orbit component} is a disk with a single boundary puncture,
labelled $(w,\rho)$, and a single orbit puncture $(e,\orb)$ in its
interior, so that $\orb$ is a simple orbit.  There is a holomorphic
map from such a component if and only if $\rho$ is one of the two
chords that covers the boundary component containing $\orb$ with
multiplicity one.  

\begin{rem}
  There are other components with disk sources one might consider.
  For example, the ``shuffle curves'' from~\cite[Section~5.3]{InvPair}
  have a natural analogue, which one might call ``orbit-shuffle
  curves''. These have a west puncture that is an orbit, another west
  puncture labelled by a chord $\rho$ on the same boundary component
  as $\orb$, and a single east puncture labelled by $\rho\uplus u$,
  where $u$ is one of the two curves that covers the boundary once.
  The map to $\R\times Z$ has a single branch point in the
  interior. These curves, however, will not enter our
  considerations. This is because in an allowed constraint packet (in
  the sense of Definition~\ref{def:Allowed}), the orbits are disjoint
  from the chords. These other curves would enter if we were to try to
  define the theory over $\Alg$ from~\cite{Bordered2}, which we can
  avoid by some algebraic considerations; see
  Section~\ref{sec:Comparison}.
\end{rem}

 \begin{figure}[h]
 \centering
 \input{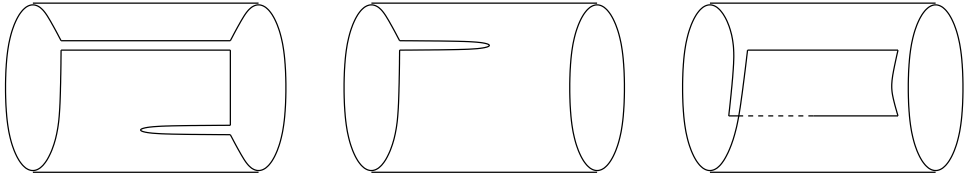}
 \caption{{\bf Curves at East infinity.}  The boundary on the left is
   glued to the source, and the boundary to the right is the ``east
   infinity'' portion.  We have illustrated from left to right: a
   split curve, an orbit curve, and an orbit shuffle curve.  The join
   curve is obtained by reflecting the leftmost picture through  a vertical axis.}
 \label{fig:EastInfinity}
 \end{figure}

\subsection{Curves at West infinity}
\label{subsec:CurvesAtWest}

Unlike in~\cite{InvPair}, generalized flowlines can degenerate also
at west infinity.  In formulating this degeneration, we need
to generalize slightly the
notion of a decorated source and pre-flowline, as follows:

\begin{defn}
  A {\em decorated source with West punctures} is a decorated source
  as in Definition~\ref{def:DecoratedSource} equipped with a further
  set of boundary punctures that are labelled $\west$.  A {\em pre-flowline
  with west punctures} is a map
  \[ u\colon (\Source,\partial\Source)\to
  (\Sigma\times[0,1]\times\R,(\alphas\times\{1\}\times \R)\cup(\betas\times\{0\}\times\R))\]
  where $\Source$ is equipped with boundary punctures $\West(\Source)$ 
  with the following  properties:
  \begin{itemize}
  \item for each west puncture $q$, 
    $\lim_{z\goesto q} u(z)$ converges to a point in
    $\Sigma\times 0\times \R$;
    i.e. if we fill in the west punctures to form a source curve 
    $\Source'$
    (without West punctures), $u$ extends 
    uniquely to a  $u'\colon \Source'
    \to \Sigma\times [0,1]\times \R$
  \item The extension $u'$ obtained as above is a pre-flowline
    in the sense of Definition~\ref{def:GenFlow}.
  \item The set of west punctures in $\Source$ 
    comes with an ordered partition into $d$-tuples $q_1,\dots,q_d$,
    so that 
    \[ \pi_{\CDisk}\circ u(q_1)=\dots=\pi_{\CDisk}\circ u(q_d).\]
    We call this the  {\em disk partition} of the punctures at West infinity.
  \end{itemize}
  A pseudo-holomorphic flowline with West punctures is a pre-flowline
  with West punctures which, when filled in, give a pseudo-holomorphic
  flowline in the sense of Definition~\ref{def:HolFlow}.
\end{defn}

Obviously, when $\West(\Source)=\emptyset$, the above definition agrees
with the usual definition of a source curve and pre-flowline
(Definitions~\ref{def:DecoratedSource} and~\ref{def:GenFlow} respectively).

For a pre-flowline $u$, let $[\West(\Source)]$ denote the set of
$d$-tuples in the disk partition; equivalently, it is the set of
equivalence classes of West punctures on $\Source$, modulo the equivalence
relation $q\sim q'$ if $t\circ u(q)=t\circ u(q')$. In particular,
\[ |[\West(\Source)]|=|\West(\Source)|/d.\]

Fix a pre-flowline $u$ with west punctures, and ${\mathbf
  q}=\{q_1,\dots,q_d\}\in[\West(\Source)]$. There is a corresponding
evaluation
\[ \evB_{\mathbf q}(u)=\pi_\Sigma(u(q_1))\times \dots\times \pi_\Sigma(u(q_d)).\]
By boundary monotonicity, 
$\evB_{\mathbf q}(u)\in  \beta_1\times \dots\times \beta_g$

Taking the product over $[\West(\Source)]$, we obtain a map
\[ \evB_\east\in (\Tb)^{[\West(\Source)]}\]

We consider moduli spaces of holomorphic curves in $\Sigma\times
{\mathbb H}$, generalizing the boundary degenerations of
Definition~\ref{def:BoundaryDegenerations}.

\begin{defn}
  A {\em boundary degeneration with West punctures}
  is a map
  \[ w\colon
  (\wSource,\partial\wSource)\to(\Sigma\times\HH,\betas\times\R)\]
  whose $d$ punctures over the point at infinity in $\HH$ are called
  {\em east punctures}; and equipped with additional punctures in $\partial\wSource$, called {\em west punctures} with the following properties:
  \begin{itemize}
  \item for each west puncture $q$, $\lim_{z\goesto q} u(z)$ converges
    to a point in $\Sigma\times \R\subset \Sigma\times \HH$; i.e. if
    we fill in in the west punctures to form a source curve $\wSource'$,
    then $\west$ extends uniquely to a continuous map 
    \[ w'\colon (\wSource',\partial\wSource')\to
    (\Sigma\times\HH,\Sigma\times \R).\]
  \item The extension $w'$ obtained as above is a boundary
    degeneration as in Definition~\ref{def:BoundaryDegenerations}.
  \item The set of west punctures in $\wSource$ comes with a partition
    into $d$-tuples 
    $q_1,\dots,q_d$ so that 
    \[\pi_\HH\circ u(q_1)=\dots=\pi_\HH\circ u(q_d).\]
  \end{itemize}
\end{defn}

\begin{defn}
  A {\em boundary degeneration level} is a finite union of boundary
  degenerations with West punctures. We can think of its source curve
  $\WestSource$ (which has possibly many components) as marked with 
  a set of East punctures $\East(\WestSource)$ and west punctures $\West(\WestSource)$;
  punctures of each type come in  $d$-tuples (again, referred to as the disk partition).
\end{defn}

For a boundary degeneration level, let $[\East(\WestSource)]$ resp.
$[\West(\WestSource)]$ and 
denote
the set of $d$-tuples in the disk partition of $\East(\WestSource)$
resp. $\West(\WestSource)$.

For a boundary degeneration level, we have, as before, evaluations
\[ \evB_\east(w)\in (\Tb)^{[\East(\WestSource)]}
\qquad \text{and}\qquad
\evB_\west(w)\in (\Tb)^{[\West(\WestSource)]}.\]

\subsection{Compactness}

As in~\cite[Section~5.4]{InvPair}, we use the Eliashberg-Givental-Hofer compactness~\cite{EGH}.

\begin{defn}
  \label{def:HolomorphicStory}
  A {\em holomorphic story}  is the following data:
  \begin{itemize}
  \item a sequence $(w_{\ell},\dots,w_{1},u,v_1,\dots,v_k)$ for some
    $\ell\geq 0$, $k\geq 0$ where $u$ is a pseudo-holomorphic flowline
    with West punctures, with source $\Source$; $\{w_i\}_{i=1}^{\ell}$
    is a sequence of boundary degeneration levels,
    where $w_i$ has source $\WestSource_i$; $\{v_i\}_{i=1}^{k}$ is a
    sequence of curves at East infinity.
  \item One-to-one correspondences between the following objects:
    $\West(\Source)$ and $\East(\WestSource_{1})$;
    $\West(\WestSource_i)$ and $\East(\WestSource_{i+1})$
    for $i=1,\dots,\ell-1$ (which respect the disk partition);
    $\East(\Source)$ and $\West(\EastSource_1)$;
    $\East(\EastSource_i)$ and $\West(\EastSource)_{i+1}$ (for 
    $i=1,\dots,k-1$),
  \end{itemize}
  so that the following conditions hold:
  \begin{itemize}
    \item $u\in\ModFlow^B(\x,\y;\Source)$
    \item $v_i\in\ModEast(\EastSource_i)$
    \item $w_{i}\in\ModWest(\WestSource_{i})$
    \item $\ev_\east(u)=\ev_\west(v_1)$ in
      $\R^{\East(\Source)}/\R\cong \R^{\West(\EastSource_1)}/\R$
      (here, and in the next few conditions, we use
      the isomorphism of product spaces
      induced by the one-to-one correspondence between punctures).
    \item $\ev_\east(v_{i})=\ev_\west(v_{i+1})$ in
      $\R^{\East(\EastSource_{i})}/\R\cong
      \R^{\West(\EastSource_{i+1})}/\R$
      for $i=1,\dots,k-1$.
    \item  $\evB_\west(u)=\evB_\east(w_{1})$ in
      $(\Tb)^{[\West(\Source)]}\cong(\Tb)^{[\East(w_1)]}$
    \item $\evB_\west(v_w)=\evB_\east(w_{i+1})$ in
      $(\Tb)^{[\West(\WestSource_i)]}\cong
      (\Tb)^{[\East(\WestSource_{i+1})]}$
      for $i=1,\dots,\ell-1$.
    \end{itemize}
    A holomorphic story with $\{k,\ell\}=\{0,1\}$ is called a {\em
      simple holomorphic comb}.  A {\em holomorphic comb of height
      $N$} is a sequence
    $(w_{j,\ell_j},\dots,w_{j,1},u_j,v_{j,1},\dots,v_{j,k_j})$ for
    $j=1,\dots,N$ of holomorphic stories with $u_j$ a stable curve in
    $\ModFlow^{B_j}(\x_j,\x_{j+1};\Source_j)$ for some sequence of
    generalized generators $\x_1,\dots,\x_{N+1}$.
\end{defn}

\begin{rem}
  In the above statement, we require the curves $u_j$ to be stable. This excludes
  where all the components of the source $\Source_j$ are disks with exactly two punctures:
  one at $+\infty$, another at $-\infty$, and $\pi_{\Sigma}\circ u_j$ is a constant map
  (to $\x_j=\x_{j+1}$).
\end{rem}

As in~\cite[Section~5.4]{InvPair} (following~\cite{EGH}), the space of
holomorphic combs can be used to construct a compactification
of $\UnparModFlow(\x,\y;\Source)$, denoted
$\ClosedModFlow(\x,\y;\Source)$; and similarly
a compactification
$\ClosedModFlow(\x,\y;\Source;\vec{\rhos})$
of $\UnparModFlow(\x,\y;\vec{\rhos})$.
This is a compactification in the following sense:

\begin{prop}
  \label{prop:Compactness}
  If $\{U_n\}$ is a sequence of holomorphic combs in a fixed homology 
  class, then $\{U_n\}$ has a subsequence which converges to a (possibly)
  nodal holomorphic curve in the same homology class. 
\end{prop}

\begin{proof}
  See~\cite[Proposition~5.24]{InvPair} for the compactness result at
  east infinity. The west infinity curves are extracted by a more
  standard Gromov compactification, rescaling from the $\CDisk$ factor
  to $\HH$, as in the proof of~\cite[Lemma 3.82]{LipshitzCyl}.
\end{proof}

Let $([0,1]\times\R)^{\vec\rhos}$ denote the space of functions
from all of the punctures on the source to $[0,1]\times \R$.
Note that  if two punctures $p$ and $q$
correspond to the same level $\rhos_i$, then 
their $t$ values coincide.
The evaluation maps
\[ 
\ev\colon \UnparModFlow^B(\x,\y,\vec{\rhos})\to ([0,1]\times
\R)^{\vec\rhos}/\R \] extend continuously to the space of holomorphic
stories, where now we evaluate on all east-most punctures.

Each curve at west infinity $w$ has a shadow $B$, which determines the
total multiplicity of the Reeb orbits. When the shadow is $\Bjk$,
there are exactly two Reeb orbits, $\orb_j$ and $\orb_k$; in this
case, we call the curve at west infinity a {\em simple boundary
  degeneration component}. There are also two special orbits that are
not matched, with corresponding shadows ${\mathcal B}_{\{j\}}$
containing no other orbit. We also call the corresponding boundary
degeneration components simple.
 A {\em simple boundary degeneration} is a
curve all of whose components are either trivial or simple west
components.

When $w$ is a simple boundary degeneration, and $\ev(w)\in
\Chamber^{\orb_k>\orb_j}$, we call $\orb_k$ the {\em top orbit} and
$\orb_j$ the {\em bottom orbit}; see
Figure~\ref{fig:ExtractBoundaryDegeneration}.

\begin{figure}[h]
 \centering
 \input{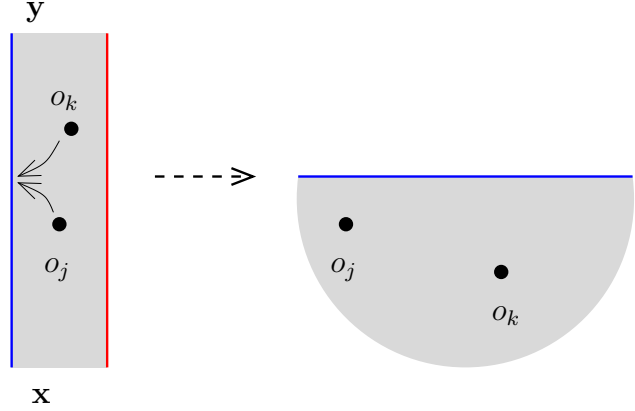}
 \caption{{\bf Extracting a boundary degeneration.}
   As the two orbits on the left come together, we can rescale
   (and rotate) to construct a boundary degeneration as pictured on the right.}
 \label{fig:ExtractBoundaryDegeneration}
\end{figure}

\subsection{Gluing}

The ends of two-dimensional moduli spaces stated in
Theorem~\ref{thm:AEnds} are modeled by certain gluing results, which
we collect here; compare~\cite[Section~5.5]{InvPair}.

The following is~\cite[Proposition~5.39]{InvPair}. To state it, use
the following terminology from there. If $(u,v)$ is a simple
holomorphic comb, with $v\in\ModEast(\EastSource)$,
a {\em smeared neighborhood} of $(u,v)$ in
$\ClosedModFlow^{B}(\x,\y,\Source,P)$ is an open neighborhood of
\[\{ (u,v')|v'\in\ModEast(\EastSource),
(u,v')\in\ClosedModFlow^B(\x,\y;\Source;P)\}.\] Also, given
$p,q\in\EastSource$, let $\compactev_{p,q}\colon
\ClosedModFlow(\x,\y;\Source)\to [-\infty,\infty]$ be the
continuous extension of $\ev_{p,q}$.

\begin{prop}
  Suppose that $(u,v)$ is a simple holomorphic comb in
  ${\ClosedModFlow}^B(\x,\y;\Source, \Partition)$.  Assume
  that $v$ is a split curve and there are two parts $P_1$ and $P_2$
  such that for each split component of $\EastSource$, its bottom
  puncture belongs to $P_1$ and its top puncture belongs to
  $P_2$. Assume that $\ind(B,\x,\y;\Source,P)=2$. Let
  $q_1\in \Partition_2$ and $q_2\in \Partition_1$ be the top and
  bottom punctures, respectively, on one of the split components of
  $\EastSource$. Then for generic $J$, there is a smeared neighborhood
  $U$ of $(u,v)$ in
  ${\ClosedModFlow}^B(\x,\y;\Source, \Partition)$ so that
  $\compactev_{q_1,q_2}\colon U \to \R_+$ is proper near $0$ and of
  odd degree near $(u,v)$.
\end{prop}

The above is proved in~\cite[Proposition~5.39]{InvPair}.

The following result is also used in~\cite{InvPair} (see especially
the proof of ~\cite[Theorem~5.61]{InvPair}):

\begin{prop}
  \label{prop:JoinCurve}
  Let $\Source$ be a source curve equipped with some ordered partition
  $\Partition$, and $\Partition'$ be the ordered partition
  where consecutive packets $P_j$ and $P_{j+1}$ in $P$ collide.
  Suppose that $\ind(B,\x,\y;\Source',\Partition')=2$.
  Suppose that $(u,v)$ is a simple holomorphic comb  in
  $\ClosedModFlow^B(\x,\y;\Source',\Partition')$, so that $u\in
  \ModFlow^B(\x,\y;\Source,\Partition)$ is a smooth point and $v$
  is a join curve.
  Let $\Delta_P\subset \R^{\Partition}$ denote the diagonal
  that specifies the collision of levels $j$ and $j+1$, and suppose that
  $u$ is a smooth, isolated point in $\ModFlow(\x,\y;\Source,\Partition)\times_{\ev_{P}}\Delta_P$.
  Then, there is a smeared neighborhood $U$ of $(u,v)\in\ClosedModFlow^B(\x,\y,\Source';\Partition')$
  with the property that $U\times_{\ev_{P'}}\Delta_{P'}$
  is homeomorphic to $[0,1)$.
\end{prop}

\begin{proof}
  Note first that since the domain of $v$ is a planar surface, $v$
  represents a smooth point in its moduli space
  $\ModEast$. (See~\cite[Proposition~5.16]{InvPair}.)  The hypothesis
  that $u$ is a smooth point includes the statement that the
  evaluation map is transverse to the diagonal where the two chords to
  be joined are mapped to the same $t$-position. Thus, $u$ and $v$ have neighborhoods
  $U_u$ and $U_v$, so that $(u,v)$ is a 
  transverse intersection point of $\ev\colon U_u\to \R^{\East(\Source)}$ and
  $\ev\colon U_v\to \R^{\West(\EastSource)}$. Thus, gluing gives an
  identification of a neighborhood of $(u,v)$ that is identified with
  $U_u\times_{\ev} U_v\times [0,1)$. In fact, the image of $\ev(U_v)$
    is the diagonal $\Delta_P$ (that determines the portion of
    $\ModFlow^B(\x,\y,\Source;P')$ where the two consecutive packets
    $P_j$ and $P_{j+1}$ collide). Thus, gluing identifies a
    neighborhood of $(u,v)$ with $(U_u\times_{\ev_P} \Delta_P)\times
    [0,1)$. By assumption, $U_u\times_{\ev_P}\Delta_P$ consists of the point $u$.
\end{proof}

In a similar vein, we have the following analogue for orbit curves. 
To state it, fix some $q_0\in \Source$, and let $p\in \Source$ be a puncture marked by an orbit.
The map
\begin{equation}
  \label{eq:EvAtP}
  s\circ \ev_{\{p\}} \colon \ModFlow(\x,\y;\Source)\to (0,1)
\end{equation}
specified by $s\circ \ev_p(u)=s\circ u(q)$, extends continuously to a map
\[ {\overline {s\circ \ev_{\{p\}}}}\colon \ClosedModFlow(\x,\y;\Source) \to [0,1].\]

\begin{prop}
  \label{prop:OrbitCurve}
  Suppose that $(u_0,v)$ is a simple holomorphic comb in
  ${\ClosedModFlow}^B(\x,\y;\Source)$, and assume that $v$ is an orbit
  curve with orbit $\orb_j$. Let $\Source_0$ denote the source for $u_0$;
  it has a distinguished boundary puncture $p_0$ labelled by some
  length one chord $\longchord_j$; and $\Source$ is obtained by 
  replacing this boundary puncture with an interior puncture $p$ labelled by 
  $\orb_j$.  Choose also the following data:
  \begin{itemize}
  \item a set of punctures $P$ on $\Source$; and let $P_0$ be
    the corresponding collection of punctures on $\Source_0$
    (i.e. with $p$ replaced by $p_0$)
  \item a smooth manifold $M$
  \item a smooth map $\phi\colon M\to \R^{P}\cong \R^{P_0}$
  \item a point $u_0\times m_0\in \ModFlow(\x,\y;\Source_0)\times_{\ev_{P_0}} M$.
  \end{itemize}
  Suppose that:
  \begin{itemize}
    \item 
      $\phi$ is transverse to
      $\ev_{P}\colon \ModFlow^B(\x,\y;\Source_0)\to \R^{P_0}=\R^{P}$
    \item $\phi$ is transverse to
      $\ev_{P}\colon
      \ModFlow^B(\x,\y;\Source)\to \R^{P}$
    \item the fibered product
      $\ModFlow^B(\x,\y;\Source_0)\times_{\ev_{P_0}} M$ is
      two-dimensional.
    \item
      $u_0\times m_0$ is an (isolated) point in
      $\ModFlow^B(\x,\u;\Source_0)\times_{\ev_{P_0}} M$.
  \end{itemize}
  Then, there is a smeared
  neighborhood $U$ of $(u_0,v)$ in $\ClosedModFlow^B(\x,\y;\Source)$
  so that 
  $U\times_{\compactev_{P_0}} M$ is homeomorphic to
  $u_0,\times m_0 \times (0,1]=(0,1]$.  Moreover, 
  \[\begin{tikzpicture}
  \node at (0,0) (X) {$u_0\times m_0\times (0,1] \subset
  \ClosedModFlow^B(\x,\y;\Source)\times_{\compactev_{P_0}} M$};
    \node at (5,0) (Y) {$[0,1]$};
    \draw[->] (X) to node[above]{\begin{tiny} $\overline{s\circ \ev}_{\{p\}}$\end{tiny}} (Y);
  \end{tikzpicture}\]
  is proper and of odd degree near
  $1$.
\end{prop}

\begin{proof}
  This proof is similar to the proof of 
  Proposition~\ref{prop:JoinCurve}.  Again,
  the moduli space $\ModEast$ for an orbit curve is smooth.  In this
  case, the evaluation map 
  \[ \ev\colon \ModEast(\EastSource,P)\to \R^{P} \] is a
  diffeomorphism near $v$, so gluing gives a smeared neighborhood of
  $(u_0,v)$ in $\overline{\ModFlow}(\x,\y;\Source)$
  which is homeomorphic to $U_0\times (0,1]$, where $U_0$ is a neighborhood
    of $u_0$ in $\ModFlow(\x,\y;\Source_0)$.
    Similarly,
    $(U_0\times (0,1])\times_{\ev_{P_0}}M$ is identified with
  $(U_0\times_{\ev_{P_0}} M)\times (0,1]$.
  Since $p$ is an interior puncture, it follows that 
  $s\circ \ev_{p}(u\times m\times t)<0$ for $t<1$; 
  it is straightforward also to see that 
  $s\circ \compactev_{p}(u\times m\times 1)=1$. The degree statement follows.
\end{proof}

The following is similar to Proposition~\ref{prop:OrbitCurve}.  To
state it, we generalize the notion of smeared neighborhood in the
following straightforward manner to simple combs $(w,u)$ where $w$ is
a curve at West infinity.

Let $(w,u)$ be a simple holomorphic curve with $w\in\ModWest(\WestSource)$,
and $u$  a pseudo-holomorphic flowline with west punctures.
A {\em smeared neighborhood} of $(w,u)$ in is an open neighborhood
in $\ClosedModFlow^B(\x,\y,\Source)$ of the set
\[ \{(w,u')\in\ModWest(\WestSource),(w,u)\in\ClosedModFlow(\x,\y,\Source)\}\]
In this neighborhood, $u$ is to be thought of as a curve with West punctures
(over the $t$-value where it meets $w$). 

\begin{prop}
  \label{prop:BoundaryDegenerationNbd}
  Suppose that $(w,u_0)$ is a simple holomorphic comb in
  ${\ClosedModFlow}^B(\x,\y;\Source)$, and assume that $w$ is a simple
  boundary degeneration with some puncture $p$ labelled by an orbit
  $\orb_j$. Let ${\mathbf q}$ denote the $d$-tuple of west punctures in $u_0$.
  Fix also the following data:
  \begin{itemize}
    \item a smooth manifold $M$
    \item a smooth map $\phi\colon M\to \R^{P_0}$, where $P_0$ denotes the
      set of punctures which do not appear in the boundary degeneration
      component
    \item an auxiliary puncture $q_0\in P_0$ with $t\circ u_0(q_0)=t \circ u_0(q)$
      for all $q\in {\mathbf q}$.
    \end{itemize}
    Assume that:
  \begin{enumerate}[label=(W-\arabic*),ref=(W-\arabic*)]
  \item $\phi$ is transverse to 
    \[ \ev_{P_0}\colon \ModFlow^{B_0}(\x,\y;\Source_0)\to \R^{P_0}, \]
    where $\Brs+B_0=B$ and $\Brs$ is the shadow of $w$.
  \item 
    $\phi\times 0 \colon M \to \R^{P_0}\times \R$ is transverse to 
    \[ \ev_{P_0} \times (\ev_{p_1}-\ev_{q_0}) \colon \ModFlow^{B}(\x,\y;\Source)\to \R^{P_0}\times \R\]
  \item 
    the fibered product
    $\ModFlow^{B}(\x,\y;\Source)\times_{\ev_{P_0}\times
      \ev_{p_1}-\ev_{q_0}} (M\times 0)$ is two-dimensional,
  \item 
    $(w\times u_0)\times m_0$ is 
    a point in $\ClosedModFlow(\x,\y;\Source)\times_{\ev_{P_0}}M$.
  \end{enumerate}
  Then there is a smeared neighborhood $U$ of $(w,u_0)$ in 
  $\ClosedModFlow^B(\x,\y;\Source)$ with the following properties:
  \begin{itemize}
  \item The map ${\overline{s\circ \ev_{p}}}\colon U\times_{\ev_{P_0}} M \to [0,1]$
    is proper near $0$.
  \item The above map has  odd degree near $(w,u_0)$.
    \end{itemize}
\end{prop}

\begin{proof}
  Each simple boundary degeneration $w$ is transversely cut out in the
  moduli space $\UnparModWest^{\Bjk}(\x;\WestSource)$. We are assuming
  that $u$ is also transversely cut out in
  $\UnparModFlow^{B_0}(\x,\y;\Source;P_2)$.  It follows from standard
  gluing results (see~\cite[Proposition~5.31]{InvPair}; see
  also~\cite[Section~5.3]{Bourgeois}) that for sufficiently small open
  neighborhoods $U_w$ and $U_{u_0}$ of $w$ and $u_0$, there is an open
  neighborhood in
  $\ClosedModFlow^B(\x,\y;\WestSource\natural\Source_0;P)$ which is
  homeomorphic to $(U_w\times_{\evB} U_{u_0})\times [0,1)\times
  (-1,1)$.  Here, the interval $[0,1)$ parameterizes the gluing
  parameter, and the interval $(-1,1)$ parameterizes the $t$ value of
  the $d$-tuple of west punctures over $u_0$.  Since $\evB\colon
  U_w\to {\mathbb T}_{\betas}$ has odd degree
  (Lemma~\ref{lem:BoundaryDegenerationsDegree1}), this neighborhood is
  identified with an odd number of copies of $U_{u_0}\times
  [0,1)\times (-1,1)$.  Choosing $U_{u_0}$ sufficiently small that
  $U_{u_0}\times_{\ev_{P_0}} M$ consists of the single point
  $(u_0,m_0)$, we have that the smeared neighborhood $U$ of $(w,u_0)$
  has the property that is $U\times_{\ev_{P_0}} M$ is identified with
  an odd number of copies of $[0,1)\times (-1,1)$.

  Pick one such component ${\overline {\mathcal M}}$, and
  let ${\mathcal M}$ denote its interior.
  Consider the
  maps 
  \[ F_1=s\circ \ev_{p} \colon {\mathcal M}\to (0,1)
  \qquad{\text{and}}\qquad F_2 =\ev_{p}-\ev_{q_0}\colon {\mathcal M}\to \R. \]
  The maps $F_1$ and $F_2$ extend 
  continuously to give maps 
  ${\overline F}_i \colon {\overline{\mathcal M}} \to \R$ for $i=1,2$;
  and in fact $F_1\colon {\overline{\mathcal M}}\to [0,1)$ is proper near $0$.

  Clearly, the restriction of ${\overline F}_2$ to
  $0\times (-1,1)$ has odd degree near
  $0$.  Moreover, since $p$ is an interior puncture, $F_1(s,t)>0$ for
  all $s>0$ and $t\in (-1,1)$. Also, ${\overline F}_1(0,t)=0$. It
  follows that $F\colon [0,1)\times (-1,1)\to \R\times \R$ has odd
  degree.  Clearly, the neighborhood of
  \[ {\mathcal M}\times_{\ev_{P_0}\times
    (\ev_{p}-\ev_{q_0})} (M\times 0) \] is given as $F_2^{-1}(0)$, and
  the restriction of ${s\circ\ev_p}$ to this set is $F_1$. It follows
  that the degree of $s\circ \ev_p$ to $U\times_{\ev_{P_0}} M$ agrees
  with an odd multiple of the degree of $F_1|_{F_2^{-1}(0)}$, which
  agrees with the degree of $F$, which we have shown to be odd.
\end{proof}

In the above proposition, we measured the distance of the boundary
degeneration by measuring $s\circ u(p)$, where $p$ is the puncture
which goes into the boundary degeneration. In the next very similar
proposition, we measure the distance to the boundary by $t\circ
u(p_1)-t\circ u(p_2)$, where now $\{p_1,p_2\}$ are the two punctures
which go into the boundary degeneration.

\begin{prop}
  \label{prop:WestInftyEnd}
  Suppose that $(w,u_0)$ is a simple holomorphic comb in
  ${\ClosedModFlow}^B(\x,\y;\Source)$, and assume that $w$ is a simple
  boundary degeneration with exactly two punctures $p_1$ and $p_2$,
  labelled by orbits
  $\orb_j$ and $\orb_k$ respectively. Fix also the following data:
  \begin{itemize}
    \item a smooth manifold $M$
    \item a smooth map $\phi\colon M\to \R^{P_0}$, where $P_0$ denotes the
      set of punctures which do not appear in the boundary degeneration
      component
    \item an auxiliary puncture $q_0\in P_0$ with $t\circ u(q_0)=t \circ u_0(q)$
      for all $q\in {\mathbf q}$.
    \end{itemize}
    Assume that:
  \begin{itemize}
  \item $\phi$ is transverse to 
    \[ \ev_{P_0}\colon \ModFlow^{B_0}(\x,\y;\Source_0)\to \R^{P_0}, \]
    where $\Brs+B_0=B$ and $\Brs$ is the shadow of $w$.
  \item $\phi$ is transverse to 
    \[ \ev_{P_0}\times (\ev_{p_1}-\ev_{q_0})\colon
    \ModFlow^{B}(\x,\y;\Source)\to \R^{P_0}\times \R, \]
    where $\Brs+B_0=B$ and $\Brs$ is the shadow of $w$.
  \item 
    the fibered product
    $\ModFlow^{B}(\x,\y;\Source)\times_{\ev_{P_0}\times
      \ev_{p_1}-\ev_{q_0}} (M\times 0)$ is two-dimensional,
  \item 
    $(w\times u_0)\times m_0$ is 
    a point in $\ClosedModFlow(\x,\y;\Source)\times_{\ev_{P_0}}M$.
  \item The map  $\evB_{\mathbf q}\colon \ModFlow(\x,\y;\Source)\to \Tb$
    is transverse to the wall $\Wall^{\orb_j=\orb_k}$.
  \end{itemize}
  Then there is a smeared neighborhood $U$ of $(w,u_0)$ in 
  $\ClosedModFlow^B(\x,\y;\Source)$ with the following properties:
  \begin{itemize}
    \item Let $f\colon U\times_{\ev_{P_0}\times (\ev_{p_1}-q_0)} (M\times 0)
      \to \R$ be defined by
      $u\mapsto t\circ u(p_1)-t\circ u(p_2)$ is proper,
      sending the portion of $U$ in the interior $\ModFlow^B(\x,\y;\Source)$
      to $\R\setminus 0$, and mapping ideal combs to $0$.
    \item
      if $\evB_{\mathbf q}(u_0)$ is in the $\Chamber^{\orb_j>\orb_k}$, then
      $f$ has odd degree near $(w,u_0)$
      for $\epsilon>0$ and even degree for $\epsilon<0$;
      if $\evB_{\mathbf q}(u_0)$  is in $\Chamber^{\orb_j<\orb_k}$
      $f$ has even degree for $\epsilon>0$ and odd degree for $\epsilon<0$.
  \end{itemize}
\end{prop}

\begin{proof}
  As in the proof of Proposition~\ref{prop:BoundaryDegenerationNbd},
  we can choose a smeared neighborhood $U$ so that
  $U\times_{\ev_{P_0}} M$ is identified with an odd number of copies
  of $[0,1)\times (-1,1)$, indexed by simple boundary degenerations
  $\{w_i\}$.  Since $\evB_{\mathbf q}$ is not on a wall, it follows
  that that for each such boundary degeneration, $t\circ w_i(p_1)-t
  \circ w_i(p_2)$ is non-zero.  Thus, we can give each component of
  $U\times_{\ev_{P_0}} M$ an {\em associated sign}, which is positive
  or negative, according to the sign of the difference $t\circ
  w_i(p_1)-t\circ w_i(p_2)$.

  Choosing $U$ sufficiently small, we can ensure
  that $t\circ u(p_1)-t\circ u(p_2)$ is non-zero; and in fact,
  on each component ${\mathcal M}$, the sign of $t \circ u(p_1)-t \circ(p_2)$
  is determined by the associated sign of the component.

  Pick any component ${\overline {\mathcal M}}$ with positive sign,
  and let ${\mathcal M}$ denote its interior. 

  Consider the maps
  \[ F_1=t \circ \ev_{p_1} - t\circ \ev_{p_2}={\mathcal M}\to \R^{>0} \qquad{\text{and}}\qquad F_2
  =\ev_{p_1}-\ev_{q_0}\colon {\mathcal M}\to \R. \]  The maps $F_1$ and
  $F_2$ extend continuously to ${\overline{\mathcal M}}$ 
  to give maps ${\overline F}_1  \colon {\overline{\mathcal M}} \to \R^{\geq 0}$ and
  ${\overline F}_2\colon {\overline{\mathcal M}}\to \R$.

  Clearly, the restriction of ${\overline F}_2$ to 
  \[ 0\times
  (-1,1)=\partial {\overline{\mathcal M}}\] 
  has odd degree near
  $0$.  Moreover, $F_1(s,t)>0$ for all $s>0$ and $t\in (-1,1)$,
  with ${\overline F}(0,t)=0$. It follows that 
  ${\overline F}\colon {\overline{\mathcal M}}\to [0,1)\times (-1,1)$
  has odd degree; and this is the same as the degree of the restriction
  $F_1|_{F_2^{-1}(0)}$. 
  
  If the component ${\overline{\mathcal M}}$ has negative associated sign,
  the map $F_1=t\circ \ev_{p_1}-t\circ \ev_{p_2}$ maps to 
  $\R^{<0}$, and indeed 
  the above argument shows that
  \[ {\overline F}\colon {\overline{\mathcal M}}\to \R^{\leq 0}\times \R \]
  has odd degree over the origin.

  Now, if $\evB_{\mathbf q}(u_0)\in \Chamber^{\orb_j>\orb_k}$, there
  is an odd number of positive associated components and an even
  number of negative associated ones; whereas if $\evB_{\mathbf
    q}(u_0)\in \Chamber^{\orb_j<\orb_k}$, there is an odd number of
  negative associated components and an even number of positive ones.
  The degree statement now follows.
\end{proof}

\subsection{Ends of one-dimensional moduli spaces}

\begin{proof}[Proof of Theorem~\ref{thm:AEnds}]  
  Consider the limiting comb in the end of the one-dimensional moduli space
  ${\widehat{\mathcal M}}$, whose existence is guaranteed by 
  Gromov compactness (cf. Proposition~\ref{prop:Compactness}).
  This limit contains no $\alpha$-boundary degenerations, because such
  a degeneration covers both basepoints $\wpt$ and $\zpt$.

  Suppose next that a $\beta$-boundary degeneration occurs in the
  limit.  By the dimension formula, the boundary degeneration must be
  simple (otherwise, the remaining curve has codimension greater than
  one). There are two cases, according to whether or not the boundary
  degeneration is special.

  Suppose first that boundary degeneration is not special. Then there
  are two different orbits $\orb_j$ and $\orb_k$ on the boundary
  degeneration, and, by the hypotheses on allowed constraint packets,
  it follows that the boundary degeneration occurs as two packets
  collide.  Moreover, since $\orb_j$ and $\orb_k$ are matched in
  $\Mdown$, and our packets are allowed, one of the two orbits is
  even, and so it is constrained to lie at the same level as a
  chord. Thus, the collision is a ``boundary degeneration collision''
  in the sense of Definition~\ref{def:BoundaryDegCollision}. Moreover,
  the hypotheses of Proposition~\ref{prop:BoundaryDegenerationNbd} are
  satisfied, using for $M$ the diagonal which specifies the chord
  packets. That proposition now ensures that there is an odd number of
  such ends, and they are all in the chamber as specified
  in~\ref{endA:BoundaryDegeneration}.

  Suppose that the boundary degeneration is special. There are two
  subcases, in the notation
  of~\ref{endA:SpecialBoundaryDegeneration}, according to whether or not
  $\sigmas=\rhos_i\setminus\{o_k\}$ is empty. 
  If $\sigmas$ is non-empty, we can 
  apply Proposition~\ref{prop:WestInftyEnd}, to find an odd number of
  ends.  If $\sigmas$ is empty, the dimension formula ensures that the
  limiting curve lies in a moduli space with expected dimension
  zero. This moduli space is empty, unless it corresponds to the
  constant flowline. That is the case where $\ell=1$, $\x=\y$, and
  $B={\mathcal B}_{\{k\}}$. These ends occur with an odd multiplicity,
  where we glue the constant flowline at $\x=\y$ to the odd number of
  simple boundary degenerations $w$ with $\evB(w)=\x$, 
  according to Lemma~\ref{lem:BoundaryDegenerationsDegree1}.
  This completes the cases where $\beta$-boundary
  degenerations occur.

  The classification of the remining ends follows very similarly to  the proof
  of~\cite[Theorem~5.61]{InvPair}.
  In the present case, we have the possibility of orbit curves forming.
  Each such end  appears with odd multiplicity by Proposition~\ref{prop:OrbitCurve}.

  The fact that the chords in the collision are disjoint from one
  another (Condition~\ref{c:Disjoint}) follows from the dimension
  formula: combining the index formula the computation from
  Example~\ref{ex:CollisionEnd}, it follows that boundary-monotone
  collisions between non-disjoint constraint packets occur in
  codimension greater than one.  

  As in~\cite{InvPair}, boundary monotonicity and the non-existence
  of $\alpha$-boundary degenerations implies that in every collision,
  the joined chords are strongly composable.

  We argue that collisions that are not visible occur with even
  multiplicity.  To see this, observe that if the limiting curve has
  two punctures $q_1$ and $q_2$, both of which are marked by the same
  orbit $\orb_i$, and $t(u(q_1))=t(u(q_2))$, then that curve can be
  obtained as a limit of curves in
  $\ModFlow(\x,\y,\Source,\rhos_1,\dots,\rhos_\ell)$, so that $q_1$
  belongs to the packet $\rhos_i$ (and $q_2$ to $\rhos_{i+1}$); or it
  can be obtained as a limit of curves in the same moduli space but
  where $q_2$ belongs to $\rhos_{i+1}$.  Thus, these two ends cancel.

  For the join curve ends as in Case~\ref{endA:Join}, the
  decomposition must satisfy Property~\ref{OneIsShort}; for other
  decompositions occur in larger codimension than one. (See the
  computation from Example~\ref{ex:JoinCurveEnd}.) 
\end{proof}

\begin{proof}[Proof of Theorem~\ref{thm:DEnds}]
  This follows as in the proof of Theorem~\ref{thm:AEnds}.
  The key difference is that in the present case, orbits are never constrained
  to lie in the same level as other Reeb chords;
  and indeed if $j$ and $k$ are matched in $\Mup$,
  i.e. both appear in a boundary degeneration, only one of $\orb_j$ or $\orb_k$
  is allowed to appear in a constraint packet. It follows that the only boundary
  degenerations that can appear are special ones. The dimension formula once
  again ensures that the remaining curve, in this case, is a constant.
\end{proof}

\newcommand\dom{\mathcal D}
\section{Type $A$ modules}
\label{sec:TypeA}

Let $\Hdown$ be a lower diagram, and $\Matching$ a matching on
$\{1,\dots,2n\}$. Let $\Mdown$ be the equivalence relation
on $\{1,\dots,2n\}$ induced
by $\Hdown$. Together, $\Matching$ and $\Mdown$ generate
an equivalence relation on
$\{1,\dots,2n\}$.
The matching $M$ is called {\em compatible} with $\Hdown$ if the
equivalence relation has one equivalence class in it.

Our aim here is to  prove the following:

\begin{thm}
  \label{thm:DefTypeA}
  Fix a lower diagram $\Hdown$ and a matching $M$ on $\{1,\dots,2n\}$
  compatible with $\Mdown$. Fix also a generic almost-complex
  structure for $\Hdown$. Let $\Amod(\Hdown,\Matching)$ be the free
  $\Field[U,V]/UV$-module generated by lower states. This can be endowed
  with the following further structures:
  \begin{itemize}
    \item A rational-valued Alexander grading $\Agr$ (Equation~\eqref{eq:DefAgr} below)
    \item An integer-valued relative grading $\Mgr$ (Equation~\eqref{eq:MgrA})
    \item A collection of maps
      \[ m_{1+\ell}\colon \Amod\otimes\overbrace{\Blg\otimes\dots\otimes\Blg}^\ell\to \Amod \]
      for $\ell\in\Z^{\geq 0}$,
      defined by counting pseudo-holomorphic 
      flows (Equation~\eqref{eq:DefAction}).
  \end{itemize}
  The result is a curved $\Ainfty$ module over $\Blg$,
  with curvature $\sum_{\{i,j\}\in\Matching} U_i U_j$,
  which has the following grading properties:
  \begin{itemize}
    \item $U$ drops $\Agr$ by $1$; $V$ raises $\Agr$ by $1$; and
      the operations $m_{1+\ell}$ preserve $\Agr$
    \item $U$ drops $\Mgr$ by $1$; $V$ drops $\Mgr$ by $1$, and 
      the operations $m_{1+\ell}$ respect $\Mgr$,
      in the sense that 
      if $\x\in\Amod$ is a homogeneous module element,
      and $a_1,\dots,a_\ell$ is a sequence of homogeneous algebra elements, then 
      $m_{1+\ell}(\x,a_1,\dots,a_\ell)$ is also homogeneous, with grading given by
      \begin{equation}
        \label{eq:MgrTypeA}
        \gr(m_{1+\ell}(\x,a_1,\dots,a_\ell))=\gr(\x)+\ell -1+ \sum_{i=1}^{\ell} \Mgr(a_i). 
      \end{equation}
    \end{itemize}
\end{thm}

Constraint packets are related to algebra elements in $\Clg$ in
Section~\ref{subsec:AlgebraicConstraints}. Once that is complete, the
proof will be given in the end of
Section~\ref{subsec:ConstructA}. Invariance properties will be dealt
with in Subsection~\ref{subsec:VaryCx}, and examples will be given in
Subsection~\ref{subsec:ExA}

\subsection{Algebra elements and constraints}
\label{subsec:AlgebraicConstraints}

An ingredient to constructing type $A$ modules to lower diagrams is
the relationship between constraint packets
(c.f. Definition~\ref{def:ConstraintPacket}) and the algebra $\Clg$, which we formulate presently.

If $\rho$ is a Reeb chord for a lower diagram, then it has a
corresponding algebra element $\bIn(\rho)$ obtained by multiplying
together the letters that represent the chord, as in
Figure~\ref{fig:ChordNamesA}.  For example, the chord $\rho$ that
covers $Z_i$ with multiplicity one and which starts and ends at
$\alpha_i$ has
\[a(\rho)=L_i \cdot R_i = \left(\sum_{\{{\mathbf{s}}\big| i\in {\mathbf{s}}, i-1\not\in
    {\mathbf{s}}\}}\Idemp{\mathbf{s}}\right)\cdot U_i \]
We generalize this construction to certain kinds of chord packets, as follows.

If $\rho$ is a Reeb chord, let $\rho^-$ denote its initial $\alpha$-arc, and $\rho^+$ denote its
terminal $\alpha$-arc.

\begin{defn}
\label{def:AlgebraicPacket}
A set of Reeb chords 
$\{\rho_1,\dots,\rho_j\}$ is called {\em algebraic} if
for any pair of distinct chords $\rho_a$ and $\rho_b$,
\begin{itemize}
  \item the chords $\rho_a$ and $\rho_b$ are on different boundary components $Z_i$ and $Z_j$,
  \item the initial points $\rho_a^-$ and $\rho_b^-$ are on different $\alpha$-curves; and
  \item the terminal points $\rho_a^+$ and $\rho_b^+$ are on different $\alpha$-curves.
\end{itemize}
\end{defn}

Let $\rhos$ be a set of Reeb chords that is algebraic, in the above sense,
we can associate an algebra element to $\rhos$, defined as follows.
Let $\alpha(\rho^+)$ be the curve
$\alpha_i$ with $\rho^+\in \alpha_i$; define $\alpha(\rho^-)$ similarly.
Let 
\[ I^-(\rhos)=\sum_{\{{\mathbf{s}}\big| \{\alpha(\rho_1^-),\dots,\alpha(\rho_j^-)\}\subset
\{\alpha_i\}_{i\in{\mathbf{s}}}\}} I_{\mathbf{s}}
\qquad\text{and}\qquad
I^+(\rhos)=\sum_{\{{\mathbf{s}}\big| \{\alpha(\rho_1^+),\dots,\alpha(\rho_j^+)\}\subset
\{\alpha_i\}_{i\in{\mathbf{s}}}\}} I_{\mathbf{s}}
\]
Then, $\bIn_0(\rhos)$ be the algebra element $a_0\in\BlgZ(2n,n)$ with
\[ a=I^- \cdot a\cdot I^+\] and whose weight $w_i(a)$ is the average local
multiplicity at $Z_i$. Let $\bIn(\rhos)$ be the image of $\bIn_0(\rhos)$
in $\Clg(n)$.

\subsection{Constructing the $\Ainfty$ module}
\label{subsec:ConstructA}

We will fix throughout a lower diagram $\Hdown$ and a compatible $\Matching$ on $\{1,\dots,2n\}$.

Sometimes it is useful to enlarge the equivalence relation to include
the points $\wpt$ and $\zpt$, as follows. We extend $\Mdown$ to a
matching $\Mdown_*$ on the set $\{\wpt,\zpt,1,\dots,2n\}$, so that $i$
is matched with $\wpt$ resp. $\zpt$ if $\Zdown_i$ can be connected to
$\wpt$ resp. $\zpt$ without crossing any $\alpha$-circles.  Then, $M$
and $\Mdown_*$ extended in this manner define an equivalence relation
on $\{\wpt,\zpt,1,\dots,2n\}$. If $M$ and $\Mdown$ are compatible, the
associated one-manifold $W(M)\cup W(\Mup_*)$ is homeomorphic to an
interval from $\wpt$ and $\zpt$. 

Traversing the arc $W(M)\cup W(\Mup_*)$ (starting at $\wpt$ and ending at
$\zpt$), we encounter placemarkers for the orbits $\{1,\dots,2n\}$ in
some order. Let $f_0\colon \{1,\dots,2n\}$ denote the order in which
the placemarkers are encountered along this arc. We will define a function
\[ f\colon \{1,\dots,2n\}\to \{1,\dots,2n\},\]
obtained by post-composing $f_0$ with the involution on $\{1,\dots,2n\}$
that switches $2i-1$ and $2i$ for $i=1,\dots,n$. Thus, 
the resulting function has $f(i)=1$ if the orbit $\orb_i$ is the second orbit we encounter
on the arc; $f(i)=2$ if $\orb_i$ is the first; 
$f(i)=3$ if $\orb_i$ is the fourth, and $f(i)=4$ if $\orb_i$ is the third, etc.
See Figure~\ref{fig:LabelOrbits}.

\begin{defn}
  \label{def:InducedOrdering} The function $f\colon \{1,\dots,2n\}$ is 
  called the {\em induced ordering} on the boundary components,
  induced by $\Mdown$ and $\Matching$.
\end{defn}

Observe that if
$i$ is chosen so that $Z_i$ and $\wpt$ are contained in the same
component of $\Sigma\setminus\betas$, then $f(i)=2n-1$.  We call the
Reeb orbit $\orb_i$ {\em even} resp. {\em odd} if $f(i)$ is even
resp. odd; i.e. $f$ induces an orbit marking, in the sense of
Definition~\ref{def:OrbitMarking}.

\begin{remark}
  The function $f$ can be thought of as a relabeling of the Reeb orbits,
  which might seem somewhat artificial at the moment. It is, however,
  a reordering which very convenient for the purpose of the pairing theorem;
  cf. Section~\ref{sec:Pairing}.
\end{remark}

 \begin{figure}[h]
 \centering
 \input{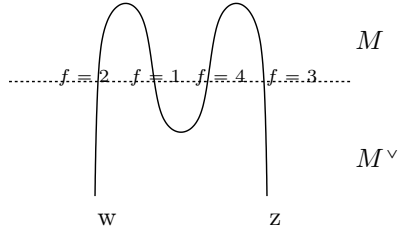}
 \caption{{\bf Labeling the orbits.} 
 The matchings $\Mdown$, $\Matching$, and the basepoint $\zpt$ order the
 Reeb orbits as indicated.}
 \label{fig:LabelOrbits}
 \end{figure}

 We will give $\Amod(\Hdown,\Matching)$ the structure of a bimodule,
 which is a right module over the idempotent ring
 $\RestrictIdempRing(n)\subset \Clg(n)$ and a left module over
 $\Ring=\Field[U,V]/U V=0$. As a left $\Ring$-module it is freely
 generated by all lower states.  The idempotent in $\Clg(n)$
 associated to a lower state is defined be
\[ \Idown(\x)=\Idemp{\alpha(\x)},\]
where $\alpha(\x)$ is as in Definition~\ref{def:UpperState}.  (Note
that this is different from the idempotent for upper states, as in
Equation~\eqref{eq:IdempOfUpper}.)  The right action of the idempotent
subalgebra of $\Clg(n)$ is specified by the condition that $\x\cdot
\Idown(\x)=\x$.

\begin{lemma}
  \label{lem:AgrDom}
  Given $\x,\y\in\States$, the quantity 
  \[ \Agr(B)=n_\zpt(B)-n_\wpt(B)+
  \sum_{i=1}^{2n} (-1)^{f(i)} \weight_i(B)\]
  is independent of the choice of $B\in\doms(\x,\y)$. 
\end{lemma}

\begin{proof}
  This follows from the fact that $\Agr(B)$  vanishes on 
  each component $B$
  of $\Sigma\setminus\betas$. (Compare the proof of Lemma~\ref{lem:GradingsWellDefined}.)
\end{proof}

Since $\Agr(B_1*B_2)=\Agr(B_1)+\Agr(B_2)$, there is a 
function $\Agr\colon \States\to \Q$, uniquely characterized up to 
an overall additive indeterminacy, so that
\begin{equation}
  \label{eq:DefAgrGen}
  \Agr(\x)-\Agr(\y)=\Agr(B) 
\end{equation}
for $B\in\doms(\x,\y)$.
This induces a $\Q$-valued grading on $\Amod(\Hdown,\Matching)=\bigoplus_{s\in\Q} \Amod(\Hdown,\Matching,s)$, 
with the convention that 
\begin{align*}
U\colon &\Amod(\Hdown,\Matching,s)\to \Amod(\Hdown,\Matching,s-1) \\
V\colon &\Amod(\Hdown,\Matching,s)\to \Amod(\Hdown,\Matching,s+1).
\end{align*}

Chose a generic admissible almost-complex structure for $\Hdown$.
We use this to  endow $\Amod(\Hdown,\Matching)$
with the structure of a right $\Ainfty$ module over
$\Clg(n)$, as follows.

\begin{defn}
  \label{def:CompatiblePacket}
  Fix a Heegaard state $\x$ and a sequence $\vec{a}=(a_1,\dots,a_\ell)$ of pure
  algebra elements in $\BlgZ(2n,n)$.
  A sequence of constraint packets $\rhos_1,\dots,\rhos_k$ is called
  \em{$(\x,\vec{a})$-compatible} if  there is a sequence 
  $1\leq k_1<\dots<k_\ell\leq k$ so that the following conditions hold:
  \begin{itemize}
  \item the constraint packets $\rhos_{k_i}$ are algebraic, in the sense
    of Definition~\ref{def:AlgebraicPacket}
  \item 
    $\Idown(\x)\cdot \bIn_0(\rhos_{k_1})\otimes\dots\otimes \bIn_0(\rhos_{k_\ell})=
    \Idown(\x)\cdot a_1\otimes\dots\otimes a_{\ell}$,
    as elements of $\Amod(\Hdown,\Matching)\otimes \ClgZ(n)^{\otimes \ell}$
  \item
    for each $t\not\in \{k_1,\dots,k_\ell\}$,
  the constraint packet $\rhos_t$ is either of the form $\{\orb_i\}$
  where $f(i)$ is odd; or it is of the form
  $\{\orb_i,\longchord_j\}$, where
  \begin{itemize}
    \item $f(i)$ is even
    \item $\{i,j\}\in\Matching$
    \item $\longchord_j$ is one of the two Reeb chords
      that covers $Z_j$ with multiplicity one. 
    \end{itemize}
\end{itemize}
Constraint packets $\rhos_j$ with $j\in\{k_1,\dots,k_\ell\}$ are called
{\em orbitless}.
Let $\llbracket \x,a_1,\dots,a_\ell\rrbracket$ 
be the set of all sequences of constraint packets $\rhos_1,\dots,\rhos_h$
that are $(\x,\vec{a})$-compatible.
\end{defn}

When considering boundary-monotone sequences, we can use
$\Clg(n)$ instead of $\ClgZ(n)$, according to the following:

\begin{lemma}
  \label{lem:NonZeroAlgElts}
  If $(\x,\rhos_1,\dots,\rhos_k)$ is a strongly boundary monotone sequence
  which is $\vec{a}=(a_1,\dots,a_{\ell})$-compatible for some sequence
  of algebra elements in $\BlgZ$;
  then the projection of 
  $\x\otimes a_1\otimes\dots\otimes a_\ell$ in $\Amod\otimes\Clg(n)^{\otimes \ell}$
  is non-zero.
\end{lemma}

\begin{proof}
  We use Proposition~\ref{prop:Ideal}.  Consider the packet at time
  $t$, $\rhos_t$, and let $\Idemp{\x_-}$ and $\Idemp{\x_+}$ be the
  idempotents immeddiately before and after. 

  We claim that $\x_-$ and $\x^+$ cannot be too far.  This follows
  from the fact that $\rhos$ contains two chords $\rho_1$ and $\rho_2$
  so that the terminal point of $\rho_1$ is the initial point of
  $\rho_2$, then the initial point of $\rho_2$ also appears in
  $I_{\x_-}$. This is immediate.

  It remains to exclude the other possibility from Proposition~\ref{prop:Ideal}.
  That can be excluded by the following reasoning. Suppose that
  $\rhos$ contains a sequence of chords $\rho_i,\dots,\rho_j$
  with the following properties:
  \begin{itemize}
  \item $\rho_t$ is supported on $\Zin_t$ for $t=i,\dots,j$
  \item either $\Idemp{\x_-}$ or $\Idemp{\x_+}$ does not contain $\alpha_i$.
  \item $\weight(\rho_t)\geq 1$ for $t=i+1,\dots,j-1$
  \end{itemize}
  A straightforward proof by induction, using boundary monotonicity,
  shows that endpoints of $\rho_j$
  are on $\alpha_j$.

  From Proposition~\ref{prop:Ideal},
  it follows that the pure algebra element 
  $\Idemp{\x_-}\cdot b_0(\rhos)\Idemp{\x_+}$ is not in ${\mathcal J}$;
  i.e. its projection to $\Blg$ is non-zero.
\end{proof}

\begin{defn}
  \label{def:UweightVweight}
  Given $(B,\rhos_1,\dots,\rhos_h)$, the {\em{$U$-weight}}
  is the quantity $\gamma$ which is be the multiplicity
  of $B$ at $\wpt$ plus the number of constraint packets among
  the $\rhos_1,\dots,\rhos_h$ consisting of a single Reeb orbit
  (or, equivalently, the total count of odd Reeb orbits appearing in the
  $\rhos_1,\dots,\rhos_h$); and the {\em{$V$-weight}}
  $\delta=\delta(B,\rhos_1,\dots,\rhos_h)$ is the
  local multiplicity of the domain $B$ at the basepoint $\zpt$.
\end{defn}

\begin{lemma}
  \label{lem:MgrDefined}
  If $\doms(\x,\y)$ is non-empty, then for $B\in\doms(\x,\y)$,
  the integer
  \begin{equation}
    \label{eq:MgrA}
    \Mgr(B)=e(B)+P(B)-\weight_\partial(B)-n_\wpt(B)-n_\zpt(B)
  \end{equation}
  is independent of $B$.
\end{lemma}

\begin{proof}
  This follows as in Lemma~\ref{lem:GradingsWellDefined},
  with the observation that now
  $e(\cald)+P(\cald)-\weight_\partial(\cald)-n_\wpt(\cald)-n_\zpt(\cald)$ vanishes
  components $\cald$ of $\Sigma\setminus\betas$.
\end{proof}

Correspondingly, we have a function $\Mgr\colon \States\to \Z$
uniquely characterized up to an overall additive constant by the property that
\begin{equation}
  \label{eq:DefMgrGen}
  \gr(\x)-\gr(\y)=\Mgr(B) 
\end{equation}
for any $B\in\doms(\x,\y)$.

\begin{lemma}
  \label{lem:CompatWithMgr}
  Fix $\x,\y\in\States$, a sequence of pure algebra elements
  $\vec{a}=(a_1,\dots,a_\ell)$, an $(\x,\vec{a})$-compatible sequence
  of constraint packets $\rhos_1,\dots,\rhos_h$, and
  $B\in\doms(\x,\y)$.  If there is a pre-flowline $u$ whose shadow is
  $B$ and whose  packet sequence is $(\rhos_1,\dots,\rhos_\ell)$,
  then
  \[ \gr(\x)+\ell-\sum_i\weight_{i=1}^{\ell}(a_i)=\gr(\y)-\Uweight(B)-\Vweight(B)+
  \ind(B,\x,\y,\rhos_1,\dots,\rhos_h).\]
\end{lemma}

\begin{proof}
  By the definition of $\gr$ and $\Mgr$,
  \[ \gr(\x)-\gr(\y)=e(B)+n_\x(B)+n_\y(B)-\weight_\partial(B)-n_\wpt(B)-n_\zpt(B).\]
  Since all the orbits in appearing in $\rhos_i$ have weight $1$, the index formula gives
  \[ \ind(B,\vec{\rhos_i})=e(B)+n_\x(B)+n_\y(B)+h-\weight_\partial(B)+\sum \iota(\chords(\rhos_i)).\]
  Indeed, $\iota(\chords(\rhos_i))=-\weight(\chords(\rhos_i))$, so
  taking the difference of these two equations, we find that
  \begin{align*}
    \gr(\x)-\gr(\y)-\ind(B,\vec{\rhos})&=-h-n_\wpt(B)-n_\zpt(B)+\sum_{i=1}^h \weight(\chords(\rhos_i)) \\
    &= -\ell-n_\wpt(B)-n_\zpt(B)-\#(\text{odd orbits})+\sum_{i=1}^{\ell}\weight(a_i) \\
    &=-\ell-\gamma(B)-\delta(B)+\sum_{i=1}^{\ell}\weight(a_i).
  \end{align*}
  Going from the second to the third line uses the fact that the packets containing
  an even orbit also contain a weight one chord.
\end{proof}

Fix a lower Heegaard state $\x$ and a sequence of pure
algebra elements $a_1,\dots,a_\ell$ so that 
$\Idown(\x)\cdot a_1\otimes \dots\otimes
  a_\ell\neq 0$.
Define
\begin{equation}
\label{eq:DefAction}
m_{1+\ell}(\x,a_1,\dots,a_\ell)=
\sum_{\{\y\in\States, 
(\rhos_1,\dots,\rhos_h)\in \llbracket \x, a_1,\dots,a_\ell\rrbracket\}} 
U^{\gamma}V^{\delta}\#\ModFlow(\x,\y,\rhos_1,\dots\rhos_h)\cdot \y,
\end{equation}
where $\gamma=\gamma(\rhos_1,\dots,\rhos_h)$ is the multiplicity at
$\wpt$ of the domain $B$ determined by the sequence
$\x,\y,\rhos_1,\dots,\rhos_h$ plus the number of constraint packets
among the $(\rhos_1,\dots,\rhos_h)$ consisting of a single Reeb orbit
  When
$\Idown(\x)\otimes a_1\otimes\dots\otimes a_\ell=0$, we define the
action $m_{1+\ell}(\x,a_1,\dots,a_\ell)=0$.

For example, the $\y$ coefficient of $m_1(\x)$ is computed by counting
points in 
\[ \ModFlow(\x,\y,\rhos_1,\dots,\rhos_h),\] where each
$\rhos_t$ is either an odd Reeb orbit, or it is the constraint packet
consisting of some even Reeb orbit covering some boundary component once
together with a Reeb chord that covers its matching (using $M$) boundary component
exactly once.

As another example, the $\y$ coefficient of $m_2(\x,U_2)$ is computed
by counting points in $\ModFlow(\x,\y;\rhos_1,\dots,\rhos_h)$, where
exactly one of $\rhos_t$ is Reeb chord $R_2 L_2$ or the Reeb chord
$L_2 R_2$, and all other constraints packets have an orbit or
an orbit and a chord as above.

We claim that the sum appearing in Equation~\eqref{eq:DefAction} is
finite, according to the following:

\begin{lemma}
  \label{lem:FiniteSum}
  Given $(\x,a_1,\dots,a_\ell)$ and $\y$, there are only finitely many
  homology classes $B$ of holomorphic disks that can represent
  $\ModFlow(\x,\y,\rhos_1,\dots,\rhos_h)$, where
  $\rhos_1,\dots,\rhos_h$ is $(\x,\vec{a})$-compatible,
  and for which one of $\Uweight=0$ or $\Vweight=0$.
\end{lemma}

\begin{proof}
  Fix $\x,\y$, and let $B\in\doms(\x,\y)$ representing some
  $(\x,\vec{a})$ compatible sequence $(\rhos_1,\dots,\rhos_h)$. 
  The following bounds are immediate:
  \begin{enumerate}[label=(b-\arabic*),ref=(b-\arabic*)]
  \item
    \label{b:EvenOdd}
  For
  $\{i,j\}\in\Mdown$, the quantity $\weight_i(B)-\weight_j(B)$ is
  independent of the choice of $B$ (depending only on $\x$ and $\y$).
  \item 
    \label{b:Two}
    $\weight_{f^{-1}(2)}(B)-n_\wpt(B)$ is independent of $B$.
  \item 
  $\weight_{f^{-1}(2n-1)}(B)-n_\zpt(B)$ is independent of $B$. 
  \item while for
  $\{i,j\}\in M$ with $f(i)$ odd,
  \[\weight_i(B)-\weight_j(B)-\#(\orb_i\in \rhos_1,\dots,\rhos_h)\]
  depends only on $\x$, $\y$, and $\vec{a}$.
  \end{enumerate}
  
  Thus, for fixed $(\x,\vec{a})$ and $\y$, 
  there are constants $c_i$ 
  so that
  \begin{equation}
    c_{i} = \weight_{f^{-1}(2i-1)}(B)-n_\wpt(B)-
    \# (\orb_{k}\in\rhos_1,\dots,\rhos_h
    ~\text{with $f(k)$ odd and $f(k)\leq 2i-1$}). \label{eq:WeightBounds} 
  \end{equation}
  The case $i=n$, together with the fact that $\weight_{f^{-1}(2n-1)}(B)-n_\zpt(B)$,
  is independent of $B$ shows that
  \[ n_\zpt(B)-n_\wpt(B)-\# (\orb_{k}\in\rhos_1,\dots,\rhos_h
  ~\text{with $f(k)$ odd})=\delta-\gamma\] is
  independent of the choice of $B$.
  Since $U V=0$, we obtain an upper bound on the $U$-power $\gamma$ appearing on any term in
  $m_{1+\ell}(\x, a_1,\dots, a_\ell)$.

  Equation~\eqref{eq:WeightBounds} implies that for any odd $i$,
  $\weight_{f^{-1}(i)}(B)\leq c_i + \gamma$.  By~\ref{b:Two}
  and~\ref{b:EvenOdd}, we can conclude a similar bound for even $i$,
  as well.  Since $\gamma$ is bounded above as above, we obtain a universal
  bound on $\weight_i(B)$ for any $B$ that contributes to
  $m_{\ell}(\x,a_1,\dots,a_\ell)$. It is elementary to see that there
  are only finitely many non-negative homology classes of $B\in
  \doms(\x,\y)$ with a universal bound on $\weight_i(B)$ bounded for
  all $i$.
\end{proof}

\begin{prop}
  \label{prop:CurvedTypeA}
  The operations defined above endow $\Amod(\Hdown,\Matching)$ with
  the structure of right $\Ainfty$ module over $\Clg(n)$ (which is
  also a free left module over $\Ring$). 
\end{prop}

\begin{proof}
  Let $\x$ be an lower Heegaard state, and $a_1,\dots,a_\ell$ be a
  sequence of algebra elements so that $\x\otimes
  a_1\otimes\dots\otimes a_\ell\neq 0$.  As usual, the $\Ainfty$
  relations are proved by looking at ends of one-dimensional moduli
  spaces. Specifically, we consider $\UnparModFlow(\x,\z,\rhos_1,\dots\rhos_h)$, where we take
  the union over all sequences of algebraic constraint packets
  $(\rhos_1,\dots,\rhos_h)\in\llbracket \x,a_1,\dots,a_\ell\rrbracket$.
  
  The condition that $\Idown(\x)\cdot a_1\otimes \dots\otimes
  a_\ell\neq 0$ ensures that the corresponding compatible constraint
  packets $(\x,\rhos_1,\dots,\rhos_h)$ are strongly boundary monotone. (See
  Lemma~\ref{lem:SBA}.)  The homology classes that have non-zero
  contribution cannot have positive multiplicity at both $\wpt$ and
  $\zpt$, since we have the relation in our algebra that $UV=0$. 
  It is easy to see that each constraint packet is 
  allowed, in the sense of Definition~\ref{def:Allowed}. 
  
  Contributions of the two-story ends (Type~\ref{endA:2Story})
  correspond to the terms 
  \[m_{\ell-i+1}(m_{i+1}(\x,a_1,\dots,a_{i}),a_{i+1},\dots,a_\ell)\]
  appearing in the $\Ainfty$ relation.

  Orbit curve ends (Type~\ref{endA:Orbit}) occur in two types. When
  the constraint packet is of the form $\{\orb_j,\longchord_k\}$ (so that
  $f(j)$ is even), its corresponding orbit curve ends corresponds to a
  term in
  \begin{equation}
    \label{eq:DiffAinf}
    m_{\ell+2}(\x,a_1,\dots,a_{i},\mu_0,a_{i+1},\dots,a_\ell).
  \end{equation}
  If the constraint packet is of the form $\{\orb_j\}$ (so that $f(j)$ is odd),
  we call the corresponding orbit curve end an {\em odd orbit curve end}.
  In the course of this proof, we will find other ends that cancel these  odd orbit curve ends.

  We turn now to visible collision ends (Type~\ref{endA:ContainedCollisions}),
  first considering visible collision ends where the two constraint packets
  $\rhos_i$ and $\rhos_{i+1}$ are orbitless. These contribution
  correspond to the terms in the $\Ainfty$ relation of the form
  \[ m_{\ell}(\x,a_1,\dots,a_{i-1},a_{i}\cdot a_{i+1},a_{i+2},\dots,a_\ell).\]

  Consider next collision ends where exactly one of $\rhos_i$ or
  $\rhos_{i+1}$ is orbitless.  In this case, we can find an
  alternative choice of $i^{th}$ and $(i+1)^{st}$ constraint packet
  $\rhos_i'$ and $\rhos_{i+1}'$, so that the corresponding ends of
  $\ModFlow(\x,\y,\rhos_1,\dots,\rhos_{i-1},\rhos_i,\rhos_{i+1},\rhos_{i+2},\dots,\rhos_\ell)$
  and
  $\ModFlow(\x,\y,\rhos_1,\dots,\rhos_{i-1},\rhos_i',\rhos_{i+1}',\rhos_{i+2},\dots,\rhos_\ell)$
  cancel. We choose $\rhos_i'=\rhos_{i+1}$ and $\rhos_{i+1}'=\rhos_i$
  except in the special case where one of $\rhos_i$ or $\rhos_{i+1}$
  (which we can assume without loss of generality is $\rhos_i$)
  consists of an even orbit $\orb_j$ and matching chord $\longchord$
  that covers $Z_k$ with multiplicity $1$, and $\rhos_{i+1}$ also
  contains some chord $\rho$ supported in $Z_k$.  In that case, there
  is a unique, possibly different long chord $\longchord'$ covering
  $Z_k$ with multiplicity one, so that $\longchord\uplus
  \rho=\rho\uplus \longchord'$.  Then,
  $\rhos_{i+1}'=\{\orb_j,\longchord'\}$, and $\rhos_{i}'=\rhos_{i+1}$;
  see Figure~\ref{fig:CollisionCancels}.  (We are cancelling here
  contributions corresponding to ends of different moduli spaces; but
  the domains $B$ each pair of moduli spaces is the same, as are the
  total number of odd orbits in the corresponding Reeb sequences; so
  the $U$ and $V$ exponents for the contributions are the same, and
  the cancellation occurs.)

 \begin{figure}[h]
 \centering
 \input{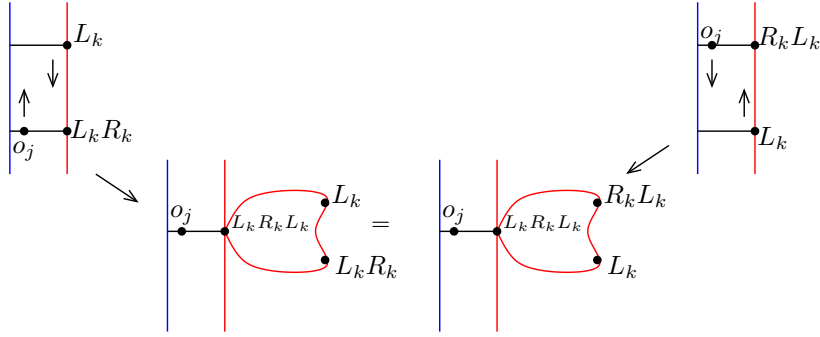}
 \caption{{\bf Cancellations of collision ends.}  
   For this collision cancellation, we use the two decompositions
   of $L_k R_k L_k$ as
   $L_k\uplus R_k L_k = L_k R_k \uplus L_k$.
   (In the picture, we assume that
   $f(k)$ and $f(j)$ are consecutive integers,
   and $f(k)$ is odd.)}
 \label{fig:CollisionCancels}
 \end{figure}

  Similar
  cancellations occur when both $\rhos_i$ and $\rhos_{i+1}$ contain
  orbits, but the orbit in $\rhos_i$ is not matched (via $\Mdown$) with the one in
  $\rhos_{i+1}$.
  
  When the orbit in $\rhos_i$ is matched with the orbit in
  $\rhos_{i+1}$ there are two kinds of collision ends: those that are
  contained (Type~\ref{endA:ContainedCollisions}), and those that are
  not (Type~\ref{endA:BoundaryDegeneration}); see
  Figure~\ref{fig:OrbitCancels}. Ends where the collision is contained
  once again cancel corresponding ends of moduli spaces where the
  order of the two packets is permuted.

  The total number of remaining (i.e. uncontained) collision ends of this
  moduli space and the one obtained by permuting $\rhos_i$ and
  $\rhos_{i+1}$ counts points in
  \[\ModFlow(\x,\z,\rhos_1,\dots,\rhos_{i-1},\sigmas,\rhos_{i+2},\dots,\rhos_\ell),\]
  where $\sigmas=\{\longchord_k\}$ is one of the two Reeb chords
  that covers the component $Z_k$ with  $f(k)$
  odd. These cancel against the odd orbit curve
  ends described above. 
  See Figure~\ref{fig:OrbitCancels}.
  Note that in this case, the homology classes $B$ and $B'$ of the curves representing the two 
  cancelling ends do not coincide: an uncontained collision end removes a boundary degeneration.
  Nonethless, the boundary degenerations considered here have $n_\wpt(v)=n_\zpt(v)=0$;
  and the total number of odd orbits remains unchanged, so the $U$ and $V$ exponents
  of the two ends coincide.

  The same cancellation occurs for ends of
  Type~\ref{endA:SpecialBoundaryDegeneration} (the ``special boundary
  degenerations''), where the even unmatched Reeb orbit (i.e. $\orb_j$
  with $f(j)=2$) is removed: it cancels with a corresponding odd
  orbit end. (See Figure~\ref{fig:SpecialCancellation}.) Cancellation
  occurs because the moduli space with the boundary degeneration end
  contributes the same $U$ and $V$ powers as the moduli space with the
  odd orbit end.  To see this, let $B$ denotes the homology class with
  the special boundary degeneration; let
  $(\rhos_1,\dots,\rhos_{\ell})$ denote the Reeb sequence for the
  moduli space containing the special boundary degeneration end; let
  $(\rhos_1,\dots,\rhos_{i-1},\sigmas,\rhos_{i+1},\dots,\rhos_{\ell})$
  be the sequence associated with the boundary degeneration end; and
  let $B'$ be the homology class where the boundary degeneration is
  removed (and the one where the corresponding odd orbit end
  occurs). Clearly, $n_\wpt(B')=n_\wpt(B)-1$, while the odd orbit end
  is associated with the sequence
  $(\rhos_1,\dots,\rhos_{i-1},\{\orb_j\},\rhos_{i+1},\dots,\rhos_\ell)$,
  which has one odd orbit more than the original sequence moduli space
  that has one additional odd orbit in its interior than
  $\rhos_1,\dots,\rhos_{\ell}$. Thus, the $U$ exponents of both moduli
  spaces coincide.  Moreover, since $n_\zpt(B)=n_\zpt(B')$, the $V$
  exponents coincide, as well.
 \begin{figure}[h]
 \centering
 \input{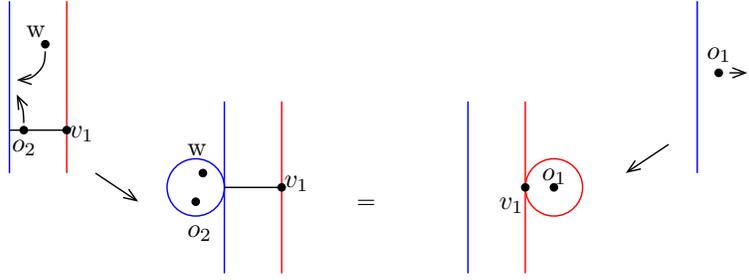}
 \caption{{\bf Special boundary degeneration cancelling 
     an odd orbit end.}  Orbits here are labelled by their $f$-values.}
 \label{fig:SpecialCancellation}
 \end{figure}

  By contrast, ends of Type~\ref{endA:SpecialBoundaryDegeneration},
  involving the other unmatched Reeb orbit $\orb_j$ (with $f(j)=1$) do not contribute.
  This is true because
  ${\mathcal B}_{\{j\}}$ also crosses the $\zpt$ basepoint, so the homology class
  contributes a multiple of $V$.
  By Equation~\eqref{eq:DefAction}, moduli spaces containing this orbit
  are counted with a multiple of $U$.  
  Since we have specialized to $UV=0$,
  this end does not contribute.  

 \begin{figure}[h]
 \centering
 \input{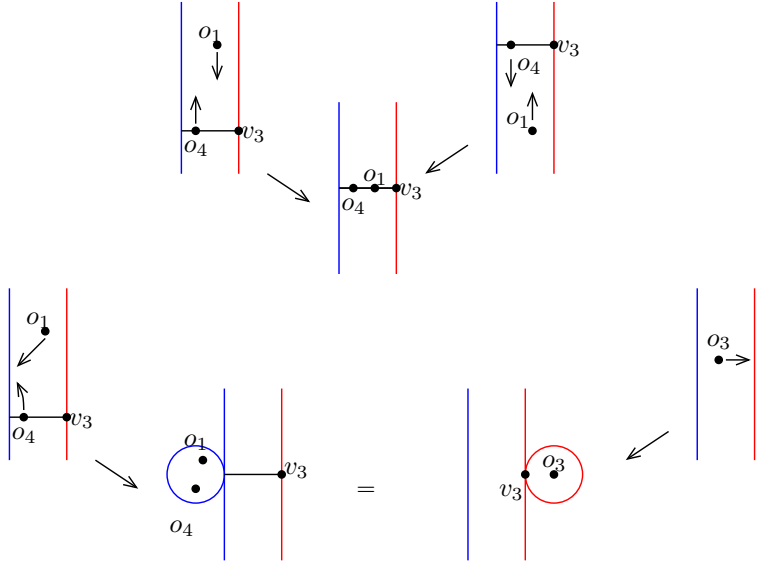}
 \caption{{\bf Cancellations when two packets contain orbits.}  The top row indicates
   cancellations of two contained collision ends.  The bottom row
   indicates the cancelation of an uncontained collision end with an
   odd orbit end. In this picture, we subscript each orbit by its $f$-value.}
 \label{fig:OrbitCancels}
 \end{figure}

  Consider next the join ends (Type~\ref{endA:Join}). We argue
  ultimately that these all count with even multiplicity, as follows.

  Suppose that  $\rhos_i$ contains some chord $\rho$ supported in
  $Z_i$, which we can split as
  $\rho=\rho_1\uplus \rho_2$. We investigate the
  corresponding join ends
  $\rhos_1,\dots,\rhos_{i-1},\sigmas,\rhos_{i+1},\dots,\rhos_\ell$,
  where
  \[\chords(\sigmas)=(\chords(\rhos_i)\setminus \{\rho\})\cup \{\rho_1,\rho_2\}.\]

  If the two endpoints of $\rho$ are distinct, then this corresponding
  join curve end cannot appear in a moduli space of boundary monotone
  curves. That is, suppose that $\rho$ has two distinct endpoints and
  it has a decomposition as $\rho=\rho_1\uplus \rho_2$. Then it is
  elementary to see that either $\rho_1^+=\rho_2^+$ or
  $\rho_1^-=\rho_2^-$; i.e. the chords $\rho_1$ and $\rho_2$ cannot
  appear in the same constraint packet, for a boundary monotone sequence.

  Consider next the case where the two endpoints of $\rho$ are the
  same, and consider there is a join curve end corresponding to a
  splitting $\rho=\rho_1\uplus \rho_2$. Such ends appear boundary
  monotone moduli spaces only when the two endpoints of $\rho_1$ (and
  hence also of $\rho_2$) are distinct. For example, we know
  $\rho_1^-=\rho_2^+$, so if $\rho_1^-=\rho_1^+$, then
  $\rho_2^+=\rho_1^+$, so $\rho_1$ and $\rho_2$ cannot appear in the
  same constraint packet.

  We are left with the case where $\rho=\rho_1\uplus\rho_2$, and the
  two endpoints of $\rho_1$ are distinct, as are the two endpoints of
  $\rho_2$.  We can then form $\rho'=\rho_2\uplus\rho_1$. 
  When $\rho_1$ or $\rho_2$ covers only half of the
  corresponding boundary component $Z_i$, this
  splitting of $\rho'$ gives rise to a join curve end in moduli space
  \[ \UnparModFlow(\rhos_1,\dots,\rhos_{i-1},\rhos_i',\rhos_{i+1},\dots,\rhos_h),\]
  where $\orbits(\rhos_i')=\orbits(\rhos_i)$ and 
  \[ \chords(\rhos_i')=(\chords(\rhos_i)\setminus \{\rho_1\uplus
  \rho_2\})\cup \{\rho_2\uplus \rho_1\}.\] Clearly,
  both  $(\rhos_1,\dots,\rhos_{i-1},\rhos_i',\rhos_{i+1},\dots\rhos_h)$ 
  is also consistent with $(\x,a_1,\dots,a_h)$.
  Moreover, this join curve end occurs with the same multiplicity as
  the corresponding join curve of
  \[
  \UnparModFlow(\rhos_1,\dots,\rhos_{i-1},\rhos_i,\rhos_{i+1},\dots,\rhos_h),\]
  corresponding to the splitting $\rho=\rho_1\uplus \rho_2$.

  We have thus set up a one-to-one correspondence between pairs of
  join curve ends of different moduli spaces that are consistent with
  the same algebra actions, so the join curve ends of the moduli
  spaces counted in the $\Ainfty$ relations cancel in pairs. 
  See
  Figure~\ref{fig:JoinsCancel} for some examples.
  \begin{figure}[h]
    \centering
    \input{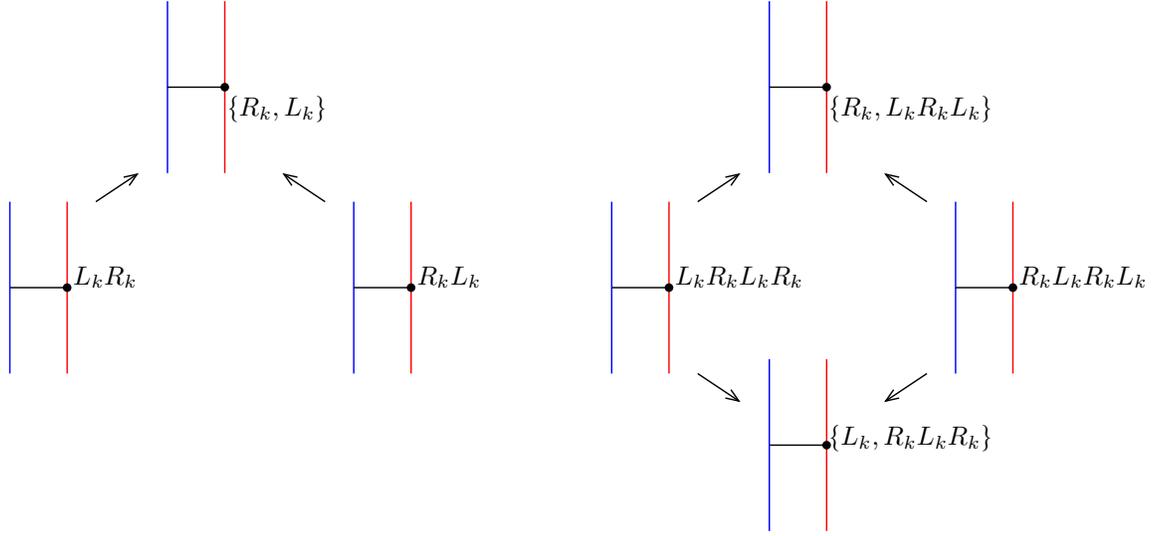}
    \caption{{\bf Cancellations of join ends.}  The two one-dimensional
      moduli spaces shown at the left (containing Reeb chords $L_k R_k$
      and $R_k L_k$) have the same ends. The two one-dimensional moduli
      spaces on the right (containing length two Reeb chords $L_k R_k L_k
      R_k$ and $R_k L_k R_k L_k$ and) each now have two different kinds
      of join curve ends; but again these two ends of the two moduli spaces
      are in one-to-one correspondence with each other.}
    \label{fig:JoinsCancel}
  \end{figure}
  (Observe that join curve ends can appear for splittings of the chord
  $(L_k R_k)^m$ or $(L_k R_k)^m$ for arbitrarily large $m$. In this
  case, the join curve at East infinity covers the cylinder with
  multiplicity $m$. Moreover, according to Theorem~\ref{thm:AEnds}
  there are only two codimension one join ends corresponding to
  splittings of $(L_k R_k)^m$; and they are
  $\{L_k,(R_k L_k)^{m-1} R_k\}$ and $\{R_k,(L_k R_k)^{m-1} L_k\}$.
  The same two ends appear when the 
  chord is replaced by $(R_k L_k)^m$.
  Figure~\ref{fig:JoinsCancel} illustrates $m=1$ and $2$.)

  This cancellation completes the verification of the $\Ainfty$ relation.
\end{proof}

We can synthesize these parts to give the following:

\begin{proof}[Proof of Theorem~\ref{thm:DefTypeA}]
  The Alexander grading was defined in Equations~\eqref{eq:DefAgrGen}
  (which is valid thanks to Lemma~\ref{lem:AgrDom}. Actions were
  defined in Equation~\eqref{eq:DefAction}, which was shown to be a
  valid definition in Lemma~\ref{lem:FiniteSum}.  The grading $\Mgr$
  was defined in Equation~\eqref{eq:DefMgrGen} (which is valid thanks to 
  Lemma~\ref{lem:MgrDefined}). The $\Ainfty$ relations were verified in 
  Proposition~\ref{prop:CurvedTypeA}. The fact that the operations
  preserve $\Agr$ is obvious from Lemma~\ref{lem:AgrDom}.
  The fact that they respect $\Mgr$ (Equation~\eqref{eq:MgrTypeA}) follows
  readily from Lemma~\ref{lem:CompatWithMgr}.
\end{proof}

\subsection{Invariance properties}
\label{subsec:VaryCx}

The following is an adaptation of~\cite[Proposition~7.19]{InvPair}:

\begin{prop}
  If $J_0$ and $J_1$ are any two generic paths of almost-complex structures,
  there is an $\Ainfty$ homotopy equivalence
  \[ \Amod(\Hdown,\Matching,J_0)\simeq \Amod(\Hdown,\Matching,J_1).\]
\end{prop}

\begin{proof}
  This is mostly standard; so we sketch the proof.
  A morphism 
  \[ f\colon \Amod(\Hdown,\Matching,J_0)\to \Amod(\Hdown,\Matching,J_1) \]
  is constructed by counting points in moduli spaces
  $\ModFlow^B(\x,\y;\rhos_1,\dots,\rhos_h)$, where now the moduli
  spaces are taken with respect to a path of paths of almost-complex
  structures that interpolate from $J_0$ to $J_1$. The proof that this
  morphism satisfies the structure equation of an $\Ainfty$
  homomorphism is similar to the proof of
  Proposition~\ref{prop:CurvedTypeA}. 
  An analogue
  of Theorem~\ref{thm:AEnds} identifies ends of these moduli spaces,
  and cancellations are analogous to the ones in
  Proposition~\ref{prop:CurvedTypeA}: two-story ends correspond to
  terms of the form
  \[ \sum_{i=1}^{1+\ell} f_{\ell-i+1}(m_i(\x,a_1,\dots,a_{i-1}),a_i,\dots,a_\ell) + 
  m_{\ell-i+1}(f_{i}(\x,a_1,\dots,a_{i-1}),a_i,\dots,a_\ell); \]
  even orbit ends contribute terms of the form 
  \[ \sum_{i=1}^\ell
  f_{\ell+2}(\x,a_1,\dots,a_{i-1},\mu_0,a_i,\dots,a_\ell); \] while
  odd orbit ends cancel with boundary degeneration collisions. Join
  curve ends cancel in pairs, and contained collisions correspond to
  terms of the form
  \[ \sum_{i=1}^{\ell} f_{\ell}(\x,\dots,a_{i-1},\mu_2(a_i,a_{i+1}),a_{i+2},\dots,a_\ell).\]
  Adding up these contributions give the $\Ainfty$ homomorphism relation.

  The homotopy inverse $g$ is constructed by turning the one-parameter
  family of paths around, and the homotopy relation
  \[ f\circ g \simeq \Id \] is verified by considering a variation of one-parameter families.
\end{proof}

More topological invariance properties of this module can be established by adapting methods from~\cite[Section~7.3]{InvPair};
the above result is the only one necessary for the purposes of this paper.

\subsection{Examples}
\label{subsec:ExA}

Consider the algebra $\cClg(1)$, using the only matching on the single
pair of strands. 
There is an isomorphism
\begin{equation}
  \label{eq:IsoClg1}
  \Psi\colon \cClg(1)\to \Ring=\Field[U,V]/UV, 
\end{equation}
with $\Psi(U_1)=U$ and $\Psi(U_2)=V$. Let $\lsup{\Ring}[\Psi]_{\Clg(1)}$ be the associated bimodule.
Note that
$\lsup{\Ring}[\Psi]_{\Clg(1)}=\lsup{\Ring}[\Psi]_{\cClg(1)}$, as the curvature in $\cClg(1)$ vanishes.

 \begin{figure}[h]
 \centering
 \input{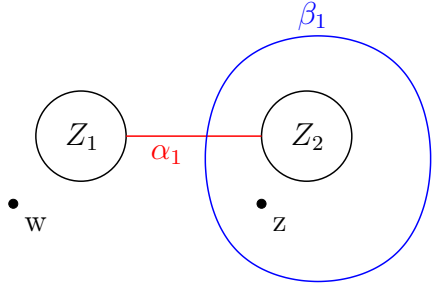}
 \caption{{\bf Standard lower diagram with $n=1$.}}
 \label{fig:SmallLower}
 \end{figure}

\begin{lemma}
  \label{lem:GlobalMinimum}
  For the standard lower diagram $\Hdown$ with two strands,
  there is an identification
  $\Amod(\Hdown)=\lsup{\Ring}[\Psi]_{\cClg(1)}$.
\end{lemma}

\begin{proof}
  Consider the diagram pictured in Figure~\ref{fig:SmallLower}.  It
  has a unique generator $x$.  Homology classes of disks are multiples
  of the two components $\dom_1$ and $\dom_2$ of $\Sigma\setminus
  \beta_1$, labelled so that $\dom_i$ contains the boundary component
  $Z_i$. A curve representing the homology class 
  $k \cdot \dom_1$ has index one if and only if contains a single Reeb
  chord on its boundary with length $k$ (and no internal
  punctures). In that case, there is a unique holomorphic
  representative, giving rise to the action $m_2(x,U_1^k)=U^k\otimes
  x$. Arguing similarly on the other side, we find that
  $m_2(x,U_2^k)=V^k\otimes x$. This completes the verification.
\end{proof}

 \begin{figure}[h]
 \centering
 \input{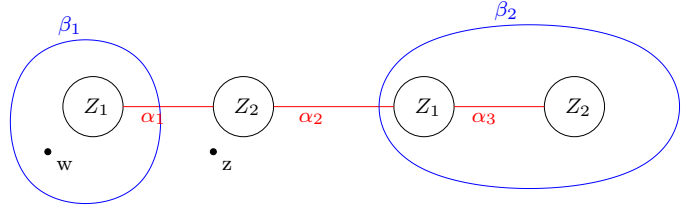}
 \caption{{\bf Standard lower diagram with $n=2$.}}
 \label{fig:Lower2}
 \end{figure}

 We consider a less trivial example; the standard lower diagram with
 $n=2$, and (compatible) matching
 $\Matching=\{\{1,4\},\{2,3\}\}$. (See Figure~\ref{fig:Lower2}.)  
 For
 simplicity, we consider the $V=0$ specialization. In this case, the
 module has a simple description, described by the following graph:

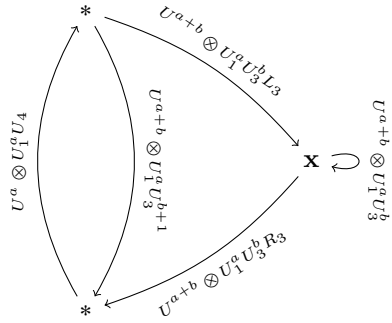
\begin{equation}
  \label{eq:MinGraph}
  \begin{tikzpicture}
    \node at (3,0) (YR2) {${\mathbf x}$} ;
    \node at (0,2) (Y) {$*$} ;
    \node at (0,-2) (X) {$*$} ;
    \draw[->] (Y) [bend left=30] to node[above,sloped] {\tiny{$U^{a+b} \otimes U_1^a U_3^{b+1}$}} (X) ;
    \draw[->] (YR2) [bend left=15] to node[below,sloped] {\tiny{$U^{a+b}\otimes 
U_1^a U_3^b R_3$}} (X) ;
    \draw[->] (X) [bend left=30] to node[above,sloped] {\tiny{$U^a\otimes U_1^a U_4$}} (Y) ;
    \draw[->] (Y) [bend left=15] to node[above,sloped] {\tiny{$U^{a+b}\otimes U_1^a U_3^{b} L_3$}} (YR2) ;
    \draw[->] (YR2) [loop right] to node[above,sloped] {\tiny{$U^{a+b}\otimes U_1^a U_3^{b}$}} (YR2);
    \end{tikzpicture}
\end{equation}
Each path from $x$ to itself (and choices of $a$ and $b$ for each edge)
gives an ${\mathcal A}_\infty$ action. The inputs are obtained by concatenating
the second tensor factor on each edge, and the output is obtained by 
multiplying together the first tensor factor in each edge. So, for example, we have
actions
\begin{align*}
  m_2(x,U_1)&= U\otimes x \\
  m_4(x,R_3,U_4,L_3)&=1\otimes x.
\end{align*}
The above two actions can be seen by looking at embedded disks in the Heegaard diagram.
Indeed, all of the other above actions are determined by the existence of the
above two actions, and the (curved) ${\mathcal A}_\infty$ structure relations. 
For example, applying the above two actions, and the ${\mathcal A}_{\infty}$
relation with inputs $R_3, U_4, L_3, U_1$, we can conclude that
\[ m_4(x,R_3, U_4, L_3 U_1)= U \otimes x.\]
Applying the ${\mathcal A}_{\infty}$ relations with input sequence
$R_3, U_4, U_1, L_3$ gives
\[ m_4(x,R_3, U_1 U_4, L_3)= U \otimes x.\]
Finally, applying the input sequence
$R_3, L_3$, we can conclude that
\[ m_2(x, U_3)= U\otimes x.\]

It is straightforward, if a little cumbersome, to generalize to the
$V$-unspecialized case; compare~\cite[Section~\ref{BK2:sec:Min}]{Bordered2}.

\section{The pairing theorem}
\label{sec:Pairing}

Start with a doubly-pointed Heegaard diagram $\HD$ for an oriented
knot $\orK$ in $S^3$, together with a decomposition along a union $Z$
of $2n$ circles along which $\HD$ decomposes as a union of a lower and
an upper diagram, $\HD=\Hdown\cup_Z \Hup$. For example, we could
consider the doubly-pointed Heegaard diagram for a knot projection as
in~\cite{AltKnots}, sliced in two along $2n$ circles corresponding to
a horizontal cut of the diagram, arranged so that the distinguished
edge is in the bottom; see for example
Figure~\ref{fig:SliceProjection}.
Letting $\Mup$ and $\Mdown$ be the matchings
induced by $\Hup$ and $\Hdown$ respectively, the hypothesis
ensures that $\Mup$ is compatible with $\Mdown$, in the sense
of Definition~\ref{def:CompatibleMatching}.

Let $C_\Ring(\HD)$ denote the Heegaard-Floer complex of the doubly-pointed diagram $\HD$:
i.e. if $C$ is the free $\Ring$-module generated by Heegaard states of $\HD$,
equipped with the differential specified by
\begin{equation}
  \label{eq:OriginalComplex}
  \partial \x =\sum_{\y\in\States(\HD)}\sum_{B\in\doms(\x,\y)}
  \# \UnparModFlow^B(\x,\y)\cdot U^{n_\wpt(B)}V^{n_\zpt(B)} \cdot  \y.
\end{equation} The homology
of this complex $C_\Ring(\HD)=(C,\partial)$, thought of as a bigraded $\Ring$-module, is the knot
invariant $\HFKsimp(K)$ discussed in the introduction.

Our aim here is to prove the following pairing theorem, describing
knot Floer homology in terms of type $D$ and type $A$ structures.

\begin{thm} 
  \label{thm:PairAwithD}
  For $\HD=\Hdown\cup_Z\Hup$ as above
  there is a quasi-isomorphism of bigraded
  chain complexes over $\Ring$
  \[\CFKsimp(\HD)\simeq \Amod(\Hdown)\DT\Dmod(\Hup).\]
\end{thm}

The proof will occupy the rest of this section.  Like the of the
pairing theorem from~\cite[Chapter~9]{InvPair}, we ``deformation 
the diagonal''. The key novelty here is to deform our orbits as well
as chords, as in~\cite{TorusMod}. However, the boundary degenerations
present in the current set-up cause this step to be more intricate.

As usual, the proof begins by inserting a sufficiently long neck, to
identify the Floer complex $\CFKsimp(\HD)$ with another complex which
is isomorphic to $C$ as a $\Ring$-module, and whose differential
$\partial^{(0)}$ counts points in a fibered product of moduli spaces
coming from the two sides. In our case, the fibered product is a
fibered product over $[0,1]\times \R$ (thought of as having
coordinates $(s,t)$) of moduli spaces of pseudo-holomorphic curves in
$\Hdown$ and $\Hup$.  This identification is spelled out in
Subsection~\ref{subsec:FiberedProduct}.  Note that this identification
works slightly differently when $n=1$. In this outline we will assume
$n>1$; the case where $n=1$ is handled in
Subsection~\ref{subsec:Nequals1}.

In Section~\ref{subsec:MatchedComplex}, we give an independent proof
that $\partial^{(0)}$ is indeed a differential. Although this is not
logically necessary (it follows from the identification with the
Heegaard Floer differential), elements of the proof will be used throughout.

We will compare the fibered product complex with another chain complex,
described in Subsection~\ref{subsec:SelfMatched}, that corresponds to
deforming the $s$-coordinate matchings to the boundary.  Again, the
underlying $\Ring$-module is $C$, but the differential $\dChanged$
now, instead of counting points in the fibered product, counts curves
where orbits for the curves in $\Hup$ match with long chords on the
$\Hdown$-side. These moduli spaces allow for some additional orbits on
the type $\Hdown$ side; details are given in
Subsection~\ref{subsec:SelfMatched}.

Interpolating between these two extremes is a sequence of complexes
(Subsection~\ref{subsec:Intermediates}), $\{(C,\partial^{(k)})\}_{k=0}^{2n}$,
where $\partial^{(2n)}=\dChanged$.

Isomorphisms $\Phi_k\colon (C,\partial^{(k)})\to (C,\partial^{(k+1)})$
are constructed in Section~\ref{subsec:Interpolate}, using moduli
spaces in which the $s$-matching on Reeb orbits are deformed by a
parameter $r\in (0,1)$.
Looking at ends of these moduli spaces give the desired relation
\[ \partial^{(k+1)}\circ \Phi_k + \Phi_k \circ \partial^{(k)}=0.\]
By an energy argument, these maps are isomorphisms.

Having fully deformed the $s$-parameter in the matching to obtain
$(C,\dChanged)$, we can now deform the $t$-parameter in the
matching as in~\cite[Chapter~9]{InvPair} (``time dilation''); see
Section~\ref{subsec:TimeDilation}.  We now proceed to establish these
steps.

\subsection{The Heegaard Floer differential and matched curves.}
\label{subsec:FiberedProduct}

The discussion here is closely modelled on~\cite[Section~9.1]{InvPair}.

Let $\Sigma_1$ denote the Heegaard surface for $\Hdown$ and $\Sigma_2$ denote the surface for $\Hup$.

\begin{defn}
  \label{def:MatchingPair}
  Let $\HD=\Hdown\#\Hup$.  States $\x_1\in\States(\Hdown)$,
  $\x_2\in\States(\Hup)$ are called {\em matching states} 
  if $\alpha(\x_1)$ is the complement of 
  $\alpha(\x_2)$ (i.e. $\Idown(\x_1)=\Iup(\x_2)$)
\end{defn}

There is a one-to-one correspondence between pairs of matching states
$\x_1$ and $\x_2$ and Heegaard states for $\Hdown\#\Hup$. Thus, the
generators of $\CFKsimp(\HD)$ correspond to the generators of
$\Amod(\Hdown)\otimes\Dmod(\Hup)$, where the latter tensor product is
over the idempotent subalgebra $\RestrictIdempRing(n)\subset \Clg(n)$.

In the Heegaard Floer differential for $\CFKsimp(\HD)$, the
holomorphic disk counting for the homology class $B$ is weighted by
$U^{n_\wpt(B)} V^{n_\zpt(B)}$.  Since in $\Ring$ we impose the relation
$UV=0$, it follows that the homology classes with non-zero contribution
have vanishing local multiplicity at $\wpt$ or $\zpt$.

For $x\in \alpha_i$, we write $\alpha(x)=i$. This definition will be used
both for $x\in\Sigma_1$ and $\Sigma_2$.

\begin{defn}
  \label{def:MatchingChords}
  Let $\rho_1$ be a Reeb chord in $\Hup$ and $\rho_2$ be a Reeb chord
  in $\Hdown$. We say that $\rho_1$ and $\rho_2$ are {\em matching
    chords} if the following two conditions hold:
  \begin{itemize}
  \item $\weight(\rho_1)=\weight(\rho_2)$ 
  \item $\alpha(\rho_1^-)\neq \alpha(\rho_2^-)$.
  \end{itemize}
\end{defn}

The first condition implies that if $\rho_1$ is a chord on $\Zin_i$,
then $\rho_2$ is a chord in $\Zout_i$. The second condition means that
the initial point of a chord is on $\alpha_{i-1}$ if and only if the
initial point of its matching chord is on $\alpha_i$.  It is
equivalent to replace the second condition by the condition
$\alpha(\rho_1^+)\neq \alpha(\rho_2^+)$.  Note that
chords in $\Zin_1$ or $\Zin_{2n}$ do not match with any chords in
$\Zout_1$ or $\Zout_{2n}$.

A more succinct description of the matching condition can be given
following our labeling conventions from Figure~\ref{fig:ChordNames}
(for the upper diagram) and~\ref{fig:ChordNamesA} (for the lower
diagram): with these conventions, $\rho_1$ and
$\rho_2$ are matching precisely if their labels are the
same.

\begin{defn}
  \label{def:MatchedPair}
  Suppose $n>1$.
  Fix two pairs $(\x_1,\x_2)$ and $(\y_1,\y_2)$ of matching states,
  i.e. where $\x_1,\y_1\in\States(\Hdown)$,
  $\x_2,\y_2\in\States(\Hup)$, so that $\x=\x_1\#\x_2$ and
  $\y=\y_1\#\y_2$ are Heegaard states for $\HD$.  
  A {\em matched pair} consists of the following data
  \begin{itemize}
  \item a holomorphic curve $u_1$ in $\Hdown$ with source $\Source_1$
    representing homology class $B_1\in\doms(\x_1,\y_1)$. 
  \item a holomorphic curve $u_2$ in $\Hup$ with source $\Source_2$ representing homology class $B_2\in\doms(\x_2,\y_2)$
  \item a bijection $\psi\colon \AllPunct(\Source_2)\to \AllPunct(\Source_1)$
  \end{itemize}
  with the following properties:
  \begin{itemize}
  \item For each $q\in \AllPunct(\Source_2)$ marked with a Reeb orbit
    or chord, the corresponding puncture $\psi(q)\in \AllPunct(\Source_1)$
    is marked with the matching Reeb orbit or chord.
  \item For each $q\in\AllPunct(\Source_2)$,
    \[ (s\circ u_1(\psi(q)),t\circ u_1(\psi(q)))=(s\circ u_2(q),t\circ u_2(q)).\]
  \end{itemize}
  If $B_1$ and $B_2$ induce $B\in \doms(\x,\y)$, let $\ModMatched^B(\x_1,\y_1;\x_2,\y_2;\Source_1,\Source_2;\psi)$
  denote the moduli space of matched pairs.
\end{defn}

\begin{lemma}
  \label{lem:SimpleLemma}
  Assume that $\rho_1,\sigma_1$ are Reeb chords on two different
  boundary components in $\Sigma_1$ and $\rho_2,\sigma_2$ are the
  matching Reeb chords in $\Sigma_2$.  The four numbers
  $\{\alpha(\rho_1^+), \alpha(\sigma_1^+),\alpha(\rho_2^+),
  \alpha(\sigma_2^+)\}$ are distinct if and only if the four numbers
  $\{\alpha(\rho_1^-),
  \alpha(\sigma_1^-),\alpha(\rho_2^-),\alpha(\sigma_2^-) \}$ are.
\end{lemma}

\begin{proof}
  Let $\xi_1$ and $\xi_2$ be matching chords, then clearly
  \[\{\alpha(\xi_1^-),\alpha(\xi_2^-)\}=\{\alpha(\xi_1^+),\alpha(\xi_2^+)\}.\]
  The lemma follows readily.
\end{proof}

\begin{lemma}
  \label{lem:BoundaryMonotone}
  Let $(\x_1,\x_2)$ and $(\y_1,\y_2)$ be two pairs of matching states,
  and fix a matched pair $(u_1,u_2)$ connecting $\x_1\#\x_2$ to
  $\y_1\#\y_2$, with $u_1\in\ModFlow(\x_1,\y_1,\rhos_1,\dots,\rhos_m)$
  and $u_2\in\ModFlow(\x_2,\y_2,\rhos_1',\dots,\rhos_m')$ (where the
  chords in $\rhos_i$ all match, in the sense of
  Definition~\ref{def:MatchingChords} with chords in $\rhos_i'$).
  Then, both $u_1$ and
  $u_2$ are strongly boundary monotone; moreover,
  $\alpha(\x_1,\rhos_1,\dots,\rhos_\ell)$ is complementary to
  $\alpha(\x_2,\rhos_1',\dots,\rhos_\ell')$ for all $\ell=0,\dots,m$.
\end{lemma}

\begin{proof}
  We prove the following by induction on $\ell$:
  \begin{enumerate}
  \item \label{item:SBM} $(\x_1,\rhos_1,\dots,\rhos_\ell)$ and
    $(\x_2,\rhos_1',\dots,\rhos_\ell')$ are strongly boundary monotone
  \item \label{item:Disjoint} $\alpha(\x_1,\rhos_1,\dots,\rhos_\ell)$ and 
    $\alpha(\x_2,\rhos_1',\dots,\rhos_\ell')$ are disjoint.
  \end{enumerate}
  The case where $\ell=0$ is simply the hypothesis that $\x_1$ and
  $\x_2$ are matching states. Given a subset ${\mathfrak s}$,
  let \[\alpha({\mathfrak s})=\{k\in 1,\dots,2n-1\big| \alpha_k\cap
  {\mathfrak s}\neq \emptyset\}.\] Recall that $\rhos^-_i$ denotes the
  set of initial points of all the chords in $\rhos_i$.  Let
  $A_-=\alpha(\rhos_i^-)$ and $A_+=\alpha(\rhos_i^+)$.  Continuity
  ensures that $A_-\subset \alpha(\x_1,\rhos_1,\dots,\rhos_{\ell-1})$
  and $B_-\subset \alpha(\x_2,\rhos'_1,\dots,\rhos'_{\ell-1})$.  
  Weak boundary monotonicity implies that $|A_-|=|\rhos_i^-|$ and $|B_-|=|{\rhos_i'}^-|$.
  By
  the induction hypothesis, $A_-$ and $B_-$ are disjoint.  From
  Lemma~\ref{lem:SimpleLemma}, it follows that
  \[ |A_-|=|\rhos_i^-|, |B_-|=|{\rhos_i'}^-|, |A_-\cap B_-|=\emptyset 
  \Rightarrow
  |A_+|=|\rhos_i^+|, |B_+|=|{\rhos_i'}^+|, |A_+\cap B_+|=\emptyset. \]
  In particular, we have verified that $A_+$ and $B_+$ are disjoint.
  It follows at once that 
  \[ \alpha(\x_1,\rhos_1,\dots,\rhos_i)=(\alpha(\x,\rhos_1,\dots,\rhos_{\ell-1})\setminus A_-)\cup A_+ \]
  is disjoint
  from 
  \[ \alpha(\x_2,\rhos_1',\dots,\rhos_i')=(\alpha(\x_2,\rhos_1',\dots,\rhos_{\ell-1}')\setminus B_-)\cup B_+, \]
  verifying Property~\eqref{item:Disjoint} for the inductive step.

  Since $|A_+|=|A_-|$, it follows that
  \[
  |\alpha(\x_1,\rhos_1,\dots,\rhos_i)|=|\alpha(\x_1,\rhos_1,\dots,\rhos_{i-1})|-|A_+\cap
  (\alpha(\x_1,\rhos_1,\dots,\rhos_{i-1})\setminus A_-)|. \] We show that 
  $A_+\cap (\alpha(\x_1,\rhos_1,\dots,\rhos_{i-1})\setminus A_-)$ is empty; for if it were 
  non-empty, there would be some chord $\rho_1\in \rhos_i$ with
  $\rho_1^+\neq \rho_1^-$, and $\alpha(\rho_1^+)\in
  \alpha(\x_1,\rhos_1,\dots,\rhos_{i-1})$. But for any chord with
  $\rho_1^+\neq \rho_1^-$, the matching chord $\rho_2$ satisfies
  $\alpha(\rho_1^+)=\alpha(\rho_2^-)$; so $\alpha(\rho_2^-)\in
  \alpha(\x_2,\rhos'_1,\dots,\rhos'_{i-1})$. But $\alpha(\rho_2^-)$ is
  contained in both $\alpha(\x_1,\rhos_1,\dots,\rhos_{i-1})$ and
  $\alpha(\x_2,\rhos_1',\dots,\rhos_{i-1}')$, violating an inductive
  hypothesis. We conclude that $|A_+\cap
  (\alpha(\x_1,\rhos_1,\dots,\rhos_{i-1})\setminus A_-|=0$, so
  \[ |\alpha(\x_1,\rhos_1,\dots,\rhos_i)|=|\alpha(\x_1,\rhos_1,\dots,\rhos_{i-1})|. \]
  A symmetric argument shows that
  \[ |\alpha(\x_2,\rhos'_1,\dots,\rhos'_i)|=|\alpha(\x_2,\rhos_1',\dots,\rhos_{i-1}')|, \]
  verifying Property~\eqref{item:SBM} for the inductive step.
\end{proof}

Recall  that
the expected dimensions of the moduli spaces $\ModFlow^{B_i}(\x_i,\y_i;\Source_i)$ for $i=1,2$ are given by the indices
\begin{align*}
  \ind(B_1,\Source_1)&=g_1+n + 2e(B_1)-\chi(\Source_1) +  2 o_1
+ c_1-2 \weight_1 \\
  \ind(B_2,\Source_2)&=g_2+n-1 + 2e(B_2)-\chi(\Source_2) +  2 o_2
+ c_2 -2 \weight_2,
\end{align*}
where $o_i$ denotes the number of interior punctures of $\Source_i$,
$c_i$ denotes the number of boundary punctures of $\Source_i$, and
$\weight_i$ denotes the total weight of $B_i$ at the boundary.
(In comparing this formula with, for example, \cite[Equation~9.8]{InvPair},
note that the convention on the Euler measures here are different from the ones used in~\cite{InvPair}; c.f. Remark~\ref{rem:EulerMeasures}.)

\begin{defn}
  The {\em index} of a matched pair
  is defined by the formula
  \begin{align}
    \label{eq:DefIndMatchedPair}
    \ind(B_1,\Source_1;B_2,\Source_2)&=
    \ind(B_1,\Source_1)+\ind(B_2,\Source_2)-c-2o\\
    &=g+2e(B_1)+2e(B_2)-\chi(\Source_1)-\chi(\Source_2)
    + 2o+c-4\weight,
    \nonumber
  \end{align}
  where
  $g=g_1+g_2+2n-1$ is the genus of $\Hup\#\Hdown$, 
  $c=c_1=c_2$, 
  $o=o_1=o_2$, and $\weight=\weight_1=\weight_2$.
\end{defn}

We can think of the moduli space
$\ModMatched^B(\x_1,\y_1;\x_2,\y_2;\Source_1,\Source_2;\psi)$ as a
fibered product
\[ \ModMatched^B(\x_1,\y_1;\x_2,\y_2;\Source_1,\Source_2;\psi)
  = \ModFlow^{B_1}(\x_1,\y_1,\Source_1)\times_\ev \ModFlow^{B_2}(\x_2,\y_2,\Source_2)\]
  where $\times_\ev$ denotes the fibered product over the evaluation maps
  at the punctures
  \[
  \ev_1\colon \ModFlow^{B_1}(\x_1,\y_1,\Source_1)\to ([0,1]\times
  \R)^k \qquad \ev_2\colon \ModFlow^{B_2}(\x_2,\y_2,\Source_2)\to
  ([0,1]\times \R)^k,\] where
  $k=|\AllPunct(\Source_1)|=|\AllPunct(\Source_2)|=c+o$.  Then, the index of the
  moduli space is the expected dimension of the moduli space of
  matched pairs inherited from its description as a fibered product
  over $([0,1]\times \R^k)$.

The moduli space $\ModMatched(\x_1,\y_1;\x_2,\y_2;\Source_1,\Source_2;\psi)$ comes with an $\R$ action
which is free except in the special case where both sides consist of trivial strips.
The quotient is denoted by
\[ \UnparModMatched(\x_1,\y_1;\x_2,\y_2;\Source_1,\Source_2;\psi)
=  \ModMatched(\x_1,\y_1;\x_2,\y_2;\Source_1,\Source_2;\psi)/\R. \]

\begin{lemma} 
  \label{lem:Transversality}
  Fix $B_1\in\doms(\x_1,\y_1)$ and $B_2\in\doms(\x_2,\y_2)$ so that
  $\weight_i(B_1)=\weight_i(B_2)$ for $i=1,\dots,2n$, and at least one of
  $n_{\wpt}(B_1)$ or $n_{\zpt}(B_1)$ vanishes.
  For generic admissible almost complex
  structures on $\Sigma_i\times [0,1]\times \R$, and
  $\ind(B_1,\Source_1;B_2,\Source_2)\leq 2$, the moduli space of
  matched pairs
  \[ \ModMatched^B(\x_1,\y_1;\x_2,\y_2;\Source_1,\Source_2;\psi)\] is
  transversely cut out by the $\dbar$-equation and the evaluation map;
  in particular, this moduli space is a manifold whose dimension is
  given by Equation~\eqref{eq:DefIndMatchedPair}.  
\end{lemma}
\begin{proof}
  This is is essentially~\cite[Lemma~9.4]{InvPair}.
\end{proof}

The data $(\Source_1,\Source_2,\psi\colon
\AllPunct(\Source_2)\to\AllPunct(\Source_1))$ of a matched pair can be
used to form a new source curve $\Source_1\natural_{\psi}\Source_2$,
obtained by gluing punctures to punctures (i.e. connected sum for
orbits and boundary connected sum for boundary punctures). If
$B_1\in\doms(\x_1,\y_1)$, $B_2\in\doms(\x_2,\y_2)$ are represented by
a matched pair, we can construct $B=B_1\natural B_2\in\doms(\x,\y)$,
where $\x=\x_1\#\x_2$ and $\y=\y_1\#\y_2$, represented by a (not
necessarily holomorphic) curve with source
$\Source=\Source_1\natural_\psi\Source_2$.

It is elementary to see that
\begin{align}
  e(B_1\natural B_2)&=e(B_1)+ e(B_2)-2 \weight \label{eq:EulerMeasures} \\
  \chi(\Source_1\natural_\psi \Source_2)&=\chi(\Source_1)+\chi(\Source_2) - 2o-c.
  \label{eq:EulerCharacteristics}
\end{align}
It is an easy consequence that
\[ \ind(B_1\natural
B_2,\Source_1\natural_\psi\Source_2)=\ind(B_1,\Source_1;B_2,\Source_2),\]
identifying expected dimension of the moduli space of curves in
$\Source_1\natural_\psi\Source_2$ with expected dimension of the
moduli spaces of matched pairs. Our goal is to refine this to an
identification of moduli spaces.

As in Definition~\ref{eq:EmbedMod}, let
\begin{align*}
  \chiEmb(B)&=g+e(B)-n_\x(B)-n_\y(B) \\
  \ind(B)&=e(B)+n_\x(B)+n_\y(B).
\end{align*}

\begin{prop}
  \label{prop:EmbeddedModuliSpaces}
  Fix $\x=\x_1\#\x_2$ and $\y=\y_1\#\y_2$, and decompose
  $B\in\doms(\x,\y)$ as $B=B_1\natural B_2$, with
  $B_i\in\doms(\x_i,\y_i)$.  Fix source curves $\Source_1$ and
  $\Source_2$ together with a one-to-one correspondence $\psi\colon
  \AllPunct(\Source_2)\to\AllPunct(\Source_1)$ which is consistent with the chord
  and orbit labels, so we can form $\Source=\Source_1\natural_\psi\Source_2$.
  Suppose that $\ModFlow^B(\x,\y;\Source)$ 
  (i.e. the moduli space for curves in $\HD$) and
  $\ModFlow^{B_i}(\x_i,\y_i;\Source_i)$
  (which are moduli spaces for curves in $\HD_i$) 
  are non-empty for $i=1,2$.
  Then, 
  $\chi(\Source)=\chiEmb(B)$ if and only if
  $\chi(\Source_i)=\chiEmb(B_i)$ for $i=1,2$; 
  and all the chords in $\AllPunct(\Source_1)$ have weight $1/2$ and
  all the orbits have length $1$.
\end{prop}

\begin{proof}  
  Let $\weight$ denote the total weight of $B_1$ at the boundary,
  $o$ denote the number of orbits in $\Source_1$, and $c$ the number of chords.
  Since $g=g_1+g_2+2n-1$, it follows immediately 
  from Equation~\eqref{eq:EulerMeasures} that
  \[ \chiEmb(B_1\natural B_2)=\chiEmb(B_1)+\chiEmb(B_2)-2\weight;\] so using
  Equation~\eqref{eq:EulerCharacteristics}, it follows that
  \[ 
  \chi(\Source)-\chiEmb(B)
    = \left(\chi(\Source_1)-\chiEmb(B_1)\right) 
    + \left(\chi(\Source_2)-\chiEmb(B_2)\right)
  + 2\weight-c-2o.\]
  Clearly, $2\weight-c-2o\geq 0$, with equality iff each orbit has weight $1$ and each chord has weight $1/2$.
  Also, the hypotheses that
  $\ModFlow^B(\Source)\neq \emptyset$ and
  $\ModFlow^{B_i}(\Source_i)\neq \emptyset$ for $i=1,2$ ensure that 
  \[ \chi(\Source)\geq \chiEmb(B),\qquad \chi(\Source_i)\geq \chiEmb(B_i) \]
  with equality iff 
  the curves are embedded.
\end{proof}

The Gromov compactification
of the space of matched curves is provided by a space of matched combs,
which we define presently.

Let $(w_{\ell_1},\dots,w_1,u,v_1,\dots,v_k)$ be a story with total
source ${\overline\Source}$.  Let $\East({\overline\Source})$ be the
eastmost punctures in ${\overline\Source}$ (i.e. these are the East
punctures on curves that are not matched with west punctures on other
curves at East infinity), $\IntPunct({\overline\Source})$ be all the
orbit-marked punctures in all the components of ${\overline\Source}$,
and
\[ \AllPunct({\overline\Source})=\East({\overline\Source})\cup\IntPunct({\overline\Source}).\]

\begin{defn}
  \label{def:dMatched}
  Given $\x_1,\y_1\in\States(\Hdown)$ and $\x_2,\y_2\in\States(\Hup)$,
  a {\em matched story from $\x=\x_1\#\x_2$ to $\y=\y_1\#\y_2$} consists of the following data:
  \begin{itemize}
    \item a 
      pair of holomorphic stories
      \[ {\overline
        u}_1=(w_{\ell_1}^1,\dots,w_1^1,u^1,v_1^1,\dots,v_{k_1}^1)
      \qquad{\text{and}} \qquad {\overline
        u}_2=(w_{\ell_2}^2,\dots,w_1^2,u^2,v_1^2,\dots,v_{k_2}^2), \]
    \item 
      a one-to-one correspondence
      \[ \psi\colon \AllPunct({\overline\Source}_2)\to
      \AllPunct({\overline\Source}_1) \] so that for all $p\in
      \AllPunct({\overline\Source}_2)$, the Reeb chord or orbit marking
      $p$ and $\psi(p)$ have matching labels.
    \end{itemize}
    satisfying the following condition for each $q\in \AllPunct({\overline\Source}_2)$:
    \[ (s\circ {\overline u}_2(q),t\circ {\overline u}_2(q))=
    (s\circ {\overline u}_1(\psi(q)),t\circ {\overline u}_1(\psi(q))).\]
    for all $q\in\AllPunct({\overline \Source}_2)$.
    The story is called {\em stable} if there are no unstable east or 
    west infinity curves on either side, and either $u^1$ or $u^2$ is stable.
\end{defn}

\begin{defn}
A matched comb of height $N$ is a sequence of stable 
matched stories running from
$(\x_1^{j},\x_2^{j})$ to $(\x_1^{j+1},\x_2^{j+1})$
for sequences of lower states $\{\x_1^j\}_{j=1}^{N+1}$ (for $\Hdown$)
and upper states $\{\x_2^j\}_{j=1}^{N+1}$ (for $\Hup$).
\end{defn}

In principle, the Gromov compactification could contain more
complicated objects: closed components contained in some story, or
$\alpha$-boundary degenerations. We exclude these possibilities in the
next two lemmas.

\begin{lemma}
  \label{lem:NoBoundaryDegenerations}
  Fix $B_1\in\doms(\x_1,\y_1)$ and $B_2\in\doms(\x_2,\y_2)$ so that
  $\weight_i(B_1)=\weight_i(B_2)$ for $i=1,\dots,2n$, and at least one of
  $n_\wpt(B_1)$ or $n_\zpt(B_1)$ vanishes, and so that $\ind(B_1\natural
  B_2)\leq 2$.  Then, curves in the Gromov compactification of
  $\ModMatched^B(\x_1,\y_1;\x_2,\y_2)$ contain no
  boundary degenerations.
\end{lemma}
  
\begin{proof}
  Fix a curve $({\overline u}_1,{\overline u}_2)$ denote the matched
  comb in the Gromov compactification of matched curve pairs, so that
  ${\overline u}_1$ denotes the portion in $\Hdown$ and ${\overline
    u}_2$ in $\Hup$.
  
  We first prove that neither ${\overline u}_1$ nor ${\overline u}_2$
  can contain a $\beta$-boundary degeneration. This is equivalent to
  showing that ${\overline u}_i$ does not contain a puncture $p_i$
  marked with a Reeb orbit, and with the property that
  $s({\overline u}_i(p_i))=0$.
  
  To see this, we prove first that if there is a puncture $p$ marked
  by a Reeb orbit in ${\overline u}_i$ with $\pi_\CDisk(p)=(0,\tau)$,
  then in fact for each $j=1,\dots,2n$, there are corresponding
  punctures $q_1$ and $q_2$ in ${\overline u}_1$ and ${\overline u}_2$
  with
  \[ \pi_{\CDisk} ({\overline u}_1(q_1))=\pi_{\CDisk}(u_2(q_2))=(0,\tau),\] so that $q_1$ is marked by some orbit that
  covers $\Zin_j$ and $q_2$ is marked by some orbit that covers
  $\Zout_j$.  This follows from the following observations:
  \begin{itemize}
    \item 
      The curve ${\overline u}_1$ contains a puncture $q_1$
      marked by Reeb orbit that covers $\Zin_j$
      with $\pi_{\CDisk}u(q)=(0,\tau)$ if and only ${\overline u}_2$ 
      contains a puncture $q_2$ marked by a Reeb orbits that covers $\Zout_j$
      with $\pi_{\CDisk}u(q)=(0,\tau)$. This is immediate from
      the matching condition.
    \item If ${\overline u}_1$ contains a puncture $q$ marked by a
      Reeb orbit that covers $\Zin_j$ with $\pi_{\CDisk}u(q)=(0,\tau)$,
      and $\{j,k\}\in\Mdown$, 
      then there is another puncture $q'$ marked by a Reeb orbit that
      covers $\Zin_k$.
      This follows at once from the fact that  $q$ is contained
      in a $\beta$-boundary
      degeneration component.
    \item If ${\overline u}_2$ contains a puncture $q$ marked
      by a Reeb orbit that covers $\Zout_j$ with $\pi_{\CDisk}u(q)=(0,\tau)$,
      then it also contains another puncture $q'$
      marked by a Reeb orbit that covers $\Zout_k$ with
      $\pi_{\CDisk}u(q')=(0,\tau)$, where $\{j,k\}\in\Mup$. 
      This follows as above.
    \end{itemize}
    Compatibility of $\Mup$ and $\Mdown$ now allows us to conclude
    that statement at the beginning of the paragraph.  Moreover, from
    that statement it follows that if ${\overline u}_1$ or ${\overline
      u}_2$ contains any boundary degeneration, then in fact it
    contains some boundary degeneration that contains $\wpt$ and
    another boundary degeneration that contains $\zpt$. But this
    violates the hypothesis that $n_\wpt(B_1)=0$ or
    $n_\zpt(B_1)=0$.

  We turn to the possibility of a $\Ta$ boundary degeneration.  Since
  $B_1$ does not cover both $\wpt$ and
  $\zpt$, our homological hypotheses on the type $A$ side ensures that
  the limiting curve contains no $\alpha$-boundary degenerations on
  the $\Hdown$ side. 

  It remains to consider the possibility that
  an $\alpha$-boundary degenerations that occurs on
  the $\Hdown$ side, with projecting to $(1,\tau)$ for some $\tau\in\R$.

  The chords and orbits in the boundary degeneration in
  ${\overline{u}_2}$ can match with chords and orbits on various East
  infinity curves in ${{\overline u}_1}$. Nonetheless, the
  $(s,t)$-projections of all of these chords and orbits are
  $(1,\tau)$, for some fixed $\tau\in\R$, by the matching
  condition. Moreover, the weight at each boundary component $Z_i$ of
  these chords in ${{\overline u}_1}$ must be at least $1$ (since the
  same is true for each boundary degeneration in $\Hdown$).  Consider
  now the main component $u_1$ of ${{\overline u}_1}$, obtained after
  removing East infinity curves.  Note that the total weights of the
  Reeb chords and orbits of an East infinity curve are the same at
  both their Eastern and their Western boundary; it follows that $u_1$
  has a chord packet $\rhos$ at $(s,t)=(1,\tau)$ with the property
  that $\weight_i(\rhos)\geq 1$ for all $i=1,\dots,2n$. Observe also the
  main component cannot have any orbits whose $s$-projection is
  $1$. It follows that $\rhos^-$ contains points on all of the
  boundary components of $\Sigma_1$ ; i.e. $\rhos^-$ contains at least
  $2n$ points. But this violates boundary monotonicity: $\rho^-$ can
  contain at most $n$ points.
\end{proof}

\begin{lemma}
  \label{lem:NoClosedCurves}
  Suppose that $n>1$.
  Fix $B_1\in\doms(\x_1,\y_1)$ and $B_2\in\doms(\x_2,\y_2)$ so that
  $\weight_i(B_1)=\weight_i(B_2)$ for $i=1,\dots,2n$, and at least one of
  $n_\wpt(B_1)$ or $n_\zpt(B_1)$ vanishes, and so that $\ind(B_1\natural
  B_2)\leq 2$.    Curves in the Gromov compactification of
  $\ModMatched^B(\x_1,\y_1;\x_2,\y_2)$ contain no
  closed components.
\end{lemma}

\begin{proof}
  By our hypotheses on the homology class, homologically non-trivial 
  closed components cannot occur
  on the $\Hdown$ side. 
  
  Suppose there is a curve in the Gromov compactification of
  $\ModMatched^B(\x_1,\y_1;\x_2,\y_2)$ which has a single closed
  component in $\Hup$, and it has multiplicity $k>0$.  Let $(u_1',u_2')$
  be the main components of this limit, so that
  $u_1'\in\ModFlow^{B_1}(\x_1,\y_1;\Source_1')$ and
  $u_2'\in\ModFlow^{B_2-k[\Sigma_2]}(\x_2,\y_2;\Source_2')$.  We will
  show now that $(u_1',u_2')$ lies in a moduli space whose expected
  dimension $\ind(u_1',u_2')$ satisfies
  \begin{equation}
    \label{eq:ConstrainedModSpaceDim}
    \ind(u_1',u_2')\leq \ind(B_1\natural B_2)+2-2nk.
  \end{equation}

  Let $o_i'$ denote the number of orbits in $\Source_i'$,
  $c_i'$ denote the number of chords in $\Source_i$, 
  $\weight$ denote $\weight_{\partial}([B_1])=\weight_{\partial}([B_2])$,
  $\weight_i'$ denote $\weight_{\partial}([B_i'])$.
  Since there are no boundary degenerations, we conclude that $B_1=B_1'$,
  and in particular  $\weight_1'=\weight$.
  Also, the matching conditions ensure that $c_1'=c_2'$.
  If $\delta$ is the number of interior punctures in $\Source_1'$ that match with punctures in the (removed) closed component, then
  $o_1'=o_2'+\delta$.

  The pair $(u_1',u_2')$ lies in a moduli space whose expected dimension is given by
  \begin{align*}
    \ind(u_1',u_2')&=\ind(B_1,\Source_1')+\ind(B_2,\Source_2)-2(\delta-1)-2 o_2'-c_2'\\
    &=\ind(B_1,\Source_1')+\ind(B_2,\Source_2)-2o_1'-c_1'+2
  \end{align*}
  To see why, note that the $\delta$ orbits in 
  $\Source_1'$ are all required to have the same $(s,t)$ projection
  (hence the term $2(\delta-1)$); the additional $2o_2'+c_2'$ constraints come
  from the orbits and chords in $\Source_2'$, which project to the same
  $(s,t)$ values as corresponding orbits and chords in $\Source_1'$.

  It is elementary to see that
  \begin{align*}
    e([\Sigma_2])+n_{\x_2}([\Sigma_2])+n_{\y_2}([\Sigma_2])&=
    2-2g_2 + 2 (g_2 + n-1)=2n.
  \end{align*}
  Since $[B_1']=[B_1]$, the index formula gives
  \begin{align*}
    \ind(u_1',u_2')-\ind(B_1\natural B_2)
    &= e(B_1')+n_{\x_1}(B_1)+n_{\y_1}(B_1)-e(B_2)-n_{\x_2}(B_2)-n_{\y_1}(B_2) \\
    &\qquad
    -2\weight_1'-2\weight_2'+2\weight
    +2o_2'+c_2'+2 \\
    &=  -2nk -2\weight_2'+2o_2'+c_2'+2 \\
    &\leq -2nk+2,
  \end{align*}
  using the obvious inequality $-2\weight_2'+2o_2'+c_2'\leq 0$;
  i.e. Inequality~\ref{eq:ConstrainedModSpaceDim} holds.

  Our hypothesis that $\ind(B_1\natural B_2)=2$ ensures that
  $(u_1,u_2')$ lives in a moduli space ${\mathcal M}$ with a free $\R$
  action (since $u_1'$ contains orbits, so it cannot be constant) and
  $\dim({\mathcal M})=4-2nk$. Since $n>1$ and $k\geq 1$, we conclude that this
  moduli space is empty.

  We have ruled out homologically non-trivial closed components.
  Homologically trivial ``ghost'' components are also ruled out
  by the dimension formula, as in~\cite[Lemma~5.57]{InvPair}.
\end{proof}

\begin{remark}
  When $n=1$, spheres can occur in zero-dimensional moduli spaces, so
  that case must be handled separately; see
  Section~\ref{subsec:Nequals1}.
\end{remark}

\begin{defn}
  Given a shadow $B\in \doms(\x_1\#\y_1,\x_2\#\y_2)$,
  the {\em embedded matched moduli space} 
  $\ModMatched^B(\x_1,\y_1;\x_2,\y_2)$ is the union of all
  $\ModMatched^B(\x_1,\y_1;\x_2,\y_2;\Source_1,\Source_2;\psi)$ 
  taken over all compatible pairs $\Source_1$ and $\Source_2$
  with $\chi(\Source_1\natural\Source_2)=\chiEmb(B)$.
  Moreover,
  \[\UnparModMatched^B(\x_1,\y_1;\x_2,\y_2)=
  \ModMatched^B(\x_1,\y_1;\x_2,\y_2)/\R.\]
\end{defn}

\begin{prop}
  \label{prop:StretchNeck}
  We can find a generic almost-complex structures
  on $\Hup$, $\Hdown$, and $\HD=\Hdown\#_Z\Hup$
  with the following property. For each $\x=\x_1\#\x_2$
  and $\y=\y_1\#\y_2\in\States(\HD)$, $B\in\doms(\x,\y)$
  and $\ind(B)=1$, 
  with at least one of $n_{\wpt}(B)$ or $n_{\zpt}(B)$ vanishing,
  there is an identification of moduli spaces
  of curves in $\HD$ with matched curves:
  \[ \UnparModFlow^B(\x_1\# \x_2,\y_1\#\y_2)\cong
  \UnparModMatched^B(\x_1,\y_1; \x_2,\y_2).\]
\end{prop}

\begin{proof}
  With Lemmas~\ref{lem:NoBoundaryDegenerations}
  and~\ref{lem:NoClosedCurves} in place, this is a standard gluing
  argument; see~\cite[Proposition~9.6]{InvPair}.
\end{proof}

\begin{defn}
  \label{def:MatchComplex}
  Let $\HD=\Hup\#_Z\Hdown$ be a decomposition of a 
  doubly-pointed Heegaard diagram for a knot in $S^3$,
  and assume that $n>1$.
  The {\em chain complex of matched curves} is the pair $(C,\partial^{(0)})$,
  where, 
  as before, $C$ denotes the free $\Ring$-module generated by Heegaard states 
  for $\HD$ or, equivalently, matching pairs of states (as in Definition~\ref{def:MatchingPair}).
  The operator
  \[\partial^{(0)}\colon C\to C \]
  is the $\Ring$-module endomorphism specified by
  \begin{equation}
    \label{eq:DefD0}
    \partial^{(0)}(\x_1\#\x_2)=\sum_{(\y_1,\y_2)} \sum_{\{B\mid \ind(B)=1\}}
  \#\left(\UnparModMatched^B(\x_1,\y_1;\x_2,\y_2)\right)\cdot
  U^{n_\wpt(B)}V^{n_\zpt(B)} \cdot \y_1\#\y_2.
  \end{equation}
\end{defn}

\subsection{The chain complex of matched curves}
\label{subsec:MatchedComplex}

Although it is not technically necessary (it can be thought of as a
consequence of Proposition~\ref{prop:StretchNeck}, we include here a proof
that $\partial^{(0)}$ induces a differential. The methods appearing in
the proof will appear again in the proof of
Theorem~\ref{thm:PairAwithD}.

The following is a slight adaptation of~\cite[Proposition~9.16]{InvPair}:

\begin{lemma}
  \label{lem:dMatchedCompactify}
  Suppose that $\ind(B_1,\Source_1;B_2,\Source_2)=2$, then for generic
  $J$, the moduli space
  $\ModMatched^{B}(\x_1,\y_1;\x_2,\y_2;\Source_1,\Source_2)$ can be
  compactified by adding the following objects:
  \begin{enumerate}[label=(ME-\arabic*),ref=(ME-\arabic*)]
  \item two-story matched holomorphic curves
    (i.e. neither story contains curves at East or
    West infinity)
  \item
    \label{def:JJ}
    matched stories $(u^1,v^1)$ and $(u^2,v^2)$
    where the only
    non-trivial components of $v_1$ and $v_2$ are join curves, and
    where $\West(v^1)$ and $\West(v^2)$ are the two distinct length one Reeb chords
    that cover the same boundary component,
  \item 
    \label{def:OO} matched stories $(u^1,v^1)$ and $(u^2,v^2)$ where the only
    non-trivial components of $v^1$ and $v^2$ are orbit curves, and
    $\West(v^1)$ and $\West(v^2)$ are the two distinct length one Reeb chords that
    cover the same boundary component.
  \end{enumerate}
\end{lemma}

 \begin{figure}[h]
 \centering
 \input{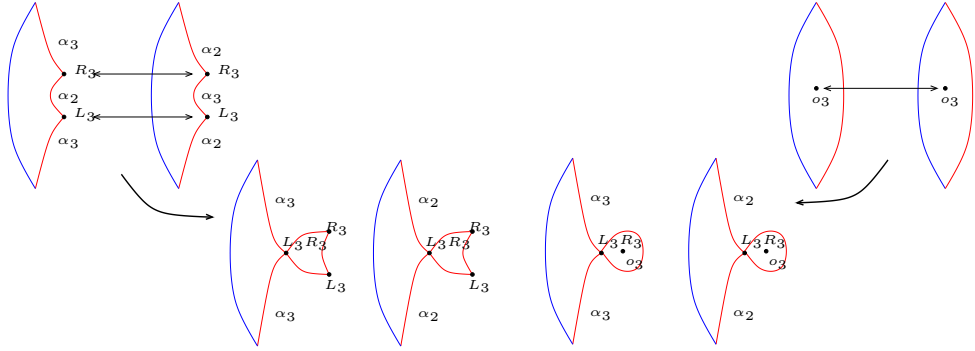}
 \caption{{\bf Cancellation of join ends with orbit ends.}}
 \label{fig:JoinOrbit}
 \end{figure}

\begin{proof}
  The main components are strongly boundary monotone by the argument
  from Lemma~\ref{lem:BoundaryMonotone}. This rules out boundary
  double points, as in~\cite[Lemma~5.56]{InvPair}.  Boundary
  degenerations are ruled out as in
  Lemma~\ref{lem:NoBoundaryDegenerations}.  Dimension considerations
  rule out closed components as in Lemma~\ref{lem:NoClosedCurves}.
  
  We follow proof of~\cite[Proposition~9.16]{InvPair}.  Suppose there
  is some matched story that appears in the boundary of the moduli
  space. We could write that matched story as a pair of stories with a
  matching condition on the eastmost punctures; 
  instead we combine the curves at East infinity
  into a single sequence, writing, more symmetrically,
  $(u_1,v_1,\dots,v_\ell,u_2)$. Let
  $(\Source_1',\EastSource_1,\dots,\EastSource_\ell,\Source_2')$
  denote the corresponding sequence of source curves.  Let
  $\EastSource=\EastSource_1\natural\dots\natural\EastSource_\ell$
  (where we glue along all boundary punctures and all interior
  punctures, as well).  There is an induced partition $P_1$ on the
  east punctures of $\Source_1$: two punctures are in the same
  partition if they are assigned to same components of
  $\EastSource$. There is a similar partition $P_2$ on the east
  punctures of $\Source_2$.  Combining Equations~\eqref{eq:DefIndMatchedPair}
  and~\eqref{eq:EulerMeasures}, we see that
  \[ \ind(B_1,\Source_1;B_2,\Source_2)=g-\chi(\Source_1)-\chi(\Source_2)+2e(B)+c+2o,\]
  where $c=|\East(\Source_1)|=|\East(\Source_2)|$
  and $o=|\IntPunct(\Source_1)|=|\IntPunct(\Source_2)|$.
  The limit curves
  $(u_1,u_2)$ live in a fibered product $\ModFlow$ whose dimension (after dividing out by the free $\R$ action) is
  given by
  \begin{align*}
    \dim(\ModFlow)&=\ind(B_1,\Source_1',P_1)+\ind(B_2,\Source_2',P_2)-k-1-2o', \\
    &= g-\chi(\Source_1')-\chi(\Source_2')+2e(B)+k-1+2o';
  \end{align*}
  where $k$ is the number of components of $\EastSource$, which agrees
  with $|P_1|=|P_2|$; and $o'=|\IntPunct(\Source_1')|=|\IntPunct(\Source_2')|$ 
  is the number of interior punctures in $\Source_1'$. By elementary topology,
  \[ \chi(\Source_1)+\chi(\Source_2)-c-2o=\chi(\Source_1')+\chi(\Source_2')+\chi(\EastSource)-c_1-c_2-2o',\]
  where $c_i=|\East(\Source_i')|$ for $i=1,2$.
  It follows that
  \[ \dim(\ModFlow)=\ind(B_1,\Source_1;B_2,\Source_2)+(\chi(\EastSource)-k)+(k-c_1)+(k-c_2)-1.\]
  Since $\chi(\EastSource)\leq k$ and $k\leq c_i$ for $i=1,2$, we conclude that
  if $\dim(\ModFlow)\geq 0$ and $\ind(B_1,\Source_1;B_2,\Source_2)=2$,
  then exactly one of the following three possibilities can occur:
  \begin{enumerate}[label=(d-\arabic*),ref=(d-\arabic*)]
  \item 
    \label{case:SplitJoin}
    all the components of $\EastSource$ are disks, and 
    $\{c_1,c_2\}=\{k,k+1\}$
  \item 
    \label{case:Annulus}
    $k=c_1=c_2$, exactly one component in $\EastSource$ is an annulus,
    and all other components are trivial strips.
  \end{enumerate}
  
  Case~\ref{case:SplitJoin} is excluded, as follows. Suppose that
  $c_1=k+1$ and $c_2=k$, so that there are two punctures on $\Source_1'$
  labelled by two chords $\rho_1$ and $\rho_2$ (for $\Hdown$)
  that are matched
  by $\EastSource$; and $\rho_1\uplus\rho_2$ is the corresponding
  chord in $\Source_2'$. Both chords $\rho_i$ have weight $1/2$, so
  clearly $\{\rho_1^-,\rho_2^-\}$ are the two $\alpha$-curves on some
  fixed boundary component, while $\rho_1\uplus\rho_2^-$ is also an
  $\alpha$-curve on the corresponding boundary component in
  $\Hup$. But this violates Lemma~\ref{lem:BoundaryMonotone},
  according to which the sets of $\alpha$-occupied curves are complementary.
  The case where $c_1=k$ and $c_2=k+1$ is excluded the same way.
  
  Consider next Case~\ref{case:Annulus}. Drop all the trivial strip components,
  letting $\EastSource_0$ be the annulus. Dropping all unstable components
  from east infinity, we have 
  $\EastSource_0=\EastSource_1\natural\dots\natural\EastSource_\ell$. 
  Obviously, we cannot have $\ell=1$: for an orbit at East infinity cannot match
  with an interior puncture. Consider next $\ell=2$.
  This can occur in two ways. Suppose
  $\EastSource=\EastSource_1\natural\EastSource_2$.  We have that
  \[ 0=\chi(\EastSource)=\chi(\EastSource_1)+\chi(\Source_2)-2o-c;\]
  where $o$ resp. $c$ denotes the number of interior resp.  boundary
  punctures in $\EastSource_1$ (which are matched with corresponding
  punctures in $\EastSource_2$).  Clearly, this forces $\EastSource_1$
  and $\EastSource_2$ to be disks, with either $o=0$ and $c=2$ or
  $o=1$ and $c=0$. The first case is Case~\ref{def:JJ}, and the second
  is Case~\ref{def:OO}.  The same Euler characteristic considerations
  show that decompositions with $\ell>2$ have unstable components.

  It remains to consider the case where there is more than one story.
  In that case, by the additivity of the index, it follows that there are 
  two stories, both have index $1$, and neither has curves at East infinity.
\end{proof}

\begin{remark}
  In~\cite[Proposition~9.16]{InvPair}, Case~\ref{case:Annulus} is
  ruled by the existence of a basepoint on the boundary, while
  Case~\ref{case:SplitJoin} does occur,
  unlike in the above proof, where it is ruled out by combinatorics
  of the Reeb chords. 
\end{remark}

\begin{lemma}
\label{lem:Join}
Let $(u^1,v^1)$ and $(u^2,v^2)$ be a matched story with
sources $(\Source_1,\EastSource_1)$ and $(\Source_2,\EastSource_2)$
with the property that for $i=1,2$, 
the only non-trivial component of $v^i$  is a
join curve forming a length one Reeb chord, then there are arbitrary
small open neighborhoods
$U$ of $(u^1,v^1)\times (u^2,v^2)$
in $\ModFlow^{B_1}(\x_1,\y_1,\Source_1\natural\EastSource_1)\times
\ModFlow^{B_2}(\x_2,\y_2,\Source_2\natural\EastSource_2)$
so that
$\partial{\overline U}$ 
meets 
$\ModFlow^B(\x_1,\y_1,\Source_1\natural\EastSource_1;
\x_2,\y_2;\Source_2;\Source_2\natural\EastSource_2)$ in an odd number of points.
The same conclusion holds if for each $i=1,2$, the only non-trivial component
of $v^i$ is an orbit curve.
\end{lemma}

\begin{proof}
  If $(u_1,v_1)$ and $(v_2,v_2)$ is a matched story where the
  only non-trivial components of $v^1$ and $v^2$ are join curves,
  the result follows from Proposition~\ref{prop:JoinCurve}.
  When $v^1$ and $v^2$ are orbit curves, the result follows from
  Proposition~\ref{prop:OrbitCurve}.
  See Figure~\ref{fig:DSquaredZero}.
\end{proof}

 \begin{figure}[h]
 \centering
 \input{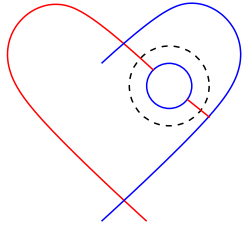}
 \caption{{\bf Cancellation of join ends with orbit ends; an example.}}
 \label{fig:DSquaredZero}
 \end{figure}

\begin{prop}
  \label{prop:dMatchedZqZero}
  The endomorphism $\partial^{(0)}$ is a differential.
\end{prop}

\begin{proof}
  Fix $\x_1\#\x_2, \y_1\#\y_2\in\States(\Hdown\#\Hup)$ and $B\in
  \pi_2(\x_1\#\x_2,\y_1\#\y_2)$ with $\ind(B)=2$.  Consider the ends
  of $\ModFlow^B(\x_1,\x_2;\y_1,\y_2)$, consisting of embedded,
  matched holomorphic curves representing $B$. By
  Lemma~\ref{lem:dMatchedCompactify}, ends of these moduli spaces are
  either two-story matched curves; of they correspond to
  matched combs, with height two, of two types (with join
  curves or orbit curves at East infinity).  These two types of ends
  cancel in pairs according to Lemma~\ref{lem:Join}; so the total
  number of two-story ends is even, and those count the $\y_1\#\y_2$
  coefficient of
  $\partial^{(0)}\circ \partial^{(0)}(\x_1\#\x_2)$.
\end{proof}

We identify $(C,\partial^{(0)})$ with the Heegaard-Floer chain
complex $\CFKsimp(\HD)$ associated to the doubly-pointed Heegaard
diagram $\HD$, via the following adaptation
of~\cite[Theorem~9.10]{InvPair}:

\begin{thm}
  \label{thm:NeckStretch}
  Let $\HD$ be a Heegaard diagram representing $K$, equipped with a
  decomposition $\HD=\Hdown\cup_Z \Hup$ as a union of an upper and a
  lower diagram along $2n$ circles, with $n>1$.
  For suitable choices of almost-complex structures $J$ used to define
  $\CFKsimp(K)$, there is an isomorphism (of chain complexes)
  $\CFKsimp(K)\cong (C,\partial^{(0)})$.
\end{thm}

\begin{proof}
  This follows from Lipshitz's reformulation of Heegaard Floer
  homology (\cite[Theorem~2]{LipshitzCyl}) and
  Proposition~\ref{prop:StretchNeck}.
\end{proof}

\subsection{Self-matched curves}
\label{subsec:SelfMatched}

\begin{defn}
  \label{def:SelfMatched}
  Let $\Source_1$ be a decorated source.  Partition
  let $\IntPunctEv(\Source_1)$ resp. $\IntPunctOdd(\Source_1)$
  denote the set of interior punctures of $\Source_1$
  that are labelled by even resp. odd orbits.
  A {\em self-marked source} is a decorated
  source, together with an injection $\phi\colon \IntPunctEv(\Source_1)\to
  \East(\Source_1)$ with the property that if $p\in\IntPunctEv(\Source_1)$
  is marked by some orbit $\orb_j$, then $\phi(p)$ is marked by a length
  one chord that covers the boundary component $Z_k$ with
  $\{j,k\}\in \Mup$. A {\em self-matched curve} $u$ is an element
  $u\in\ModFlow^{B_1}(\x,\y;\Source_1,\phi)$, subject to the following additional
  constraints: for each puncture $p\in\IntPunctEv(\Source_1)$,
  \begin{equation}
    \label{eq:SelfMatching}
    t\circ u(p)=t\circ u(\phi(p)).
  \end{equation}
\end{defn}

Note although ``self-matched curves'' are supported in $\Hdown$, the
matching condition depends on $\Hup$, through its induced matching
$\Mup$.

\begin{defn}
  \label{def:SelfMatchedCurvePair}
  A {\em self-matched curve pair} consists of the following data:
  \begin{itemize}
  \item a self-matched source
    $(\Source_1,\phi\colon \IntPunctEv(\Source_1)\to \East(\Source_1))$
  \item a marked source $\Source_2$
  \item an injection $\psi\colon \AllPunct(\Source_2)\to
    \East(\Source_1)$, where
    $\AllPunct(\Source_2)=\IntPunct(\Source_2)\cup\East(\Source_2)$
  \item a pseudo-holomorphic self-matched curve $u_1$ with source 
    $\Source_1$
    in $\Hdown$ and 
  \item a pseudo-holomorphic curve $u_2$ with source $\Source_2$ in $\Hup$,
  \end{itemize}
  satisfying the following properties
  \begin{enumerate}[label=(smpc-\arabic*),ref=(smcp-\arabic*)]
  \item 
    \label{smcp:1to1}
    The map $\psi$ is a one-to-one correspondence between
    $\East(\Source_2)$ and $\East(\Source_1)\setminus
    \phi(\IntPunctEv(\Source_2))$.
  \item If $p\in \IntPunct(\Source_2)$, is labelled by an 
    orbit $\orb_j$, then $\phi(p)$ is labelled by 
    a length one Reeb chord that covers the boundary component of $\Zin_j$.
  \item If $p\in\East(\Source_2)$, then
    the name of the Reeb chord in $\Hup$ marking $p$
    is the same as the name of the Reeb chord in $\Hdown$ marking $\psi(p)$.
  \item For  each puncture $q\in \AllPunct(\Source_2)$:
    \begin{equation}
      \label{eq:SelfMatchedCurvePair}
      t\circ u_1(\psi(q))=t\circ u_2(q).
    \end{equation} 
  \end{enumerate}
  Let
  $\ModMatchedChanged^{B_1,B_2}(\x,\y;\Source_1,\Source_2,\phi,\psi)$
  denote the moduli space of self-matched curve pairs. 
\end{defn}

Note that the even orbits in $\Source_1$ are constrained by the
self-matching condition; and the odd orbits in $\Source_1$ are not constrained by any additional condition.

Let $(u_1,u_2)$ be a self-matched curve pair with sources $\Source_1$
and $\Source_2$. Let $o_1^+$ resp. $o_1^-$ denote the number of even
resp. odd orbits in $\Source_1$; let $c_i$ denote the number of
boundary punctures in $\Source_i$. For example, note that
$c_1=c_2+o_2+o_1^+$. 

Given homology classes $B_1\in\pi_2(\x_1,\y_1)$ and
$B_2\in\pi_2(\x_2,\y_2)$, we construct a corresponding class 
$B=B_1\# B_2\in\pi_2(\x_1\#\y_1,\x_2\#\y_2)$, as follows. First, we
obtain a homology class $B_2'$ from $B_2$ by summing, for each
puncture in $\Source_1$ labelled by some orbit $\orb_j$ with $f(j)=2k-1$,
all the components of $\Sigma_2\setminus \betas$ that contain boundary
components $\Zout_\ell$ with $f(\ell)\leq 2k$. Similarly, we obtain a homology
class $B_1'$ from $B_1$ by adding, for each puncture in $\Source_1$
labelled by $\orb_j$ with $f(j)=2k-1$, all the components of
$\Sigma_1\setminus\betas$ that contain boundary components $\Zin_\ell$
with
$f(\ell)\leq 2k$.
Let $B_1\# B_2=B_1'\natural B_2'$. 

We formalize now the expected dimension of the moduli space of pairs
of curves, with the time constraints coming from
Equations~\eqref{eq:SelfMatching} and~\eqref{eq:SelfMatchedCurvePair}.

\begin{defn}
  For a self-matched curve pair
  with
  $\chi(\Source_i)=\chiEmb(B_i)$, define the {\em index of the self-matched curve pair} by
  \[ \ind^{\sharp}(B_1,\Source_1;B_2,\Source_2)
  = \ind(B_1,\Source_1)+\ind(B_2,\Source_2)-c_1 \]
\end{defn}

\begin{prop}
  For a self-matched curve pair $(u_1,u_2)$, as above the domain
  $B_1\sharp B_2$ has the following properties:
  \begin{align*}
    B_1\sharp B_2&\in \doms(\x_1\#\x_2,\y_1\#\y_2) \\
    n_\wpt(B_1\sharp B_2)&=n_\wpt(B_1)+o^-_1 \\
    n_\zpt(B_1\sharp B_2)&=n_\zpt(B_1) \\
  \end{align*}
  Moreover, if $\chi(\Source_i)=\chiEmb(B_i)$ for $i=1,2$, then
  \[ \ind(B_1\sharp B_2)=\ind^{\sharp}(B_1,\Source_1;B_2,\Source_2).\]
\end{prop}
\begin{proof}
  This is a straightforward consequence of the fact that if ${\mathcal
    D}$ is any of the components in $\Sigma_i\setminus\betas$ then
  $\ind(B_i+{\mathcal D})=\ind(B_i)+2$.
\end{proof}

Let $\x=\x_1\#\x_2$ and $\y=\y_1\#\y_2$.
Define
\begin{align*} \ModMatchedChanged^B&(\x,\y)\\
  &=
\bigcup_{\left\{\begin{tiny}\begin{array}{r}
B_1\in \pi_2(\x_1,\y_1) \\B_2\in\pi_2(\x_2,\y_2)
  \end{array}\end{tiny}\Big|~B=B_1\sharp B_2\right\}}
\bigcup_{\{(\Source_1,\Source_2,\phi,\psi)\big| \ind^{\sharp}(B_1,\Source_1;B_2,\Source_2)=\ind(B)\}}
\ModMatchedChanged^{B_1,B_2}(\x,\y,\phi,\psi).
\end{align*}
We use these moduli spaces to construct an endomorphism of $C$, specified by its values
$\x=\x_1\#\x_2$ by
\begin{align} \dChanged &(\x)
  &=\sum_{\y} 
  \sum_{\{B\in\pi_2(\x,\y)|\ind(B)=1\}}
  \#\left(\frac{\ModMatchedChanged^B(\x,\y)}{\R}\right)\cdot   U^{n_\wpt(B)}V^{n_\zpt(B)}
\cdot \y.
\label{eq:dChanged}
\end{align}

\begin{prop}
  \label{prop:dChangedSqZero}
  The endomorphism $\dChanged$ satisfies the identity  $\dChanged\circ \dChanged=0$.
\end{prop}

As usual, the proof involves understanding the ends of one-dimensional
moduli spaces of self-matched curve pairs. These ends contain terms
counted in $\dChanged\circ\dChanged$; and all other terms cancel. We
give the proof after some preliminary results; see Figures~\ref{fig:JJandXO}, ~\ref{fig:OXoddWX},
and~\ref{fig:OXevenXW} for pictures.

 \begin{figure}[h]
 \centering
 \input{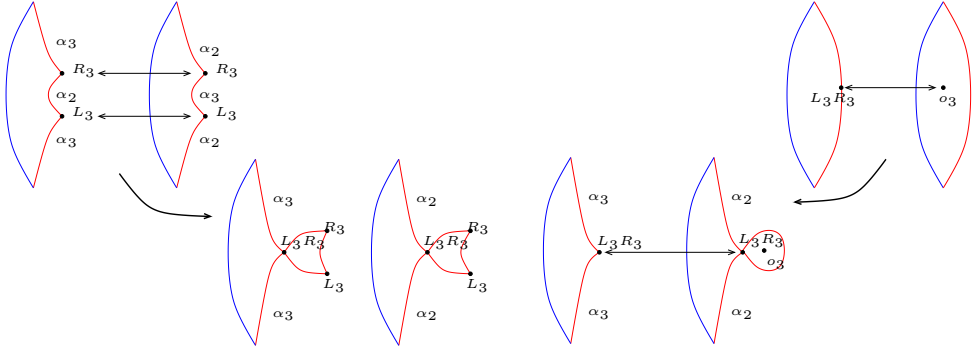}
 \caption{{\bf Ends of Type~\ref{SMCP:XO} cancel ends of type~\ref{SMCP:JJ}.}
 This is a combination of Lemmas~\ref{lem:JJ} and~\ref{lem:XO}.}
 \label{fig:JJandXO}
 \end{figure}

 \begin{figure}[h]
 \centering
 \input{OXoddWX.pstex_t}
 \caption{{\bf Ends of Type~\ref{SMCP:OX} with an odd orbit cancel against ends of Type~\ref{SMCP:WX}.}
 This cancellation is a combination of Lemmas~\ref{lem:OXodd} and~\ref{lem:WX}.}
 \label{fig:OXoddWX}
 \end{figure}

 \begin{figure}[h]
 \centering
 \input{OXevenXW.pstex_t}
 \caption{{\bf Ends of Type~\ref{SMCP:OX} with an even orbit cancel against ends of Type~\ref{SMCP:XW}.}
    This cancellation is a combination of Lemmas~\ref{lem:OXeven} and~\ref{lem:XW}.}
 \label{fig:OXevenXW}
 \end{figure}

\begin{defn}
  If ${\overline\Source}$ is the source of a holomorphic story
  $(w_\ell,\dots,w_1,u,v_1,\dots,v_m)$, let
  $\IntPunctOdd({\overline\Source})$ resp
  $\IntPunctEv({\overline\Source})\subset \IntPunct({\overline\Source})$
  denote the subset of punctures in ${\overline\Source}$ marked by
  Reeb orbits which are odd resp. even.  A {\em self-matched story}
  consists of the following data:
  \begin{itemize}
  \item a holomorphic story ${\overline u}=(w_\ell,\dots,w_1,u,v_1,\dots,v_m)$
  \item an injective map $\phi\colon \IntPunctEv({\overline\Source})\to
    \East({\overline\Source})$ with the property that if
    $p\in {\overline\Source}$ is a puncture marked by some orbit
    $\orb_j$, then $\phi(p)$ is marked by some chord $\longchord_k$
    that covers $Z_k$ with multiplicity one, so that $\{j,k\}\in
    \Mup$.
  \end{itemize}
\end{defn}
  
\begin{defn}
  A {\em self-matched story pair} 
  consists of the following data:
  \begin{itemize}
  \item a self-matched holomorphic story ${\overline u}_1$ in $\Hdown$,
    with source ${\overline\Source}_1$
  \item a holomorphic story ${\overline u}_2$  in $\Hup$,
    with source ${\overline\Source}_2$
  \item a one-to-one correspondence
    \[\psi\colon \AllPunct({\overline\Source}_2)\to
    \East({\overline\Source}_1)\setminus
    \phi(\IntPunct({\overline \Source}_1),\]
  \end{itemize}
  satisfying the following properties:
  \begin{itemize}
  \item if $q\in \IntPunct({\overline\Source}_2)$ is marked by an orbit
    $\orb_j$ (in $\Hup$), then $\psi(q)$ is marked by one of the two
    length $1$ chords that covers $\Zin_j$
  \item if $q\in\East({\overline\Source}_2)$,
    the Reeb chord in $\Hup$ label on $q$ has the same
    name as the Reeb chord in $\Hdown$ that marks
    $\psi(q)$
  \item for each $q\in \AllPunct(\Source_2)$, 
\[      t\circ {\overline u}_1(\psi(q))=
t\circ {\overline u}_2(q).\]
  \end{itemize}
\end{defn}

\begin{prop}
  \label{prop:SMCP-Ends}
  Suppose that $\ind(B_1,\Source_1;B_2,\Source_2)=2$.
  Every comb pair appearing in the boundary of the moduli space
  $\ModMatchedChanged^B(\x,\y;\Source_1,\Source_2)$ 
  is of one of the following types:
  \begin{enumerate}[label=(sME-\arabic*),ref=(sME-\arabic*)]
    \item 
      \label{SMCP:2Story}
      a two-story self-matched curve pair
    \item
      \label{SMCP:JoinCurve}
      a self-matched story pair  of the form $(u_1,v_1),u_2$,
      where $v_1$ is a join curve
    \item 
      \label{SMCP:XO}
      a self-matched story pair of the form $u_1,(u_2,v_2)$
      where $v_2$ is an orbit curve
    \item\label{SMCP:OX} a self-matched story pair $(u_1,v_1),u_2$ where
      $v_1$ is an orbit curve
    \item 
      \label{SMCP:JJ}
      a self-matched story pair of the form $(u_1,v_1),(u_2,v_2)$
      where $v_1$ and $v_2$ are join curves, and the corresponding
      punctures in $u_1$ and $u_2$ are the two distinct length one Reeb 
      chords that cover the same boundary component.
    \item 
      \label{SMCP:WX} 
       a self-matched story pair of the form $(w_1,u_1),u_2$,
      where $w_1$ is a simple boundary degeneration
    \item
      \label{SMCP:XW}
      a self-matched story pair of the form $u_1,(w_2,u_2)$,
      where $w_2$ is a simple boundary degeneration.
  \end{enumerate}
\end{prop}

\begin{proof}
  Suppose that we have a sequence of self-matched curve pairs
  with fixed sources $\Source_1$ and $\Source_2$, representing
  fixed homology classes $B_1$ and $B_2$.
  The index of these is computed nby
  \begin{align*}
    \ind^{\sharp}(\Source_1,\Source_2)&=
    \ind(B_1;\Source_1)+\ind(B_2;\Source_2)-o_+-c_2-o_2\\
    &= d_1+n_{\x_1}(B_1)+n_{\y_1}(B_1) + e(B_1)-2\weight^1_\partial \\
    &\qquad
    + d_2+n_{\x_2}(B_2)+n_{\y_2}(B_2)
    +e(B_2)-2\weight^2_\partial+2o_-+o_++o_2+c_1,
  \end{align*}
  where $o_+$ resp. $o_-$ denotes the number of punctures
  in $\Source_1$ marked by even resp. odd orbits; $c_2$ denotes the number of East
  punctures in $\Source_2$ and $o_2$ the number of interior punctures.

  Take a Gromov limit, and assume that it does not contain any curves at West infinity,
  and that it consists of a singly story.
  Consider  the main component $(u_1',u_2')$,
  with sources $\Source_1'$ and $\Source_2'$.  This limit inherits
  certain matching conditions, whose expected dimension we will now
  express.
  Let $\lambda_+$ denote the number of East punctures
  that arise as limits of East marked orbits in $\Source_1$.
  Let $\lambda_-$ denote the number of East punctures
  that are not matched with any other punctures in either $\Source_i'$.
  (These punctures arise as limits of interior punctures in $\Source_1$
  marked by odd orbits.)
  Let $c_2'$ denote the number of East punctures in
  $\Source_2'$ and $o_2'$ denote the number of its interior punctures.

  Each
  interior puncture of $\Source_1'$ marked by an even orbit is
  constrained to lie at the same $t$-level as some corresponding East
  puncture $\Source_1'$; and each puncture of $\Source_2'$ is
  constrained to lie at the same $t$-level as a corresponding East
  puncture in $\Source_1'$.  Thus, the expected dimension of the moduli
  space in which $(u_1',u_2')$ lives is computed by
  \begin{align*}
    \ind(u_1',u_2')&\leq
    \ind(B_1;\Source_1')+\ind(B_2;\Source_2')-o_+'-\lambda_+-c_2'-o_2' \\
    &= d_1+n_{\x_1}(B_1)+n_{\x_2}(B_1)+e(B_1)-2\weight^1_\partial   \\
    &\qquad +d_2+ n_{\x_1}(B_1)+n_{\x_2}(B_1)+e(B_2)-2\weight^2_\partial+c_1'+o_+'+2o_-'+o_2',
  \end{align*}

  Thus,
  \[ \ind(u_1',u_2')-\ind^\sharp(\Source_1,\Source_2)
  \leq  (c_1'-\lambda_--\lambda_+-c_1)+(o_+'-o_+)+(2o'_-+\lambda_--2o_-)+(o_2'-o_2).\]

  Each quantity in parentheses is clearly non-positive. Since
  we assumed that $\ind^\sharp(\Source_1,\Source_2)=2$, at most one of the quantities
  can be $-1$ (and the others zero) for $(u_1',u_2')$ to exist.

  If $2o_+'+\lambda_--2o_-=-1$, then $\lambda_-=1$. In this case,
  an (odd) orbit curve forms in $\Source_1$ (Case~\ref{SMCP:OX}). On the other hand,
  if $2o_+'+\lambda_--2o_-=0$, then $\lambda_-=0$.

  Also, $o_+'+\lambda_+\leq o_+$, so if $o'_+-o_+=-1$, then
  $\lambda_+=0$ or $1$.  Assume first that $\lambda_+=0$. In this
  case, $\Source_1'$ must contain a multiple even orbit. In that case,
  it is also the case that $c_1'<c_1$, which is a
  contradiction. Assume next that $\lambda_+=1$. In that case,
  $c_1'=c_1+1$. This is the case of an (even) orbit curve appearing in
  $\Source_1$ (Case~\ref{SMCP:OX}, again).

  If $c_1'-c_1=-1$, there are two possibilities. A join curve can form
  on $\Source_1$ without one on $\Source_2$
  (Case~\ref{SMCP:JoinCurve}), or it can form on both sides
  (Case~\ref{SMCP:JJ}).

  If $o_2'-o_2=-1$, we have an orbit curve forming on $\Source_2$ (Case~\ref{SMCP:XO}).

  This finishes the cases where curves at West infinity do not form.

  Suppose that the limiting curve $(u',v')$ contains weight $2k$
  boundary degeneration with $2m$ distinct orbits in it. 
  Then, $(u',v')$ lies in a moduli space whose index is computed by
  \[ \ind(B_1,\Source_1)+\ind(B_1',\Source_1')-(c_1-2k)-(2m-1):\]
  there are $c_1-2k$ height constraints coming from chords in
  $\Source_1'=\Source_1$ that are not matched with orbits in the boundary
  degeneration, and the remaining group of $2m$ chords are required to occur at
  the same $t$-value.

  \begin{align*}
    \ind(B_1,\Source_1')&\leq \ind(B_1,\Source_1) -c_1+c_1'=\ind(B_1,\Source_1)-2(k-m) \\
    \ind(B_2',\Source_2')&\leq \ind(B_2,\Source_2) -2k.
  \end{align*}
  Thus,
  \[ \ind(u_1',u_2')\leq \ind^{\sharp}-2k+1.\]
  Thus, for the limiting object to be non-empty, we require $k=1$,
  i.e. the boundary degeneration is simple; this is
  Case~\ref{SMCP:XW}. More generally, each boundary degeneration
  level carries codimension at least $1$, so it follows that no more
  than one boundary degeneration can occur, and if it does, then there
  are no other curves at East infinity.

  Suppose next that there is a weight $2k$ boundary degeneration West
  infinity on the $\Source_1$-side; let $a$ denote the total weight of
  the orbits, and let $m$ denote the number
  of chords in $\Source_1'$ that are matched with the boundary degeration. 
  Suppose that $m>0$. 
  Then, $(u_1',u_2')$ lives in a moduli space with expected dimension
  \[ \ind(u_1',u_2')=\ind(B_1',\Source_1')+\ind(B_2,\Source_2')-c_1'+1,\]
  since there are $c_1'-m$ matching comditions coming from the chords
  that do not go into the boundary degeneration, and the remaining $m$
  chords are required to lie at the same height, imposing $m-1$
  further constraints. It is straightforward to see that
  \begin{align*}
    n_{\x_1}(B_1')+n_{\y_1}(B_1')+e(B_1')&=n_{\x_1}(B_1)+n_{\y_1}(B_1)+e(B_1)-4k \\
    \weight^1_\partial(B_1')&=\weight^1_{\partial}(B_1) -2k\\
    c_1'&\leq c_1-a+m \\
    o_1'&\leq o_1-a \\ 
    \weight^1_\partial(B_1')&=\weight^1_{\partial}(B)-a
    \end{align*}
    so $\ind(u_1',u_2')\leq \ind^{\sharp}+ 1-2k$. 
    Thus, $k>1$ forces $(u_1',u_2')$ to be in an empty moduli space;
    the case $k=1$ is allowed, and it is 
    Case~\ref{SMCP:WX}. 

    We turn our attention to $m=0$.
    In this case, 
    \[ \ind(u_1',u_2')=\ind(B_1',\Source_1')+\ind(B_2,\Source_2')-c_1',\]
    and the boundary degeneration is forced to contain exactly one odd
    orbit: it is a special boundary degeneration with weight $2k=2a$.
    In this case, computing as above, we find that  $\ind(u_1',u_2')\leq \ind^{\sharp}-2k$. The
    moduli space is once again empty if $k>1$. When $k=1$, there is a
    special case where it is non-empty: when $(u_1',u_2')$ is
    constant. This case also falls under Case~\ref{SMCP:WX}.

\end{proof}

Let $(u_1,v_1),(u_2,v_2)$ be a self-matched story pair where $v_1$ and
$v_2$ are join curves, as in Case~\ref{SMCP:JJ}.  Then, $(u_1,u_2)$
are self-matched pairs. Similarly, if $u_1,(u_2,v_2)$ is a
self-matched story pair where $v_2$ is an orbit curve as in
Case~\ref{SMCP:OX}, then $(u_1,u_2)$ are also self-matched pairs.  In
these cases, we call $(u_1,u_2)$ the {\em trimming} of the
corresponding limit curve.

\begin{lemma}
  \label{lem:JJ}
  The number of ends of Type~\ref{SMCP:JJ} coincides with the number
  of self-matched pairs $(u_1,\Source_1,u_2,\Source_2,\phi,\psi)$ 
  for which $\East(\Source_2)$
  consists of one length $1$ chord and all other chords have length
  $1/2$.
\end{lemma}

\begin{proof}
  Let $\ModMatchedX(\x,\y,\Source_1;\Source_2)$ be the moduli space of
  self-matched curve pairs, where exactly one of the boundary punctures $q$ of
  $\Source_2$ is labelled by a length $1$ chord, which we denote $\longchord$ (and all others
  boundary punctures are marked by length $1/2$ chords). This moduli space
  embeds in a larger moduli space
  ${\widetilde\ModMatchedX}$, where we drop the condition 
  from Equation~\eqref{eq:SelfMatchedCurvePair} for the puncture $q\in \Source_2$ marked
  by the length $1$ chord. This moduli space in turn admits an evaluation map
  \[\ev_q-\ev_{\psi(q)}\colon \widetilde\ModMatchedX \to \R \]
  with $0$ as a regular value, whose preimage is $\ModMatchedX(\x,\y,\Source_1;\Source_2)$.

  Given  $u=(u_1,u_2)\in \ModMatchedX(\x,\y;\Source_1,\Source_2)$, let  ${\widetilde U}\subset \widetilde\ModMatchedX$
  be a neighborhood so that $\ev_q-\ev_{\psi(q)}\colon {\widetilde U}\to (-\epsilon,\epsilon)$ is a diffeomorphism
  (i.e. so that $u$ is the preimage of $0$).
  
  Let $\Source_1'=\Source_1\natural_{\psi(q)} \EastSource_1$ and $\Source_2'=\Source_2\natural_q \EastSource_2$.
  Thus, $\Source_2$ comes with an extra pair of punctures $\{q_1,q_2\}$ (labelled by 
  chords which can be joined to form $\longchord$).
  Let $\BigModMatched$ be the moduli space like $\ModMatchedChanged(\x,\y,\Source_1',\Source_2')$,
  except now that at the two distinguished punctures $\{q_1,q_2\}$ in $\Source_2'$ coming from the east infinity
  curve, we do not require the time constraint from Equation~\eqref{eq:SelfMatchedCurvePair}.
  There is a map $F=(F_1,F_2,F_3)\colon \BigModMatched \to \R^3$, with components
  \[
    F_1=t\circ \ev_{\psi(q_1)}-t\circ \ev_{q_1}, \qquad
    F_2=t\circ \ev_{q_1}-t\circ \ev_{q_2} \qquad
    F_3=t\circ \ev_{\psi(q_1)}-t\circ \ev_{\psi(q_2)},
    \]
  with the property that if $\Delta\subset \R^2$ denotes the diagonal, th
  Type~\ref{SMCP:JJ}
  end of $\ModMatchedChanged(\x,\y,\Source_1',\Source_2')$
  at $(u_1,\Source_1,u_2,\Source_2,\phi,\psi)$.

  Gluing the east infinity curves to $u_1$ and $u_2$, we get a gluing map
  \[ \gamma \colon {\widetilde U}\times (0,\epsilon)\times (0,\epsilon) \to \BigModMatched.\]
  The gluing map continuously extends to a  map 
  from ${\widetilde U}\times [0,\epsilon)\times [0,\epsilon)$ 
  to the Gromov compactification of $\BigModMatched$
  so that for all $r_1,r_2>0$,
  \begin{align*}
    \gamma(u_1\times u_2\times \{0\}\times \{0\})&=((u_1,v_1),(u_2,v_2)) \\
    F_1\circ \gamma|_{{\widetilde U}\times \{0\}\times \{0\}}&=
    {\ev_{\psi(q)}-\ev_{q}}|_{\widetilde U} \\
    F_2\circ \gamma({\widetilde U}\times \{r_1\}\times \{0\})&=0 \\
    F_2\circ \gamma({\widetilde U}\times \{r_1\}\times \{r_2\})&>0 \\
    F_3\circ \gamma({\widetilde U}\times \{0\}\times \{r_2\})&=0 \\
    F_3\circ \gamma({\widetilde U}\times \{r_1\}\times \{r_2\})&>0.
  \end{align*}
  It follows that $(F\circ \gamma)^{-1}(\{0\}\times \Delta)$ has a
  single endpoint over the origin, and that is the point
  $((u_1,v_1),(u_2,v_2))$, giving the stated correspondence
  betwen ends and self-matched curves stated in the lemma.
\end{proof}

\begin{lemma}
  \label{lem:XO}
  The number of  ends of Type~\ref{SMCP:XO} coincides with the number 
  of self-matched pairs for which $\East(\Source_2)$ consists of one length $1$ chord
  and all other chords have length $1/2$.
\end{lemma}

\begin{proof}
  Let $\ModMatchedX(\x,\y,\Source_1,\Source_2)$, ${\widetilde\ModMatchedX}$, and $U$ be as in the proof of Lemma~\ref{lem:JJ}.
  Gluing the orbit curve to $\Source_2$ now gives a gluing map
  \[ \gamma\colon U \times (0,\epsilon)\to \ModMatchedChanged(\x,\y,\Source_1,\Source_2'),\]
  where now $\Source_2'=\Source_2\natural\EastSource$ is obtained by gluing on an orbit curve, which we denote here by $v_2$.
  (Note that $\Source_2'$ and $v_2$ mean different objects than they did in the proof of Lemma~\ref{lem:JJ}.)
  Gluing extends continuously to the Gromov compactification, giving 
  \[ \gamma\colon U \times [0,\epsilon)\to \overline\ModMatchedChanged(\x,\y,\Source_1,\Source_2'),\]
  so that for all $r>0$,
  \begin{align*}
    \gamma(u_1,\times u_2\times \{0\})&=(u_1,(u_2,v_2)) \\
    \ev_{\longchord_i}-\ev_{\orb_i'}|_{\gamma(U\times \{0\})}&=\ev_{\longchord_i}-\ev_{\longchord_i'} \\
    s\circ \ev_{\orb_i'}(\gamma(u,0)&=1 \\
    s\circ \ev_{\orb_i'}(\gamma(u,r)&<1.
  \end{align*}
  There is a neighborhood $U'$ of $(u_1,(u_2,v_2))$ in ${\overline\ModMatchedChanged}(\x,\y,\Source_1,\Source_2')$ so that
  \[ U'\cap\ModMatchedChanged(\x,\y;\Source_1,\Source_2')=(\ev_{\orb_i}-\ev_{\orb'_i})^{-1}(0),\]
  so that $s\circ \ev_{\orb_i'}\colon U'\to (1-\epsilon,1]$ is a proper map near $1$.
  It follows that the ends 
  $U'\cap\ModMatchedChanged(\x,\y;\Source_1,\Source_2')=(\ev_{\orb_i}-\ev_{\orb'_i})^{-1}(0)$
  is modelled on the preimage at $s=1$.
\end{proof}

\begin{defn}
  \label{def:PartialSMCP}
  Let $X$ be a set consisting of one or two Reeb chords.  A
  {\em{$X$-partially self-matched curve pair }} (or $X$-partial-smcp) is the data $(\Source_1,\phi,\Source_2,\psi)$ as
  in Definition~\ref{def:SelfMatchedCurvePair}, except
Condition~\ref{smcp:1to1} is replaced by the following:
  \begin{itemize}
    \item 
     $\East(\Source_1)$ is partitioned into three disjoint sets:
     \[ \phi(\IntPunctEv(\Source_1))\qquad \psi(\East(\Source_2))\qquad X',\]
     where $X'$ is a set with $|X'|=|X|$, and the labels on the punctures of $X'$
     are specified in $X$.
   \end{itemize}
     Moreover, if $|X|=2$, let $X'=\{p_1,p_2\}$. In this case, we also require 
    \begin{equation}
      \label{eq:NearMatchedConstraint}
      t\circ u_1(p_1)=t\circ u_1(p_2).
    \end{equation}
    Let $\NearModMatched{X}{B}(\x,\y)$ denote the moduli space of $X$-partial smcps,
  with homology class $B$.
\end{defn}

Let $(u_1,v_1),u_2$ be a self-matched story pair where $v_1$ is an
orbit curve; and let $\longchord$ be the chord labelling the boundary
puncture of $v_1$. If the orbit in $v_1$ is odd, then $(u_1, u_2)$ is
a $\{\longchord\}$-partial smcp.  If the orbit is even, then
$(u_1,u_2)$ is a $X$-partial smcp, where $X$ consists of
two length one chords
that cover $\Mup$-matched boundary components.  Similarly, if
$(w_1,v_1),u_2$ is a self-matched story pair where $w_1$ is a boundary
degeneration, the curves $u_1$ and $u_2$ is a partial smcp with one
remaining chord, and if $u_1,(w_2,u_2)$ is a self-matched story pair
where $w_2$ is a boundary degeneration, then $(u_1,u_2)$ is a
$\{\longchord\}$-partial smcp, where $\longchord$ is a length one chord  
that covers some boundary component.  In all the above cases, we
call $(u_1,u_2)$ the {\em{trimming}} of the self-matched story pair.
The next few lemmas show that the number of ends are determined by the
trimmings of the limiting configurations.

\begin{lemma}
  \label{lem:OXodd}
  Suppose that $\orb_j$ is an odd orbit, and let $\longchord_j$  be the 
  chord of length one that covers the corresponding boundary component
  $\Zin_j$.  The number of curves in
  $\NearModMatched{\{\longchord_j\}}{B}(\x,\y;\Source_1,\Source_2)$
  has the same parity as the number of ends of
  $\ModMatchedChanged^B(\x,\y;\Source_1',\Source_2')$ of
  Type~\ref{SMCP:OX}, where $v_1$ is an orbit curve that has a
  boundary puncture marked by $\orb_j$.
\end{lemma}

\begin{proof}
  This follows from the gluing
  \[ \gamma \colon
  \NearModMatched{\longchord_j}{B}(\x,\y;\Source_1',\Source_2) \times (0,\epsilon) \to
  \ModMatched^B(\x,\y,\Source_1,\Source_2), \] 
  gluing the orbit curve to $\Source_1'$ at $\longchord_j$,
  which parameterizes the end with $s(u(q))\goesto 0$,
  where here $q$ denotes the $\orb_j$-marked puncture;
  cf. Proposition~\ref{prop:OrbitCurve}.
\end{proof}

The following is a variant of Proposition~\ref{prop:BoundaryDegenerationNbd}:

\begin{lemma}
  \label{lem:WX}
  Suppose that $\{j,k\}\in\Mup$, where $f(j)$ is odd. Let $\longchord_j$
  be a chord of length one that covers an boundary component $\Zin_j$.
  The number of curves in
  $\NearModMatched{\{\longchord_j\}}{B}(\x,\y;\Source_1,\Source_2)$ has
  the same parity as the number of ends of
  $\ModMatchedChanged^B(\x,\y;\Source_1',\Source_2')$ of
  Type~\ref{SMCP:WX}, where $w_1$ is a (smooth) simple boundary degeneration
  that contains $\orb_j$. 
\end{lemma}

\begin{proof}
  Let $\Source_1'=\WestSource\natural\Source_1$. 
  This curve contains three distinguished punctures, $q_1$ and $q_2$ marked by orbits
  (coming from $\WestSource$), labelled so that $q_2$ is the even orbit; and a third puncture $q_3$ that 
  marked by a length $1$ chord, with $q_3=\phi(q_2)$.
  (In the notation of the lemma statement, $q_3$ is labelled by $\longchord_j$.)
  Consider a moduli space ${\BigModMatched}$ containing $\ModMatchedChanged(\x,\y,\Source_1',\Source_2)$,
  where we drop the height constraint from Equation~\eqref{eq:SelfMatching} for the orbit $q_2$ and its corresponding
  chord $q_3=\phi(q_2)$. Thus, there is an evaluation map 
  \[ t\circ \ev_{q_1}-t\circ \ev_{q_2}\colon \BigModMatched\to \R \]
  so that $0$ is a regular value, and whose preimage is identified with 
  $\ModMatchedChanged(\x,\y,\Source_1',\Source_2)$.
  We will consider a map $F=(F_1,F_2)\colon \BigModMatched \to \R^2$ whose components are given by 
  \[ F_1= t\circ \ev_{q_1}-t\circ \ev_{q_2}\qquad F_2=t\circ \ev_{q_1}-t\circ\ev_{q_3}.\]
  This map extends continuously to the Gromov compactification of $\BigModMatched$.

  Fix $((u_1,w_1),u_2)$ in the Gromov compactification.
  Gluing  gives a map
  \[ \gamma\colon  \ModFlow\times_{\Tb}
  \ModWest(\WestSource)\times (-\epsilon,\epsilon)\times (0,\epsilon)\to \BigModMatched.\]
  Here, the $(-\epsilon,\epsilon)$ specifies the $t$-coordinate where the gluing
  is performed, and $[0,\epsilon)$ represents the gluing scale. 
  We restrict this to $(u_1,u_2)\times w_1\in \ModFlow\times_{\Tb}\ModDeg$,
  and denote the resulting map
  \[ \gamma\colon (-\epsilon,\epsilon)\times(0,\epsilon)\to \BigModMatched.\]

  Consider the two punctures $q_1$ and $q_2$ in $\WestSource$
  Since $w_1$ is generic,  $t\circ w_1(q_1)\neq t\circ w_1(q_2)$.
  Assume $t \circ w_1(q_1)>t\circ w_1(q_2)$.
  The map
  \[ (t\circ \ev_{q_1}-t\circ \ev_{q_3}) \circ\gamma(\cdot,0) \colon(-\epsilon,\epsilon)\to (-\epsilon,\epsilon) \]
  has odd degree. By our assumption, for all $r>0$,
  \[    F_1\circ \gamma(r,t)>0.\]
  Clearly, also
  \[    F_1(0,t)=0.\]
  Thus, 
  \[ F\circ \gamma\colon (-\epsilon,\epsilon)\times [0,\epsilon) \to \R^{\geq 0}\times \R \]
  is proper map of degree $1$ to points in $\R^{\geq 0}\times \R$ near the origin.
  The moduli space consists of points in the preimage of $\R^{>0}\times \{0\}$, a smooth one-manifold whose end consists of the preimage
  of the origin. 

  The same argument applies when $t\circ w_1(q_1)<t\circ w_1(q_2)$, with minor modifications.
\end{proof}

\begin{lemma}
  \label{lem:OXeven}
  Suppose that $\orb_j$ is an even orbit, and let 
  $\longchord_j$ be a chord of length one that covers a the corresponding boundary component $\Zin_j$,
  and $\longchord_k$ be a chord of length one that covers $\Zin_k$, with $\{j,k\}\in \Mup$.
  The number of curves in
  $\NearModMatched{\{\longchord_j,\longchord_k\}}{B}(\x,\y;\Source_1,\Source_2)$
  has the same parity as the number of ends of 
  $\ModMatchedChanged^B(\x,\y;\Source_1',\Source_2')$ 
  of Type~\ref{SMCP:OX}, where $v_1$ is an orbit curve with boundary puncture marked by $\longchord_j$.
\end{lemma}

\begin{proof}
  Let ${\widetilde\ModMatchedX}$ denote the  moduli space containing
  $\NearModMatched{\{\longchord_j,\longchord_k\}}{B}(\x,\y,\Source_1,\Source_2)$,
  of data $(u_1,u_2,\Source_1,\phi,\Source_2,\psi)$ satisfying the
  conditions from Definition~\ref{def:PartialSMCP}, except that for
  the two distinguished punctures $q_1$ and $q_2$ on $\Source_1$
  labelled by $\longchord_j$ and $\longchord_k$ (not contained
  contained in $\phi(\IntPunctEv(\Source_1))$ or
  $\psi(\East(\Source_2))$), we no longer require that $t\circ
  u_1(q_1)=t\circ u_1(q_2)$; i.e. we drop the constraint from
  Equation~\eqref{eq:NearMatchedConstraint}.

  We have a map
  \[t\circ \ev_{q_1}-t\circ \ev_{q_2}\colon {\widetilde\ModMatchedX}\to \R \]
  so that $0$ is a regular value, and
  \[ (t\circ \ev_{q_1}-t\circ \ev_{q_2})^{-1}(0)= \NearModMatched{\{\longchord_j,\longchord_k\}}{B}(\x,\y,\Source_1,\Source_2).\]

  Each $(u_1,u_2)\in  \NearModMatched{\{\longchord_j,\longchord_k\}}{B}(\x,\y,\Source_1,\Source_2)$
  neighborhood ${\widetilde U}\in{\widetilde M}$, so that
  $F_1\colon {\widetilde U} \to (-\epsilon,\epsilon)$
  is a homeomorphism. 

  Let $\BigModMatched$ be the moduli space like ${\widetilde\ModMatchedX}$,
  except now $q_1$ is marked by the orbit $\orb_j$ rather than $\longchord_j$.
  In particular, 
  we have the map 
  \[ F_1=t\circ\ev_{q_1}-t\circ \ev_{q_2}\colon \BigModMatched\to \R \]
  with regular value $0$, so that
  \[ F_1^{-1}(0)=\ModMatchedChanged^B(\x,\y,\Source_1',\Source_2').\]
  Let $F_2=s\circ \ev_{q_1}$, we have a map
  \[ F=(F_1,F_2)\colon \BigModMatched\to \R\times \R^{< 1}. \]

  Gluing on the orbit curve
  at $q_1$ gives a map $\gamma\colon {\widetilde U}\times [0,r) \to
  \ModMatchedChanged^B(\x,\y,\Source_1',\Source_2')$.
  We have that
  \[ F_1\circ \gamma \colon {\widetilde U}\times (0)\to (-\epsilon,\epsilon) \]
  has degree $1$ near $0$; and for all ${\widetilde u}\in {\widetilde U}$, $r\in (0,\epsilon)$.
  \begin{align*}
    F_2\circ \gamma({\widetilde u},0)&=1 \\
    F_2 \circ \gamma({\widetilde u},r)&<1 
  \end{align*}
  Thus, 
  \[ F\circ \gamma({\widetilde U}\times [0,\epsilon)) \to \R\times \R^{\leq 1} \]
  is proper of degree $1$ for points near $(0,1)$ in $\R\times \R^{\leq 1}$.
  It follows at once that the smooth manifold
  \[ 
  \ModMatchedChanged^B(\x,\y,\Source_1',\Source_2')=F_1^{-1}(\{0\}) \]
  has one end over the point $(0,1)\in \R\times \R^{\leq 1}$.
\end{proof}

\begin{lemma}
  \label{lem:XW}
  Suppose that $\{j,k\}\in\Mup$, and let $\longchord_j$ and
  $\longchord_k$ be chords of length one that cover the boundary
  components $\Zin_j$ and $\Zin_k$ respectively.  The number of curves in
  $\NearModMatched{\{\longchord_j,\longchord_k\}}{B}(\x,\y;\Source_1,\Source_2)$ has
  the same parity as the number of ends of
  $\ModMatchedChanged^B(\x,\y;\Source_1,\Source_2')$ of
  Type~\ref{SMCP:XW}, where $w_2$ is a simple boundary degeneration
  that contains $\orb_j$ and $\orb_k$.
\end{lemma}

\begin{proof}
  Let ${\widetilde\ModMatchedX}$ be as in Lemma~\ref{lem:OXeven}.

  Let $\BigModMatched$ be the moduli space containing the moduli space
  of self-matched curve pairs
  $\ModMatchedChanged^{B}(\x,\y,\Source_1,\Source_2')$ as in
  Definition~\ref{def:SelfMatchedCurvePair}, except that now there are
  two distinguished (interior) punctures $q_1,q_2$ in $\Source_2$
  marked by orbits $\orb_j$ and $\orb_k$ respectively, where we do not
  impose the corresponding height constraints (from
  Equation~\eqref{eq:SelfMatchedCurvePair}).  Note that $u_1$
  represents the homology class $B_1$ and $u_2$ represents the
  homology class $B_2+{\mathcal D}$, where ${\mathcal D}$ is the
  elementary domain $\Sigma_2$ containing $\Zout_j$ and $\Zout_k$.
  Consider the map
  $F=(F_1,F_2,F_3) \colon \BigModMatched\to \R^3$ whose three components are 
  the are evaluation maps
  \begin{align*}
    F_1&=t\circ \ev_{\psi(q_1)}-t\circ \ev_{q_1} \\
    F_2&=t\circ \ev_{\psi(q_1)}-t\circ \ev_{\psi(q_2)}\\
    F_3&=t\circ \ev_{q_1}-t\circ \ev_{q_2}. 
  \end{align*}
  Clearly, 
  \[ \ModMatchedChanged^{B_1,B_2+{\mathcal D}}(\x,\y,\Source_1,\Source_2')=F^{-1}(\Delta\times \{0\}), \]
  where $\Delta\subset \R^2$ is the diagonal.

  For sufficiently small open subsets ${\widetilde U}\subset
  {\widetilde \ModMatchedChanged}$, gluing gives a map
  \[ \gamma\colon   (-\epsilon,\epsilon)\times {\widetilde \ModMatched}\times_{\Tb} \ModDeg({\mathcal D})\times
   [0,\epsilon) \to \BigModMatched, \] 
  where the $(-\epsilon,\epsilon)$ factor 
  specifies the $t$-coordinate where the
  gluing is performed, and the $[0,\epsilon)$ represents the gluing scale.
  The fibered product over $\Tb$ is taken with respect to an evaluation map
  $\ev_t \colon {\ModFlow}\to \Tb$, defined by
  \[ \ev_t(u_1)={u_1}^{-1}(0,t)\in \Tb; \] and the
  (degree one) evaluation map $\ev^\beta\colon \ModWest\to \Tb$.
  Assume that $\ev^{\beta}(w_1)$ is a regular value, so that there
  is an open neighborhood $W$ of $w_2$ so that $\ev\colon W \to \Tb$
  is a local diffeomorphism.  Restricting to some sufficiently small
  neighborhood ${\widetilde U}\subset {\widetilde \ModFlow}$, we
  can guarantee that for $t\in (-\epsilon,\epsilon)$,
  $\ev_t({\widetilde U}) \subset \ev(W)$, so
  ${\widetilde U}\times_{\Tb} W \cong {\widetilde U}$. 
  Further shrinking ${\widetilde U}$ if needed, we can assume that
  \[ t\circ \ev_{q_1}-t\circ \ev_{q_2} \colon {\widetilde U} \to
  (-\epsilon,\epsilon) \] is a homeomorphism, so that the preimage of
  $0$ is the nearly self-matched curve pair $(u_1,u_2)$.  We
  abbreviate the gluing map (suppressing the choice of $W$), writing
  instead
  \[ \gamma\colon {\widetilde U}\times (-\epsilon,\epsilon)\times
  (0,\epsilon)\to
  {\BigModMatched}.\]
  This map has a natural extension to 
  ${\widetilde U}\times (-\epsilon,\epsilon)\times
  [0,\epsilon)$ to the Gromov compactification of ${\BigModMatched}$.

  For any fixed ${\widetilde u}\in {\widetilde U}$, $r\in (0,1)$, 
  \[ F_1\colon \gamma(\cdot,{\widetilde u},r)
  \colon (-\epsilon,\epsilon)\to (-\epsilon,\epsilon) \]
  has degree one, since the same statement holds when $r=0$.
  Also, for any fixed $t\in (-\epsilon,\epsilon)$, 
  \[ F_2\circ \gamma(t,\cdot,r)
  \colon {\widetilde U}\to (-\epsilon,\epsilon)\]
  has degree one, since the same statement holds setting $r=0$
  Finally, 
  \[ 
    F_3(t,{\widetilde u},0)= 0\qquad{\text{and}}\qquad
    F_3(t,{\widetilde u},r)> 0. 
    \]
  It follows that 
  \[ F\circ \gamma \colon (-\epsilon,\epsilon)\times
  {\widetilde\ModMatched}\times [0,\epsilon) \to 
  \R\times \R \times \R^{\geq 0}\] has degree $1$ near the origin; and the smooth
manifold $F^{-1}(\Delta\times \{0\})$ has one end over the origin.
\end{proof}

\begin{proof}[Proof of Proposition~\ref{prop:dChangedSqZero}]
  As usual, we consider index two moduli spaces. Consider their ends,
  as in Proposition~\ref{prop:SMCP-Ends}.
  Combining Lemma~\ref{lem:JJ} and \ref{lem:XO}, it follows that
  the count of ends of Type~\ref{SMCP:JJ} cancels with
  the ends of Type~\ref{SMCP:XO}.  Combining Lemma~\ref{lem:WX} and
  \ref{lem:OXodd}, it follows that ends of of Type~\ref{SMCP:WX}
  cancel with ends of Type~\ref{SMCP:OX}, where the orbit curve is
  odd, as illustrated in Figure~\ref{fig:OXoddWX}.  Combining
  Lemma~\ref{lem:OXeven} and~\ref{lem:XW}, it follows that ends of
  Type~\ref{SMCP:XW} cancel with ends of Type~\ref{SMCP:OX} where the
  orbit curve is even, as illustrated in Figure~\ref{fig:OXevenXW}.
  
  The ends that are not accounted for are of Type~\ref{SMCP:2Story};
  and these ends count the $\y$ coefficient of
  $\dChanged\circ\dChanged$.
\end{proof}

\subsection{Intermediate complexes}
\label{subsec:Intermediates}

Note that the self-matched compatible pairs from
Definition~\ref{def:SelfMatchedCurvePair} use the orbits curves
quite differently from Definition~\ref{def:MatchedPair}; and so the
chain complex defined using the two objects
$(C,\partial^{(0)})$ and $(C,\partial_{\natural})$ seem quite different.
To construct a homopy equivalence between these complexes, we will use a
sequence of intermediate complexes, defined here.

Fix an integer $\ell$ and a marked source $\Source$. Let
$\IntPunct^{\leq \ell}(\Source)\subset \IntPunct(\Source)$ denote the
subset of those punctures $q\in \Source$ that are marked by orbits
$\orb_j$ with $f(j)\leq \ell$.  Define $\IntPunct^{\ell}(\Source)$ and
$\IntPunct^{>\ell}(\Source)$ analogously

\begin{defn}
  \label{def:IntermediateComplexes}
  Let $\ell$ be an integer between $0,\dots,2n$.
  An $\ell$-self-matched curve pair  is the following data.
  \begin{itemize}
    \item a holomorphic curve $u_1$ in $\Hdown$, with source $\Source_1$
    \item a holomorphic curve $u_2$ in $\Hup$, with source $\Source_2$
    \item an injection 
      $\phi\colon \IntPunctEv^{\leq \ell}(\Source_1) \to \East(\Source_1)$
    \item an injection $\psi\colon \East(\Source_2)\to\East(\Source_1)$
    \end{itemize}
    with the following properties:
    \begin{itemize}
    \item $\East(\Source_1)$ is a union of three disjoint sets,
      $\phi(\IntPunctEv^{\leq \ell}(\Source_1))$, $\psi(\East(\Source_2))$, and
      $\IntPunct^{\leq \ell}(\Source_1)$
    \item If $p\in \IntPunctEv^{\leq \ell}(\Source_1)$ is labelled by some orbit $\orb_j$,
      then $\phi(p)$ is marked by a length one Reeb chord that covers
      the boundary component of $\Zin_k$, where $\{j,k\}\in\Mup$, and
      \[ t \circ u_1(\phi(p))=t\circ u_2(p). \]
    \item If $q\in \East^{\leq \ell}(\Source_2)$ , then $\psi(q)$ is labelled by a
      length one Reeb chord that covers the boundary component of
      $\Zin_j$, and
      \[ t\circ u_1(\psi(q))=t\circ u_2(q). \]
    \item If $q\in \East^{> \ell}(\Source_2)\cup\East(\Source_2)$,
      then the marking on the Reeb chord or orbit
      of $q$ is the same as the marking on the Reeb chord or orbit of
      $\psi(q)$, and 
      \[ (s\circ u_1(\psi(q)),t \circ u_1(\psi(q)))=
      (s \circ u_2(q),t \circ u_2(q)).\]
      \end{itemize}
    Let $\ModInt{\ell}^{B_1,B_2}(\x,\y,\Source_1,\Source_2,\phi,\psi)$ denote the
    moduli space of $\ell$-morphism matched curves.
\begin{align*} \ModMor{\ell}^B&(\x,\y)\\
  &=
\bigcup_{\left\{\begin{tiny}\begin{array}{r}
B_1\in \pi_2(\x_1,\y_1) \\B_2\in\pi_2(\x_2,\y_2)
  \end{array}\end{tiny}\Big|~B=B_1\natural B_2\right\}}
\bigcup_{\{(\Source_1,\Source_2,\phi,\psi)\big| \ind^{\natural}(B_1,\Source_1;B_2,\Source_2)=\ind(B)\}}
\ModMor{\ell}^{B_1,B_2}(\x,\y,\phi,\psi).
\end{align*}
\end{defn}

\begin{defn}
  Fix $\ell\in \{0,\dots,2n\}$. 
  Consider $C$ equipped with the endomorphism
  determined by
  \[ \dInt{\ell}(\x)=\sum_{\y} \sum_{\{B\in\doms(\x,\y)\big| \ind(B)=1\}}
  \# \left(\frac{\ModInt{\ell}^B(\x,\y)}{\R}\right)\cdot \y \]
\end{defn}

\begin{remark}
  Note that $\dInt{0}$ is the operator from
  Equation~\eqref{eq:DefD0}; and  $\dInt{2n}$ is the operator
  $\dChanged$ is the operator
  from Equation~\eqref{eq:dChanged}
\end{remark}

Unlike the earlier cases, a sequence of $\ell$-self-matched curve
pairs can have a Gromov limit to a pair of curves
$((w_1,u_1),(w_2,u_2))$ where both $w_1$ and $w_2$ are simple boundary
degenerations, so that each puncture in $w_2$ has a corresponding
puncture in $w_1$, which is marked by the same orbit.  This can happen
in the special case where the odd orbit $\orb_j$ in $w_2$ has $f(j)=\ell$.

We will formulate the end counts in terms of the following types of curves
(which naturally arise in Gromov limits of $\ell$-self-matched curve pairs):

\begin{defn}
  \label{def:IntermediateComplexesx}
  Let $\rho$ be a Reeb chord in $\Sigma_1$, and 
  let $\ell$ be an integer between $0,\dots,2n$.
  An {\em $\ell$-self-matched curve pair with remaining $\{\rho\}$}  is the following data.
  \begin{itemize}
    \item a holomorphic curve $u_1$ in $\Hdown$, with source $\Source_1$
    \item a holomorphic curve $u_2$ in $\Hup$, with source $\Source_2$
    \item an injection 
      $\phi\colon \IntPunctEv^{\leq \ell}(\Source_1)\to \East(\Source_1)$
    \item an injection $\psi\colon \East(\Source_2)\to\East(\Source_1)$
    \end{itemize}
    with the following properties:
    \begin{itemize}
    \item $\East(\Source_1)$ is a union of four disjoint sets,
      \[ \phi(\East^{\leq \ell}_+(\Source_1)), \qquad
      \psi(\East(\Source_2)\cup\East(\Source_2)), \qquad
      \East^{<\ell}_-(\Source_1), \qquad \{q_0\}, \]
      where $q_0$ is a puncture labelled by the Reeb chord $\rho$.
    \item If $p\in \IntPunctEv^{\leq \ell}(\Source_1)$ is labelled by some orbit $\orb_j$,
      then $\phi(p)$ is marked by a length one Reeb chord that covers
      the boundary component of $\Zin_k$, where $\{j,k\}\in\Mup$, and
      \[ t \circ u_1(\phi(p))=t\circ u_2(p). \]
    \item If $q\in \East^{\leq \ell}(\Source_2)\cup \East(\Source_2)$, 
      then $\psi(q)$ is labelled by a
      length one Reeb chord that covers the boundary component of
      $\Zin_j$, and
      \[ t\circ u_1(\psi(q))=t\circ u_2(q). \]
    \item If $q\in \East^{> \ell}(\Source_2)$,
      then the marking on the Reeb chord or orbit
      of $q$ is the same as the marking on the Reeb chord or orbit of
      $\psi(q)$, and 
      \[ (s\circ u_1(\psi(q)),t \circ u_1(\psi(q)))=
      (s \circ u_2(q),t \circ u_2(q)).\]
      \end{itemize}
  We denote the moduli space of these data by 
  $\NearModMatched{\ell,\{\rho\}}{B}(\x,\y,\Source_1,\Source_2)$.
\end{defn}

 \begin{figure}[h]
 \centering
 \input{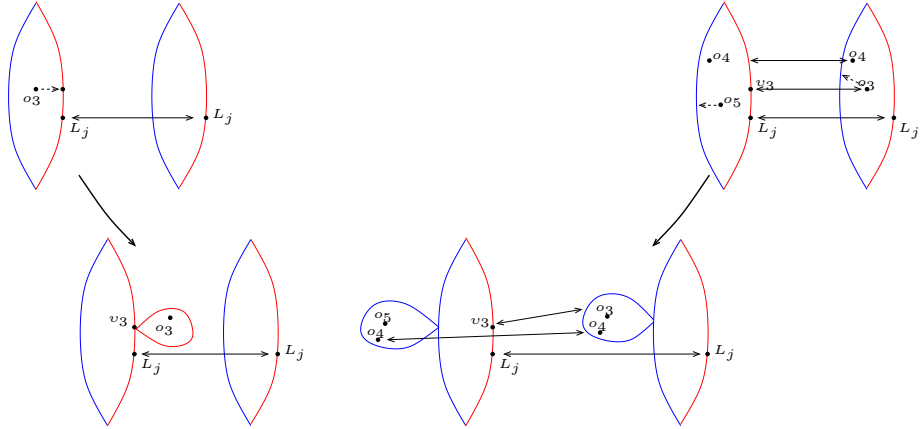}
 \caption{Ends of Type~\ref{SMCP:OX} 
   where the orbit is labelled by $\orb_j$ with 
   $f(j)=\ell$ is odd,
   cancel against curves with boundary degenerations on both sides.}
 \label{fig:OXWW}
 \end{figure}

\begin{lemma}
  \label{lem:WW}
  Let $j$ be so that $f(j)=\ell$ is odd.
  The number of
  curves in
  $\NearModMatched{\ell;\{\longchord_j\}}{B}(\x,\y;\Source_1,\Source_2)$
  has the same parity as the number of ends of
  $\ModInt{\ell}^{B'}(\x,\y;\Source_1',\Source_2')$ of
  the form $((u_1,w_1),(u_2,w_2))$, where
  $\Source_1'=\WestSource_1\natural \Source_1$,
  $\Source_2'=\WestSource_2\natural\Source_2$, $w_1$ and $w_2$ are
  simple boundary degenerations with sources $\WestSource_1$ and
  $\WestSource_2$ both of which contain a puncture marked with
  $\orb_j$, and
  $B'_i=B_i+{\mathcal D}_i$, where ${\mathcal D}_i$ is the shadow of
  $w_i$ for $i=1,2$.
\end{lemma}

\begin{proof}
  Let $q_1$ and $q_2$ be the two punctures on $\Source_2'$ coming from
  $w_2$, labelled so that $q_1$ is labelled by the odd Reeb orbit and
  $q_2$ by the even one. In particular, $\psi(q_1)$ is a puncture on
  $\Source_1'$ labelled by $\longchord_j$.
  There is a
  moduli space ${\BigModMatched}$ which is like
  $\NearModMatched{\ell;\{\longchord_j\}}{B}(\x,\y,\Source_1',\Source_2')$, except we now drop the
  conditions that 
  \[ t\circ u_1(\psi(q_1))=t\circ u_2(q_1)\qquad{\text{and}}\qquad
  t\circ  u_1(\psi(q_2))=t\circ u_2(q_2). \] 
  Thus, we have a map
  \[ F=(F_1,F_2,F_3,F_4)\colon \BigModMatched\to \R\times \R\times (0,1)\times (0,1) \]
  with components
  \begin{align*}
    F_1&=t\circ u_1(\psi(q_1))-t\circ u_2(q_1),\\
    F_2&=t\circ  u_1(\psi(q_2))-t\circ u_2(q_2),\\
    F_3&=s\circ u_1(\psi(q_2)), \\
    F_4&=s\circ u_2(q_2) 
  \end{align*}
  so that $F^{-1}(0\times \Delta)=\NearModMatched{\ell;\{\longchord_j\}}{B}(\x,\y,\Source_1',\Source_2')$.

  Fix $(u_1,u_2)\in\NearModMatched{\ell;\{\longchord_j\}}{B}(\x,\y;\Source_1,\Source_2)$. 
  Gluing $w_1$ and $w_2$ to $u_1$ and $u_2$ gives a map
  \[ \gamma\colon \times (-\epsilon,\epsilon)\times
  (0,\epsilon)\times (-\epsilon,\epsilon)\times (0,\epsilon)\to
  \BigModMatched.\] (Note we are now gluing to both sides, giving two
  time parameters and two scale parameters.)  

  The gluing map extends
  to a map from $(-\epsilon,\epsilon)\times
  [0,\epsilon)\times (-\epsilon,\epsilon)\times [0,\epsilon)$
  to the Gromov compactification of $\BigModMatched$.
  This extension satisfies the following properties,
  for all $t_1,t_2\in(-\epsilon,\epsilon)$ and $r_1,r_2\in(0,\epsilon)$:
  \begin{itemize}
    \item 
      $F_1\circ \gamma(\cdot, r_1, t_1 
      r_2)\colon (-\epsilon,\epsilon)\to (-\epsilon,\epsilon)$
      has degree one, since the same holds for $r_1=0$.
    \item 
      $F_2\circ \gamma(t_1, r_1, \cdot,
      r_2)\colon (-\epsilon,\epsilon)\to (-\epsilon,\epsilon)$
      has degree one, since the same holds for $r_2=0$.
    \item 
      $F_3(t_1,\cdot,t_2,r_2)\colon [0,\epsilon)\to (0,1]$
      has degree one near $1$.
    \item 
      $F_4(t_1,r_1,t_2,\cdot)\colon [0,\epsilon)\to (0,1]$
      has degree one near $1$.
  \end{itemize}
  It follows at once that 
  $(F\circ \gamma)^{-1}(\{0\}\times \{0\}\times \Delta)$ is a one-manifold with a single
  end over the origin.
\end{proof}

\begin{prop}
  The endomorphism $\dInt{\ell}$ satisfies
  $\dInt{\ell}\circ\dInt{\ell}=0$.
\end{prop}

\begin{proof}
  This is a straightforward synthesis of Propositions~\ref{prop:dMatchedZqZero}
  and~\ref{prop:dChangedSqZero}. Again, we look at two-dimensional
  moduli spaces. We find that their ends are of the following types:
  \begin{itemize}
  \item Two-story $\ell$-self-matched curve pairs.
  \item Type~\ref{def:JJ}.
  \item Type~\ref{def:OO}, if the orbit curves  are marked by $\orb_j$ with $j\not\in\Omega_k$
  \item Type~\ref{SMCP:XO}, if the odd orbit $\orb_j$ has $f(j)\leq \ell$.
  \item Type~\ref{SMCP:OX}, where the orbit curve on the left is marked by an odd orbit $\orb_j$,
    with $f(j)\leq \ell$.
  \item Type~\ref{SMCP:WX}.
  \item Type~\ref{SMCP:XW}.
  \item $\Omega$-matched story pair ends, of the form $(w_1,u_1),(w_2,u_2)$, 
    which occurs when $\ell$ is odd and 
    both $w_1$ and $w_2$ contain the puncture with $f$-value equal to $\ell+1$.
  \end{itemize}
  Ends of Type~\ref{def:OO} and~\ref{SMCP:XO} cancel with ends of
  Type~\ref{def:JJ} by Lemma~\ref{lem:Join} 
  when the orbit is has $f$-value greater than $\ell$; and a combination
  of Lemmas~\ref{lem:JJ} combined with Lemma~\ref{lem:XO} otherwise.
  Similarly, ends of type~\ref{SMCP:OX} where the orbit is an odd orbit in $\Omega_k$
  cancel with those of Type~\ref{SMCP:WX}, as in Lemmas~\ref{lem:OXodd} and~\ref{lem:WX}.
  Ends of Type~\ref{SMCP:OX} where the orbit is even and and has $f$-value less than or equal to $\ell$
  cancel with those of Type~\ref{SMCP:XW} where the orbit has $f$-value less than ore equal to $\ell$.
  By Lemma~\ref{lem:WW} and~\ref{lem:OXodd}, 
  the remaining ends of the form $(w_1,u_1),(w_2,u_2)$ 
  cancel with 
  ends of Type~\ref{SMCP:OX} where the orbit is odd, has and has 
  $f$-value equal to $\ell$; see Figure~\ref{fig:OXWW}.
\end{proof}

Consider $C$ equipped with the endomorphism
\begin{align*}
 \partial^{(k)} &(\x)
  &=\sum_{\y} 
  \sum_{\{B\in\pi_2(\x,\y)|\ind(B)=1\}}
  \#\left(\frac{\ModInt{\ell}^B(\x,\y)}{\R}\right)\cdot   U^{n_\wpt(B)}V^{n_\zpt(B)}
\cdot \y.
\end{align*}

\subsection{Interpolating between the intermediate complexes}
\label{subsec:Interpolate}

In this section we prove the following:

\begin{prop}
  \label{prop:Intermediates}
  When $n>1$, 
  there is an isomorphism of chain complexes over 
  $\Ring$
  \[ \Phi_{\ell}\colon (C,\partial^{(\ell)})\to (C,\partial^{(\ell+1)})\]
\end{prop}

The map $\Phi_\ell$ will be constructed by counting curves which generalize
the $\Omega$-matched curves, as follows.

\begin{defn}
  \label{def:MorMatch}
  Let $\ell$ be an integer between $1,\dots,2n$.
  An $\ell$-morphism matched curve, is the following data.
  \begin{itemize}
    \item a holomorphic curve $u_1$ in $\Hdown$, with source $\Source_1$
    \item a holomorphic curve $u_2$ in $\Hup$, with source $\Source_2$
    \item a subset $X\subset \East^{\ell}_+(\Source_1)$, 
      (i.e. which is empty if $\ell$ is odd)
    \item a subset $Y\subset \East^{\ell}_-(\Source_1)$
    \item an injection 
      $\phi\colon \IntPunctEv^{<\ell}(\Source_1)\cup X \to \East(\Source_1)$
    \item an injection $\psi\colon \East(\Source_2)\to\East(\Source_1)$
    \item a real number $t_0\in \R$
    \end{itemize}
    with the following properties:
    \begin{itemize}
    \item $\East(\Source_1)$ is a union of four disjoint sets,
      $\phi(\East^{<\ell}_+(\Source_1)\cup X)$, $\psi(\East(\Source_2))$,
      $\East^{<\ell}_-(\Source_1)$, and $Y$.
    \item $X$
      contains all punctures $q\in \East^{\ell}_+(\Source_1)$
      with $t\circ u_1(q)>t_0$ 
    \item $Y$ contains all punctures $q\in \East^{\ell}_-(\Source_1)$
      with $t\circ u_1(q)<t_0$ 
    \item If $p\in \IntPunctEv^{<\ell}(\Source_1)\cup X$ is labelled by some orbit $\orb_j$,
      then $\phi(p)$ is marked by a length one Reeb chord that covers
      the boundary component of $\Zin_k$, where $\{j,k\}\in\Mup$, and
      \[ t \circ u_1(\phi(p))=t\circ u_2(p). \]
    \item If $q\in \East^{<\ell}(\Source_2)$ or $q\in\East^{\ell}(\Source_2)$
      and $t \circ u_2(q)<t_0$, then $\psi(q)$ is labelled by a
      length one Reeb chord that covers the boundary component of
      $\Zin_j$, and
      \[ t\circ u_1(\psi(q))=t\circ u_2(q) \]
    \item If $q\in \East^{> \ell}(\Source_2)$ or
      $q\in \East^{\ell}(\Source_2)$ and $t \circ u_2(q)>t_0$, 
      then the marking on the Reeb chord or orbit
      of $q$ is the same as the marking on the Reeb chord or orbit of
      $\psi(q)$, and 
      \[ (s\circ u_1(\psi(q)),t \circ u_1(\psi(q)))=
      (s \circ u_2(q),t \circ u_2(q)).\]
    \item 
      If $q\in \East^{\ell}(\Source_2)$ and $t\circ u_2(q)=t_0$,
      then $\psi(q)$ is also labelled by the same Reeb orbit, and
      \[ t\circ u_1(\psi(q))=t\circ u_2(q)=t_0.\]
      Moreover, the following inequalities hold on the $s$ projection.
      Order the punctures $q\in \East^{\ell}(\Source_2)$ with $t\circ u_1(q)=t_0$ 
      $\{q_i\}_{i=1}^m$ so that the sequence $\{s\circ u_2(q_i)\}_{i=1}^m$ is 
      increasing; then 
      \[   s\circ u_2(q_i)< s\circ u_1(\psi(q_i)) \]
      and, if $i<m$,
      \[ s\circ u_1(\psi(q_i))<s\circ u_2(q_{i+1}).\]
      \end{itemize}
    Each $\ell$-morphism matched curve has three associated integers,
    $m_-$, $m$, and $m_+$, where $m_-$ resp. $m_+$ denotes the number of punctures $q$
    in $\East^{\ell}(\Source_2)$ with $t(q)<t_0$ resp. $t(q)>t_0$;
    and $m$ (as above) is the number of punctures 
    in $\East^{\ell}(\Source_2)$ with $t(q)=t_0$.
    The triple $(m_-,m,m_+)$ is called the {\em profile} of the $\ell$-morphism
    matched curve.

    Let $\ModMor{\ell}(\x,\Source_1,\Source_2,\phi,\psi)$ denote the
    moduli space of $\ell$-morphism matched curves.
\begin{align*} \ModMor{\ell}^B&(\x,\y)\\
  &=
\bigcup_{\left\{\begin{tiny}\begin{array}{r}
B_1\in \pi_2(\x_1,\y_1) \\B_2\in\pi_2(\x_2,\y_2)
  \end{array}\end{tiny}\Big|~B=B_1\natural B_2\right\}}
\bigcup_{\{(\Source_1,\Source_2,\phi,\psi)\big| \ind^{\natural}(B_1,\Source_1;B_2,\Source_2)=\ind(B)\}}
\ModMor{\ell}^{B_1,B_2}(\x,\y,\phi,\psi).
\end{align*}
\end{defn}

\begin{defn}
  Define a map $h_\ell\colon (C,\partial^{(\ell)})\to (C,\partial^{(\ell+1)})$
  by the formula
\begin{align*}
 h_\ell &(\x)
  &=\sum_{\y} 
  \sum_{\{B\in\pi_2(\x,\y)|\ind(B)=1\}}
  \#\ModMor{\ell}(\x,\y)\cdot   U^{n_\wpt(B)}V^{n_\zpt(B)}
\cdot \y.
\end{align*}
\end{defn}

\begin{lemma}
  \label{lem:ChainMap}
  \[ \partial^{(\ell+1)}\circ h_\ell + h_\ell\circ \partial^{(\ell)}
  = \partial^{(\ell+1)}+\partial^{(\ell)}. \]
\end{lemma}

\begin{proof}
  Consider ends of one-dimensional moduli spaces
  $\ModMor{\ell}(\x,\y)$.
  There are two cases, according to the parity of $\ell$. Suppose that
  $\ell$ is odd.

  There are ends as involving punctures other than the ones whose
  $t$-projection is $t_0$. These ends are as in the intermediate
  complexes in Proposition~\ref{prop:Intermediates}. Many of these ends cancel in pairs
  in the proof of that proposition, leaving the two-story buildings, 
  and the ends that involve the special $t_0$-level.

  Those ends in turn can be classified, as follows.  Let
  $\{q_1,\dots,q_m\}\in \East^{\ell}(\Source_2)$ be the punctures with
  $t\circ u_1(q_i)=t_0$, labelled as in Definition~\ref{def:MorMatch}.
  \begin{enumerate}[label=($\flat$-\arabic*),ref=($\flat$-\arabic*)]
    \item \label{end:OffRight} The end corresponds to $s(u_1(q_m))\goesto 1$;
      in this case, there is a Gromov limit
      to $((u_1,v_1),u_2)$,
      where $v_1$ is an orbit curve (for an orbit whose $f$-value is $\ell$),
      attached at the level $t_0$.
      These ends are labelled by integers $(m_-,m-1,m_+)$,
      where $m_-$ and $m_+$ are defined as in Definition~\ref{def:MorMatch}.
    \item \label{end:OffLeft}
      The end corresponds to $s(u_2(\psi(q_1)))\goesto 0$;
      in this case, there is a Gromov limit
      to $((w_1,u_1),(w_2,v_2))$, where 
      $w_1$ and $w_2$ are simple boundary degenerations,
      both of which contain an orbit $\orb_j$
      with $f(j)=\ell$.
      These ends are labelled by integers $(m_-,m-1,m_+)$.
    \item
      \label{end:HorizCollision}
      Pairs $(u_1,u_2)$ with $s(u_1(q_i))=s(u_2(\psi(q_i)))$
      or $s(u_2(\psi(q_i)))= s(u_1(q_{i+1}))$.
      These ends are labelled $(m_-,m,m_+,j)$ where
      with the convention $j=2i-1$, if $s(u_1(q_i))=s(u_2(\psi(q_i)))$;
      and $j=2i$ if $s(u_2(\psi(q_i)))= s(u_1(q_{i+1}))$.
    \item 
      \label{end:OrbFromBelow}
      there is some $q\in \East^{\ell}(\Source_2)$ 
      with $t(u_2(q))=t_0$, but $q$ arises as a limit point 
      of punctures with  $t(u_2(q))<t_0$.
      Let $s_0=s(u_1(\psi(q)))=s(u_2(q))$.
      These ends are labelled $(m_-,m,m_+,j)$ 
      where
      \begin{equation}
        j =\left\{\begin{array}{ll}
        1 & {\text{if $s_0< s\circ u_1(q_1)$}} \\
        2i-1 & {\text{if $s\circ u_2(\psi(q_{i-1}))<s_0 <s\circ u_2(q_i)$}} \\
        2i &{\text{if $s\circ u_2(q_i)< s_0< s\circ u_1(\psi(q_{i+1}))$}} \\
        2m &{\text{if $s\circ u_1(q_{m})< s_0$}}
        \end{array}\right.
        \label{eq:DefOfj}
      \end{equation}
      Let  $m_+$ here be one greater than the number of
      $q\in\East^{\ell}(\Source_2)$ with $t(u_2(q))>t_0$;
      this is $m_+$ for the curves before taking the Gromov limit.
    \item \label{end:OrbFromAbove} there is some $q\in \East^{\ell}(\Source_2)$ 
      with $t(u_2(q))=t_0$, but $q$ arises as a limit point
      of punctures with $t(u_2(q))> t_0$.
      These ends are labelled $(m_-,m,m_+,j)$,
      where $j$ is defined as in Equation~\eqref{eq:DefOfj};
      where now $m_-$ is computed before taking the Gromov limit.
    \item \label{end:OrbFromAboveA} there is some $q\in \East^{\ell}(\Source_1)$ 
      with $t(u_1(q))=t_0$, arising as a limit point
      of punctures with $t(u_1(q))> t_0$.
      These ends are labelled $(m_-,m,m_+,j)$,
      where $j$ is defined as in Equation~\eqref{eq:DefOfj};
      and $m_-$ is computed before taking the Gromov limit.
  \end{enumerate}

 \begin{figure}[h]
 \centering
 \input{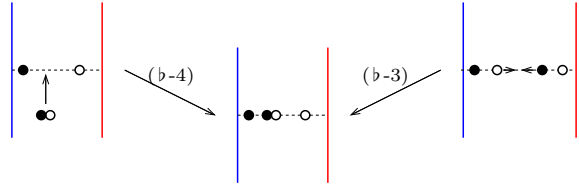}
 \caption{{\bf Ends of Types~\ref{end:OrbFromBelow} and~\ref{end:HorizCollision}
     cancel}
   We have drawn here the strip: the light dots represent the images under the 
   projection to $[0,1]\times \R$ of the punctures on $\Source_1$; 
   the dark ones represent the punctures on $\Source_2$.   
 \label{fig:MorMod1}}
 \end{figure}

  End of Type~\ref{end:HorizCollision} $(m_-,m,m_+,j)$
  cancel with end of Type~\ref{end:OrbFromBelow} $(m_-,m-1,m_++1,j)$
  except when $m=1$; see Figure~\ref{fig:MorMod1}.

  Ends of Type~\ref{end:OrbFromAbove} $(m_-,m,m_+,j)$ cancel in pairs
  except in the special case where $j=m$.

 \begin{figure}[h]
 \centering
 \input{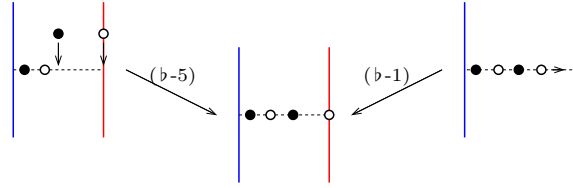}
 \caption{{\bf Ends of Type~~\ref{end:OffRight}
     cancel certain ends of Type\ref{end:OrbFromAbove},}
   when the ennd of Type~\ref{end:OrbFromAbove} has the form
   $(m_-,m,m_+,m)$.
 \label{fig:MorMod2}}
 \end{figure}

 \begin{figure}[h]
 \centering
 \input{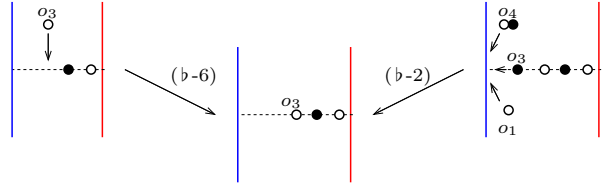}
 \caption{{\bf Ends of Types~\ref{end:OrbFromAboveA} and~\ref{end:OffLeft}
     cancel}
 \label{fig:MorMod3}}
 \end{figure}

  Ends of Type~\ref{end:OffRight} $(m_-,m,m_+)$
  cancel with ends of Type~\ref{end:OrbFromAbove} 
  $(m_-+1,m-1,m_+,m-1)$ except when $m=1$.
  
  Ends of Type~\ref{end:OffLeft}  $(m_-,m,m_+)$ cancel with ends of
  Type~\ref{end:OrbFromAboveA} $(m_-+1,m-1,m_+)$ 
  except when $m=1$.

  The remaining ends are: Type~\ref{end:HorizCollision} $(m_-,1,m_+,1)$,
  Type~\ref{end:OffRight} $(m_-,1,m_+)$, and Type~\ref{end:OffLeft}
  $(m_-,1,m_+)$. Now, provided that $m_->0$,
  ends of Type~\ref{end:HorizCollision} $(m_-,1,m_+,1)$
  correspond to ends of Type~\ref{end:OffRight} and~\ref{end:OffLeft}
  $(m_--1,1,m_++1)$.

  After these further cancellations, the remaining ends are
  of Type~\ref{end:HorizCollision} $(0,1,m_+,1)$ -- which
  correspond  to the terms in $\partial^{(\ell)}$ -- and ends of type
  Type~\ref{end:OffRight} and~\ref{end:OffLeft}
  $(m_-,1,0)$ -- which correspond to the terms in $\partial^{(\ell+1)}$.

  The remaining two-story buildings count terms
  in $\partial^{(\ell+1)}\circ h_{\ell}+h_{\ell}\circ \partial^{(\ell)}$,
  verifying that 
  $\partial^{(\ell+1)}\circ h_{\ell}+h_{\ell}\circ \partial^{(\ell)}=0$ when
  $\ell$ is odd.

  This discussion requires slight modifications in case $\ell$ is even.
  Let
  $\{q_1,\dots,q_m\}\in \East^{\ell}(\Source_2)$ be the punctures with
  $t\circ u_1(q_i)=t_0$, labelled as in Definition~\ref{def:MorMatch}.
  \begin{enumerate}[label=($\flat'$-\arabic*),ref=($\flat'$-\arabic*)]
    \item \label{end:eOffRight} The end corresponds to $s(u_1(q_m))\goesto 1$;
      in this case, there is a Gromov limit
      to $((u_1,v_1),u_2)$,
      where $v_1$ is an orbit curve (for an orbit whose $f$-value is $\ell$),
      attached at the level $t_0$.
    \item \label{end:eOffLeft}
      The end corresponds to $s(u_2(\psi(q_1)))\goesto 0$;
      in this case, there is a Gromov limit
      to $((w_1,u_1),u_2)$, where 
      $w_1$ is a simple boundary degeneration containing 
      an orbit $\orb_j$
      with $f(j)=\ell$,
      and $u_2$ contains an extra unmatched chord $\longchord_k$,
      which covers the boundary component $\Zdown_k$ so that $f(k)=\ell-2$.
    \item
      \label{end:eHorizCollision}
      Pairs $(u_1,u_2)$ with $s(u_1(q_i))=s(u_2(\psi(q_i)))$
      or $s(u_2(\psi(q_i)))= s(u_1(q_{i+1}))$.
    \item 
      \label{end:eOrbFromBelow}
      There is some $q\in \East^{\ell}(\Source_2)$ 
      with $t(u_2(q))=t_0$, but $q$ arises as a limit point 
      of punctures with  $t(u_2(q))<t_0$.
    \item \label{end:eOrbFromAboveA} There is some $q\in \East^{\ell}(\Source_1)$ 
      with $t(u_1(q))=t_0$, arising as a limit point
      of punctures with $t(u_1(q))> t_0$.
      Note that in this case there as an extra puncture $q'$ (the limit of $\phi(q)$)
      with $t(q')=t_0$ on $u_1$ labelled by $\longchord_k$, where
      $f(k)=\ell-1$.
  \end{enumerate}
  With the above remarks in place, verification of the stated relation 
  when $\ell$ is even proceeds much as before. The most significant
  difference is that the cancellation of ends of Type~\ref{end:eOrbFromAboveA}
  with those of Type~\ref{end:eOffLeft} is slightly simpler than the corresponding
  cancellation in the odd case; see Figure~\ref{fig:eMorMod}.
 \begin{figure}[h]
 \centering
 \input{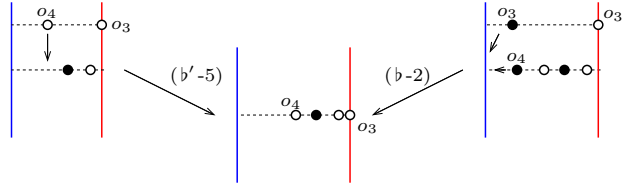}
 \caption{{\bf Ends of Types~\ref{end:eOrbFromAboveA} and~\ref{end:eOffLeft}
     cancel}
 \label{fig:eMorMod}}
 \end{figure}

\end{proof}

\subsection{Time dilation}
\label{subsec:TimeDilation}

Consider $(C,\dChanged)$. The differential $\dChanged$ counts self-matched curve pairs.
As in~\cite{InvPair}, we deform the matching appearing. In our case,
we deform the constraints appearing in the definition
of a self-matched curve pair (Equation~\eqref{eq:SelfMatchedCurvePair})
by conditions indexed by a real parameter $T$, as follows:
\[       T\cdot t\circ u_1(\psi(q))=t\circ u_2(q). \]
The corresponding moduli spaces are denoted 
$\ModMatchedChanged(T;\x,\y)$, which we call  the {\em moduli space of $T$-modified 
paritally self-matched curve pairs}.

Taking the limit as $T\goesto\infty$, the curves converge to combs,
whose algebraic information is contained in their main components.
These limiting objects are natural analogues of the ``trimmed simple ideal matched curves'' from~\cite[Definition~9.31]{InvPair}:

\begin{defn}
  \label{def:tsic}
  A {\em trimmed simple ideal partially self-matched curve} is a pair of
  holomorphic combs $(u_1,u_2)$ connecting two Heegaard states
  generators $\x=\x_1\#\x_2$ and $\y=\y_1\#\y_2$, where $u_1$ is a
  self-matched curve (in the sense of
  Definition~\ref{def:SelfMatched}), equipped with a a one-to-one
  correspondence $\varphi\colon \East(\Source_1)\setminus
  \phi(\IntPunctEv(\Source_1))\to\East(\Source_2)$ such that
  either 
  one of $u_1$ or $u_2$ is trivial, and the other has index $1$,
  and 
  $\East(\Source_1)\setminus\phi(\IntPunctEv(\Source_1))$ and $\East(\Source_2)$ are empty; or all of 
  the following conditions hold:
  \begin{enumerate}[label=(TSIC-\arabic*),ref=(TSIC-\arabic*)]
    \item 
      \label{TSIC:LeftIsCurve} The comb $u_1$ is a holomorphic curve for $\Hdown$ asymptotic to a sequence of non-empty sets of 
      Reeb chords $\vec{\rhos}=(\rhos_1,\dots,\rhos_m)$
    \item \label{TSIC:Ind1} $u_1$ has index $1$ with respect to $\vec{\rhos}$.
    \item \label{TSIC:Ind2} $u_2$ is a height $m$ holomorphic building for $\Hup$ with no components at east infinity
    \item \label{TSIC:Ind3} each story of $u_2$ has index one.
    \item \label{TSIC:SBM}
      $u_1$ and $u_2$ are strongly boundary monotone
    \item
      \label{TSIC:Composable}
        for each $i=1,\dots,m$, the east punctures of the $i^{th}$ story of $u_2$ are labelled, in order,
      by a non-empty sequence of Reeb chords $(-\rho^i_1,\dots,-\rho^i_{\ell_i})$ with the property that
      the sequence of singleton sets of chords $\vec{\rho}^i=(\{\rho^i_1\},\dots,\{\rho^i_{\ell_i}\})$ are composable.
    \item 
      \label{TSIC:Matching}
      The composition of the sequence of singleton sets of Reeb chords $\rho^i$ on the $i^{th}$ story of $u_2$
      coincides with the $i^{th}$ set of reeb chords $\rhos_i$ in the partition for $w$. 
  \end{enumerate}
  Let $\ModMatchedTSIC^{B_1\natural B_2}(\x,\y)$ denote the moduli space of trimmed simple ideal partially
  self-matched curves.
\end{defn}

\begin{prop}
  \label{prop:TGoestoInf}
  Fix $\x=\x_1\#\x_2,\y=\y_1\#\y_2\in\States(\HD=\Hup\#\Hdown)$,
  $B_1\in\doms(\x_1,\y_1)$,
  and 
  $B_2\in\doms(\x_2,\y_2)$.
  For each generic $T$ sufficiently large,
  \[ \#\ModMatchedTSIC^{B_1\natural B_2}(\x,\y)=
  \#\ModMatchedChanged^{B_1\natural B_2}(T;\x,\y).\]
\end{prop}

\begin{proof}
  The proof is as in~\cite[Proposition~9.40]{InvPair}; the key
  difference being that in~\cite{InvPair}, there is no self-matching
  (in particular $u_1$ a holomorphic curve rather than a self-matched
  curve), but this does not affect the argument.

  In a little more detail, Gromov compactness shows that as
  $T\goesto\infty$, the $T$-selfmatched curves converge to a pair of
  combs $U_1$ and $U_2$ (in $\Hdown$ and $\Hup$ respectively)
  satisfying a matching condition. Throwing out the East infinity
  curves, we arrive at a pair of combs $(u_1,u_2)$.  The matching
  conditions (Condition~\ref{TSIC:Matching}) is clear; the fact that
  $u_1$ is a (one-story) holomorphic curve (Condition~\ref{TSIC:LeftIsCurve})
  of index one (Property~\ref{TSIC:Ind1}) follows
  from the index formula. 

  Boundary monotonicity of $u_1$ follows from
  Lemma~\ref{lem:StrongMonotoneClosed}. It follows then from
  Lemma~\ref{lem:NonZeroAlgElts} that the algebra elements from each
  packet are non-zero. It now follows that $u_2$ has no
  $\alpha$-boundary degenerations; for such a boundary degeneration
  would give rise to a vanishing algebra element. Having eliminated
  $\alpha$-boundary degenerations from $u_2$,
  Condition~\ref{TSIC:Ind2} and~\ref{TSIC:Ind3} follows from the index
  formula. 
  Boundary monotonicity of $u_2$ follows now from Lemma~\ref{lem:NonZeroAlgElts}
  combined with Proposition~\ref{prop:SBD}.
  The matching conditions~\ref{TSIC:Matching} is straightforward.

  Conversely, the existence of $T$-self-matched curves for sufficiently
  large $T$ follows from a gluing argument as in the proof 
  of~\cite[Proposition~9.40]{InvPair}.
\end{proof}

\subsection{Putting together the pieces}

We can now assemble the steps to provide the main theorem:

\begin{proof}[Proof of Theorem~\ref{thm:PairAwithD}]
  Start from the complex $C_\Ring(\HD)=(C,\partial)$ for the
  doubly-pointed Heegaard diagram, with differential as in
  Equation~\eqref{eq:OriginalComplex}. When $n>1$, Theorem~\ref{thm:NeckStretch}
  (neck stretching) identifies $C_\Ring(\HD)\simeq
  (C,\partial^{(0)})$, where the latter differential counts matched
  holomorphic curves.  Proposition~\ref{prop:Intermediates}
  gives the sequence of isomorphisms
  \[ (C,\partial^{(0)})\cong\cdots \cong (C,\partial^{(2n)}). \]
  Note that $(C,\partial^{(2n)})=(C,\dChanged)$.
  When $n=1$, 

  Next, we replace the
  differential $\dChanged$ by a new differential $\dChangedT$ which
  counts $T$-modified partially self-matched pairs; i.e. points in
  $\ModMatchedChanged(T;\x,\y)$. When $T=1$, clearly
  $\dChanged=\dChangedT$.  The chain homotopy type of
  $(C,\dChangedT)$ is independent of the choice of $T$: i.e. varying
  $T$ gives chain homotopy equivalences between the various choices of
  complex.  (This is the anlogoue of~\cite[Proposition~9.22]{InvPair},
  with the understanding that now, in one-dimensional families, we
  have orbit curve end cancellation ends in addition to the
  cancellation of join curve ends as in~\cite{InvPair};
  cf. Lemma~\ref{lem:Join} above.)  Taking $T$ sufficiently large
  as in Proposition~\ref{prop:TGoestoInf}, and composing homotopy
  equivalences, we find that $C_{\Ring}(\HD)$ is chain homotopic to
  $(C,\partial')$, where now $\partial'$ counts trimmed simple ideal
  partially self-matched curves (Definition~\ref{def:tsic}). 
  Since $u_1$ is strongly boundary monotone (which can be phrased in terms
  of chord packets, thanks to Lemma~\ref{lem:SBA}),
  Lemma~\ref{lem:NonZeroAlgElts} guarantees that the 
  objects counted in $\partial'$ correspond to the algebraic
  counts appearing in the differential on $\Amod(\Hup)\DT\Dmod(\Hdown)$.
\end{proof}

\subsection{The case where $n=1$}
\label{subsec:Nequals1}

The case where $n=1$ works technically a little differently from the
case where $n>1$.  The key distinguishing feature is that in the case
where $n=1$, closed components do exist in the Gromov
compactification. (It is also, of course a bit simpler, since we have to
deal with deforming only one pair of matched orbits.)

As we shall see in our proof of Theorem~\ref{thm:MainTheorem}
(Section~\ref{sec:Comparison}), we will need the case $n=1$ only in a
very specific special case: gluing on the standard lower diagram,
which has the property that any homotopy class that covers both $Z_1$
and $Z_2$ also covers the two basepoints $\wpt$ and $\zpt$. This property 
would allow us to simplify the arguments considerably; but in the interest
of giving a clean statement of Theorem~\ref{thm:PairAwithD}, we give a
proof when $n=1$ without these restrictions hypotheses.

\subsubsection{Matched curves}

Consider the notion of matched curves (as in
Definition~\ref{def:MatchedPair}), except where the objects
${\overline u}_1$ and ${\overline u}_2$ are stories, rather than
simply curves. When $n>1$, Lemma~\ref{lem:NoClosedCurves} shows that
in sufficiently small index (and in homology classes not covering both
$\wpt$ and $\zpt$), the combs contain no closed components;
Lemma~\ref{lem:NoBoundaryDegenerations} shows that they can contain no
boundary degenerations. Thus, with these hypotheses, the matched
stories are automatically matched curves.

This is no longer the case where $n=1$. Specifically,
Lemma~\ref{lem:NoClosedCurves} fails in this case: moduli spaces of
self-matched stories are expected to contain closed components (on the
$\Hup$ side); and indeed, after removing those components, we obtain a (suitably) generalized matched curve in a moduli space of the same expected dimension.

We formalize these curves as follows:

\begin{defn}
  A  {\em special matched pair}
  consists of 
  \begin{itemize}
  \item a holomorphic curve $u_1$ in $\Hdown$ with source $\Source_1$ representing homology class $B_1\in\doms(\x_1,\y_1)$
  \item a holomorphic curve $u_2$ in $\Hup$ with source $\Source_2$ representing homology class $B_2\in\doms(\x_2,\y_2)$
  \item a subset $X\subset \East(\Source_1)$ of punctures marked by the orbit $\orb_1$, 
  \item a subset $Y\subset \East(\Source_1)$ of punctures marked by the orbit $\orb_2$, 
  \item a one-to-one correspondence $\phi\colon X\to Y$
  \item an injection $\psi\colon \East(\Source_2)\to \East(\Source_1)$
  \end{itemize}
  with the following properties:
  \begin{itemize}
  \item $\East(\Source_1)$ is a disjoint union of
    $\psi(\East(\Source_2))$, $X$, and $Y$.
  \item For each $q\in \East(\Source_2)$ is marked with a Reeb orbit
    or chord in $\Hup$, the corresponding puncture $\psi(q)\in
    \East(\Source_1)$ is marked with the Reeb orbit or chord in
    $\Hdown$ with the same name.
  \item For each $q\in\East(\Source_2)$,
    \[ (s\circ u_1(\psi(q)),t\circ u_1(\psi(q)))=(s\circ u_2(q),t\circ u_2(q)).\]
  \item For each $p\in X$
    \[ (s\circ u_1(\phi(p)),t\circ u_1(\phi(p)))=(s\circ u_2(p),t\circ u_2(p)).\]
  \end{itemize}
  If $B_1$ and $B_2$ induce $B\in \doms(\x,\y)$, let $\SModMatched^B(\x_1,\y_1;\x_2,\y_2;\Source_1,\Source_2;\psi,\phi)$
  denote the modul space of matched pairs.
\end{defn}

We have the following analogue of Lemmas~\ref{lem:NoClosedCurves} for
special matched curves.

\begin{lemma}
  \label{lem:NoBoundaryDegenerationsMatchedNone}
  Suppose $n=1$.
  Fix $B_1\in\doms(\x_1,\y_1)$ and $B_2\in\doms(\x_2,\y_2)$ so that
  $\weight_i(B_1)=\weight_i(B_2)$ for $i=1,\dots,2n$, and at least one of
  $n_\wpt(B_1)$ or $n_\zpt(B_1)$ vanishes, and so that $\ind(B_1\natural
  B_2)\leq 2$.  Then, curves in the Gromov compactification of
  $\SModMatched^B(\x_1,\y_1;\x_2,\y_2)$ contain no
  closed components or boundary degenerations.
\end{lemma}

\begin{proof}
  For a closed component to form, at least two pairs of matched orbits
  must come together (i.e. it could either be that two punctures in $X$ along with 
  their two corresponding punctures in $Y$; or punctures $q_1$ and $q_2$ in $\Source_2$
  along with their matching punctures $p_1$ and $p_2$ in $\Source_1$).
  In any case, this occurs in codimension $2$, and is therefore
  excluded from the index computation. 

  Let $({\overline u}_1,{\overline u}_2)$ denote a Gromov limit.
  Suppose that ${\overline u}_2$ contains a boundary degeneration. It
  follows that there are two punctures in ${\overline \Source}_2$
  marked by Reeb orbits that cover both boundary components, and which
  project to $(1,\tau)$.  There must be two matching punctures in
  ${\overline \Source}_1$,
  which project to the same point $(1,\tau)$. 
  It follows that $\Source_1$ contains either a closed component projecting to $(1,\tau)$,
  or two boundary degenerations (covering $\Zin_1$ and $\Zin_2$). In either case,
  it follows that $n_\wpt(B_1)>0$ and $n_\zpt(B_1)>0$, violating our hypothesis.

  Similarly, if ${\overline u}_1$ contains a boundary degeneration,
  then it follows that ${\overline u}_2$ must also contain a closed component
  or a boundary degeneration, contradicting the above.
\end{proof}

We can form an endomorphism of $C$, obtained by counting index one
special matched curves. We denote this by $(C,\partial^{(0)})$.
In view of the above compactness result, $\partial^{(0)}$ is a differential.
(Compare Proposition~\ref{prop:dMatchedZqZero}.)

\begin{thm}
  \label{thm:NeckStretchNOne}
  When $n=1$, 
  suitable choices of almost-complex structures $J$ used to define
  $\CFKsimp(K)$, there is an isomorphism (of chain complexes)
  $\CFKsimp(K)\cong (C,\partial^{(0)})$.
\end{thm}
\begin{proof}
  Using the neck stretching from Theorem~\ref{thm:NeckStretch}, we
  finding limiting objects which are now matched stories. Those
  stories contain closed components on the $\Hup$-side. Removing each
  such closed component, we arrive at a corresponding special matched
  curve.  Conversely, given a special matched curve, at each puncture
  $p\in X$, and matching puncture $\phi(p)\in Y$, we attach a sphere
  that covers $\Hup$. Gluing that sphere gives
  the stated identification.

  To justify this, we need to observe that the count of curves
  after gluing spheres at each
  $\{p,\phi(p)\}$ puncture (with $p\in X$) agrees with the original count of
  special marked curves. It suffices to verify this in a special case;
  see Figure~\ref{fig:StabN1}. 

    \begin{figure}[h]
      \centering
      \input{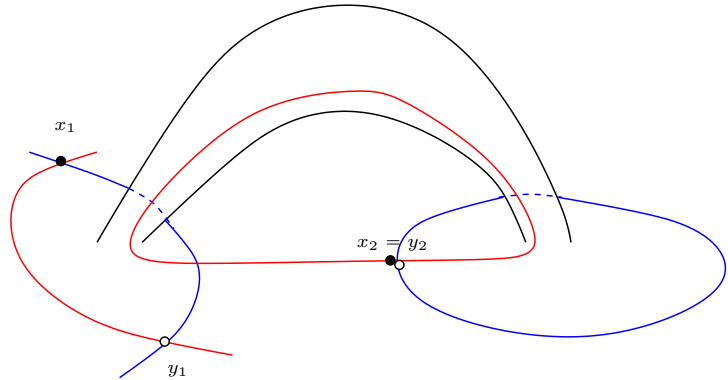}
      \caption{{\bf Local contributions of spheres}.
            \label{fig:StabN1}}
    \end{figure}

    We have exhibited a region in a Heegaard diagram. It is known that
    for any almost-complex structure, the displayed shadow from
    $\{x_1,x_2\}$ to $\{y_1,y_2\}$ has an odd number of
    representatives. View the annular region as an upper diagram. The
    moduli space from the lower diagram now consists of a bigon from
    $x_1$ to $y_1$ superimposed on a bigon from $x_2$ to itself. The
    first moduli space is one-dimensional, with a free $\R$ action,
    and an orbit $\orb_1$ in the interior. The second moduli space is
    two-dimensional, parameterized by a cut and an $\R$ action, with
    an orbit $\orb_2$ in its interior. We can scale the $\R$ action
    and the cut perameter so that $\orb_1$ and $\orb_2$ occur at the
    same $(s,t)$ coordinate; gluing in the sphere gives rise to the
    needed differential.
\end{proof}

\begin{rem}
  The gluing of closed components is very similar to the stabilization
  invariance proof in Heegaard Floer homology~\cite[Section~10]{HolDisk}.
  The key difference is that here we are gluing spheres with two punctures,
  rather than the one puncture considered there.
\end{rem}

\subsubsection{Comparison with self-matched curves}

The intermediate complexes considered when $n=1$ have a slightly
different form.  For $(C,\partial^{(1)})$, we count points in moduli
spaces satisfying the matching conditions for special matched curves,
except now punctures on the $\Hup$-side marked by $\orb_1$ are matched
with punctures on the $\Hdown$-side, marked by the corresponding length one
Reeb chord.  Similarly, in $(C,\partial^{(2)}$, all orbits on the $\Hup$-side
are matched with the corresponding length one Reeb chord; while orbits on the $\Hup$ side are matched with long chords on the $\Hup$ side.

We have the following analogue of Proposition~\ref{prop:Intermediates}:

\begin{prop}
  \label{prop:Intermediatesx}
  There is an isomorphism of chain complexes over $\Ring$
  \[ \Phi\colon (C,\partial^{(0)})\to (C,\partial^{(2)})\]
\end{prop}

\begin{proof}
  We construct first an isomorphism
  $\Phi_0\colon (C,\partial^{(0)})\to (C,\partial^{(1)})$
  as in the proof of Proposition~\ref{prop:Intermediates}.
  Specifically, we modify the definition of $0$-morphism matched curves
  as in Definition~\ref{def:MorMatch}, which comes equipped with a special
  time $t_0$ with the following properties:
  \begin{itemize}
  \item Each puncture in $\Source_2$ 
    marked by $\orb_2$ is matched with a length one chord in $\Source_1$
    that covers the corresponding boundary component, with the same $t$-projection.
  \item Each puncture in $\Source_2$ marked with the orbit $\orb_1$
    and which projects to $t>t_0$ is matched with an orbit on $\Source_1$
    with the same $(s,t)$ projection.
  \item Each puncture in $\Source_2$ marked with the orbit $\orb_1$
    and which projects to $t<t_0$ is matched with a chord on $\Source_1$
    marked with the corresponding length one chord, with the same $t$-projection.
  \item The punctures $\{q_i\}_{i=1}^m$ in $\Source_2$
    that project to  $t=t_0$ are all marked by $\orb_1$
    have corresponding punctures $\{\psi(q_i)\}_{i=1}^m$
    in $\Source_1$ that project to $t=t_0$.
    For all $i=1,\dots,m$
    $s\circ u_2(q_i)<s\circ u_1(\psi(q_i))$; and moreover
    for $i=1,\dots,m-1$,
    and $s\circ u_1(\psi(q_i))<s\circ u_2(q_{i+1})$.
  \item The remaining punctures on $\Source_1$ marked with $\orb_1$ are paired
    off with punctures on $\Source_1$ marked with $\orb_2$, with the same $(s,t)$ projection.
    \end{itemize}
  Define a map $h_0\colon C\to C$ counting rigid such objects, as
  in the proof of Lemma~\ref{lem:ChainMap}.

  We adapt the proof of Lemma~\ref{lem:ChainMap}, to show that
  \[ \partial^{(1)}\circ h_0 + h_0\circ \partial^{(0)}=\partial^{(0)}+\partial^{(1)}.\]
  In this adaptation, we consider once again ends of one-dimensional
  moduli spaces. Some of the ends considered in that proof cannot
  occur: ends of Type~\ref{end:OffRight} can occur.
  Ends of
  Type~\ref{end:OffLeft} are excluded: the argument of Lemma~\ref{lem:NoBoundaryDegenerationsMatchedNone} would show that if a boundary degeneration occurs in $\Hup$, then in fact the homology class on $\Hdown$ has $n_\wpt>0$ and $n_\zpt>0$.
  Ends of
  Types~\ref{end:HorizCollision},~\ref{end:OrbFromBelow}, and~\ref{end:OrbFromAbove} can occur; but ends of
  Type~\ref{end:OrbFromAboveA} are replaced by ends where there are
  two matching punctures on $\Source_1$, $q$ and $\phi(q)$, labelled
  by $\orb_1$ and $\orb_2$, which project to the same $(s,t)$
  coordinate.  These latter ends cancel in pairs (where the double
  puncture comes from $t>t_0$ or $t<t_0$. The remaining ends cancel in
  pairs as in the proof of Lemma~\ref{lem:ChainMap}.  (Note
  also that there are no join curve ends to the moduli spaces, since
  the boundary Reeb chords appearing in join curves do not exist for
  $n=1$ diagrams.)

  The isomorphism $\Phi_0$ is now given by $\Id+h_0$.

  A similar isomorphism is constructed
  $\Phi_1\colon (C,\partial^{(1)})\to (C,\partial^{(2)})$ is constructed analogously.
\end{proof}

We wish to compare $(C,\partial^{(2)})$ with the chain complex
$(C,\dChanged)$ defined by counting self-matched curves
(Definition~\ref{def:SelfMatched}). (Note that for $\partial^{(2)}$,
we allow punctures in $\Source_1$ marked by $\orb_2$ to project to the
same $(s,t)$-coordinate as some puncture marked by $\orb_1$; while
punctures in $\Source_2$ marked by $\orb_i$ project to the same
$t$-coordinate as some boundary puncture on $\Source_1$ marked by a
corresponding length $1$ Reeb chord.) An isomorphism between the two
complexes is constructed as follows:

\begin{prop}
  \label{prop:nEqualsOne}
  When $n=1$, there is an isomorphism of complexes over $\Ring$
  \[ (C,\partial^{(0)})\cong (C,\partial^{(2)}). \]
\end{prop}

\begin{proof}
  As in the proof of Proposition~\ref{prop:Intermediates}, we
  construct an isomorphism 
  \[ \Phi\colon
  (C,\partial^{(2)})\to(C,\dChanged),\] by counting certain curves.

  The curves we count in this morphism are equipped with a
  distinguished $t$-value $t_0$, with the following properties:
  \begin{itemize}
  \item All punctures on $\Source_2$ labelled with some orbit
    are paired off with punctures in $\Source_1$ labelled with
    the matching length one Reeb chord.
  \item The punctures $\{q_i\}_{i=0}^{2m-1}$ on $\Source_1$
    with $t(q_i)=t_0$ have the following properties:
    \begin{itemize}
    \item 
      $q_i$ is marked by the orbit $\orb_j$ for $j=1,2$
      where $i\equiv j\pmod{2}$
    \item
      $s(u_1(q_i))$ is a monotone increasing function of $i=0,\dots,2m-1$.
      \end{itemize}
  \item The remaining punctures $q$ on $\Source_1$ labelled with the orbit
    $\orb_2$ are paired off with punctures $\phi(q)$ in $\Source_1$:
    \item
    if $t(q)>t_0$, then $\phi(q)$ is labelled with $\orb_1$,
    and $q$ and $\phi(q)$ have the same $(s,t)$ projection.
  \item
    if $t(q)<t_0$, then $\phi(q)$ is labelled with the length one chord
    covering $\Zin_1$; and $t(u(\phi(q)))=t(u(q))$.
  \end{itemize}

    Counting such curves induces a map
    $h\colon (C,\partial^{(2)})\to (C,\dChanged)$.
    We claim that
    \[ \dChanged \circ h + h \circ \partial^{(2)}=\partial^{(2)}+\dChanged.\]
    This is obtained by looking at ends of one-dimensional moduli spaces of the above kind.
    The following kinds of ends can occur
    \begin{enumerate}[label=($\flat$'-\arabic*),ref=($\flat$'-\arabic*)]
    \item \label{end:OffRightx} $s(u_1(q_{2m-1}))\goesto 1$,
      so that in the Gromov limit, we get $((u_1,v_1),u_2)$,
      where $v_1$ is an orbit curve (with orbit $\orb_1$)
      attached at the level $t_0$.
    \item \label{end:OffLeftx}
      Ends where $s(u_2(\psi(q_0)))\goesto 0$;
      in this case, there is a Gromov limit
      to $((w_1,u_1),u_2)$, where 
      $w_1$ is a simple boundary degeneration containing $\orb_2$ and
      $\wpt$. 
    \item
      \label{end:HorizCollisionx}
      Pairs $(u_1,u_2)$ with $s(u_1(q_i))=s(u_2(\psi(q_i)))$
      or $s(u_2(\psi(q_i)))= s(u_1(q_{i+1}))$.
    \item 
      \label{end:OrbFromBelowx}
      there is some $q\in \East(\Source_2)$
      labelled by an orbit
      with $t(u_2(q))=t_0$, but $q$ arises as a limit point 
      of punctures with  $t(u_2(q))<t_0$.
    \item \label{end:OrbFromAbovex} there is some $q\in \East(\Source_2)$
      labelled by an orbit
      with $t(u_2(q))=t_0$, but $q$ arises as a limit point
      of punctures with $t(u_2(q))> t_0$.
    \item \label{end:OffLeftxGen}
      Ends where $s(u_2(\psi(q)))\goesto 0$
      and $t(u_1(q))>t_0$;
      in this case, there is a Gromov limit
      to $((u_1,v_1),u_2)$, where 
      $v_1$ is an orbit curve.
    \item \label{end:OffRightxGen}
      Ends where $s(u_2(\psi(q)))\goesto 1$
      and $t(u_1(q))>t_0$;
      in this case, there is a Gromov limit
      to $((w_1,u_1),v_2)$, where 
      $w_1$ is a simple boundary degeneration.
    \end{enumerate}
    Consider ends of Type~\ref{end:OffLeftxGen}. When $w_1$ contains
    $\orb_1$, it also contains $\zpt$. These cases do not count
    algebraically, since we have specialized to $UV=0$. Thus, ends of
    Type~\ref{end:OffLeftxGen} count when $w_1$ contains $\orb_2$ and
    $\wpt$; and these cancel against ends of 
    Type~\ref{end:OffRightxGen}.

    Ends of type~\ref{end:OffRightx} drop out in pairs, or cancel with ends
    of Type~\ref{end:OrbFromAbovex}, 
    when the latter orbit is labelled $\orb_2$.
    Ends of Type~\ref{end:OffLeftx} drop out with orbits of
    Type~\ref{end:OrbFromAbovex},
    when the latter orbit is labelled $\orb_1$.
    Ends of Type~\ref{end:HorizCollisionx}
    cancel with ends of Type~\ref{end:OrbFromBelowx} (noting that the latter
    punctures come in pairs). See Figure~\ref{fig:MorNOne} for an illustration.
    \begin{figure}[h]
      \centering
      \input{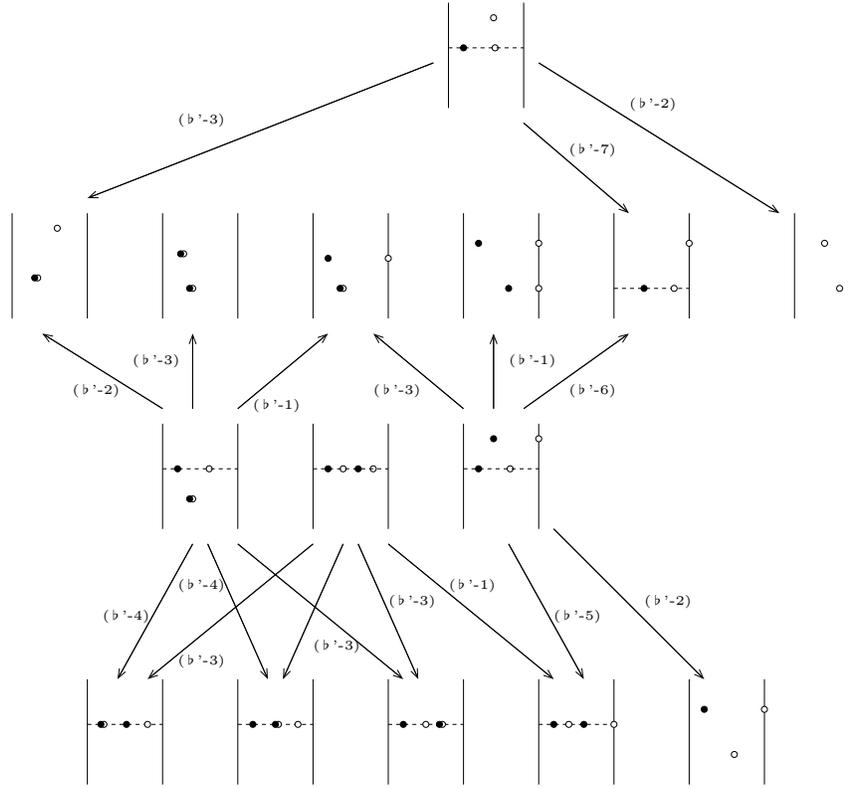}
      \caption{{\bf Cancellation of ends in the moduli spaces in the
          construction of $h\colon (C,\partial^{(2)})\to
          (C,\dChanged)$.} The drawings are shorthand: we have
        illustrated the projections of punctures in $\Source_1$ to the
        strip, coloring the ones labelled by $\orb_1$ white and those
        labelled by $\orb_2$ black. There are four moduli spaces,
        with arrows coming out of them; and their ends are at the ends
        of the arrows. Ends with two incoming arrows cancel, and ends
        with only one incoming arrow
        are the terms in $\partial_\sharp$ and $\partial^{(2)}$.}
      \label{fig:MorNOne}
    \end{figure}
\end{proof}

\newcommand\OmegaInEv{\widecheck{\mathbf\Omega}^+}
\newcommand\OmegaInOdd{\widecheck{\mathbf\Omega}^-}
\newcommand\orW{\vec{W}}
\newcommand\pin{\widehat{p}}
\newcommand\pout{\widecheck{p}}
\newcommand\DAmodBig{\lsup{\star}\DAmod}
\newcommand\ClginBig{\lsup{\star}\Clgin}
\newcommand\ClgoutBig{\lsup{\star}\Clgout}
\newcommand\rhosin{\rhos^{\wedge}}
\newcommand\rhosout{\rhos^{\vee}}
\newcommand\SourceMid{\Source}
\newcommand\IdempIn{\widecheck{I}}
\newcommand\IdempOut{\widehat{I}}
\section{Bimodules}
\label{sec:Bimodules}

In this section, we describe how to associate a type DA bimodule to a
middle Heegaard diagram (cf. Definition~\ref{def:MiddleDiagram}),
together with a matching on the incoming boundary components.  Loosely
speaking, the incoming boundary is treated as type $A$, and the
outgoing as type $D$.  (See~\cite{Bimodules} for the corresponding
construction in bordered Floer homology.)

In more detail, fix a middle diagram
\begin{align*} \Hmid=(\Sigma_0,(\Zin_1,&\dots,\Zin_{2m}),(\Zout_1,\dots,\Zout_{2n}),
\{\alphain_1,\dots,\alphain_{2m-1}\},
\{\alphaout_1,\dots,\alphaout_{2n-1}\},\\
&\{\alpha^c_1,\dots,\alpha^c_{g}\},
\{\beta_1,\dots,\beta_{g+m+n-1}\}), 
\end{align*}
and let $\MatchIn$ be a matching on $\{1,\dots,2m\}$, thought of as indexing the
components of $\Zin$. The Heegaard diagram induces a matching $\Mmid$
on all the boundary components of $\Sigma_0$.<

Together, $\MatchIn$ and $\Mmid$ given an equivalence
relation on the components of $\partial\Sigma$.
\begin{defn}
  \label{def:CompatibleDA}
  We say that $\MatchIn$ is
  {\em compatible} with $\Hmid$ if every equivalence class has some
  component of $\Zout$ in it. 
\end{defn}
Form $\Wmid=W(\Hmid)$ as in Definition~\ref{def:AssociatedW}, and
$\Win=W(\MatchIn)$.  The compatibility condition is equivalent to the
condition that the one-manifold $W=\Wmid\cup\Win$ has no closed
components.

\begin{defn}
  Let $\Hmid$ be a middle diagram, equipped with a matching $\MatchIn$ on the
  incoming boundary components.
  The {\em full incoming} algebra 
  $\ClginBig(\Hmid)$ and {\em full outgoing algebra} 
  $\ClgoutBig(\Hmid)$ are defined by
  \[\Clgin(\Hmid)=\bigoplus_{k=0}^{2m-1}\Clg(2m,k);
  \qquad \Clgout(\Hmid)=\bigoplus_{k=0}^{2n-1}\Clg(2n,k).\]
  We will be primarily interested in the $k=n$, summands,
  which we call the {\em incoming algebra} and the {\em outgoing algebra}
  respectively:
  \[\Clgin(\Hmid)=\Clg(2m,m);
  \qquad \Clgout(\Hmid)=\Clg(2n,n).\]
\end{defn}

Each middle Heegaard state $\x$ determines two subsets
\[ \alphain(\x)\subset \{1,\dots,2m\}\qquad
{\text{resp.}}\qquad
\alphaout(\x)\subset \{1,\dots,2n\}\]
consisting of those $i\in\{1,\dots,2m\}$ resp.  $\{1,\dots,2n\}$
with $\x\cap \alphain_i\neq \emptyset$ resp 
$\x\cap\alphaout_i\neq\emptyset$.
Each middle Heegaard state $\x$ has an {\em idempotent type}
$k=|\alphain(\x)|$. 

Let $\DAmodBig(\Hmid)$ be the $\Field$-vector space
spanned by
the middle Heegaard states of $\Hmid$.
Let 
\[ \IdempIn(\x)=\Idemp{\alphain(\x)}
\qquad{\text{and}}\qquad
\IdempOut(\x)=\Idemp{\{1,\dots,2n-1\}\setminus \alphaout(\x)}.\]
An
$\IdempRing(2n,k-m+n)-\IdempRing(2m,k)$-bimodule structure is specified by
\[  \IdempOut(\x)\cdot \x
\cdot \IdempIn(\x)=\x.\]
There is a splitting
\[ \DAmodBig(\Hmid)=\bigoplus_{k\in\Z}
~~~{\lsup{\IdempRing(2n,k-m+n)}\DAmod(\Hmid)}_{\IdempRing(2m,k)},\]
where $k$ is the idempotent type of $\x$. We will be primarily
interested in the summand where $k=m$,
\[ \lsup{\IdempRing(2n,n)}\DAmod(\Hmid)_{\IdempRing(2m,m)}.\]

Our goal here is to endow $\DAmod(\Hmid)$ with the structure of a type $DA$ bimodule structure
$\DAmod(\Hmid,\MatchIn)=\lsup{\Clgout}\DAmod(\Hmid)_{\Clgin}$, where
$\Clgin=\Clgin(\Hmid)$ and $\Clgout=\Clgout(\Hmid)$. 

To equip $\DAmod(\Hmid)$ with the structure of a $DA$ bimodule, we
choose further an orientation $\orW$ on $W=\Wmid\cup\Win$.  Each boundary
component $\Zin_i$ or $\Zout_j$ of $\Hmid$ corresponds to some point
on $W$.

This data specifies an orbit marking, in the following sense:

\begin{defn}
  \label{def:OrbitMarkingDA}
  A {\em special orbit} in $\Hmid$ covers an orbit in $\Zin_i$ that is matched
  with $\Zout_j$. 
  An {\em orbit marking} in a middle diagram $\Hmid$ is a partition of the
  simple orbits of $\Zin$ so that:
  \begin{itemize}
  \item $\MatchIn$ matches even with odd orbits,
  \item if components of $\Zin$ are matched by $\Mmid$, then one one
    is even and the other is odd.
  \end{itemize}
  Let $\OmegaInEv\subset \{1,\dots,2m\}$ denote the even boundary components of $\Zin$;
  and $\OmegaInOdd\subset \{1,\dots,2m\}$ denote the odd ones.
\end{defn}

An orientation on $W$ is equivalent to an orbit marking: each segment
of $W$ in $\Wmid$ is oriented from even to odd, and each segment in $\Win$
is oriented from odd to even.

The orientation on $W$ also induces a pair of functions
\[ \sigma\colon \{1,\dots,2m\}\to \{1,\dots,2n\}, \qquad \tau\colon
\{1,\dots,2m\}\to \{1,\dots,2n\},\] the {\em starting} and {\em
  terminal} points respectively.  Namely, $\sigma(p)=i$ and
$\tau(p)=j$ if $\Zin_p$ is contained on the oriented interval in $W$
starting at $\Zout_i$, and $\tau(p)=j$ if $\Zin_p$ is contained on the
oriented interval terminating in $\Zout_j$.

See Figure~\ref{fig:OrbitMarkingDA} for an example.

 \begin{figure}[h]
 \centering
 \input{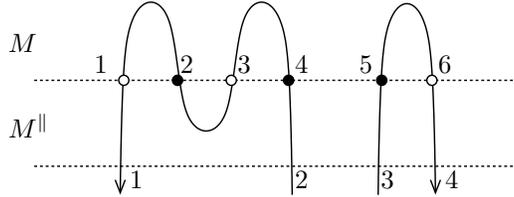}
 \caption{{\bf Orbit markings in a middle diagram.}
   The even orbits (1,3,6) are colored white and the odd ones (2,4,5) are colored black.
   Moreover, 
   $\sigma(1)=\sigma(2)=\sigma(3)=\sigma(4)=2$;
   $\tau(1)=\tau(2)=\tau(3)=\tau(4)=1$;
   $\sigma(5)=\sigma(6)=3$;
   $\tau(5)=\tau(6)=4$.}
 \label{fig:OrbitMarkingDA}
 \end{figure}

The $DA$ bimodule $\DAmod(\Hmid)$ depends on the incoming matching,
but we typically suppress data from the notation; in fact, even the
curvature element in the incoming algebra $\Clgin$ depends on this
choice. The bimodule also depends on an orientation of $W$, which we
also suppress.

\subsection{Type $DA$ bimodules}
\label{subsec:DAmodConstruction}

We adapt the definitions from Section~\ref{subsec:AlgebraicConstraints}
in the following straightforwarde manner.
Suppose that $\rhos$ is a set of Reeb chords for $\Hmid$
that is algebraic in the sense of Definition~\ref{def:AlgebraicPacket},
and that is supported entirely on $\Zin$. For each $\rho\in \rhos$, 
$\alpha(\rho^+)$ resp.
$\alpha(\rho^-)$ be the curve $\alphain_i$ with $\rho^+\in\alphain_i$
resp. $\rho^-\in\alphain_i$. 
Let \[ I^-(\rhos)=\sum_{\{{\mathbf{s}}\big| \{\alphain(\rho_1^-),\dots,\alphain(\rho_j^-)\}\subset
\{\alphain_i\}_{i\in{\mathbf{s}}}\}} I_{\mathbf{s}}
\qquad\text{and}\qquad
I^+(\rhos)=\sum_{\{{\mathbf{s}}\big| \{\alphain(\rho_1^+),\dots,\alphain(\rho_j^+)\}\subset
\{\alphain_i\}_{i\in{\mathbf{s}}}\}} I_{\mathbf{s}}
\]
Then, $\bIn_0(\rhos)$ be the algebra element $a_0\in\BlgZ(2m,m)$ with
$a=I^- \cdot a\cdot I^+$ and whose weight $w_i(a)$ is the average local
multiplicity at $\Zin_i$ for $i=1,\dots,2m$. 
Let $\bIn(\rhos)$ be the image of $\bIn_0(\rhos)$
in $\Clgin$.

The definition of constraint packets has the following immediate
generalization to middle diagrams:
\begin{defn}
  \label{def:CompatiblePacketDA}
  Fix a Heegaard state $\x$ and a sequence $\vec{a}=(a_1,\dots,a_\ell)$ of pure
  algebra elements of $\Clgin(\Hmid)$. 
  A sequence of constraint packets $\rhos_1,\dots,\rhos_k$ is called
  \em{$(\x,\vec{a})$-compatible} if  there is a sequence 
  $1\leq k_1<\dots<k_\ell\leq k$ so that the following conditions hold:
  \begin{itemize}
  \item the constraint packets $\rhos_{k_i}$ consist of chords
    in $\Zin$, and they are algebraic, in the sense
    of Definition~\ref{def:AlgebraicPacket},
  \item  $\IdempIn(\x)\cdot \bIn(\rhos_{k_1})\otimes\dots\otimes \bIn(\rhos_{k_\ell})=
    {\mathbf I}(\x)\cdot a_1\otimes\dots\otimes a_{\ell}$,
    as elements of $\DAmod(\Hdown)\otimes \Clgin^{\otimes \ell}$
  \item
    for each $t\not\in \{k_1,\dots,k_\ell\}$,
    the constraint packet $\rhos_t$ is one 
    of the following  types:
    \begin{enumerate}[label=($DA\rhos$-\arabic*),ref=($DA\rhos$-\arabic*)]
    \item 
      \label{eq:OddOrbit}
      it consists of the orbit $\{\orb_i\}$,
      where $\orb_i$ is {\em odd} for the $\orW$-induced orbit marking.
    \item 
      \label{eq:EvenOrbit}
      it is of the form
      $\{\orb_i,\longchord_j\}$, where $\orb_i$ is even
      $\{i,j\}\in\Matching$, and
      $\longchord_j$ is one of the two Reeb chords
        that covers $\Zin_j$ with multiplicity one. 
    \item 
      \label{eq:OutOrbit}
      it is of the form $\{\orb_k\}$, where $\orb_k$ is the simple Reeb orbit
      around some component in $\Zout$
    \item 
      \label{eq:OutChord} it is of the form $\{\rho_j\}$, where $\rho_j$ is a Reeb chord
      of length $1/2$ supported in some $\Zout$.
  \end{enumerate}
\end{itemize}
Let $\llbracket \x,a_1,\dots,a_\ell\rrbracket$ 
be the set of all sequences of constraint packets $\rhos_1,\dots,\rhos_k$
that are $(\x,\vec{a})$-compatible.
\end{defn}

Given $(\rhos_1,\dots,\rhos_k)\in\llbracket
\x,a_1,\dots,a_{\ell}\rrbracket$, we can consider
$\pi_2(\x,\rhos_1,\dots,\rhos_k,\y)$, the space of homology
classes of maps with asymptotics at $\Zin$ 
specified by the given Reeb chords.

We define the index $\ind(B,\x,\y;\vec{\rhos})$ exactly as in the case
of type A modules (Definition~\ref{def:IndexTypeA}):
\begin{align*}
  \ind(B,\x,\y;\vec{\rhos}) &= e(B)+n_\x(B)+n_\y(B)+\ell   \label{eq:EmbIndA} \\
  & \qquad -\weight(\vec{\rhos})+\iota(\chords(\vec{\rhos}))+\sum_{o\in\orb(\vec\rhos)} (1-\weight(o)), \nonumber
\end{align*}

\begin{lemma}
  \label{lem:MgrAndAgr}
  Given $B\in\doms(\x,\y)$,
  the quantity
  \[ \Mgr(B)=e(B)+n_\x(B)+n_\y(B)-\weight_\partial(B) \] is indendent
  of $B$ (depending only on $\x$ and $\y$).  
  Also, if $\{j,k\}\in\Mmid$, oriented in $\Wmid$ from $k$ to $j$, then the integer
  \[ \Agr_{\{j,k\}}=\weight_{j}(B)-\weight_{k}(B),\] 
  is independent of the choice of
  $B$ (depending only on $\x$ and $\y$).
\end{lemma}

\begin{proof}
  Clearly, $\{j,k\}\in\Mmid$ if and only if $i$ and $j$ are contained
  in the same component $\Bjk$ of
  \[ \Sigma_0\setminus(\beta_1\cup\dots\cup\beta_{g+m+n-1}).\]
  Now, 
  \[ \pi_2(\x,\y)\cong \bigoplus_{\{j,k\}} \Z\cdot \Bjk.\]

  The lemma follows from the following:
  \begin{align*}
    e(\Bjk)+n_{\x}(\Bjk)+n_{\y}(\Bjk)&=2 \\
    \weight_j(\Bjk)&=\weight_k(\Bjk)=1.
  \end{align*}
\end{proof}

We can think of $\Agr(B)$ as defining an element ${\mathbb
  A}(B)\in H^1(\Wmid,\partial \Wmid)\cong \Z^{m+n}$, the {\em Alexander grading}.  
Now, the Alexander grading, with
values in $H^1(\Wmid,\partial \Wmid)$; and the $\delta$-grading $\gr$, with
values in $\Q$, are determined up to overall constants by
\begin{align*}
  \Agr(\x)-\Agr(\y)&=\Agr(B) \\
  \gr(\x)-\gr(\y)&=e(B)+n_\x(B)+n_\y(B)-\weight_\partial(B)
\end{align*}

Given $B\in\pi_2(\x,\rhos_1,\dots,\rhos_h,\y)$, 
let $\bOut_0(B)=a\in\BlgZ(2n,n)$ be the algebra element 
with $\Iup(\x)\cdot a=a$ and whose weight
at $i\in\{1,\dots,2n\}$ agrees with the local multiplicity of
$B$ at $\Zout_i$.
Let $\bOut(B)$ denote the image of $\bOut_0(B)$ in $\Clgout$.
Given, an $(\x,\vec{a})$-compatible sequence of constraint packets
$\rhos_1,\dots,\rhos_h$, we define a corresponding monomial in the
variables $U_1,\dots,U_{2n}$, denoted $\gamma(\rhos_1,\dots,\rhos_h)$, to
be product over all packets of Type~\ref{eq:OddOrbit} of the element
$U_{\tau(i)}$.
Let
\begin{equation}
  \label{eq:bOut}
  \bOut(B,\rhos_1,\dots,\rhos_k)=\gamma(\rhos_1,\dots,\rhos_h)\cdot
  \bOut(B).
\end{equation}

Let $\ModFlow(\x,\y,\rhos_1,\dots\rhos_h)$ denote the moduli space of
flowlines as in Definition~\ref{def:GenFlow}, with the understanding
that now the asymptotics at $\pm \infty$ go to middle Heegaard states
$\x$ and $\y$.  Note also that for middle diagrams, the number of
$\beta$-curves is given by $d=g+m+n-1$.

Define
\begin{align}
\label{eq:DefDA-Action}
\delta^1_{\ell+1}&(\x,a_1,\dots,a_\ell)\\
&=
\sum_{
\left\{\begin{tiny}
\begin{array}{r}
\y\in\States \\
(\rhos_1,\dots,\rhos_k)\in \llbracket \x,a_1,\dots,a_\ell\rrbracket \\
B\in\pi_2(\x,\rhos_1,\dots,\rhos_h,\y)
\end{array}
\end{tiny}\Big| \ind(B,\rhos_1,\dots,\rhos_h)=1\right\}}
\!\!\!\!\!\!\!\!\!\!\!\!\!\!\!\!\!\!\!\!\!\!\!\!\#\UnparModFlow(\x,\y,\rhos_1,\dots\rhos_h)\cdot \bOut(B,\rhos_1,\dots,\rhos_h)\otimes \y.
\nonumber
\end{align}

Lemma~\ref{lem:CompatWithMgr} has the following straightforward adaptation:

\begin{lemma}
  \label{lem:CompatWithMgrDA}
  Fix $\x,\y\in\States$, a sequence of pure algebra elements
  $\vec{a}=(a_1,\dots,a_\ell)$, an $(\x,\vec{a})$-compatible sequence
  of constraint packets $\rhos_1,\dots,\rhos_h$, and
  $B\in\doms(\x,\y)$.  If there is a pre-flowline $u$ whose shadow is
  $B$ and whose  packet sequence is $(\rhos_1,\dots,\rhos_\ell)$,
  then
  \[ \gr(\x)+\ell-\sum_{i=1}^{\ell}\weight(a_i)=\gr(\y)-\weight_{\Zout}(\bOut(B))+
  \ind(B,\x,\y,\rhos_1,\dots,\rhos_h).\]
\end{lemma}

\begin{proof}
  We have that
  \[ \gr(\x)-\gr(\y)=e(B)+n_\x(B)+n_\y(B)-\weight_\partial(B), \]
  and
  \[
  \ind(B)=e(B)+n_\x(B)+n_\y(B)+h-\weight_\partial(B)+\sum\iota(\chords(\rhos_i)).
  \]
  Taking the difference, we find that
  \[ \gr(\x)-\gr(\y)-\ind(B,\x,\y,\vec{\rhos})
  =-h-\sum\iota(\chords(\rhos_i)).\]
  Now, if $\rhos_i$ is an algebraic packet, then
  $\iota(\chords(\rhos_i))=-\weight(\rhos_i)$;
  if it is of Type~\ref{eq:OddOrbit},
  $\iota(\chords(\rhos_i))=0$;
  if it is of Type~\ref{eq:EvenOrbit},
  $\iota(\chords(\rhos_i))=-1$;
  if it is of Type~\ref{eq:OutOrbit},
  the contribution is $0$;
  if it is of Type~\ref{eq:OutChord},
  we get $-1/2$. 
  It follows that
  \begin{align*}
    \gr(\x)-\gr(\y)-\ind(B,\x,\y,\vec{\rhos})
    &=-\ell+\left(\sum_{i=1}^{\ell}\weight_{\Zin}(a_i)\right) \\
    & \qquad
  -\#(\text{odd orbits coming in})-\weight_{\Zout}(B) \\
  &=-\ell+\left(\sum_{i=1}^{\ell}\weight_{\Zin}(a_i)\right)
  -\weight(\bOut(B,\vec{\rhos_i})).
  \end{align*}
\end{proof}

\begin{lemma}
  Given $\x$ and a sequence of algebra elements $(a_1,\dots,a_\ell)$,
  there are only a finite number of non-negative homology classes
  $B\in\pi_2(\x,\y,\vec{\rhos})$ where $\rhos_1,\dots,\rhos_h$ are
  $(\x,\vec{a})$-compatible and with
  $\ind(B,\x,\y,\vec{\rhos})=1$.
\end{lemma}

\begin{proof}
  According to Lemma~\ref{lem:MgrAndAgr}, $(\x,\vec{a})$, $\y$, and
  $\ind(B,\x,\y,\vec{\rhos})=1$ determines the total weight of $b$.
  The lemma also shows that 
  $\{i,j\}\in\Mmid$, then $\weight_i(B)-\weight_j(B)$ is independent
  of $B$ (depending only on $\x$ and $\y$). 

  Suppose next  that
  $\{i,j\}\in\MatchIn$, and $i$ is an odd orbit, then if
  $c_{i,j}=\weight_i(a_1\otimes\dots \otimes a_\ell)-
  \weight_j(a_1\otimes\dots\otimes a_\ell)$, then clearly
  \[ c_{i,j}\leq  \weight_i(B)-\weight_j(B)\leq c_{i,j}+\weight_{\tau(i)}(b).\]
  Finally,
  if $i\in\Zout$ is $W$-initial, then
  $\weight_i(B)=\weight_i(b)$.

  Since every in-coming boundary component is equivalent (using the
  equivalence relation generated by $\Mmid$ and $\MatchIn$) to an
  $W$-initial out boundary component, we have universal upper bounds
  on the weights of $B$ at all of its boundary points (which we also
  assumed to be non-negative).  Since the map $B\mapsto
  \bigoplus_{i}\weight_i(B)$ gives an injection of $\pi_2(\x,\y)$
  into $\Z^{m+n}$, the lemma follows.
\end{proof}

Theorem~\ref{thm:AEnds} has the following straightforward adaptation
to middle diagrams.

\begin{remark}
  We will need the following theorem in the case where the in-coming
  sequence of packets are compatible with some sequence of algebra
  elements $(a_1,\dots,a_\ell)$. To underscore the similarity with
  Theorem~\ref{thm:AEnds}, we have stated it under slightly weaker
  hypotheses. (Compare Definition~\ref{def:Allowed}.)
\end{remark}

\begin{thm}
  \label{thm:DAEnds} Choose a middle diagram $\Hmid$, and fix a
  compatible matching $\Mup$. Choose also an orbit marking
  (Definition~\ref{def:OrbitMarkingDA}).  Fix a lower Heegaard state $\x$
  and a sequence of constraint packets $\vec{\rhos}$ with the
  following properties:
  \begin{itemize}
  \item $(\x,\vec{\rhos})$ is strongly
    boundary monotone.
  \item The chords appearing in each packet  $\rhos_i$ are disjoint from one another
  \item Each packet contains at most one orbit, and that orbit is simple.
  \item If a packet contains an even (in-coming) orbit, then it 
    contains exactly one other Reeb chord, as well; and that chord
    is disjoint from the orbit.
  \item If the packet contains an orbit which is not even,
    then it contains no other chord.
  \end{itemize}
  Let $\y$ be a
  lower Heegaard state, and $B\in\pi_2(\x,\y)$, whose local multiplicity
  vanishes somewhere. Choose $\Source$ and
  $\vec{P}$ so that $[\vec{P}]=(\rhos_1,\dots,\rhos_\ell)$ and
  so that the $\chi(\Source)=\chiEmb(B)$;
  and suppose that $\ind(B,\x,\y;\vec{\rhos})=2$, and abbreviate
  $\UnparModFlow=\UnparModFlow^B(\x,\y;\Source;{\vec{P}})$. The total
  number of ends of $\UnparModFlow$ of the following types are even in
  number:
  \begin{enumerate}[label=(DAE-\arabic*),ref=(DAE-\arabic*)]
  \item \label{endDA:2Story}
    Two-story ends, which are of the form
    \[ \UnparModFlow(\x,\w;\Source_1;\rhos_1,\dots,\rhos_i)\times 
    \UnparModFlow(\w,\y;\Source_2;\rhos_{i+1},\dots,\rhos_{\ell}),\]
    taken over all lower Heegaard states $\w$ and choices of $\Source_1$ and $\Source_2$ so that
    $\Source_1\natural \Source_2=\Source$, and $B_1\natural B_2=B$.
  \item 
    \label{endDA:Orbit}
    Orbit curve ends, of the form
    $\UnparModFlow^B(\x,\y,\Source';\rhos_1,\dots,\rhos_{i-1},\sigmas,\rhos_{i+1},\dots,\rhos_{\ell})$,
    where $\orbits(\sigmas)=\orbits(\rhos_i)\setminus\{\orb_r\}$,
    $\chords(\sigmas)=\chords(\rhos_i)\cup\{\longchord_r\}$ where
    $\longchord_r$ is a Reeb chord that covers the boundary component
    $Z_r$ with multiplicity $1$.
  \item 
    \label{endDA:ContainedCollisions}
    Contained collision ends for two consecutive packets $\rhos_i$ and
    $\rhos_{i+1}$, which correspond to points in
    $\UnparModFlow^{B'}(\x,\y,\Source;\rhos_1,\dots,\rhos_{i-1},\sigmas,\rhos_{i+2},\dots,\dots,\rhos_\ell)$
    with the following properties:
    \begin{itemize}
    \item The collision is visible.
    \item The packets $\rhos_i$ and $\rhos_{i+1}$ are strongly composable.
    \item The packet $\sigmas$ is a contained collision of $\rhos_i$
      and $\rhos_{i+1}$
    \item 
      The chords in $\sigmas$ are disjoint from one
      another.
    \end{itemize}
  \item
    \label{endDA:Join}
    Join ends, of the form
    $\UnparModFlow^B(\x,\y,\Source';\rhos_1,\dots,\rhos_{i-1},\sigmas,\rhos_{i+1},\dots,\rhos_{\ell})$,
    $\orbits(\sigmas)=\orbits(\rhos_i)$, and the following conditions hold:
    \begin{itemize}
      \item $(\x,\rhos_1,\dots,\rhos_{i-1},\sigmas,\rhos_{i+1},\dots,\rhos_{\ell})$
        is strongly boundary monotone.
      \item There is some $\rho\in \chords(\rhos_i)$ with the property
        that $\rho=\rho_1\uplus \rho_2$, and
        $\chords(\sigmas)=(\chords(\rhos_i)\setminus \{\rho\})\cup
        \{\rho_1,\rho_2\}$.
      \item  In the above decomposition, at least one of $\rho_1$ and 
        $\rho_2$ covers only half of a boundary component.
      \end{itemize}
  \item \label{endDA:BoundaryDegeneration}
    Boundary degeneration collisions $\sigmas$
    between two consecutive packets $\rhos_i$ and $\rhos_{i+1}$;
    when
    $\orb_j\in\orbits(\rhos_i)$, $\orb_k\in\orbits(\rhos_{i+1})$
    and $\{j,k\}\in\Mmid$.
    When $\sigmas=\rhos_i\setminus \{\orb_j,\orb_k\}$
    is non-empty, 
    these correspond to points in
    \[ \ModFlow^{B'}(\x,\y,\Source';\rhos_1,\dots,
    \rhos_{i-1},\sigmas,\rhos_{i+1},\dots,\rhos_\ell)\]
    in the chamber $\Chamber^{\orb_j<\orb_k}$, where
    the homology class $B'$ 
    is obtained from $B$ by removing 
    a copy of $\Bjk$. When $\sigmas=\emptyset$, then $\ell=1$,
    $\x=\y$, $B'=0$, and the end is unique.
    \end{enumerate}
\end{thm}

\begin{proof}
  The proof is exactly as in the proof of Theorem~\ref{thm:AEnds}.
  The key difference is that in a middle diagram, we do not treat
  special boundary degenerations separately. (Theorem~\ref{thm:AEnds}
  would have a had a similar statement if we had treated the marked
  points $\wpt$ and $\zpt$ as orbits around the punctures.)
\end{proof}

\begin{prop}
  \label{prop:DAmid}
  Let $\Hmid$ be a middle diagram that is compatible with a given
  matching $M$ on its incoming boundary. Choose an orientation on
  $W=W(\Hmid)\cup W(M)$. The  $\IdempRing(2n)-\IdempRing(2m)$-bimodule
  $\DAmod(\Hmid)$, equipped the operations 
  \[\delta^1_{\ell+1}\colon \DAmod(\Hmid)\otimes\Clgin^{\otimes \ell}\to \Clgout\otimes\DAmod(\Hmid)\]
  defined above endows $\DAmod(\Hmid)$ with the structure of 
  a curved $\Clg(n,\Mout)-\Clg(m,M)$ $DA$ bimodule (cf. Equations~\eqref{eq:CurvedDAbimoduleRelation1} and~\eqref{eq:CurvedDAbimoduleRelation2}),
  where $\Mout$ is the matching on $\{1,\dots,2n\}$ induced by 
  $\Mmid$ and $M$.
\end{prop}

\begin{proof}
  This is a combination of Propositions~\ref{prop:CurvedTypeD} and
  Proposition~\ref{prop:CurvedTypeA}.

  In more detail, look at ends of one-dimensional moduli spaces.  Note
  that the moduli spaces that contribute to the outgoing algebra
  element cannot cover all of the outgoing boundary with positive
  weight, so we can restrict attention to homology classes $B$ that do
  not cover all of $\Sigma$; i.e. Theorem~\ref{thm:DAEnds} applies.

  Consider first the $\Ainfty$ relation with no incoming algebra
  elements.

  When $\x\in\Chamber^{\orb>\orb'}$, there is a corresponding end of
  $\ModFlow(\x,\x,\{\orb\},\{\orb'\})$.  In turn, that moduli space
  contributes to $\mu_2\circ(\Id\otimes \delta^1)\circ \delta^1$ only
  in two cases: 
  \begin{itemize}
  \item when both $\orb$ and $\orb'$ are Reeb orbits on
    the out-going boundary, or 
  \item one of the two is a Reeb orbit on the
    out-going boundary and the other is an odd orbit on the in-coming
    boundary.
  \end{itemize}
  Each equivalence class of orbits contains exactly one pair of orbits
  which can be paired in a simple boundary degeneration as above: the
  boundary degeneration is the unique simple boundary degeneration that
  contains the orbit corresponding to the endpoint of the given
  $W$-equivalence class. Explicitly, when the $W$-equivalence class
  contains no in-coming boundary components, then $\orb$ and
  $\orb'$ are the two (out-going) orbits in the equivalence
  class; otherwise, if $\orb$ is outgoing and it is paired with 
  an in-coming $\orb'$, then that $\orb'$ must be the 
  last odd orbit in the $W$-equivalence class (under the ordering induced
  by its orientation).

  By switching the order of $\orb$ and $\orb'$ if needed (since
  $\x\in\Chamber^{\orb>\orb'}$ or $\Chamber^{\orb'>\orb}$), we can conclude
  each equivalence class of orbits contributes $U_j U_k$,
  where $\Zout_j$ and $\Zout_k$ are the two out-going boundary components
  in the equivalnce class.
  Thus, these boundary degenerations to
  $\mu_2\circ(\Id\otimes\delta^1)\circ\delta^1(\x)$ gives a term of
  the form $\mu_0^{\Clgout}\otimes \x$. 
  
  As in the proof of Proposition~\ref{prop:CurvedTypeA}, there are orbit 
  curve ends which can be identified with
  $\delta^1(\x,\mu^{\Clgin}_0)$.

  All other contributions cancel in pairs as in the proof of
  Proposition~\ref{prop:CurvedTypeA}, verifying the weighted type $DA$
  structure equation with no algebra inputs,
  Equation~\eqref{eq:CurvedDAbimoduleRelation1}.
  Note that we do have collision ends when packets on $\Zin$ collide
  with packets on $\Zout$; but once again, these cancel with ends 
  where the two packets are permuted.
\end{proof}

\subsection{Another pairing theorem}

Let $\Hup$ be an upper diagram and $\Hmid$ be a middle diagram, so
that $\partial\Hup$ is identified with $\partial^{\vee}\Hmid$.  Then,
we can form a new upper diagram $\Hmid\#\Hup$.  Evidently, there is a
one-to-one correspondence between pairs of states $\x$ and $\y$, where
$\x$ is an partial Heegaard state for $\Hmid$ and $\y$ is an upper
Heegaard state for $\Hup$, and
$\alpha(\x)=\{1,\dots,2n\}\setminus\alphaout(\y)$.

We have the following analogue of the pairing theorem
Theorem~\ref{thm:PairAwithD}; compare also~\cite[Theorem~11]{Bimodules}:

\begin{thm} 
  \label{thm:PairDAwithD}
  Let $\Hmid$ and $\Hup$ be as above.
  Let $\Clg_1=\Clg(\Hup)=\Clgin(\Hmid)$;
  $\Clg_2=\Clgout(\Hmid)$.
  Under the above hypotheses, there is a quasi-isomorphism
  of curved type $D$ structures
  \[\lsup{\Clg_2}\Dmod(\Hmid\#\Hup) 
  \simeq
  \lsup{\Clg_2}\DAmod(\Hmid)_{\Clg_1}\DT \lsup{\Clg_1}\Dmod(\Hup).\]
\end{thm}

\subsection{Proof of the DA bimodule pairing theorem}

The proof of Theorem~\ref{thm:PairDAwithD} is very similar to the
proof of Theorem~\ref{thm:PairAwithD}. The key algebraic difference is
that in the present context, there is a curvature in the result; and
analytically, the curvature is produced by boundary
degenerations. We give the details presently.

Let $\Hup_2$ be an upper diagram with $2m$ outgoing boundary
components and $\Hmid$ be a middle diagram with $2m$ incoming boundary
components and $2n$ outgoing ones. Fix an identification between
$\Zout(\Hup_1)$ and $\Zin(\Hmid)$, and use this to form the upper
$\Hup=\Hmid\#\Hup_2$.

Definition~\ref{def:MatchingPair} has the following straightforward
generalization:

\begin{defn}
  Let $\Hup=\Hmid\#\Hup_2$.  States $\x_1\in\States(\Hmid)$,
  $\x_2\in\States(\Hup)$ are called {\em matching states} 
  if $\alphain(\x_1)$ is the complement of 
  $\alpha(\x_2)$ (i.e. $\Idown(\x_1)=\Iup(\x_2)$).
\end{defn}

There is a one-to-one correspondence between pairs of matching states
$\x_1$ and $\x_2$, and upper states for $\Hup$; and hence there is a
one-to-one correpsondence between generators of
$\DAmod(\Hmid)\DT\Dmod(\Hup_2)$ and $\Dmod(\Hup)$.

\begin{defn}
Suppose that $(\x_1,\y_1)$ and $(\x_2,\y_2)$ are matching states,
and $B_1\in \doms(\x_1,\y_1)$ and $B_2\in\doms(\x_2,\y_2)$.
We say that $B_1$ and $B_2$ are {\em matching domains}
if the local multiplicities of $B_2$
around each component of $\Zout_i(\Hup_2)$ coincide with
the local multiplicities around each component of 
$\Zin_i(\Hmid)$. 
\end{defn}

For $\x=\x_1\#\x_2$ and $\y=\y_1\#\y_2$, there is a one-to-one
correspondence between $B\in\doms(\x,\y)$ and matching domains
$B_1\in\doms(\x_1,\y_1)$ and $B_2\in\doms(\x_2,\x_2)$. In that case,
we write  $B=B_1\# B_2$.

Definition~\ref{def:MatchedPair} has the following generalization:

\begin{defn}
  \label{def:MatchedPairDA}
  Suppose $n>1$.
  Fix two pairs $(\x_1,\x_2)$ and $(\y_1,\y_2)$ of matching states,
  i.e. where $\x_1,\y_1\in\States(\Hmid)$,
  $\x_2,\y_2\in\States(\Hup_2)$, so that $\x=\x_1\#\x_2$ and
  $\y=\y_1\#\y_2$ are Heegaard states for $\HD$.  
  A {\em matched pair} consists of the following data
  \begin{itemize}
  \item a holomorphic curve $u_1$ in $\Hmid$ with source $\Source_1$ representing homology class $B_1\in\doms(\x_1,\y_1)$
  \item a holomorphic curve $u_2$ in $\Hup_2$ with source $\Source_2$ representing homology class $B_2\in\doms(\x_2,\y_2)$
  \item a bijection $\psi\colon \AllPunct(\Source_2)\to \AllPunctIn(\Source_1)$,
    where $\AllPunctIn(\Source_1)\subset\AllPunct(\Source_1)$ denotes the set
    of punctures on $\Source_1$ that are labelled by chords and orbits
    supported on $\Zin$,
  \end{itemize}
  with the following properties:
  \begin{itemize}
  \item For each $q\in \AllPunct(\Source_2)$ is marked with a Reeb orbit
    or chord, the corresponding puncture $\psi(q)\in \AllPunct(\Source_1)$
    is marked with the matching Reeb orbit or chord
    (in the sense of Definition~\ref{def:MatchingChords}).
  \item For each $q\in\AllPunct(\Source_2)$,
    \[ (s\circ u_1(\psi(q)),t\circ u_1(\psi(q)))=(s\circ u_2(q),t\circ u_2(q)).\]
  \end{itemize}
  If $(u_1,u_2)$ is a matched pair with homology class
  $B_1$ and $B_2$, then $B_1$ and $B_2$ are matching domains.
  Let $\ModMatched^B(\x_1,\y_1;\x_2,\y_2;\Source_1,\Source_2;\psi)$
  denote the modul space of matched pairs with shadow $B=B_1\# B_2$.
\end{defn}

We have the following analogue of Lemma~\ref{lem:BoundaryMonotone}:

\begin{lemma}
  \label{lem:BoundaryMonotoneDA}
  Let $(\x_1,\x_2)$ and $(\y_1,\y_2)$ be two pairs of matching states,
  and fix a matched pair $(u_1,u_2)$ connecting $\x_1\#\x_2$ to
  $\y_1\#\y_2$, with $u_1\in\ModFlow(\x_1,\y_1,\rhos_1,\dots,\rhos_m)$
  and $u_2\in\ModFlow(\x_2,\y_2,\rhos_1',\dots,\rhos_m')$ (where the
  chords in $\rhos_i$ all match, in the sense of
  Definition~\ref{def:MatchingChords} with chords in $\rhos_i'$).
  Then, both $u_1$ and
  $u_2$ are strongly boundary monotone; moreover,
  $\alpha(\x_1,\rhos_1,\dots,\rhos_\ell)$ is complementary to
  $\alpha(\x_2,\rhos_1',\dots,\rhos_\ell')$ for all $\ell=0,\dots,m$.
\end{lemma}

Given a holomorphic curve $u_1$ for a middle diagram, let $\cin_1$
denote the number of chords in $u_1$ on $\Zin(\Hmid)$ $\oin_1$ denote
the number of orbits in $u_1$ on $\Zin(\Hmid)$, and $\win_1$ denote
the total weight of $u_1$ at $\Zin(\Hmid)$.  

If $(u_1,u_2)$ is a
matched pair of curves, if $c_2$, $o_2$, and $w_2$ denote the number
of chords, orbits, and total weight at $\Zout(\Hup_2)$ respectively,
then $\cin_1=c_2$, $\oin_1=o_2$, $\win_1=\weight_2$.
The space of matched pairs lies in a moduli space whose expected dimension is
given by
\begin{equation}
  \label{eq:DefIndMatchedPairDA}
  \ind(B_1,\Source_1;B_2,\Source_2)=
\ind(B_1,\Source_1)+\ind(B_2,\Source_2)-c_2-2o_2.
\end{equation}

\begin{defn}
The {\em embedded index} of a matched pair is defined by the formula
\begin{align}
    \label{eq:DefEmbInd}
    \ind(B_1;B_2)&=e(B_1)+n_{\x_1}(B_1)+n_{\y_1}(B_1)
    + e(B_2)+n_{\x_2}(B_2)+n_{\y_2}(B_2)-2\weight_{\Zin}(B_1) \nonumber \\
    &= e(B_1\natural B_2)+n_{\x}(B_1\natural B_2)+n_{\y}(B_1\natural B_2).
  \end{align}
\end{defn}

Lemma~\ref{lem:TransversalityDA} has the following analogue:

\begin{lemma} 
  \label{lem:TransversalityDA}
  Fix $B_1\in\doms(\x_1,\y_1)$ and $B_2\in\doms(\x_2,\y_2)$ so that
  $\weight_i(B_1)=\weight_i(B_2)$ for $i=1,\dots,2m$, 
  and the local multiplicity of $B_1$ vanishes somewhere.
  For generic admissible almost complex
  structures on $\Sigma_i\times [0,1]\times \R$, and
  $\ind(B_1,\Source_1;B_2,\Source_2)\leq 2$, the moduli space of
  matched pairs
  \[ \ModMatched^B(\x_1,\y_1;\x_2,\y_2;\Source_1,\Source_2;\psi)\] is
  transversely cut out by the $\dbar$-equation and the evaluation map;
  in particular, this moduli space is a manifold whose dimension is
  given by Equation~\eqref{eq:DefIndMatchedPairDA}.  
\end{lemma}

Proposition~\ref{prop:EmbeddedModuliSpaces} has the following analogue:
\begin{prop}
  \label{prop:EmbeddedModuliSpacesDA}
  Fix $\x=\x_1\#\x_2$ and $\y=\y_1\#\y_2$, and decompose
  $B\in\doms(\x,\y)$ as $B=B_1\natural B_2$, with
  $B_i\in\doms(\x_i,\y_i)$.  Fix source curves $\Source_1$ and
  $\Source_2$ together with a one-to-one correspondence $\psi\colon
  \AllPunct(\Source_2)\to\AllPunct(\Source_1)$ which is consistent with the chord
  and orbit labels, so we can form $\Source=\Source_1\natural_\psi\Source_2$.
  Suppose that $\ModFlow^B(\x,\y;\Source)$ 
  (i.e. the moduli space for curves in $\HD$) and
  $\ModFlow^{B_i}(\x_i,\y_i;\Source_i)$
  (which are moduli spaces for curves in $\HD_i$) 
  are non-empty for $i=1,2$.
  Then,
  $\ind(B_1,\Source_1;B_2,\Source_2)\leq \ind(B)$
  if and only if
  $\chi(\Source_i)=\chiEmb(B_i)$ for $i=1,2$; 
  and all the chords in $\Source_1$ have weight $1/2$ and all 
  the orbits in $\Source_1$ have weight $1$.
\end{prop}

In the present case, the Gromov compactification can contain
$\beta$-boundary degenerations. Nonetheless, excluding the
$\alpha$-boundary degenerations and closed curve components work as
before, as in the following. Compare Lemmas~\ref{lem:NoBoundaryDegenerations}
and~\ref{lem:NoClosedCurves} respectively: 

\begin{lemma}
  \label{lem:NoBoundaryDegenerationsDA}
  Fix $B_1\in\doms(\x_1,\y_1)$ and $B_2\in\doms(\x_2,\y_2)$ with so
  that 
  \begin{itemize}
    \item the local multiplicities of $B_1$ along $\Zin(\HD_1)$ agree
      with the local multiplicities of $\Zout(\HD_2)$ 
      \item $B_1$ has
      vanishing local multiplicity somewhere.
      \end{itemize}
      Then, curves in the Gromov
  compactification of
  $\ModMatched^B(\x_1,\y_1;\x_2,\y_2;\Source_1,\Source_2)$ contain no
  $\alpha$-boundary degenerations.
\end{lemma}

\begin{lemma}
  \label{lem:NoClosedCurvesDA}
  Suppose that $m>1$.  Fix matching domains $B_1\in\doms(\x_1,\y_1)$
  and $B_2\in\doms(\x_2,\y_2)$, so that $B_1$ has a vanishing local multiplicity
  somewhere, and so that
  $\ind(B_1\natural B_2)\leq 2$.  Curves in the Gromov
  compactification of
  $\ModMatched^B(\x_1,\y_1;\x_2,\y_2;\Source_1,\Source_2)$ contain no
  closed components.
\end{lemma}

\begin{lemma}
  \label{lem:NoBoundaryDegenerations2DA}
  Fix matching domains
  $B_1\in\doms(\x_1,\y_1)$ and $B_2\in\doms(\x_2,\y_2)$ so that
  $w_i(B_1)=w_i(B_2)$ so that
  $\ind(B_1\natural B_2)\leq 2$.  Then, curves in the Gromov
  compactification of
  $\ModMatched^B(\x_1,\y_1;\x_2,\y_2;\Source_1,\Source_2)$ contain no
  $\beta$-boundary degenerations, except in the special case where
  $B=\Bjk$ for some matched pair $j$ and $k$.
\end{lemma}

\begin{proof}
  The proof of Lemma~\ref{lem:NoBoundaryDegenerations} actuall shows that
  if there is $\beta$-boundary degeneration (on either side), then 
  in fact, it forms along with a sequence of boundary degenerations
  that contain all the orbits in a single component of $W$.
  A straightforward computation shows that the remaining components
  have index $0$, and therefore, if it has a pseudo-holomorphic representative,
  the remainder must correspond to a constant map. It follows that
  $B=\Bjk$.
\end{proof}

\subsubsection{Type $D$ structures of matched curves}

\begin{defn}
  \label{def:MatchTypeD}
  Let $\Hup=\Hmid\cup \Hup_2$ be a decomposition of an upper diagram.
  Let $X$ denote the vector spacegenerated by 
  upper states for $\Hup$, equipped with an operator
  \[ \delta^1_{\natural}\colon X \to \Clg\otimes X \]
  characterized by
  \[ \delta^1_{\natural}(\x)=\sum_{\{B=B_1\# B_2\big| ind(B)=1\}}
 \#\ModMatched^B(\x,\y)\cdot  \bOut(B)\otimes \y.\]
\end{defn}

The analogue of Theorem~\ref{thm:NeckStretch} is the following:

\begin{thm}
  \label{thm:NeckStretchDA}
  Provided $m>1$.
  $(X,\delta^1_{\natural})$ is a curved type $D$ structure, which is
  homotopy equivalent to $\Dmod(\Hup)$.
\end{thm}

\begin{proof}
  The stretching argument from Theorem~\ref{thm:NeckStretchDA}
  proves that $\delta^1_{\natural}$ agrees with $\delta^1$
  for a suitable complex structure on $\Hup$.
  It follows that $\delta^1_{\natural}$ is a curved type $D$ structure.
\end{proof}

\begin{rem}
  The above argument shows that boundary degenerations which are
  allowed in Lemma~\ref{lem:NoBoundaryDegenerations2DA} indeed do
  occur; and their algebraic count is $1$.
\end{rem}

\subsubsection{Self-matched curves}

We adapt Definition~\ref{def:SelfMatched} to middle diagrams as follows:

\begin{defn}
  \label{def:SelfMatchedDA}
  Let $\Source_1$ be a decorated source for $\Hmid$.  
  according to whether the punctures are marked by chords or orbits on
  $\Zout$ or $\Zin$ respectively.  Further partition
  the punctures of $\Source_1$ labelled by chords and orbits in $\Zin$,
  as
  \[ \EastIn(\Source_1)\cup \OmegaInEv(\Source_1)\cup
  \OmegaInOdd(\Source_1),\] where $\EastIn(\Source_1)$ consists of
  boundary punctures, $\OmegaInEv(\Source_1)$ consists of interior
  punctures marked by even Reeb orbits, and $\OmegaInOdd(\Source_1)$
  consists of interior punctures marked by odd Reeb orbits. A {\em
    self-marked source} is a decorated source, together with an
  injection $\phi\colon \OmegaInEv(\Source_1)\to
  \EastIn(\Source_1)$ with the property that if
  $p\in\OmegaInEv(\Source_1)$ is marked by some orbit $\orb_j$, then
  $\phi(p)$ is marked by a length one chord that covers the boundary
  component $Z_k$ so that $\{j,k\}\in \Mup$. A {\em self-matched
    curve} $u$ is an element
  $u\in\ModFlow^{B_1}(\x,\y;\Source_1,\phi)$, subject to the following
  additional constraints: for each puncture
  $p\in\OmegaInEv(\Source_1)$,
  \begin{equation}
    \label{eq:SelfMatchingDA}
    t\circ u(p)=t\circ u(\phi(p)).
  \end{equation}
  Let
  $\ModMatchedChanged^{B_1,B_2}(\x,\y;\Source_1,\Source_2,\phi,\psi)$
  denote the moduli space self-matched curve pairs. 
\end{defn}

Let $X$ be the $\Field$-vector space generated by matching states
$\x_1$ and $\x_2$. 
We use the self-matched moduli spaces to construct
a map
\[ \DChanged\colon X \to \Blg\otimes X \]
defined as in Definition~\ref{def:MatchTypeD},
only using $\ModMatchedChanged^B(\x,\y)$ instead of
$\ModMatched^B(\x,\y)$.

\begin{proof}[Proof of Theorem~\ref{thm:PairDAwithD}]
  The proof is very similar to the proof of
  Theorem~\ref{thm:PairAwithD}.  Theorem~\ref{thm:NeckStretchDA}
  provides the $D$-module isomorphism
  $\Dmod(\Hup)\cong(X,\delta^1_{\natural})$.  Suppose $m>1$.  

  We wish to define a sequence of intermediate complexes Specifically,
  in the definitions of the intermediate complexes, we chose a
  particular ordering on $\{1,\dots,m\}$ in which to deform the
  matching condition; specifically, choose a one-to-one correspondence
  $f\colon \{1,\dots,2m\}\to\{1,\dots,2m\}$ so that if
  $n_1,\dots,n_{2k}$ are the orbits in $\Zin$ on a given component of
  $W$, arranged in the opposite to the order the appear in $W$ (with
  its orientation), then
  $\{f(n_2),f(n_1),f(n_3),\dots,f(n_{2k}),f(n_{2k-1})\}$ is a sequence
  of consecutive integers. See Figure~\ref{fig:LabelOrbitDA} for an
  example.
  
  We can now define the sequence of intermediate type $D$ structures
  $(X,\delta^1_\ell\colon X\to \Clg\otimes X)$ using operators
  $\delta^1_\ell$ that count $\ell$-self-matched curve pairs, adapting
  Definition~\ref{def:IntermediateComplexes}.

 \begin{figure}[h]
 \centering
 \input{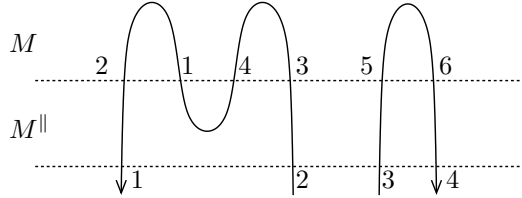}
 \caption{{\bf Labeling the orbits for a middle diagram.} 
   For this picture, we have listed a valid numbering $f$.}
 \label{fig:LabelOrbitDA}
 \end{figure}

 Isomorphisms are constructed as in
 Proposition~\ref{prop:Intermediates}.  Finally, the identification of
 $(X,\DChanged)$ with the tensor product is achieved by time dilation.

  The case where $m=1$ is handled separately, using special matched
  pairs,  as in Subsection~\ref{subsec:Nequals1}.
\end{proof}

\section{Extending the $DA$ bimodules}
\label{sec:ExtendDA}

As defined so far, our type DA bimodules are curved bimodules over
$\Clg$. Our aim here is to extend these modules to modules over
$\Blg$. Specifically, if $\Hmid$ is a middle diagram, let
$\Blgin=\Blgin(\Hmid)=\bigoplus_{k=0}^{2m-1}\Blg(2m,k)$,
$\Blgout=\Blgout(\Hmid)=\bigoplus_{k=0}^{2n-1}\Blg(2n,k)$, so 
that
\begin{align*}
 \Clgin&= 
\left(\sum_{\{\x\mid \x\cap \{0,
    2m\}=\emptyset\}} \Idemp{\x}\right)\cdot \Blgin\cdot 
\left(\sum_{\{\x\mid \x\cap \{0, 2m\}=\emptyset\}}
  \Idemp{\x}\right) \\ \\
 \Clgout&= 
\left(\sum_{\{\x\mid \x\cap \{0,
    2m\}=\emptyset\}} \Idemp{\x}\right)\cdot \Blgout\cdot 
\left(\sum_{\{\x\mid \x\cap \{0, 2m\}=\emptyset\}}
  \Idemp{\x}\right).
\end{align*}
In particular, there are inclusion maps
${\iota}\colon \Clgin\to \Blgin$ and ${\iota}\colon \Clgout\to \Blgout$.
In the next statement we will use the corresponding bimodules
$\lsup{\Blgin}[\iota]_{\Clgin}$ and
$\lsup{\Blgout}[\iota]_{\Clgout}$ 

\begin{prop}
  \label{prop:ExtendDA}
  Let $\Hmid$ be a middle diagram, and $\lsup{\Clgout}\DAmod(\Hmid)_{\Clgin}$
  be the corresponding type $DA$ bimodule.
  Then, there is a type $DA$ bimodule 
  $\lsup{\Blgout}\DAmod_{\Blgin}$, with the property that
  \[ \lsup{\Blgout}[\iota]_{\Clgout}~
  \DT \lsup{\Clgout}\DAmod(\Hmid)_{\Clgin}=
  \lsup{\Blgout}\DAmodExt_{\Blgin}\DT~\lsup{\Blgin}[\iota]_{\Clgin}.\]
\end{prop}

The bimodule over $\Blg$ in the statement of
Proposition~\ref{prop:ExtendDA} is constructed by extending the notion of
middle diagrams, and defining their associated type $DA$ bimodules, as
we do presently.  The proof of the above theorem is given in
Section~\ref{sec:DestabilizationTheorem}.

\subsection{Extended middle diagrams}
We define now relevant generalization of middle diagrams used
in the construction of the module from Proposition~\ref{prop:ExtendDA}.

\begin{defn}
An {\em extended middle diagram} is a Heegaard diagram obtained 
from an upper diagram with $2m+2n+2$ boundary components 
$\{Z_0,\dots,Z_{2m+2n+1}\}$, and adding a new arc $\alpha_{2m+2n+1}$ connecting $Z_0$ and
$Z_{2m+2n+1}$ , 
and adding two new beta circles, so that the $\beta$ circles now are labelled
$\{\beta_i\}_{i=1}^{g+m+n+1}$.

We think of $2m$ of these components as
``incoming'' boundary, writing $\Zin_i=Z_i$ for
$i=1,\dots,2m$; $2n$ ``outgoing'' boundary components, writing
$\Zout_i=Z_{2m+2n+2-i}$ for $i=1,\dots,2n$; and two ``middle'' boundary
components $\Zmid_0=Z_0$ and $\Zmid_1=Z_{2m+1}$.  We also let
$\alphain_i=\alpha_i$ for $i=0,\dots,2m$; 
$\alphaout_i=\alpha_{2m+2n+1-i}$ for $i=0,\dots,2n$.
We write this data
\begin{align*} \HmidExt=(\Sigma_0,\{\Zin_i\}_{i=1}^{2m},
\{\Zout_i\}_{i=1}^{2n},
\{\Zmid_0,\Zmid_1\},
\{\alphain_i\}_{i=0}^{2m},
\{\alphaout_i\}_{i=0}^{2n},
\{\alpha^c_i\}_{i=1}^g,
\{\beta_i\}_{i=1}^{g+m+n+1}\}).
\end{align*}
\end{defn}

Note that  $\alphain_0$ connects $\Zin_1$ to
$\Zmid_0$, $\alphaout_0$ connects $\Zmid_0$ to $\Zout_1$,
$\alphain_{2m}$ connects $\Zin_{2m}$ to $\Zmid_1$, and
$\alphaout_{2n}$ connects $\Zmid_1$ to $\Zout_{2n}$.

\begin{example}
  \label{ex:HmidEx}
  Let 
  \[\Hmid=(\Sigma_0,\{\Zin_i\}_{i=1}^{2m}, \{\Zout_i\}_{i=1}^{2n},
  \{\alphain_i\}_{i=1}^{2m-1}, \{\alphaout_i\}_{i=1}^{2n-1},
  \{\alpha^c_i\}_{i=1}^g, \{\beta_i\}_{i=1}^{g+m+n-1})\] be a middle
  diagram.  We can form an extended middle diagram as follows.
  Connect $\Zin_1$ and $\Zout_1$ by a path that crosses only
  $\beta$-circles.  Remove a disk centered at a point on the path from
  $\Sigma_0$ to obtain a new boundary component $\Zmid_0$, and
  introduce new arcs $\alphain_0$ connecting $\Zin_1$ to $\Zmid_0$ and
  $\alphaout_0$ connecting $\Zmid_0$ to $\alphaout_1$. Add also a new
  small $\beta$-circle $\beta_{0}$ in a collar neighborhood of
  $\Zmid_0$.  Connect $\Zin_{2m}$ and $\Zout_{2n}$ by another path,
  remove another disk to obtain the boundary component $\Zmid_1$, and
  introduce arcs $\alphain_{2m}$ connecting this to $\Zin_{2m}$ and
  $\alphaout_{2n}$ connecting it to $\Zout_{2n}$, and introduce a
  circle $\beta_{g+m+n}$ which encricles $\Zmid_1$, intersecting
  $\alphain_{2m}$ and $\alphaout_{2n}$ in one point apiece. In this
  manner, we obtain a new extended middle diagram, $\HmidExt$, called
  a {\em stabilization} of $\Hmid$.
\end{example}

See Figure~\ref{fig:ExtMidDiag} for an example.

\begin{figure}[h]
 \centering
 \input{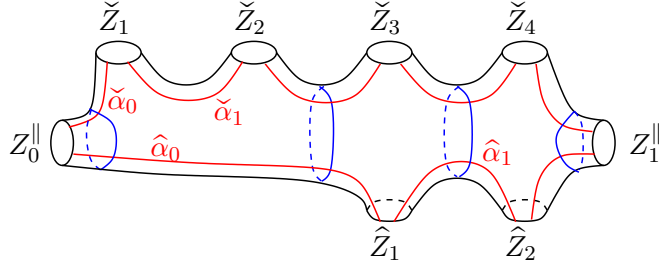}
 \caption{{\bf Extended middle diagram of local minimum.} 
 \label{fig:ExtMidDiag}}
 \end{figure}

\begin{defn}
  An {\em extended partial Heegaard state} is a $g+m+n+1$-tuple of
  points, where one lies on each $\beta$-circle, one lies on each
  $\alpha$-circle, and no more than one lies on each $\alpha$-arc.
  Let $\alphain(\x)\subset \{0,\dots,2m\}$ be the set of
  $i\in\{0,\dots,2n\}$ so that $\x\cap \alphain_i\neq \emptyset$; and
  $\alphaout(\x)\subset\{0,\dots,2n\}$ be the set of
  $i\in\{0,\dots,2n\}$ so that $\x\cap \alphain_i\neq \emptyset$.  An
  exended middle Heegaard state has an idempotent type
  $k=|\alphain(\x)|$.
\end{defn}

\subsection{$\DA$ bimodules for extended diagrams}

An exteded middle diagram induces a matching on the components of
$\partial\Sigma$; in fact, if $\HmidExt$ is an extension of $\Hmid$,
then the matchings are the same.

\begin{defn}
  Let $\HmidExt$ be an extended middle diagram compatible with a matching $M$ on
  the incoming boundary components. The {\em incoming algebra}
  $\Blgin(\HmidExt)$ and the {\em outgoing algebra}
    $\Blgout(\HmidExt)$ are defined by
  \[\Blgin(\Hmid)=\bigoplus_{k=0}^{2m-1}\Blg(2m,k);
  \qquad \Blgout(\Hmid)=\bigoplus_{k=0}^{2n-1}\Blg(2n,k).\]
\end{defn}

We will define a type $DA$ bimodule $\DAmodExt(\HmidExt)=
\lsup{\Blgout}\DAmodExt(\HmidExt)_{\Blgin}$. As a vector space,
it is spanned by the extended middle Heegaard states of $\HmidExt$.
\[ \Idemp{\{0,\dots,2n\}\setminus \alphaout(\x)} \cdot \x
\cdot \Idemp{\alphain(\x)}=\x.\]

Again, there is a splitting
\[ \DAmodExt(\HmidExt)=\bigoplus_{k\in\Z}
~~~{\lsup{\IdempRing(2n,k-m+n)}\DAmodExt(\Hmid)}_{\IdempRing(2m,k)},\] where $k$
  is the idempotent type of $\x$. We will be primarily interested in
the summand where $k=m$,
\[ \lsup{\IdempRing(2n,n)}\DAmodExt(\Hmid)_{\IdempRing(2m,m)}.\]

The material from Section~\ref{subsec:DAmodConstruction}
has a straightforward adaptation to extended diagrams,
with the understanding that the pseudo-holomorphic curves in consideration
now have zero multiplicity at the middle boundary $\Zmid_0$ and $\Zmid_1$.

In particular, the definition of $(\x,\vec{a})$-compatible sequences
given in Definition~\ref{def:CompatiblePacketDA}
can be readily defined when $\x$ is an extended partial Heegaard state,
and the sequence $\vec{a}=(a_1,\dots,a_{\ell})$ lies in $\Blgin(\HmidExt)$
(rather than $\Clgin(\Hmid)$, as it was there).

Observe given $(B,\rhos_1,\dots,\rhos_h)$, the associated element
$\bOut(B,\rhos_1,\dots,\rhos_h)$ defined as in Equation~\eqref{eq:bOut}
naturally lies in $\Blgout(\HmidExt)$ for extended diagrams, rather than merely 
in $\Clgout$.

With these observations in place, Equation~\eqref{eq:DefAction} 
induces now a map
\[ \delta^1_{\ell+1}\colon \DAmodExt(\HmidExt)\otimes\Blgin^{\otimes \ell}\to
\Blgout\otimes \DAmodExt(\HmidExt).\]
Proposition~\ref{prop:DAmid} has the following analogue for extended middle diagrams:

\begin{prop}
  \label{prop:DAmidEx}
  Let $\HmidExt$ be an extended middle diagram that is compatible with a given
  matching $M$ on its incoming boundary. Choose an orientation on
  $W=W(\Hmid)\cup W(M)$. The  $\IdempRing(2n)-\IdempRing(2m)$-bimodule
  $\DAmodExt(\HmidExt)$, equipped the operations 
  $\delta^1_{\ell+1}\colon \DAmodExt(\HmidExt)\otimes\Blgin^{\otimes m}\to \Blgout\otimes\DAmodExt(\HmidExt)$
  defined above endows $\DAmodExt(\HmidExt)$ with the structure of 
  a curved $\Blgin-\Blgout$ type $DA$ bimodule.
\end{prop}

\begin{proof}
  Theorem~\ref{thm:DAEnds} has a straightforward adaptation to
  extended middle diagrams.  The proposition then is an immediate
  consequence.
\end{proof}

\subsection{Destabilizing extended diagrams}
\label{sec:DestabilizationTheorem}

Our aim now is to prove the following version of
Proposition~\ref{prop:ExtendDA}:

\begin{prop}
  \label{prop:ExtendDAPrecise}
  Let $\HmidExt$ be an extension of a middle diagram $\Hmid$, as in
  Example~\ref{ex:HmidEx}. The type $DA$  bimodules
  associated to $\Hmid$ and $\HmidExt$ are related by the formula
  \[ \lsup{\Blgout}[\iota]_{\Clgout}~
  \DT \lsup{\Clgout}\DAmod(\Hmid)_{\Clgin}=
  \lsup{\Blgout}\DAmodExt_{\Blgin}(\HmidExt)\DT~\lsup{\Blgin}[\iota]_{\Clgin}.\]
\end{prop}

\begin{figure}[h]
 \centering
 \input{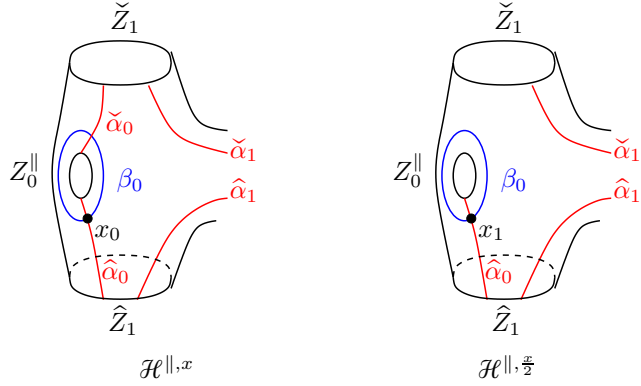}
 \caption{{\bf Removing $\alphain_0$.}
 At the left, we have  a portion of an extended middle diagram;
 at the right, we have the corresponding half-extended middle diagram. \label{fig:ExtRemove1}}
 \end{figure}

This is proved in two steps.

{\bf{Step 1: Remove $\alphain_0$ and $\alphain_m$.}}
  From $\Hmid$ we constructed $\HmidExt$. There is another diagram,
  $\HmidHalfExt$ which is obtained from $\HmidExt$ by removing
  $\alphain_0$ and $\alphain_{2m+1}$. (See Figure~\ref{fig:ExtRemove1}
  for an illustration of removing $\alphain_0$.)
  We call such diagrams {\em half-extended middle diagrams}.
  We can define the associated DA bimodule as before, except that now
  the input algebra for a half extended middle diagram is obviously
  $\Clgin$, whereas the output algebra is {\em a priori} $\Blgout$.

  \begin{lemma}
    \label{lem:RemoveAlphaIns}
    The DA bimodules for an extended middle diagram and a half-extended 
    middle diagram are related by
  \[ \lsup{\Blgout}\DAmodHalfExt(\HmidHalfExt)_{\Clgin}=~
  \lsup{\Blgout}\DAmodExt(\HmidExt)_{\Blgin}\DT~\lsup{\Blgin}[\iota]_{\Clgin}.\]
  \end{lemma}

  \begin{proof}
    Recall that $\beta_0$ and $\beta_{g+m+n}$ denote the two
    $\beta$-circles that encircle $\Zmid_0$ and $\Zmid_1$
    respectively.  Let $x_0=\beta_0\cap\alphaout_0$ and
    $x_{2m+1}=\beta_{g+m+n}\cap\alphaout_{2n}$; and
    $y_0=\beta_0\cap\alphain_0$ and
    $y_{2n+1}=\beta_{g+m+n}\cap\alphaout_{2n+1}$; The generators of
    $\lsup{\Blgout}\DAmodExt(\HmidExt)_{\Blgin}\DT~\lsup{\Blgin}[\iota]_{\Clgin}$
    are those middle states with $\{x_0,x_{2m+1}\}\subset \x$, which
    in turn are the states generating $\DAmodHalfExt(\HmidHalfExt)$.
    Similarly, the holomorphic curves counted in the module actions
    for
    $\lsup{\Blgout}\DAmodExt(\HmidExt)_{\Blgin}\DT~\lsup{\Blgin}[\iota]_{\Clgin}$
    cannot have an $\alpha$-interval which maps to $\alphain_0$ or
    $\alphain_{2m+1}$: for the endpoints would have to either be at
    Reeb chords, but the incoming algebra does not allow Reeb chords
    with boundary on $\alphain_0$ or $\alphain_{2m+1}$; or they would
    have to be $\pm \infty$ punctures, but the generators we are
    considering contain $x_0$ and $x_{2m+1}$, not the corresponding
    $y_0$ or $y_{2n+1}$.
    Thus, the holomorphic curves used to define the $\Ainfty$ operations on
    $\lsup{\Blgout}\DAmodExt(\HmidExt)_{\Blgin}\DT~\lsup{\Blgin}[\iota]_{\Clgin}$
    coincide with the ones used to define 
    $\lsup{\Blgout}\DAmodExt(\HmidHalfExt)_{\Clgin}$.
  \end{proof}
    
{\bf{Step 2: Remove $\alphaout_0$ and $\alphaout_{2n}$.}}
This is a neck stretching argument, in the spirit of Step 1 in the
pairing theorem.

We start by analyzing moduli spaces that are relevant after the neck
stretching.  To this end, let $\Hmid_0$ consist of a punctured disk,
whose puncture we think of as the (filled) middle boundary $\Zmid$,
and whose boundary we label ${\mathcal Z}$, equipped with a single
embedded $\beta$-circle that separates the puncture from the boundary,
and a single $\alpha$-arc, denoted $\alpha$, that runs from the
boundary to the puncture, meeting $\beta$ in a single point, which we
denote $x$. See Figure~\ref{fig:LocalStab}.

\begin{figure}[h]
 \centering
 \input{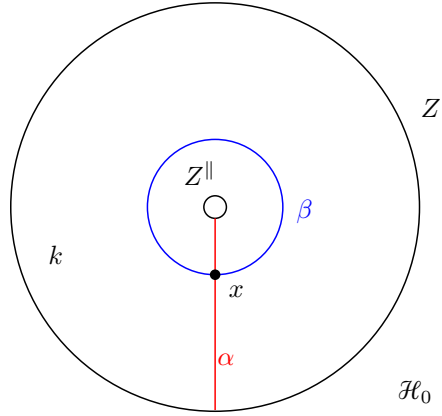}
 \caption{{\bf The diagram ${\mathcal H}_0$.}  We will consider
   homology classes with local multiplicity $k$ near the boundary $Z$,
   as indicated.}
 \label{fig:LocalStab}
 \end{figure}

A homology class of flows from $x$ to itself is determined by its
local multiplicity $k$ at the puncture; denote the space
$\ModFlow^{[k]}(x,x)$.

\begin{lemma}
  The space $\ModFlow^{[k]}(x,x)$ is $2k$-dimensional.
  There is a dense, open subset of $\ModFlow^{k}(x,x)$ consisting of
  those curves whose Reeb orbits are simple.
\end{lemma}

\begin{proof}
  The dimension is a straightforward application of the dimension formula:
  the Euler measure of the region is $k$, and the point measure is $k$.
  Smoothness is a consequence of the fact that the domain curve is also planar.
\end{proof}

Consider the evaluation map $\ev\colon \ModFlow^k(x,x)\to
\Sym^k([0,1]\times \R)$, which projects the punctures of the source
curves in $\ModFlow^{k}(x,x)$ to $[0,1]\times \R$. 

\begin{lemma}
  \label{lem:EvaluationDegreeOne}
  The map
  $\ev\colon \ModFlow^{[k]}(x,x)\to \Sym^k([0,1]\times \R)$
  is a proper map of odd degree.
\end{lemma}
 
\begin{proof}
  Properness can be thought of as a consequence of Gromov's compactness
  theorem, since any non-constant holomorphic curves from $x$ to $x$
  projects to some symmetric product $\Sym^k([0,1]\times\R)$. It
  remains to compute the degree, which we do with a model computation.

  Invert the Heegaard surface, so that it is a disk $D$ in the complex
  plane, $Z$ corresponds to the origin, $\beta$ corresponds to the
  unit circle.  We can think of the domain of the holomorphic curve
  $D^+$ as the upper half disk (i.e. $z\in D$ whose imaginary part is
  non-negative).  The holomorphic disks we are considering are
  holomorphic maps from $D^+$ to $D$ which carry the real interval
  $[-1,1]$ in $\partial D^+$ to the real interval in $D$, the upper half
  circle in $\partial D^+$ to the boundary of $D$, and the points $\{\pm
  1\}$ to $1$. By the Schwartz reflection principle, these correspond
  to holomorphic maps $f\colon D\to D$ so that
  \begin{enumerate}[label=(B-\arabic*),ref=(B-\arabic*)]
  \item $f(\partial D)\subset \partial D$
  \item 
    $f(D\cap \R)\subset \R$; i.e. ${\overline{f(z)}}=f({\overline z})$.
  \end{enumerate}
  Consider those $f$ for which $f(D\cap \R)\neq 0$. (These correspond
  to those holomorphic curves that have no Reeb chord on their boundary.)
  By classical complex analysis holomorphic maps $f$ as above can be uniquely
  written in the form
  \[ f(z)=\prod_{i=1}^{k}\frac{(z-\alpha_i)(z-{\overline\alpha_i})}
  {(1-{\overline\alpha}_iz)\cdot (1-{\alpha}_i z)},\]
  where $\{\alpha_1,\dots,\alpha_k\}\in\Sym^k(D^+)$.
  This shows that $\ev$ induces a homeomorphism, for suitable choices
  of complex structure $\beta$. It follows that the degree is odd in general.
\end{proof}

\begin{prop}
  \label{prop:ExtendMainStep}
  There is an identification
  \[ \lsup{\Blgout}\DAmodHalfExt(\HmidHalfExt)_{\Clgin}
  \simeq 
  \lsup{\Blgout}[\iota]_{\Clgout}\DT~ \lsup{\Clgout}\DAmod(\Hmid)_{\Clgin}\]
\end{prop}

\begin{proof}
  Start from $\Dmod(\Hmid)$. We will add first $\Zmid_0$, $\beta_0$,
  $\alphaout_0$ to obtain a diagram ${\mathcal H}'$; and then add
  $\Zmid_1$, $\beta_{g+m+n}$, and $\alphaout_{2n}$ to obtain $\HmidHalfExt$.
  
  Fix a curve $\gamma_0$ in ${\mathcal H}'$ which is parallel to
  $\beta_0$, dividing $\HmidHalfExt$ into two components, one of which
  is homeomorphic to $\Hmid_0$, and another looks like $\HmidHalfExt$
  with $\beta_0$ removed.  Stretch the neck normal to $\gamma_0$, so
  that $\HmidHalfExt={\mathcal H}'\cup \Hmid_0$.
  The diagram ${\mathcal H}'$ can also be used to define a $DA$ bimodule
  of the form $\lsup{\Blgout}\DAmodP_{\Clgin}$.

  Let $\ModFlow^B(\x,\y;\rhos_1,\dots,\rhos_k)$ be some moduli space of
  holomorphic curves which appears in the definition of the module
  actions for $\Hmid$, equipped with a puncture point $z_0$ (thought
  of as a filling of $\Zmid_0$). We have an analogous 
  moduli spaces $\ModFlow^B_{{\mathcal H}'}(\x,\y;\rhos_1,\dots,\rhos_k)$ which
  define the actions on $\lsup{\Blgout}\DAmodP_{\Clgin}$.

  Stretching the neck gives a fibered product description:
  \[ \ModFlow^B_{{\mathcal H}'}(\x,\y;\rhos_1,\dots,\rhos_k)
  =\ModFlow^B(\x,\y;\rhos_1,\dots,\rhos_k)\times_{\Sym^d([0,1]\times\R)}
  \ModFlow^{[d]}(\Hmid_0),\]
  where the fibered product is taken over the evaluation at the puncture $z_0$
  \[ \ModFlow^B(\x,\y;\rhos_1,\dots,\rhos_k)\to
  \Sym^d([0,1]\times\R) \] (where here $d$ depends on the homotopy
  class of $B$), and the evaluation map from
  Lemma~\ref{lem:EvaluationDegreeOne}. Indeed, by
  Lemma~\ref{lem:EvaluationDegreeOne}, it follows that
  $\DAmod(\Hmid)\cong \DAmodP({\mathcal H}')$.

  Adding $\beta_{g+m+n}$, $\Zmid_1$, and $\alphaout_{2n}$ to
  ${\mathcal H}'$ in the same manner, we obtain an 
  isomorphism $\DAmodP({\mathcal H}')\cong \DAmodExt(\HmidExt)$.
\end{proof}

We now have all the ingredients to prove the following:

\begin{proof}
  [Proof of Proposition~\ref{prop:ExtendDAPrecise}] Combine
  Lemma~\ref{lem:RemoveAlphaIns} and
  Proposition~\ref{prop:ExtendMainStep}.
\end{proof}

\section{Some algebraically defined bimodules}
\label{sec:AlgDA}

Having defined the holomorphic objects of study, we now turn to their
computation.  In~\cite{Bordered2}, we associated bimodules to
crossings, maxima, and minima.  Their description did not involve
pseudo-holomorphic curves: rather, they were defined via explicit,
algebraic descriptions. These bimodules were defined over the algebra
$\Alg$, and some versions were defined over a related algebra
$\DuAlg$. In this section, we adapt these constructions to the curved
framework. These modified bimodules will play a central role in an
explicit computation of the holomorphic objects
(e.g. Theorem~\ref{thm:MainTheorem} and Theorem~\ref{thm:ComputeD}
below).

\subsection{ Algebraically defined, curved bimodules}
\label{subsec:FormalModules}

Consider the $DA$ bimodules from~\cite{Bordered2} associated to a
positive crossings, a negative crossing, a maximum, and a minimum,
$\lsup{\Alg_2}\Pos^i_{\Alg_1}$, $\lsup{\Alg_2}\Neg^i_{\Alg_1}$,
$\lsup{\Alg_2}\Max^c_{\Alg_1}$, and $\lsup{\Alg_2}\Min^c_{\Alg_1}$.
(We suppress here the parameters (matching and strand numbers) of the algebras
appearing in these bimodules.) By slightly modifying the construction, we obtain corresponding
curved bimodules over $\cBlg$, which are related to the original bimodule as follows:

\begin{prop}
  \label{prop:CurvedDABimodules}
  There are curved bimodules associated to crossings, maxima, and
  minima over $\Blg$, $\lsup{\cBlg_2}\Pos^i_{\cBlg_1}$,
  $\lsup{\cBlg_2}\Neg^i_{\cBlg_1}$, $\lsup{\cBlg_2}\Max^c_{\cBlg_1}$, and
  $\lsup{\cBlg_2}\Min^c_{\cBlg_1}$, which are related to the corresponding
  bimodules over $\Alg$ defined in~\cite{Bordered2}
  by the following relations:
  \begin{align*} 
    \lsup{\Alg_2}T_{\cBlg_2} \DT \lsup{\cBlg_2}\Pos^i_{\cBlg_1} &\simeq
    \lsup{\Alg_2}\Pos^i_{\Alg_1}\DT \lsup{\Alg_1}T_{\cBlg_1} \\
    \lsup{\Alg_2}T_{\cBlg_2} \DT \lsup{\cBlg_2}\Neg^i_{\cBlg_1} &\simeq
    \lsup{\Alg_2}\Neg^i_{\Alg_1}\DT \lsup{\Alg_1}T_{\cBlg_1} \\
    \lsup{\Alg_2}T_{\cBlg_2} \DT \lsup{\cBlg_2}\Max^c_{\cBlg_1} &\simeq
    \lsup{\Alg_2}\Max^c_{\Alg_1}\DT \lsup{\Alg_1}T_{\cBlg_1} \\
    \lsup{\Alg_2}T_{\cBlg_2} \DT \lsup{\cBlg_2}\Min^c_{\cBlg_1} &\simeq
    \lsup{\Alg_2}\Min^c_{\Alg_1}\DT \lsup{\Alg_1}T_{\cBlg_1},
  \end{align*}
  where $\lsup{\Alg_i}T_{\cBlg_i}$ is the $\Blg$-to-$\Alg$ transformer
  from Definition~\ref{def:Transformer}.
\end{prop}

\begin{proof}
  This follows from the fact that the bimodules defined over $\Alg$
  are ``standard'' in the sense
  of~\cite[Section~\ref{BK2:sec:Algebras}]{Bordered2}, which we recall
  presently.  First, recall that the algebras $\Alg_i$ are equipped
  with preferred elements $C_i$, and each is equipped with a
  subalgebra $\Blg_i\subset \Alg_i$, with the property that $d C_i\in
  \Blg_i$; indeed, $d C_i$ is the curvature $\mu_0^{\Blg_i}$
  associated to the matching.  A {\em standard sequence} is a sequence
  $a_1,\dots,a_{\ell-1}$ of elements of $\Alg_1$ with the property
  that each $a_i$ is either equal to $C_1$ or it is an element of
  $\Blg_1\subset \Alg_1$.
  \begin{defn}
    \label{def:Standard}
    A type $DA$ bimodule $\lsup{\Alg_2}X_{\Alg_1}$ is called {\em standard} if
    the following conditions hold:
    \begin{enumerate}[label=(DA-\arabic*),ref=(DA-\arabic*)]
    \item
      \label{prop:Adapted}
      The bimodule $X$ is finite-dimensional (over $\Field$),
      and it has a  a $\Q$-grading $\Delta$ and a further grading $\Agr$, as follows. 
      Think of $\Alg_i$ as associated to a union of points $Y_i$,
      and fix a cobordism $W_1$ from $Y_1$ to $Y_2$.
      The weights of a
      a homogenous algebra $a$ element can be viewed as an element
      of $H^0(Y_i)$; taking the coboundary then induces an element
      $\Agr(a)$ in $H^1(W,\partial W)$.
      These various gradings are related by the formulas:
      \begin{align*}
        \Delta(\delta^1_{\ell}(\x,a_1,\dots,a_{\ell-1}))&=
        \Delta(\x)+\Delta(a_1)+\dots+\Delta(a_{\ell-1})-\ell+2 \\
        \Agr(\delta^1_{\ell}(\x,a_1,\dots,a_{\ell-1}))&=
        \Agr(\x)+\Agr(a_1)+\dots+\Agr(a_{\ell}));
      \end{align*}
      (This is the condition that $X$ is {\em adapted to $W_1$}
      in the sense
      of~\cite[Definition~\ref{BK2:def:AdaptedDA}]{Bordered2}.)
    \item
      The one-manifold $W$ is compatible with the
      matching $\Matching_1$ used in the algebra $\Alg_1$, in the
      sense that
      $W\cup W(\Matching_1)$ has no closed components.
    \item\label{prop:LandsInB} For any standard sequence of elements $a_1,\dots,a_{\ell-1}$
      with at least some $a_i\in\Blg_1$,
      $\delta^1_{\ell}(\x,a_1,\dots,a_{\ell-1})\in\Blg_2\otimes X$.
    \item
      \label{prop:CStandard}
      For any $\x\in X$,
      \[
      C^2\otimes \x +
      \sum_{\ell=0}^{\infty}\delta^1_{1+\ell}(\x,\overbrace{C^1,\dots,C^1}^{\ell})\in
      \Blg_2\otimes X.\]
      \end{enumerate}
  \end{defn}
  If $\lsup{\Alg_2}X_{\Alg_1}$ is standard in the above sense, then
  the actions on 
  $\lsup{\Alg_2}X_{\Alg_1}\DT \lsup{\Alg_1}T_{\cBlg_1}$ are sums of
  actions on $X$ where the input is a standard sequence in the
  above sense. It follows from Property~\ref{prop:LandsInB} that
  the output lies in $\Blg_2$, except in the special case of the
  $\delta^1_1$ action on the tensor product, in which case
  Proposition~\ref{prop:CStandard} expresses the output in
  $\Blg_2\otimes X$ plus $C^2\otimes X$.  Dropping the term in
  $C^2\otimes X$, we obtain actions defining
  $\lsup{\cBlg_2}X_{\cBlg_1}$. As the notation suggests, the operations so defined give a curved
  DA bimodule. Indeed,  the curved bimodule relations can be seen as direct
  consequence of the ordinary bimodule relations for
  $\lsup{\Alg_2}X_{\Alg_1}$: in our construction of
  $\lsup{\Blg_2}X_{\Blg_1}$, we have dropped the term $C^2\otimes \x$,
  but its contribution to the $\Ainfty$ relation in $\Blg_2\otimes X$
  is precisely $(\partial C^2) \otimes \x=\mu_0^{\Blg_2}\otimes \x$;
  moreover, the terms involving for $\mu_0^{\Blg_1}$ account for the
  terms in the $\Ainfty$ relation containing $\partial C^1$.

  By construction,
  \[ \lsup{\Alg_2}T_{\cBlg_2}\DT \lsup{\cBlg_2}X_{\cBlg_1}
  =\lsup{\Alg_2}X_{\Alg_1}\DT \lsup{\Alg_1}T_{\cBlg_1}:\]
  
  The proposition now follows since all the bimodules listed above
  are standard. (See~\cite[Proposition~\ref{BK2:prop:PosExt} and
    Theorems~\ref{BK2:thm:MaxDA} and~\ref{BK2:thm:MinDual}]{Bordered2}.)
\end{proof}

(Note that the actions on $\lsup{\cBlg_2} X_{\cBlg_1}$ are precisely
the actions $\gamma_k$ defined
in~\cite[Section~\ref{BK2:sec:Fast}]{Bordered2}.)

\subsection{Bimodules over $\DuAlg$}

Recall that in~\cite{Bordered2}, we considered an algebra $\DuAlg$ that
was dual to $\Alg$. Sometimes, it is convenient to work with bimodules over this algebra
(and its quotient, described in Subsection~\ref{subsec:nDuAlg}).

It is natural to consider idempotents in $\DuAlg$
that are complementary to those in $\Alg$. In the present paper, when
we write $\cBlg$, we understand $\Blg(2n,n)$, with curvature specified by some matching 
$\Matching$. Correspondingly, when we
write $\DuAlg$ without decoration, we understand $\DuAlg(2n,n+1,\Matching)$
from~\cite{Bordered2}.

In~\cite[Section~\ref{BK2:sec:DuAlg}]{Bordered2}, we constructed
bimodules associated to crossings over the algebras $\DuAlg$, denoted
$\lsup{\DuAlg_1}\Pos^i_{\DuAlg_2}$ and $\lsup{\DuAlg_1}\Neg^i_{\DuAlg_2}$, which
become homotopy equivalent to the type $DA$ bimodule over $\Alg$,
after tensoring with the canonical $DD$ bimodule. (See
Proposition~\ref{prop:DualityRelationship} below).  Before giving the
precise statement, we describe a similar bimodule for a local minimum,
as well. (This latter bimodule was not needed in~\cite{Bordered2}; rather, it
was sufficient to have only its corresponding type $DD$ module.)

\subsubsection{The local minimum $\lsup{\DuAlg_1}\Min^c_{\DuAlg_2}$}
\label{subsec:DDmin}

Let $\lsup{\Blg_2}\Min^c_{\Blg_1}$ denote the DA bimodule for a
minimum from~\cite{Bordered2}.  We define now the corresponding $DA$
bimodule $\lsup{\DuAlg_1}\Min^c_{\DuAlg_2}$ (i.e. where the input
algebra $\DuAlg_2$ has $2$ fewer strands than the output algebra
$\DuAlg_1$).  This discussion is is very similar to the construction
of the corresponding bimodule for a local maximum (over $\Alg$)
from~\cite[Section~\ref{BK2:sec:Max}]{Bordered2}; see
also~\cite[Section~\ref{BK1:sec:Crit}]{BorderedKnots}.

Let $\phi_c\colon \{1,\dots,2n\}\to \{1,\dots,2n+2\}$ be the map
\begin{equation}
\label{eq:DefInsert}
\phi_c(j)=\left\{\begin{array}{ll}
j &{\text{if $j< c$}} \\
j+2 &{\text{if $j\geq c$.}}
\end{array}\right.
\end{equation}
Let 
\begin{equation}
  \label{eq:MaxAlgebras}
  \DuAlg_1=\DuAlg(n+1,\Matching_1)
  \qquad{\text{and}}\qquad
  \DuAlg_2=\Alg(n,\Matching_2),
\end{equation}
where $\Matching_2$ is obtained from $\Matching_1$ by the property
that:
\begin{align}
\{\phi_c(i),\phi_c(j)\}\in\Matching_1&\Rightarrow
\{i,j\}\in\Matching_2; \nonumber \\
\{c,\phi_c(s)\},\{c+1,\phi_c(t)\}\in\Matching_1&\Rightarrow\{s,t\}\in\Matching_2. \label{eq:M1andM2}
\end{align}

We call an idempotent state $\y$ for $\DuAlg_1$ an {\em allowed idempotent state for $\DuAlg_1$} if
\begin{equation}
  \label{eq:AllowedIdempotents}
  |\y\cap\{c-1,c,c+1\}|\leq 2\qquad\text{and}\qquad c\in\y.
\end{equation}
There is a map $\psi'$ from
allowed idempotent states $\y$ for $\DuAlg_1$ to idempotent states  for $\Blg_2(n)$,
where $\x=\psi'(\y)\subset \{0,\dots,2n\}$ is characterized by
\begin{equation}
  \label{eq:SpecifyPsi}
  |\y\cap \{c-1,c,c+1\}| + 
  |\x\cap \{c-1\}| =2~\qquad{\text{and}}~\qquad \phi_c(\x)\cap \y=\emptyset.
\end{equation}
Similarly, we can define $\psi$ from idempotent states for $\DuAlg_1$ to idempotent states for $\DuAlg_2$ by 
\begin{equation}
\label{eq:SpecifyPsiPrimed}
\psi(\y)=\{0,\dots,2n\}\setminus\psi'(\y). 
\end{equation}

We have the following:

\begin{lemma}
  \label{lem:ConstructDeltaTwo}
  If $\x$ is an allowed idempotent state (for $\DuAlg_1$) and $\y$ is
  an idempotent state for $\DuAlg_2$ so that $\psi(\x)$ and $\y$ are
  close enough (i.e. $\Idemp{\psi(\x)}\cdot \Blg_1\cdot \y\neq 0$),
  then there is an allowed idempotent state $\z$ (for $\DuAlg_1$) with
  $\psi(\z)=\y$ so that there is a map
  \[ \Phi_{\x}\colon 
  \Idemp{\psi(\x)}\cdot \DuAlg_2\cdot \Idemp{\y}\to
  \Idemp{\x}\cdot \DuAlg_1\cdot \Idemp{\z}\]
  with the following properties:
  \begin{itemize}
    \item $\Phi_\x$ maps the portion of $\Idemp{\psi(\x)}\cdot \Blg_2\cdot \Idemp{\y}$ with weights
      $(v_1,\dots,v_{2n})$ surjectively onto the portion of 
      $\Idemp{\x}\cdot \Blg_1\cdot\Idemp{\z}$ with 
      $w_{\phi_c(i)}=v_i$ and $w_{c}=w_{c+1}=0$, 
    \item $\Phi_{\x}$ further satisfies the relations
      \begin{align*}
      \Phi_{\x}(U_i\cdot a)&=U_{\phi_c(i)} \cdot \Phi_{\x}(a)  \\
      \Phi_{\x}(E_i\cdot a)&=
      \left\{\begin{array}{ll}
      E_{\phi_c(i)}\cdot \Phi_{\x}(a) & {\text{if $i\neq t$}} \\
      (E_{\phi_c(t)}+E_{c}\llbracket E_{\phi_c(t)},E_{c+1}\rrbracket)\cdot \Phi_\x(a) &{\text{if $i=t$}}
      \end{array}\right.
          \end{align*}
      for any $i\in 1,\dots,2n$ and $a\in \Idemp{\psi(\x)}\cdot \DuAlg_2\cdot \Idemp{\y}$.
    \end{itemize}
    Moreover, the state $\z$ is uniquely characterized by the existence of 
    such a map $\Phi_\x$.
\end{lemma}

\begin{proof}
  This is a straightforward adaptation
  of~\cite[Lemma~\ref{BK2:lem:ConstructDeltaTwo}]{Bordered2}; the only
  novelty is that here we are using the algebra elements $E_i$
  extending $\Blg$ to $\DuAlg$ rather than the algebra elements $C_p$
  extending $\Blg$ to $\Alg$ as in that lemma. 
\end{proof}

\begin{defn}
  Let $\lsup{\DuAlg_1}\Min^c_{\DuAlg_2}$ be the vector space generated
  by elements ${\mathbf Q}_\y$ generated by
  allowed idempotent states $\y$ for $\DuAlg_1$. Endow this with 
  the structure of a 
  $\IdempRing(\DuAlg_1)-\IdempRing(\DuAlg_2)$ bimodule
  by
  \[ \Idemp{\psi(\y)}\cdot {\mathbf Q}_\y\cdot \Idemp{\y}={\mathbf Q}_\y.
  \]
 Let 
  \[ \delta^1_1\colon \lsup{\DuAlg_1}\Max^c_{\DuAlg_2}
  \to \DuAlg_1\otimes \lsup{\DuAlg_1}\Max^c_{\DuAlg_2} \]
  be the map specified by
  \[ \delta^1_1({\mathbf Q}_{\y})= \Idemp{\y}\cdot \left(R_{c+1} R_{c} +
    L_{c} L_{c+1} + U_c E_{c+1}\right)\otimes \sum_{\z} {\mathbf
    Q}_{\z}.\] where the sum is taken over all allowed idempotents
  $\z$ for $\DuAlg_2$. 
  Let 
  \[\delta^1_2\colon \lsup{\DuAlg_1}\Max^c_{\DuAlg_2}\otimes
  \DuAlg_2\to \DuAlg_1\otimes \lsup{\DuAlg_1}\Max^c_{\DuAlg_2} \]
  be the map characterized by the property that if
  $a=\Idemp{\psi(\x)}\cdot a\cdot \Idemp{\y}\in\DuAlg_1$ is a non-zero
  algebra element,
  then 
  $\delta^1_2({\mathbf Q}_\x\cdot a)=\Phi_{\x}\cdot {\mathbf Q}_{\z}$,
  where $\z$ is as in Lemma~\ref{lem:ConstructDeltaTwo}.
\end{defn}

\begin{prop}
  The above specified actions $\delta^1_1$ and $\delta^1_2$ (and
  $\delta^1_\ell=0$ for all $\ell>2$) give
  $\lsup{\DuAlg_2}\Min^c_{\DuAlg_1}$ the structure of a $DA$
  bimodule, equipped with a $\Delta$-grading and a grading by $H^1(W,\partial)$,
  where $W$ is the one-manifold specified in the diagram.
\end{prop}

\begin{proof}
  The proof is
  straightforward. (Compare~\cite[Theorem~\ref{BK2:thm:MaxDA}]{Bordered2}.)
  The differential of the term $E_{c}\llbracket E_{c+1},E_t\rrbracket$
  in $\Phi_\x(E_s)$, which is $U_{c} \cdot \llbracket
  E_{c+1},E_t\rrbracket$, cancels against the anti-commutator of
  the other term in
  $\Phi_{\x}(E_t)$, $E_{\phi_c(t)}$,
  with the term $\delta^1_1(\x)=U_c\cdot E_{c+1}$.
\end{proof}

\subsubsection{$DA$ bimodules over $\DuAlg$}

The bimodule over $\DuAlg$ associated to a local minimum defined above
and the bimodules over $\DuAlg$ associated to crossings
in~\cite[Section~\ref{BK2:sec:DualCross}]{Bordered2} are related to
the corresponding bimodules over $\Alg$.

In a little more detail, given an integer $n$, a matching $\Matching$
on $\{1,\dots,2n\}$, and an integer $i\in 1,\dots,2n-1$, there are
associated algebras which we abbreviate $\Alg_1$, $\DuAlg_1$,
$\Alg_2$, and $\DuAlg_2$, and imodules $\lsup{\Alg_2}\Pos^i_{\Alg_1}$
and
$\lsup{\Alg_2}\Neg^i_{\Alg_1}$. In~\cite[Section~\ref{BK2:sec:DualCross}]{Bordered2},
we also associated bimodules $\lsup{\DuAlg_1}\Pos_{\DuAlg_2}$
and $\lsup{\DuAlg_1}\Neg_{\DuAlg_2}$.

\begin{prop}
  \label{prop:DualityRelationship}
  For algebras $\Alg_j$ and $\DuAlg_j$ as above, we have relations
  \begin{align*} 
    \lsup{\Alg_2}\Pos^i_{\Alg_1}\DT \lsup{\Alg_1,\DuAlg_1}\CanonDD &\simeq
    \lsup{\DuAlg_1}\Pos^i_{\DuAlg_2}\DT \lsup{\DuAlg_2,\Alg_1}\CanonDD \\
    \lsup{\Alg_2}\Neg^i_{\Alg_1}\DT \lsup{\Alg_1,\DuAlg_2}\CanonDD &\simeq
    \lsup{\DuAlg_1}\Neg^i_{\DuAlg_2}\DT \lsup{\DuAlg_2,\Alg_2}\CanonDD
  \end{align*}
  Similarly, now using algebras $\DuAlg_j$ as in Equation~\ref{eq:MaxAlgebras}
  (and corresponding algebras $\Alg_j$), we have that
  \begin{align*}
  \lsup{\Alg_2}\Min_{\Alg_1}\DT \lsup{\Alg_1,\DuAlg_1}\CanonDD &\simeq
    \lsup{\DuAlg_1}\Min_{\DuAlg_2}\DT \lsup{\DuAlg_2,\Alg_2}\CanonDD
  \end{align*}
\end{prop}

\begin{proof}
  Note that the $DA$ bimodules over $\DuAlg$ are all bounded, so the
  above tensor products all make sense.  The relation involving the
  positive crossing
  is~\cite[Proposition~\ref{BK2:prop:PosExt}]{Bordered2}.  The
  relationship with the negative crossing follow formally, since both
  negative crossing bimodules are constructed as ``opposite modules''
  to the positive crossing bimodules;
  see~\cite[Definition~\ref{BK2:def:NegCrossing} and
  Subsection~\ref{BK2:subsec:DuAlgNegCross}]{Bordered2}.
 
  For the relation involving the minimum, note that
  in~\cite[Section~\ref{BK2:subsec:DDmin}]{Bordered2},
  we defined the type $DD$ bimodule associated to a minimum,
  denoted $\lsup{\Alg,\DuAlg}\Min$.
  In Theorem~\ref{BK2:thm:MinDual}, it is verified that
  \[\lsup{\Alg_2}\Min^c_{\Alg_1}\DT \lsup{\Alg_1,\DuAlg_1}\CanonDD
  \simeq \lsup{\Alg_2,\DuAlg_1}\Min.\]
  The similar verification that 
  \[\lsup{\Alg_2,\DuAlg_1}\Min\simeq
  \lsup{\DuAlg_1}\Min_{\DuAlg_2}\DT \lsup{\DuAlg_2,\Alg_2}\CanonDD\]
  is a straightforward exercise in the definitions.
\end{proof}

\begin{rem}
  Observe that for all $X\in \Min^c$,
  \begin{equation}
    \label{eq:IdealToIdeal}
    \delta^1_2(X,\llbracket E_i,E_j\rrbracket)=
    \left\{\begin{array}{ll}
        \llbracket E_{\phi_c(i),\phi_c(j)}\rrbracket\otimes X &{\text{if $\{i,j\}\in M_2$,
            $\{i,j\}\neq \{s,t\}$}} \\
        \llbracket E_{\phi_c(s)},E_{c}\rrbracket\cdot
        \llbracket E_{\phi_c(t)},E_{c+1}\rrbracket\otimes X
        &{\text{if $\{i,j\}=\{s,t\}$.}}
      \end{array}\right.
  \end{equation}
\end{rem}

\subsection{A new algebra $\nDuAlg$}
\label{subsec:nDuAlg}

We will find it convenient to work in a quotient of $\DuAlg$.
Specifically,
let $\nDuAlg$ be the quotient of $\DuAlg$ by the relations $\llbracket E_i,
E_j\rrbracket = 1$ for all $\{i,j\}\in\Matching$.  There is a quotient
map $q\colon \DuAlg\to\nDuAlg$.

Recall that $\DuAlg$ is equipped with a $\Delta$-grading defined by
\[ \Delta(a)=\#(\text{$E_j$~that divide $a$})-\sum_{i} \weight_i(a).\]
Since $\llbracket E_i,E_j\rrbracket$ is homogeneous
with $\Delta(\llbracket E_i,E_j\rrbracket)=0$, it follows that the
$\Delta$ grading descends to $\nDuAlg$.

The weight vector, which was a grading on $\DuAlg$, no longer descends
to $\nDuAlg$. However, there is an Alexander grading on $\DuAlg$,
induced by the matching and the orientation on its associated
one-manifold, defined by
\[ \AlexGr(a)=\bigoplus_{\{i,j\}\in\Matching} \weight_i(a)-\weight_j(a), \]
where $W$ represents an arc oriented from $i$ to $j$.
This grading descends to $\nDuAlg$.

Evidently, the set of algebra
elements in $\nDuAlg$ with a fixed $\Delta$-grading is
finite-dimensional (unlike for $\DuAlg$). It follows that bimodules
over $\nDuAlg$ automatically satisfy the boundedness properties
required to form their tensor products, as
in~\cite[Proposition~\ref{BK2:prop:AdaptedTensorProd}]{Bordered2}.
Specifically, a bimodule is called {\em adapted $W$} if it has a grading
by $\Q$ and $H^1(W)$, as in Property~\ref{prop:Adapted} from
Definition~\ref{def:Standard} (using the algebras
$\nDuAlg$ in place of $\Alg$, equipped with the above Alexander grading).

\begin{prop}
  \label{prop:AdaptedTensorProdnDuAlg}
  Choose $W_1\colon Y_1\to Y_2$, $W_2\colon Y_2\to Y_3$, $\nDuAlg_1$,
  $\nDuAlg_2$, and $\nDuAlg_3$ as above.  Suppose moreover that $W_1\cup
  W_2$ has no closed components, i.e. it is a disjoint union of
  finitely many intervals joining $Y_1$ to $Y_3$.  Given any two
  bimodules $\lsup{\nDuAlg_2}X^1_{\nDuAlg_1}$ and
  $\lsup{\nDuAlg_3}X^2_{\nDuAlg_2}$ adapted to $W_1$ and $W_2$ respectively,
  we can form their tensor product
  $\lsup{\nDuAlg_3}X^2_{\nDuAlg_2}\DT~\lsup{\nDuAlg_2}X^1_{\nDuAlg_1}$ (i.e. only
  finitely many terms in the infinite sums in its definition
  are non-zero); and moreover, it is a bimodule that is adapted to
  $W=W_1\cup W_2$.
\end{prop}

\begin{proof}
  This follows exactly as
  in~\cite[Proposition~\ref{BK2:prop:AdaptedTensorProd}]{Bordered2},
  in view of the fact that the set of algebra elements in $\nDuAlg$
  with a fixed $\Delta$-grading is finite-dimensional. (The same
  property holds for the algebras $\Alg$ considered in that
  proposition.)
\end{proof}

We construct bimodules over $\nDuAlg$ that are related to bimodules over 
$\DuAlg$ in the following:

\begin{lemma}
  \label{lem:BimodulesOvernDuAlg}
  There are $DA$ bimodules
  $\lsup{\nDuAlg_1}\Pos^i_{\nDuAlg_2}$, 
  $\lsup{\nDuAlg_1}\Neg^i_{\nDuAlg_2}$, 
  $\lsup{\nDuAlg_1}\Min^c_{\nDuAlg_2}$
  related to the above bimodules by the relations
  \begin{align} 
  \lsup{\nDuAlg_1}
  q_{\DuAlg_1} \DT~ \lsup{\DuAlg_1}\Min^c_{\DuAlg_2} &\simeq
    \lsup{\nDuAlg_1}\Min^c_{\nDuAlg_2}\DT~ \lsup{\nDuAlg_2}q_{\DuAlg_2} \label{eq:MinnDuAlg}\\
  \lsup{\nDuAlg_1}q_{\DuAlg_1} \DT~ \lsup{\DuAlg_1}\Pos^i_{\DuAlg_2} &\simeq
    \lsup{\nDuAlg_1}\Pos_{\nDuAlg_2}\DT~ \lsup{\nDuAlg_2}q_{\DuAlg_2} \label{eq:PosnDuaAlg}\\
  \lsup{\nDuAlg_1}q_{\DuAlg_1} \DT~ \lsup{\DuAlg_1}\Neg^i_{\DuAlg_2} &\simeq
    \lsup{\nDuAlg_1}\Neg_{\nDuAlg_2}\DT~ \lsup{\nDuAlg_2}q_{\DuAlg_2} 
\end{align}
\end{lemma}

\begin{proof}
  Equation~\eqref{eq:IdealToIdeal} ensures that
  all the operations where some element of the
  input algebra lies in the ideal in $\DuAlg$ whose quotient is
  $\nDuAlg$ have the property that the output is contained in the corresponding 
  ideal in the output algebra. 
  It follows at once that there is an induced module
  $\lsup{\nDuAlg_1}\Min^c_{\nDuAlg_2}$ satsisfying
  Equation~\eqref{eq:MinnDuAlg}.

  The bimodule associated to a positive crossing is constructed in
  Section~\cite[Section~\ref{BK2:sec:DuAlg}]{Bordered2}. It is constructed
  so that all operations commute with multiplication by 
  $\llbracket E_i,E_j\rrbracket$. Specifically, there are relations
  \begin{align*}
    \delta^1_2(X,\llbracket E_i,E_j\rrbracket\cdot a)&=\llbracket
     E_{\tau(i)},E_{\tau(j)}\rrbracket\cdot \delta^1_2(X,a)\nonumber\\
    \delta^1_3(X,\llbracket E_i,E_j\rrbracket\cdot a_1,a_2)&=
    \delta^1_3(X,a_1,\llbracket E_i,E_j\rrbracket\cdot a_2)
    =
    \llbracket E_\tau(i),E_\tau(j)\rrbracket\cdot 
    \delta^1_3(X,a_1,a_2),
  \end{align*}
  where $\tau\colon \{1,\dots,2n\}\to \{1,\dots,2n\}$ is the
  transposition that switches the two adjacent elements, corresponding
  to the strands that enter the crossing.  (Note that $\delta^1_k=0$
  for $k>3$.)  It follows at once that there is a bimodule
  $\lsup{\nDuAlg_1}\Pos_{\nDuAlg_2}$ satisfying
  Equation~\eqref{eq:PosnDuaAlg}.

  The corresponding construction for the negative crossing follow
  similarly. (Compare~\cite[Section~\ref{BK2:subsec:DuAlgNegCross}]{Bordered2}.)
\end{proof}

We wish to state an analogue of
Proposition~\ref{prop:DualityRelationship}, using the algebras $\nDuAlg$. To this end, we will
define the canonical type $DD$ bimodule
$\lsup{\cBlg,\nDuAlg}\CanonDD$.  
Generators of $\lsup{\cBlg,\nDuAlg}\CanonDD$, as a vector space,
correspond to $n$-element subsets $\x\subset\{0,\dots,2n\}$, i.e.
$I$-states for $\Blg(2n,n)$. Let $\gen_\x$ be the generator corresponding to $\x$. The  left $\IdempRing(2n,n)\otimes \IdempRing(2n,n+1)$-module 
structure is specified by
\[ (\Idemp{\x}\otimes \Idemp{\{0,\dots,2n\}\setminus \x}) \cdot \gen_\x = \gen_x.\]
The differential is specified by the element
\[
A = \sum_{i=1}^{2n} \left(L_i\otimes R_i + R_i\otimes L_i\right) + \sum_{i=1}^{2n}
U_i\otimes E_i \in
  \Blg\otimes \nDuAlg,\]
  \[ \delta^1 \colon \CanonDD \to \Alg\otimes\nDuAlg \otimes
  \CanonDD.\] by $\delta^1(v)=A\otimes v$.
That is, $\lsup{\cBlg,\nDuAlg}\CanonDD$ is obtained from the description
of $\lsup{\Alg,\DuAlg}\CanonDD$ by dropping the terms involving $C_{\{i,j\}}$,
and then taking the quotient to go from $\DuAlg$ to $\nDuAlg$, or, more formally,
\[ \lsup{\Alg}T_{\cBlg}\DT \lsup{\cBlg,\nDuAlg}\CanonDD=
\lsup{\nDuAlg_1}q_{\DuAlg}\DT \lsup{\DuAlg,\Alg}\CanonDD.\]

For future reference, we describe analogous type $DD$ bimodules for 
a local minimum and for a crossing.

\subsubsection{Algebraically defined curved $DD$-bimodule of a local minimum}

Fix some positive integers $n$ and $c$,
with $1\leq c\leq 2n+1$, and a matching $\Matching_1$ on $\{1,\dots,2n+2\}$.
Let 
\[
  \nDuAlg_1=\nDuAlg(n+1,\Matching_1)
  \qquad{\text{and}}\qquad
  \cBlg_2=\cBlg(n,\Matching_2),
\]
where $\Matching_2$ is induced from $\Matching_1$ as in Equation~\ref{eq:M1andM2}.
 Define $\lsup{\cBlg_2,\nDuAlg_1}\Min_c=\Min$, as follows.
 
 We call an idempotent state $\y$ for $\Blg_1$ an {\em allowed idempotent state for $\Blg_1$} if
 \[ |\y\cap\{c-1,c+1\}|\neq\emptyset\qquad\text{and}\qquad c\not\in\y.\] 
 Note that the allowed idempotent states $\y$ for $\Blg_1$ 
 are those that are of the form $\{0,\dots,2n+2\}\setminus \y'$, where
 $\y'$ is an allowed idempotent state for $\DuAlg_1$, in the sense of Equation~\eqref{eq:AllowedIdempotents}.

 As a vector space, $\Min$ is spanned by vectors that are in
 one-to-one correspondence with allowed idempotent states for $\nDuAlg_1$;
 if $\y$ is an allowed idempotent state for $\nDuAlg_1$, let ${\mathbf P}_\y$ be its corresponding generator.
 The bimodule structure, over the rings of idempotents $\IdempRing(\Blg_2)\otimes \IdempRing(\nDuAlg_1)$
 is specified by the relation
 \[ (\Idemp{\psi(\y)}\otimes \Idemp{\{0,\dots,2n+2\}\setminus \y})\cdot {\mathbf P}_{\y}= {\mathbf P}_\y,\]
 where $\psi$ is as in Equation~\eqref{eq:SpecifyPsi}.

 The differential is specified by the the following element in $\Blg_2\otimes\nDuAlg_1$:
 \begin{align}                                                                                                                                        
   A''&=(1\otimes L_{c} L_{c+1}) +                                                                                                                        
   (1\otimes R_{c+1} R_{c})                                                                                                                             
   + \sum_{j=1}^{2n} R_{j} \otimes L_{\phi(j)} + L_{j} \otimes R_{\phi_c(j)} + U_{j}\otimes E_{\phi_c(j)}
   \label{eq:defDDmin}\\
   &                                                                                                                                                    
   + 1\otimes E_{c} U_{c+1}                                                                                                                           
   + U_{s}\otimes E_{c+1}.
   \nonumber                                                                                                                                            
 \end{align}

 (Note that in~\cite[Section~\ref{BK2:subsec:DDmin}]{Bordered2}, we described a $DD$ bimodule
 $\lsup{\nDuAlg_1,\Alg_2}\Min_c$ associated to a local minimum. The above description
 is obtained by elminiating all the terms involving $C_{i,j}$, and specializing from $\DuAlg$
 to $\nDuAlg$.)

 The bimodule can be described in a little more detail, in terms of the classification of
 idempotents for $\nDuAlg_1$ into three types,  labelled $\XX$, $\YY$, and $\ZZ$:
 \begin{itemize}
 \item $\y$ is of type $\XX$ if $\y\cap \{c-1,c,c+1\}=\{c,c+1\}$,
 \item $\y$ is of type $\YY$ if $\y\cap \{c-1,c,c+1\}=\{c-1,c\}$,
 \item $\y$ is of type $\ZZ$ if $\y\cap \{c-1,c,c+1\}=\{c-1,c+1\}$. 
 \end{itemize}
 There is a corresponding classification of the generators ${\mathbf
   P}_{\y}$ into $\XX$, $\YY$, and $\ZZ$.

 With respect to this decomposition, 
 terms in  the differential are of the following four types:
 \begin{enumerate}[label=(P-\arabic*),ref=(P-\arabic*)]
 \item 
   \label{type:OutsideLR}
   $R_{j}\otimes L_{\phi_c(j)}$
   and $L_{c}\otimes R_{\phi_c(j)}$ for all $j\in \{1,\dots,2n\}\setminus \{c-1,c\}$; these connect
   generators of the same type. 
 \item
   \label{type:UC}
   $U_{j}\otimes E_{\phi_c(j)}$
 \item 
   \label{type:UC2}
   $1\otimes E_{c} U_{c+1}$
 \item
   \label{type:Curved}
   $U_s \otimes E_{c+1}$
 \item 
   \label{type:InsideCup}
   Terms in the diagram below that connect  generators
   of different types.
   
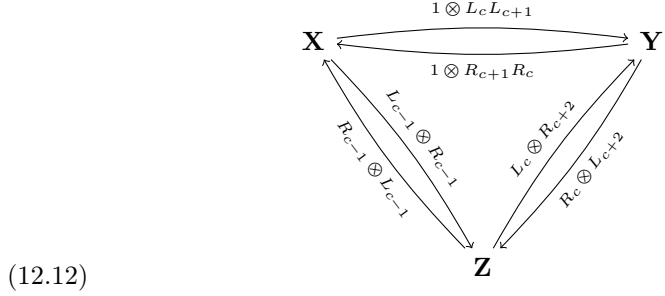
\begin{equation}
     \label{eq:CritDiag}
     \begin{tikzpicture}[scale=1.5]
       \node at (-1.5,0) (X) {$\XX$} ;
       \node at (1.5,0) (Y) {$\YY$} ;
       \node at (0,-2) (Z) {$\ZZ$} ;
       \draw[<-] (X) [bend right=7] to node[below,sloped] {\tiny{$1\otimes R_{c+1} R_{c}$}}  (Y)  ;
       \draw[<-] (Y) [bend right=7] to node[above,sloped] {\tiny{$1\otimes L_{c} L_{c+1}$}}  (X)  ;
       \draw[<-] (X) [bend right=7] to node[below,sloped] {\tiny{$R_{c-1}\otimes L_{c-1}$}}  (Z)  ;
       \draw[<-] (Z) [bend right=7] to node[above,sloped] {\tiny{$L_{c-1}\otimes R_{c-1}$}}  (X)  ;
       \draw[<-] (Z) [bend right=7] to node[below,sloped] {\tiny{$R_{c} \otimes L_{c+2}$}}  (Y)  ;
       \draw[<-] (Y) [bend right=7] to node[above,sloped] {\tiny{$L_{c}\otimes R_{c+2}$}}  (Z)  ;
     \end{tikzpicture}
   \end{equation}
 \end{enumerate}

 With the understanding that
 if $c=1$, then the terms containing $L_{c-1}$ or $R_{c-1}$ are missing; 
 similarly, if $c=2n$, the terms containing $R_{c+2}$ and $L_{c+2}$ are missing.

\subsubsection{Algebraically defined curved bimodule of a positive crossing}
\label{subsec:AlgPosCross}

We construct the bimodle $\lsup{\cBlg_2,\nDuAlg_1}\Pos_i$.
Fix $i$ with $1\leq i\leq 2n-1$,
fix a matching $\Matching$ on $\{1,\dots,2n\}$.
Let $\tau=\tau_i\colon \{1,\dots,2n\}\to  \{1,\dots,2n\}$ 
be the transposition that 
switches $i$ and $i+1$, and let
$\tau(\Matching)$ be the induced matching; i.e.
$\{j,k\}\in\Matching$ iff $\{\tau(j),\tau(k)\}\in\tau(\Matching)$
Let 
\begin{equation}
  \label{eq:DefA1A2}
  \nDuAlg_1=\nDuAlg(n,\tau_i(\Matching))
  \qquad{\text{and}}\qquad
  \cBlg_2=
  \cBlg(n,\Matching).
\end{equation}

\begin{figure}[ht]
\input{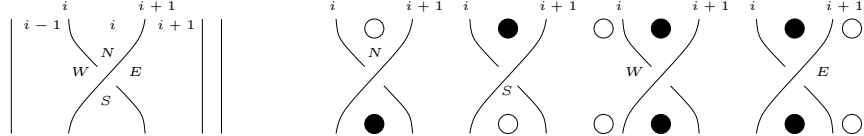}
\caption{\label{fig:PosCrossDD} {\bf{Positive crossing $DD$ bimodule generators.}}  
The four generator types are pictured to the right.}
\end{figure}

As an $\IdempRing(\Blg_1)-\IdempRing(\nDuAlg_2)$-bimodule, $\Pos_i$ is the submodule
of $\IdempRing(\Blg_1)\otimes_{\Field}\IdempRing(\DuAlg_2)$ generated by elements
$\Idemp{\x}\otimes \Idemp{\y}$ where $\x\cap\y=\emptyset$ or 
\begin{equation}
  \label{eq:PosGens}
  \x\cap\y=\{i\}\qquad{\text{and}}\qquad \{0,\dots,2n\}\setminus (\x\cup \y) =\{i-1\}~\text{or}~\{i+1\}.
\end{equation}
Generators can be classified into four types,
$\North$, $\South$, $\West$, and $\East$:  for generators of type $\North$ 
the subsets $\x$ and $\y$ are complementary subsets of $\{0,\dots,2n\}$ 
and $i\in \x$; 
for generators of type $\South$, 
$\x$ and $\y$ are complementary subsets of $\{0,\dots,2n\}$ with $i\in \y$;
for generators of type $\West$, $i-1\not\in \x$ and $i-1\not\in \y$,
and $\x\cap\y=\{i\}$;
for generators of type $\East$, $i+1\not\in \x$ and $i+1\not\in \y$,
and $\x\cap\y=\{i\}$.

The differential has the following types of terms:
\begin{enumerate}[label=(P-\arabic*),ref=(P-\arabic*)]
\item 
  \label{type:OutsideLRP}
  $R_j\otimes L_j$
  and $L_j\otimes R_j$ for all $j\in \{1,\dots,2n\}\setminus \{i,i+1\}$; these connect
  generators of the same type. 
\item
  \label{type:UCP}
  $U_{j}\otimes E_{\tau(j)}$
  for all $j=1,\dots,2n$
\item 
  \label{type:InsideP}
  Terms in the diagram below that connect  generators
  of different types:
  
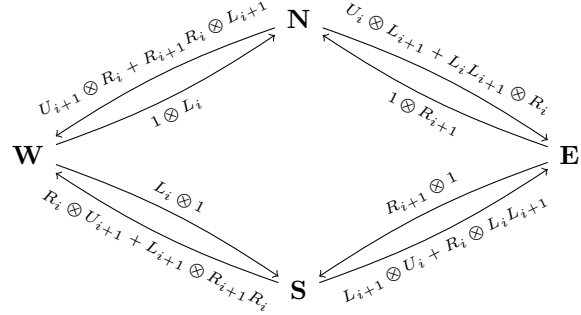
\begin{equation}
    \label{eq:PositiveCrossing}
    \begin{tikzpicture}[scale=1.8]
    \node at (0,3) (N) {$\North$} ;
    \node at (-2,2) (W) {$\West$} ;
    \node at (2,2) (E) {$\East$} ;
    \node at (0,1) (S) {$\South$} ;
    \draw[->] (S) [bend left=7] to node[below,sloped] {\tiny{$R_i\otimes U_{i+1}+L_{i+1}\otimes R_{i+1}R_i$}}  (W)  ;
    \draw[->] (W) [bend left=7] to node[above,sloped] {\tiny{$L_{i}\otimes 1$}}  (S)  ;
    \draw[->] (E)[bend right=7] to node[above,sloped] {\tiny{$R_{i+1}\otimes 1$}}  (S)  ;
    \draw[->] (S)[bend right=7] to node[below,sloped] {\tiny{$L_{i+1}\otimes U_i + R_i \otimes L_{i} L_{i+1}$}} (E) ;
    \draw[->] (W)[bend right=7] to node[below,sloped] {\tiny{$1\otimes L_i$}} (N) ;
    \draw[->] (N)[bend right=7] to node[above,sloped] {\tiny{$U_{i+1}\otimes R_i + R_{i+1} R_i \otimes L_{i+1}$}} (W) ;
    \draw[->] (E)[bend left=7] to node[below,sloped]{\tiny{$1\otimes R_{i+1}$}} (N) ;
    \draw[->] (N)[bend left=7] to node[above,sloped]{\tiny{$U_{i}\otimes L_{i+1} + L_{i} L_{i+1}\otimes R_i$}} (E) ;
  \end{tikzpicture}
\end{equation}
\end{enumerate}

Note that for a generator of type $\East$, the terms of
Type~\ref{type:OutsideLRP} with $j=i+2$ vanish; while for one of type
$\West$, the terms of Type~\ref{type:OutsideLRP} with $j=i-1$ vanish.

There is a $\Q$-grading on $\Pos_i$ defined by
\begin{equation}
  \label{eq:DeltaGradingPos}
  \Delta(\North)=\Delta(\South)=\Delta(\West)-\OneHalf=\Delta(\East)-\OneHalf,
\end{equation}
which is homological in the sense that if $(a\otimes b)\otimes Y$
appears with non-zero coefficient in $\delta^1(X)$, then 
\[ \Delta(X)-1=\Delta(Y)+\Delta(a)+\Delta(b)
=\Delta(Y)-\weight(a)-\weight(b)+\# \text{($E$ in $b$)}.\]

There is an additional $\MGradingSet=\Q^{2n}$-valued grading on
$\Pos_i$. 
The grading set can be thought of as the vector space
spanned by the strands in the picture.)  

Fix a standard basis
$e_1,\dots,e_{2n}$ for $\Q^{2n}$ (where we can think of $e_k$ as
labelling the strand whose output is the $k^{th}$ spot from the left,
in the algebra corresponding to $\Blg$), and set
\begin{equation}
  \label{eq:AlgGradeCrossing}
  \begin{array}{llll}
  \gr(\North)=\frac{e_i +e_{i+1}}{4} & 
  \gr(\West)=\frac{e_{i}-e_{i+1}}{4} & 
  \gr(\East)=\frac{-e_{i}+e_{i+1}}{4} &
  \gr(\South)=\frac{-e_{i}-e_{i+1}}{4}.
  \end{array}
\end{equation}
This induces an Alexander grading on the bimodule, in the following
sense.  Let $\Agr^{\Matching}$ denote the Alexander grading on the
algebra from Subsection~\ref{subsec:Gradings}.
Now, if
$(a\otimes b)\otimes Y$
appears with non-zero multiplicity in $\delta^1(X)$,
then
\[
\Pi(\gr(X))=\Agr^{\Matching}(a)-\Agr^{\tau_i(\Matching)}(b)+\Pi(\gr(Y)),\]
where $\Pi\colon \Q^{2n}\to \Q^{n}$ is the quotient by the
matching. 

(More abstractly, if $W_1$ is the $1$-manifold with $n$ components
associated to the matching, and $W_2$ denotes $1$-manifold with $2n$
components appearing in the picture of the crossing, the map $\Pi$ is
induced by the inclusion map $H^1(W_2,\partial W_2)\to H^1(W_2\cup
W_1,\partial W_2\cup W_1)\cong H^1(W_1,\partial W_1)$.)

The definition is almost the same as the $DD$ bimodule
$\lsup{\Alg_2,\DuAlg_1}\Pos_i$ associated to a crossing
from~\cite[Section~\ref{BK2:subsec:DDcross}]{Bordered2}, except that the
terms involving $C_{\{i,j\}}$ are dropped; i.e. by its construction, it
  is clear that
\[ \lsup{\nDuAlg_1}q_{\DuAlg_1} \DT\lsup{\Alg_2,\DuAlg_1}\Pos_i
= \lsup{\Alg_2}T_{\cBlg_2}\DT \lsup{\cBlg_2,\nDuAlg_1}\Pos_i. \]

\subsubsection{Algebraically defined curved bimodule of a negative crossing}
Taking opposite modules, we can form
\[{\overline{\lsup{\Alg_1,\DuAlg_2}
    \Pos_i}}=\overline{\Pos}_i^{\Alg_1,\DuAlg_2}
=\lsup{\Alg_1^{\opp},(\DuAlg_2)^{\opp}}{\overline{\Pos}}_i.\] Combining
this with the identification $\Opposite$ of $\Alg_1$ and $\DuAlg_2$
with their opposites (an identification that switches the roles of $R_i$ and $L_i$), we arrive at a type $DD$ bimodule, denoted
$\lsup{\Alg_1,\DuAlg_2}\Neg_i$.
Concretely, we reverse all the arrows in Figure~\ref{eq:PositiveCrossing} and replace algebra elements $L_j$ and $R_j$
with $R_j$ and $L_j$ respectively (leaving $U_j$ unmodified).

  Note that this is related to the $DD$ bimodule associated to a crossing
  $\lsup{\DuAlg,\Alg}\Pos$ from~\cite[Section~\ref{BK2:subsec:DDcross}]{Bordered2}
  by the relation
  \[ \lsup{\Alg}T_{\cBlg}\DT \lsup{\cBlg,\nDuAlg}\Pos
  = \lsup{\nDuAlg_1}q_{\DuAlg} \DT \lsup{\DuAlg,\Alg}\Pos.\]

\subsection{Curved $DA$ bimodules over $\Blg$ and those over $\nDuAlg$}

We have the following analogue of Proposition~\ref{prop:DualityRelationship}:

\begin{prop}
  \label{prop:BimodulesOvernDuAlg}
  There is the following relationship between the curved bimodules
  associated to a minimum $\lsup{\cBlg_2}\Min^c_{\cBlg_1}$,
  $\lsup{\cBlg_2}\Pos^i_{\cBlg_1}$, and $\lsup{\cBlg_2}\Neg^i_{\cBlg_1}$
  (from Proposition~\ref{prop:CurvedDABimodules})
  and the corresponding $DA$ bimodules over $\nDuAlg$ (from Lemma~\ref{lem:BimodulesOvernDuAlg}):
  \begin{align*} 
    \lsup{\cBlg_2}\Min^c_{\cBlg_1}\DT \lsup{\cBlg_1,\nDuAlg_1}\CanonDD &\simeq
    \lsup{\nDuAlg_1}\Min^c_{\nDuAlg_2}\DT \lsup{\nDuAlg_2,\cBlg_2}\CanonDD \\
    \lsup{\cBlg_2}\Pos^i_{\cBlg_1}\DT \lsup{\cBlg_1,\nDuAlg_1}\CanonDD &\simeq
    \lsup{\nDuAlg_1}\Pos^i_{\nDuAlg_2}\DT \lsup{\nDuAlg_2,\cBlg_2}\CanonDD \\
    \lsup{\cBlg_2}\Neg^i_{\cBlg_1}\DT \lsup{\cBlg_1,\nDuAlg_2}\CanonDD &\simeq
    \lsup{\nDuAlg_1}\Neg^i_{\nDuAlg_2}\DT \lsup{\nDuAlg_2,\cBlg_2}\CanonDD 
  \end{align*}
\end{prop}

\begin{proof}
  It is straightforward to see that
  $\lsup{\cBlg_1}\Min^c_{\cBlg_2}\DT \lsup{\cBlg_2,\nDuAlg_2}\CanonDD
  =\lsup{\cBlg_1,\nDuAlg_2}\Min_c$, where the latter is the type $DD$ bimodule 
  from Subsection~\ref{subsec:DDmin}.
  The verification:
  $\lsup{\nDuAlg_1}\Min^c_{\nDuAlg_2}\DT \lsup{\nDuAlg_2,\cBlg_2}\CanonDD
  =\lsup{\cBlg_1,\nDuAlg_2}\Min_c$
  follows similarly.

  The identification
  $\lsup{\cBlg_2,\nDuAlg}\Pos_i=~\lsup{\cBlg_2}\Pos^i_{\cBlg_1}\DT
  \lsup{\cBlg_1,\nDuAlg_1}\CanonDD$
  is similarly straightforward.
  A homotopy equivalence
  $\lsup{\nDuAlg_1}\Pos^i_{\nDuAlg_2}\DT \lsup{\nDuAlg_2,\cBlg_2}\CanonDD\simeq 
    \lsup{\nDuAlg_1,\cBlg_2}\Pos_i$
  is given exactly as in the proof
  of~\cite[Lemma~\ref{BK2:lem:DuAlgDualCross}]{Bordered2}
  (where it is shown that
  $\lsup{\DuAlg_1}\Pos^i_{\DuAlg_2}\DT \lsup{\DuAlg_2,\Alg_2}\CanonDD\simeq 
  \lsup{\DuAlg,\Alg}\Pos_i$).

  The relation with a negative crossing follows formally, since $\lsup{\cBlg,\nDuAlg}\CanonDD$
  is preserved by the symmetry induced by taking the algebra to its opposite.
\end{proof}

\subsection{Restricting idempotents}
Consider $\Blg(n)$, and let
\begin{equation}
  \label{eq:DefCidemp}
  \Cidemp(n)=\sum_{\{\x\big|\x\cap\{0,2n\}=\emptyset\}}\Idemp{\x}.
\end{equation}
Recall that $\Clg(n)$ is the subalgebra of $\Blg(n)$ determined by
$\Clg(n)=\Cidemp(n)\cdot\Blg(n)\cdot\Cidemp(n)$.  We can think of the
inclusion map $i\colon \Clg(n)\to \Blg(n)$ as a curved type $DA$ bimodule
$\lsup{\cBlg(n)}i_{\cClg(n)}$. As usual, we drop $n$ (in the notation
for $\cClg$, $\cBlg$, and $\Cidemp$) when it
clear from the context.

\begin{lemma}
  \label{lem:RestrictIdempotents}
  There are $DA$ bimodules
  $\lsup{\cClg_2}\Pos^i_{\cClg_1}$, 
  $\lsup{\cClg_2}\Neg_{\cClg_1}$, 
  $\lsup{\cClg_2}\Min_{\cClg_1}$
  related to the above bimodules by the relations
  \begin{align*} 
  \lsup{\cBlg_2}i_{\cClg_2} \DT~ \lsup{\cClg_2}\Pos_{\cClg_1} &\simeq~
    \lsup{\cBlg_2}\Pos_{\cBlg_1}\DT~ \lsup{\cBlg_1}i_{\cClg_1} \\
  \lsup{\cBlg_2}i_{\cClg_2} \DT~ \lsup{\cClg_2}\Neg_{\cClg_1} &\simeq~
    \lsup{\cBlg_2}\Neg_{\cBlg_1}\DT~ \lsup{\cBlg_1}i_{\cClg_1} \\
  \lsup{\cBlg_2}i_{\cClg_2} \DT~ \lsup{\cClg_2}\Min_{\cClg_1}  &\simeq~ 
    \lsup{\cBlg_2}\Min_{\cBlg_1}\DT~ \lsup{\cBlg_1}i_{\cClg_1} 
\end{align*}
\end{lemma}

\begin{proof}
  It suffices to notice that all three bimodules
  $\lsup{\cBlg_2}X_{\cBlg_1}=\lsup{\cBlg_2}\Min_{\cBlg_1}$, 
  $\lsup{\cBlg_2}\Pos^i_{\cBlg_1}$, and $\lsup{\cBlg_2}\Neg^i_{\cBlg_1}$,
  and have the property that if
  $X\cdot \Cidemp=\Cidemp\cdot X$, with $\Cidemp$ as in Equation~\eqref{eq:DefIota}.
  (Note that $\Cidemp$
  depends on the underlying algebra; in particular for $\Min$, the left and
  right ``$\Cidemp$'' differ.) This in turn it follows from properties of
  these three objects as (ordinary) bimodules over the idempotent
  algebras;
  c.f.~\cite[Proposition~\ref{BK2:prop:RestrictIdempotent}]{Bordered2}.
\end{proof}

\section{Computing the $DD$ bimodules of standard middle diagrams}
\label{sec:ComputeDDmods}

Given an extended middle diagram $\HmidExt$, we can think of the bimodule
\[\lsup{\cBlgout}\DAmodExt(\HmidExt)_{\cBlgin}\DT
~\lsup{\cBlgin,\nDuAlgin}\CanonDD\]
as the type $DD$ bimodule 
associated to $\HmidExt$.  
Our aim here is to compute these bimodules
for some standard middle diagrams, expressed in terms of the
algebraically defined $DD$ bimodules from Section~\ref{sec:AlgDA}.

\begin{thm}
  \label{thm:ComputeDDmods}
  For each number $n$ of strands equipped with some matching
  $\MatchIn$, there are extended middle Heegaard diagrams for local
  minima (provided that $\{c,c+1\}\not\in\MatchIn$), positive, and
  negative crossings, denoted $\Hmin{c}$, $\Hpos{i}$ $\Hneg{i}$; and
  whose associated $DA$ bimodules
  $\lsup{\cBlg_2}\DAmodExt(\Hmin{c})_{\cBlg_1}$,
  $\lsup{\cBlg_2}\DAmodExt(\Hpos{i})_{\cBlg_1}$,
  $\lsup{\cBlg_2}\DAmodExt(\Hneg{i})_{\cBlg_1}$
  are related to the
  bimodules from Lemma~\ref{lem:BimodulesOvernDuAlg} by:
  \begin{align}
    \lsup{\nDuAlg_1}\Min^c_{\nDuAlg_2} \DT~ \lsup{\nDuAlg_2,\cBlg_2}\CanonDD &\simeq
    \lsup{\cBlg_2}\DAmodExt(\Hmin{c})_{\cBlg_1} \DT~ \lsup{\cBlg_1,\nDuAlg_1}\CanonDD \label{eq:ComputeMinDD} \\
    \lsup{\nDuAlg_1}\Pos^i_{\nDuAlg_2} \DT~ \lsup{\nDuAlg_2,\cBlg_2}\CanonDD &\simeq
    \lsup{\cBlg_2}\DAmodExt(\Hpos{i})_{\cBlg_1} \DT~ \lsup{\cBlg_1,\nDuAlg_1}\CanonDD \label{eq:ComputePosDD}\\
    \lsup{\nDuAlg_2}\Neg_{\nDuAlg_1} \DT~ \lsup{\nDuAlg_1,\cBlg_1}\CanonDD &\simeq
    \lsup{\cBlg_2}\DAmodExt(\Hneg{i})_{\cBlg_1} \DT~ \lsup{\cBlg_1,\nDuAlg_1}\CanonDD. \label{eq:ComputeNegDD}
  \end{align}
\end{thm}

We split the verification of Theorem~\ref{thm:ComputeDDmods} into
parts (Subsections~\ref{subsec:LocalMin}, \ref{subsec:Pos}, and
\ref{subsec:Neg}), starting first with a warm-up
(Subsection~\ref{subsec:CanonDD}).

\subsection{The canonical $DD$ bimodule}
\label{subsec:CanonDD}

Although this is not logically needed for the subsequent computations,
we compute the bimodule associated to the middle diagram for the identity cobordism, as follows:

\begin{prop}
  \label{prop:ComputeIdDD}
  For the (extended middle) Heegaard diagram $\Hid$ for the identity cobordism
  from Figure~\ref{fig:IdDiag},
  stabilized as in Section~\ref{sec:ExtendDA}, we have that
  \[ \lsup{\cBlg}\DAmod(\Hid)_{\cBlg}\DT \lsup{\cBlg,\nDuAlg}\CanonDD
  \simeq \lsup{\cBlg,\nDuAlg}\CanonDD, \]
\end{prop}

\begin{proof}
  The stabilized identity diagram Figure~\ref{fig:IdDiag} is redrawn in
  Figure~\ref{fig:IdDiag2}.

\begin{figure}[h]
 \centering
 \input{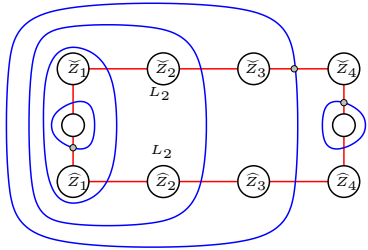}
 \caption{{\bf Extended middle diagram of the identity, again.}}
 \label{fig:IdDiag2}
 \end{figure}

 We claim that for the diagram, 
 \[\delta^1_2(\One,L_i)=L_i\otimes\One \qquad
 \delta^1_2(\One,R_i)=R_i\otimes \One,\qquad
 \delta^1_2(\One,U_i)=U_i\otimes \One.\]
The holomorphic disks counting the first two kinds of actions
are polygons (in fact, in cases where $1<i<2n$, they are rectangles).
For example, the rectangles counting the operations
\[
  \delta^1_2(\Idemp{\{2,3,4\}},L_2)=L_2\otimes \Idemp{\{1,3,4\}} \qquad
  \delta^1_2(\Idemp{\{1,3,4\}},R_2)=R_2\otimes \Idemp{\{2,3,4\}} 
  \]
are shown in the top line of Figure~\ref{fig:IdDomains}.

The holomorphic disk representing the action
$\delta^1_2(\One,U_i)=U_i\otimes \One$ is a (twice punctured) annulus,
which is the region bounded by $\beta_{i-1}$ and $\beta_{i}$ (labelled
so that $\beta_0$ and $\beta_{2n}$ encircle the two ``middle''
boundary components $\Zmid_0$ and $\Zmid_1$); but the precise
combinatorics are dictated by the idempotent type of the input. When
only one of $i-1$ or $i$ is present in the in-coming generator, there
is a cut from the generator to $\Zin_i$, and another cut along
$\alphaout_{i-1}\cup\alphaout_{i}$, which might or might not cut through
$\Zout_i$. In both cases, this annulus always has an odd number of
pseudo-holomorphic representatives that can be realized as branched
double-covers of the disk, and the output algebra element is always
$U_i$.

When both $i-1$ and $i$ are present in the in-coming generator, there
are two cuts going in towards $\Zin_i$; one goes exactly until
$\Zin_i$ and the other not quite. The total number of such ends is odd.
To see this, we consider the one-dimensional moduli space 
with the given homotopy class, and which contains the Reeb orbit.
This has one end, which is a boundary degeneration. The other ends of this moduli space occur when a cut goes exactly out till $\Zin_i$; thus there is an odd number of such ends.

It follows that the generators of $\lsup{\cBlg}\DA(\Hid)_{\cBlg}\DT
\lsup{\cBlg,\nDuAlg}\CanonDD$ correspond to complementary idempotents,
and the terms $L_i\otimes R_i$, $R_i\otimes L_i$, and $U_i\otimes E_i$
appear with non-zero coefficient in the differential. 

The $DD$ bimodule inherits a $\Delta$-grading, defined by  by
\begin{equation}
  \label{eq:DeltaGradenAlg}
  \Delta ((a\otimes b)\otimes \x)=\# \text{($E$ in $b$)}-\weight(a)-\weight(b). 
\end{equation}
By the Alexander grading, 
the coefficients $(a\otimes b)\x$ in $\delta^1(\y)$
satisfy the relation that $\weight(a)=\weight(b)$. 
Thus, $L_i\otimes R_i$, $R_i\otimes L_i$, and $U_i\otimes E_i$ are all the elements
with grading $-1$, and all other algebra elements have
$\leq -2$, so they cannot appear in the differential.
It follows that 
\[ \lsup{\cBlg}\DA(\Hid)_{\cBlg}\DT
\lsup{\cBlg,\nDuAlg}\CanonDD=\lsup{\cBlg,\nDuAlg}\CanonDD.\]
\end{proof}

\begin{figure}[h]
 \centering
 \input{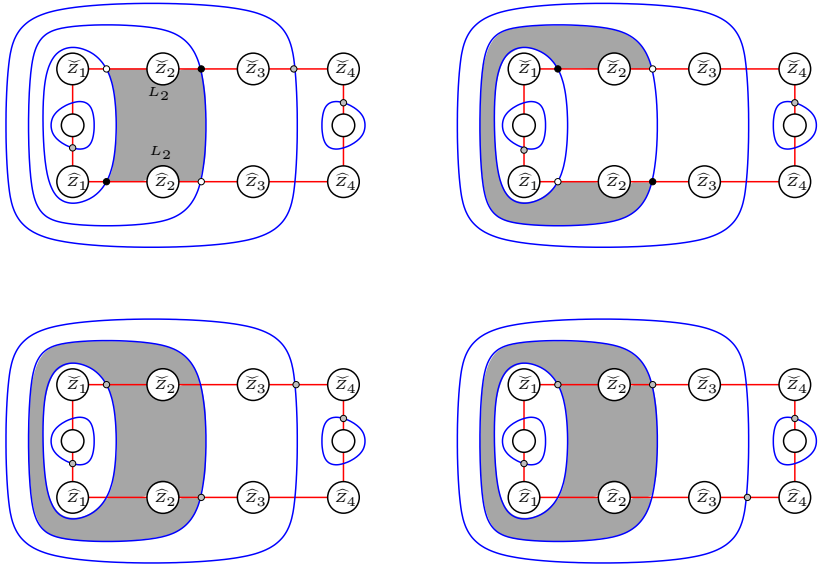}
 \caption{{\bf Some holomorphic disks.} 
   Redrawing the picture from Figure~\ref{fig:IdDiag}.}
 \label{fig:IdDomains}
 \end{figure}

\subsection{A local minimum}
\label{subsec:LocalMin}

 Take the standard middle diagram for the local minimum occuring between
 strands $c$ and $c+1$, and consider its stabilized middle diagram $\Hmin{c}$,
 as pictured in Figure~\ref{fig:MinHDiag}. Equip the strands with a
 matching $\Matching$ with $\{c,c+1\}\not\in\Matching$ (i.e., so that
 $\Matching$ is compatible with $\Hmid$.)

\begin{figure}[h]
 \centering
 \input{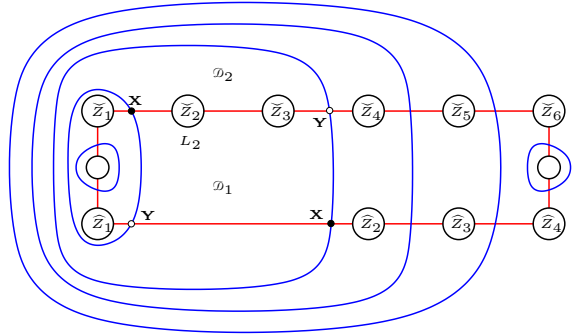}
 \caption{{\bf Heegaard diagram for a minimum, stabilized.} Compare Figure~\ref{fig:MinimumHeeg}.  We have labelled two
   domains.}
 \label{fig:MinHDiag}
 \end{figure}

We now prove the following:
\begin{prop}
  \label{prop:MinimumComputation}
  For each $n$, there is a Heegaard diagram for
  which Equation~\eqref{eq:ComputeMinDD} holds.
\end{prop}

\begin{proof}
  It is clear that the Heegaard states are in one-to-one
  correspondence with allowed idempotents for $\y$; the
  correspondence is simply given by $\x\mapsto \alphain(\x)$.

  We verify 
  actions that connect different terms of the following form:
  \[    \begin{tikzpicture}[scale=1.5]
      \node at (-1.5,0) (X) {$\XX$} ;
      \node at (1.5,0) (Y) {$\YY$} ;
      \node at (0,-2) (Z) {$\ZZ$} ;
      \draw[->] (X) [bend right=7] to node[below,sloped] {\tiny{$1\otimes (R_{c}, R_{c+1})$}}  (Y)  ;
      \draw[->] (Y) [bend right=7] to node[above,sloped] {\tiny{$1\otimes (L_{c+1}, L_{c})$}}  (X)  ;
      \draw[->] (Y) [bend right=7] to node[above,sloped] {\tiny{$L_{c-1}\otimes L_{c-1}$}}  (Z)  ;
      \draw[->] (Z) [bend right=7] to node[below,sloped] {\tiny{$R_{c-1}\otimes R_{c-1}$}}  (Y)  ;
      \draw[->] (Z) [bend right=7] to node[above,sloped] {\tiny{$R_c\otimes R_{c+2}$}}  (X)  ;
      \draw[->] (X) [bend right=7] to node[below,sloped] {\tiny{$L_c\otimes L_{c+2}$}}  (Z)  ;
      \draw[->] (X) [loop above] to node[above,sloped]{\tiny{$1\otimes (R_{c},U_{c+1},L_{c})+ U_t\otimes U_c$}} (X);
      \draw[->] (Y) [loop above] to node[above,sloped]{\tiny{$1\otimes (L_{c+1},U_{c},R_{c+1}) + q \cdot U_s\otimes U_{c+1}$}} (Y);
      \draw[->] (Z) [loop below] to node[below,sloped]{\tiny{$1\otimes (R_{c},U_{c+1},L_{c})+U_t\otimes U_c$}} (Z);
    \end{tikzpicture}
    \]
    (Note, though, that there are other actions as well.)
    Here, $q=0$ or $1$ (though we shall see afterwards that it must be $1$).
    We also prove that $\delta^1_2(\XX,U_{c+1})=0=\delta^1_2(\YY,U_c)=\delta^1_2(\ZZ,U_{c+1})$.

    To set up the computation, we introduce some notation.
    Let $\{\phi_c(s),c\}\in
    \Matching_1$ and $\{c+1,\phi_c(t)\}\in \Matching_2$, and orient the
    strand through the minimum as indicated in Figure~\ref{fig:OrientMin};
    i.e. $c$ corresponds to an even orbit, $c+1$ to an odd orbit,
    and 
    $\tau(c)=\tau(c+1)=t$. (See Figure~\ref{fig:OrientMin}.)

  The six arrows that connect different generator types are represented by
  polygons. For example, the verification $\XX$ appears with non-zero
  multiplicity in $\delta^1_3(\YY,L_{c+1},L_c)$ can be seen by considering
  the quadrialteral (denoted ${\mathcal D}_1$ in
  Figure~\ref{fig:MinHDiag}), and noting that it represents a unique
  rigid holomorphic flow-line. Traversing the $\alpha$-curves, 
  $L_{c+1}$ occurs before $L_c$.

  \begin{figure}[h]
    \centering
    \input{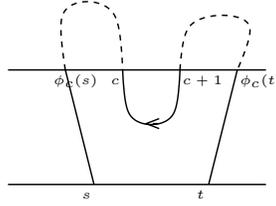}
    \caption{{\bf Orienting a local minimum.}}
    \label{fig:OrientMin}
  \end{figure}

  Next we consider other terms from $\XX$ to $\XX$.  Consider
  holomorphic curves with shadow the annulus, ${\mathcal D}_1 +
  {\mathcal D}_2$.  If we cut exactly to $\Zin_c$, there is a
  holomorphic disk containing $e_{c+1}$ in its interior; if we cut
  further, exactly to $\Zin_{c+1}$, the holomorphic disk is
  interpreted as a $\delta^1_4$. Thus, this homotopy class gives two
  terms
  \[ \begin{CD}
    \XX@>{U_{t}\otimes U_{c}}>{{\mathcal D}_1+{\mathcal D}_2}> \XX
  \end{CD}\qquad{\text{and}}\qquad \begin{CD}
    \XX@>{1\otimes (R_{c},U_{c+1},L_{c})}>{{\mathcal D}_1+{\mathcal D}_2}> \XX.
    \end{CD}
  \]
  Observe that there are no other holomorphic representatives for this
  homology class: the order of the algebra elements $R_c$, $U_{c+1}$, and
  $L_c$ is pre-determined by the order they appear on the
  $\alpha$-arc.
  Moreover, $\delta^1_2(\XX,U_{c+1})=0$.

  Consider next $\YY$. In this case, if we cut exactly to
  $\Zin_{c+1}$, there is a holomorphic disk containing $e_c$ in its
  interior. But this homotopy class alone does not contribute to the
  differential, since $e_c$ is an ``even'' Reeb orbit.

  For this to count in an action for the DA bimodule, there must also be some $v_{s}$ occuring
  at the same time. When $s$ is distinct from $c-1$ and $c+2$, 
  it is easy to find a corresponding action
  \[ \begin{CD}
    \YY@>{U_{s}\otimes U_{c+1}}>{{\mathcal D}_1+{\mathcal D}_2 + {\mathcal D}_s}> \YY
  \end{CD}, \]
  for a choice of ${\mathcal D}_s$ which is an annular domain containing $\Zin_{\phi_c(s)}$ and $\Zout_s$ in its interior.

  Cutting further, we obtain the action 
  \[ 
  \begin{CD}
    \YY@>{1\otimes (L_{c+1}, U_c, R_{c+1})}>{{\mathcal D}_1+{\mathcal D}_2}> \YY,
  \end{CD}
  \]
  (which is obvious for any choice of $s$).
  Moreover, $\delta^1_2(\YY,U_{c})=0$.

  Tensor with the identity type $DD$ bimodule, we obtain terms
  \[\begin{CD}
    \XX @>{U_{t}\otimes E_{c} + 1\otimes (U_c E_{c+1})}>> \XX \\
    \YY @>{q \cdot (U_{s}\otimes E_{c}) + 1\otimes (E_{c} U_{c+1})}>>\YY.
  \end{CD}\]
  So far, we have verified that $q=1$ only under hypotheses on $s$; but
  $q=1$ is forced from algebraic considerations, and the existing  other terms, as follows.
  We refer to the structural equation for a $DD$ bimodule simply as $\delta\circ\delta=0$:
  this includes both terms that multiply terms in $\delta^1$ with other such terms,
  and terms that differentiate terms in $\delta^1$.
  The term in $\delta\circ \delta$ arising from anti-commuting the terms
  $\YY\mapsto (U_{s}\otimes E_{\phi_c(s)})\otimes \YY$ and the above
  $\YY\mapsto((1\otimes (E_{c} U_{c+1}))\otimes\YY$ gives a term
  $\YY\to (U_{s}\otimes U_{c+1})\otimes \YY$. Note that $(1\otimes U_{c+1})\otimes \YY\neq 0$,
  so we need a term in $\delta\circ\delta$ to cancel this term $(U_{s}\otimes U_{c+1})\otimes \YY$. In fact, 
  the only
  possible term that can cancel $U_{s}\otimes U_{c+1}$ is the differential of $U_{s}\otimes E_{c+1}$.

  Starting at $\ZZ$, the space of almost-complex structures has a
  chamber structure: pseudo-holomorphic flows correspond to annuli
  obtained by cutting the annulus $A=\cald_1\cup\cald_2$ in along
  $\alphain_{c-1}$ (from the left) and $\alphain_{c+1}$ (from the
  right). To determine which $\Ainfty$ operations these induce
  involves understanding whether the cut from the left reaches the
  boundary punctures before the cuts on the right.

  \begin{figure}[h]
    \centering
    \input{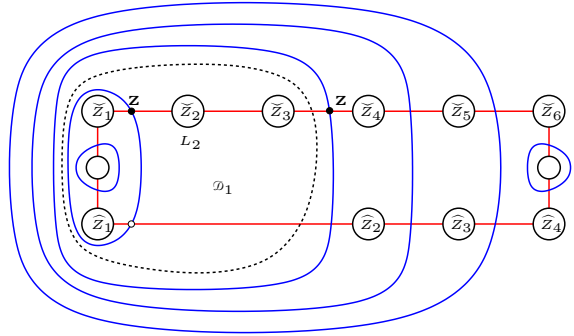}
    \caption{{\bf Stretch along the dotted curve.}}
    \label{fig:MinimumChamber}
  \end{figure}

  Choose a complex structure where the length of the curve
  $(\alphain_{c-1}\cup\alphain_{c})\cap (\cald_1\cup\cald_2)$ is much
  shorter than $(\alphain_{c+1}\cup\alphain_{c})\cap
  (\cald_1\cup\cald_2)$.   (This choice
  is equivalent to stretching sufficiently much normal to the dotted
  curve in Figure~\ref{fig:MinimumChamber}.)
  Thus, for the pseudo-holomorphic flowlines,
  the cuts from the left reach all the punctures before the cuts from
  the right.   In this
  case, the annulus supports exactly two non-trivial operations
  \[\begin{CD}
    \ZZ@>{U_{t}\otimes U_{c}}>{{\mathcal D}_1+{\mathcal D}_2}> \ZZ
  \end{CD}\qquad{\text{and}}\qquad \begin{CD}
    \ZZ@>{1\otimes (R_{c},U_{c+1},L_{c})}>{{\mathcal D}_1+{\mathcal D}_2}> \ZZ
  \end{CD}
  \]
  (which looks like the operations from $\XX$ to itself.)
  Thus, 
  as in the case of generators of type $\XX$, we obtain terms 
  in the $DD$ bimodule of the form.
  \[     \begin{CD}
    \ZZ @>{U_{t}\otimes E_{c} + 1\otimes (U_c E_{c+1})}>> \ZZ 
    \end{CD}
  \]

  Having verified some of the terms in the type $DD$ bimodule 
  $\lsup{\Blg_2}\DAmod(\Hmin{c})_{\Blg_1}\DT \lsup{\Blg_1,\nDuAlg_1}\CanonDD$,
  algebra also forces some additional terms
  \[\begin{CD}
    \XX @>{1\otimes L_{c} L_{c+1}E_{c} E_{c+1}}>> \YY \\
    \YY @>{1\otimes R_{c+1}R_c E_{c} E_{c+1}}>> \XX.
  \end{CD}\] (The first are needed cancel terms in $\delta\circ
  \delta$ arising from composing the terms from $\XX$ to $\YY$
  labelled by $1\otimes L_{c} L_{c+1}$, with terms from $\XX$ to $\XX$
  labelled $1\otimes U_c E_{c+1}$ or (on the other side) terms from
  $\YY$ to $\YY$ labelled $1\otimes E_{c} U_{c+1}$. The second terms
  follow similarly.)

  Maslov index considerations  allow the following
  additional types of terms in the differential: $L_i\otimes R_i E_c
  E_{c+1}$, $R_i\otimes L_i E_c E_{c+1}$, and $U_i\otimes E_i E_c
  E_{c+1}$. But these would contribute terms to $\delta\circ\delta$
  that cannot be cancelled. Consider the case where the generator is
  of type $\XX$.  Then, $d(U_i\otimes E_i E_c E_{c+1})$ contains the
  non-zero term $U_i \otimes E_i E_c U_{c+1}$, which can be factored
  as $(U_i\otimes E_i)\cdot (1\otimes E_{c} U_{c+1})$, but $1\otimes
  E_c U_{c+1}$ does not connect two generators of type $\XX$. The same
  argument for type $\YY$ idempotents, using the other non-zero term
  $U_i\otimes E_i U_c E_{c+1}$, excludes the existence of terms of the
  form $U_i\otimes E_i E_c U_{c+1}$. Similar considerations apply to
  the other types of terms listed above.

  A $DD$ isomomorphism from
  \[\lsup{\Blg_2}\DAmod(\Hmin{c})_{\Blg_1}\DT \lsup{\Blg_1,\nDuAlg_1}\CanonDD \to \lsup{\cBlg_2,\nDuAlg_1}\Min \]
  is now defined  by 
  \[ h^1(\XX)=(1+E_c E_{c+1})\otimes \XX, \qquad h^1(\YY)=\YY,\qquad h^1(\ZZ)= (1+E_c E_{c+1})\otimes \ZZ.\]
\end{proof}
  
\subsection{A positive crossing}
\label{subsec:Pos}

Consider a positive crossing between the $i^{th}$ and $(i+1)^{st}$ strands.

A  Heegaard diagram for this crossing  (stabilized)  is
shown in Figure~\ref{fig:PosCrossDiag}. Our aim here is to verify
Theorem~\ref{thm:ComputeDDmods} for this diagram.

There are two
cases, according to whether or not $i$ and $i+1$ are matched. We
consider these two cases separately.

\begin{figure}[h]
 \centering
 \input{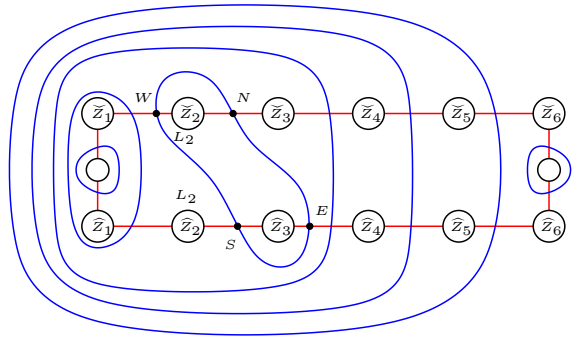}
 \caption{{\bf Positive crossing diagram.} 
   Stabilize Figure~\ref{fig:CrossDiag}.}
 \label{fig:PosCrossDiag}
 \end{figure}

\begin{figure}
    \begin{tikzpicture}[scale=1.85]
    \node at (0,4) (N) {${\mathbf N}$} ;
    \node at (-2,2.5) (W) {${\mathbf W}$} ;
    \node at (2,2.5) (E) {${\mathbf E}$} ;
    \node at (0,-1) (S) {${\mathbf S}$} ;
    \draw[->] (S) [bend left=10] to node[below,sloped] {\tiny{$R_{i} U_{i+1} \otimes E_{i} E_{i+1}+L_{i+1}\otimes R_{i+1}R_{i}
+ R_i U_\beta\otimes 1$}}  (W)  ;
    \draw[->] (W) [bend left=10] to node[above,sloped] {\tiny{$L_{i}\otimes 1$}}  (S)  ;
    \draw[->] (E) [bend right=10] to node[above,sloped] {\tiny{$R_{i+1}\otimes 1$}}  (S)  ;
    \draw[->] (S)[bend right=10] to node[below,sloped] {\tiny{$R_{i} \otimes L_{i} L_{i+1} + L_{i+1} U_{i}\otimes E_{i+1} 
E_{i}+L_{i+1}U_\alpha\otimes 1$}} (E) ;
    \draw[->] (W)[bend right=10]to node[below,sloped] {\tiny{$1\otimes L_{i}$}} (N) ;
    \draw[->] (N)[bend right=10] to node[above,sloped] {\tiny{$U_{i+1}\otimes R_{i} + R_{i+1} R_{i} \otimes L_{i+1}$
}} (W) ;
    \draw[->] (E)[bend left=10]to node[below,sloped]{\tiny{$1\otimes R_{i+1}$}} (N) ;
    \draw[->] (N)[bend left=10] to node[above,sloped]{\tiny{$U_{i}\otimes L_{i+1} + L_{i} L_{i+1}\otimes R_{i}$}} 
(E) ;
    \draw[->] (N) [loop above] to node[above]{\tiny{$U_{i}\otimes E_{i+1} + U_{i+1}\otimes E_{i}$}} (N);
    \draw[->] (W) [loop left] to node[above,sloped]{\tiny{$U_{i+1}\otimes E_{i}$}} (W);
    \draw[->] (E) [loop right] to node[above,sloped]{\tiny{$U_{i}\otimes E_{i+1}$}} (E);
    \draw[->] (E) [bend right=5] to node[above,pos=.3] {\tiny{$R_{i+1} R_{i} \otimes E_{i+1}$}} (W) ;
    \draw[->] (W) [bend right=5] to node[below,pos=.3] {\tiny{$L_{i} L_{i+1}\otimes E_{i}$}} (E) ;
    \draw[->] (S) to node[below,sloped,pos=.3] {\tiny{$L_{i+1}\otimes E_{i} R_{i+1} + R_{i}\otimes L_{i} E_{i+1}$}} (N) ;
    \end{tikzpicture}
    \caption{\label{fig:PosDD} {\bf $DD$ bimodule of a positive crossing.}}
    \end{figure}
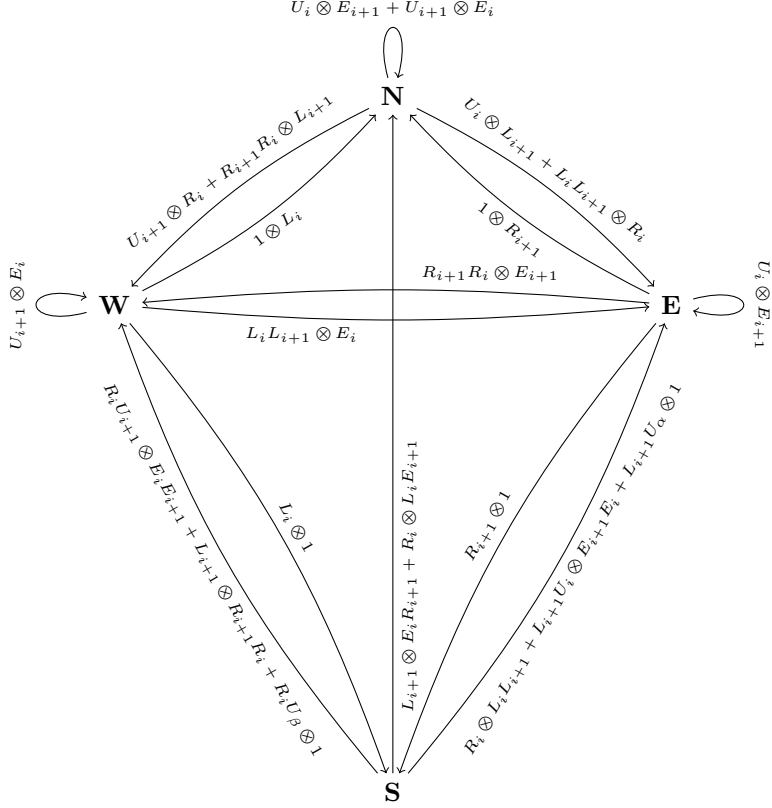

\begin{lemma}
  \label{lem:PosDDchar}
  Suppose that $i$ and $i+1$ are not matched; indeed, $\{\alpha,i\},
  \{i+1,\beta\}\in M_1$. 
  There is a unique type $DD$ bimodule over $\lsup{\Blg_2,\nDuAlg_1}M$ 
  with the following properties:
  \begin{itemize}
    \item The generators of $M$ over the idempotent ring
      are of the four types $\North$, $\South$, $\East$, and $\West$
      as in the bimodule $\lsup{\cBlg_2,\DuAlg}\Pos_i$
      described in Section~\ref{subsec:AlgPosCross}
    \item The bimodule $M$ has a $\Delta$ grading is as specified in Equation~\eqref{eq:DeltaGradingPos}
    \item The bimodule $M$ has an Alexander grading is as specified in 
      Equation~\eqref{eq:AlgGradeCrossing}.
    \item 
      The $\North$ coefficient of $\delta^1_1(\North)$ is $U_i\otimes
      E_{i+1}+U_{i+1}\otimes E_i + \sum_{j\not\in\{i,i+1\}} U_j\otimes
      E_j$.
    \item 
      The $\West$ coefficient of $\delta^1_1(\West)$ is 
      $U_{i+1}\otimes E_i+ \sum_{j\not\in\{i,i+1\}} U_j\otimes E_j$.
    \item 
      The $\East$ coefficient of $\delta^1_1(\East)$ is 
      $U_{i}\otimes E_{i+1}+ \sum_{j\not\in\{i,i+1\}} U_j\otimes E_j$.
    \item 
      The $\South$ coefficient of $\delta^1_1(\South)$ is 
      $\sum_{j\not\in\{i,i+1\}} U_j\otimes E_j$.
    \item 
      The $\East$ coeficient of $\delta^1(\West)$ is $R_{i+1}R_i\otimes E_{i+1}$.
    \item 
      The $\West$ coeficient of $\delta^1(\West)$ is $L_i L_{i+1}\otimes E_i$.
  \end{itemize}
  Moreover, there is an isomorphism
  $\lsup{\cBlg_2,\nDuAlg_1}M\cong~\lsup{\cBlg_2,\nDuAlg_1}\Pos_i$,
  where the right-hand-side bimodule is the one described in
  Section~\ref{subsec:AlgPosCross}.
\end{lemma}

\begin{proof}
  Consider the $\North$ coefficient of $\delta^1\circ \delta^1(\North)$
  There is a term of $U_{i+1}\otimes U_i$ coming from
  differentiating the self-arrows. 
  Gradings now ensure the term can arise
  only from a term of $(1\otimes L_i)\otimes \North$ in $\delta^1(\West)$
  and $(U_{i+1}\otimes R_i)\otimes \West$ in $\delta^1(\North)$.

  The coefficient of $R_i \otimes L_i E_{i+1}$ on the arrow from $\South$ to
  $\North$ is forced from the $\North$
  component of $\delta^1\circ \delta^1(\West)$.

  The term of $R_i U_{i+1}\otimes E_{i+1} E_i$ on the arrow from
  $\South$ to $\West$ is forced from the $\West$ coefficient of
  $\delta^1 \circ \delta^1(\West)$. (Note that degree considerations
  alone allow also for terms $R_i U_{i+1}\otimes E_{i} E_{i+1}$ and
  $R_i U_{i+1}\otimes 1$; but the $DD$ structure relations and our
  hypotheses about the existing coefficients eliminates these
  possibilities.)

  Symmetric arguments give $L_{i+1}\otimes R_{i+1} E_i$ on the arrow
  from $\South$ to north, $L_{i+1} U_i\otimes E_i E_{i+1}$ on the
  arrow from $\South$ to $\East$.

  Now, considering the $\South$ coefficient fo $\delta^1\circ \delta^1(\North)$
  ensures the terms $R_{i+1} R_i \otimes L_{i+1}$ on the arrow from
  $\North$ to $\West$ and $L_i L_{i+1}\otimes R_i$ on the arrow from $\North$
  to $\East$.

  The curvature term in $\delta^1\circ \delta^1(\West)$ (of $U_i
  U_\beta\otimes 1$) ensures the arrow from $\South$ to $\East$
  labelled by $R_i U_\beta$.   

  This constructs all the actions in $\lsup{\cBlg_2,\nDuAlg}M$.
  The isomorphism is provided by the map 
  \[ h^1\colon M \to \Pos_i \]
  defined by
    \[ h^1(X) = \left\{\begin{array}{ll}
        \South+ (L_2\otimes E_1)\cdot \East + (R_1\otimes E_2)\cdot 
        \West & {\text{if $X=\South$}} \\
        X &{\text{otherwise.}}
      \end{array}
    \right.\]
\end{proof}

\begin{prop}
  \label{prop:PosDDUnmatched}
  If $i$ and $i+1$ are unmatched, then Equation~\eqref{eq:ComputePosDD}
  holds.
\end{prop}

\begin{proof}
  It is straightforward to see that Lemma~\ref{lem:PosDDchar}
  computes $\lsup{\nDuAlg_1}\Pos^i_{\nDuAlg_2} \DT~
  \lsup{\nDuAlg_2,\cBlg_2}\CanonDD$.

  It remains to verify that there is a complex structure for which
  $\lsup{\cBlg_2}\DAmodExt(\Hpos{i})_{\cBlg_1} \DT~ \lsup{\cBlg_1,\nDuAlg_1}\CanonDD$
  is also computed by Lemma~\ref{lem:PosDDchar}.

  We consider the actions on
  $\lsup{\cBlg_2}\DAmodExt(\Hpos{i})_{\cBlg_1}$ which could pair to give the actions
  required by the lemma.

  \begin{figure}[h]
 \centering
 \input{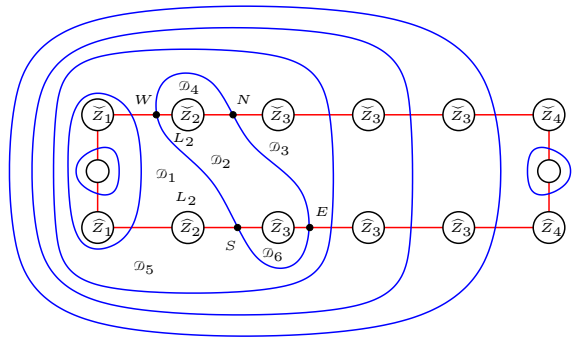}
 \caption{{\bf Labeling domains in the positive crossing diagram.} }
 \label{fig:PosCrossDoms}
 \end{figure}

  Consider labels shown in Figure~\ref{fig:PosCrossDoms}. We have polygons that exhibit
  actions 

  For $X\in\{\North,\West\}$
  \[\begin{CD}
  X@>{U_{i+1}\otimes U_i}>{\cald_2+\cald_4+\cald_6}> X 
  \end{CD}.\]
  For $X'\in\{\South,\East\}$, there is no action 
  \[\begin{CD}
  Y@>{U_{i+1}\otimes U_i}>{\cald_2+\cald_4+\cald_6}> Y 
  \end{CD},\]
  as can be seen from the geometry of that bigon.

  The following action is given by a polygon, and hence
  it exists for all choices of almost-complex structure:
  \[  \begin{CD}
    \West@>{L_i L_{i+1}\otimes U_i}>{\cald_1+\cald_2+\cald_4}>\East
    \end{CD}\]

  Consider possible actions
  \begin{equation}
    \label{eq:ExcludeThis}
    \begin{CD}
    Y'@>{U_{i}\otimes U_{i+1}}>{\cald_1+\cald_3+\cald_5}> Y'.
  \end{CD}
  \end{equation}
  for $Y'\in\{\South,\West\}$.
  We can choose a complex structure
  so that neither exists, but 
  \begin{equation}
    \label{eq:IncludeThis}
    \begin{CD}
    \East@>{R_{i+1} R_i\otimes E_{i+1}}>{\cald_1+\cald_3+\cald_5}> \West
  \end{CD}
  \end{equation}
  exists. This can be seen as follows.

  \begin{figure}[h]
 \centering
 \input{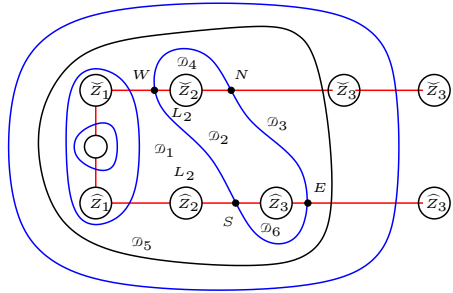}
 \caption{{\bf Stretch normal to the black curve.} }
 \label{fig:StretchPosDoms}
 \end{figure}
 First, we stretch normal to the black curve from
 Figure~\ref{fig:StretchPosDoms}. In the limit, this curve degenerates
 to the point $p$; this is indicated in Figure~\ref{fig:DegenPosDoms}
 (once we identify the two curves labelled $p$ to points).

  \begin{figure}[h]
 \centering
 \input{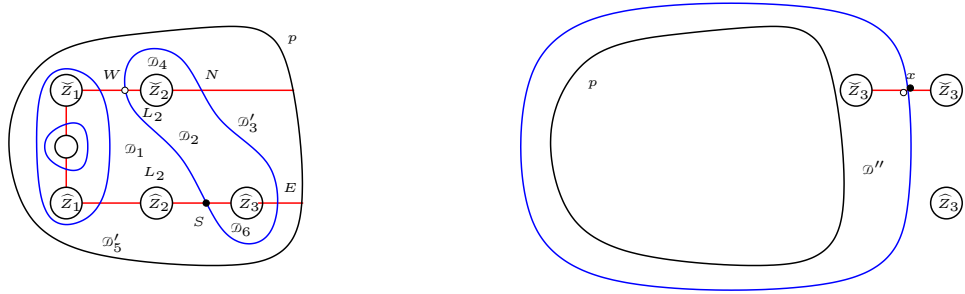}
 \caption{{\bf Degenerating the black curve from
     Figure~\ref{fig:StretchPosDoms}.} }
 \label{fig:DegenPosDoms}
 \end{figure}

 The domain $\cald_3+\cald_5$ can be thought of as the shadow of a
 pseudo-holomorphic flow from $\South$ to $\West$ with Maslov index
 $0$.  As the domain is degenerated, the domain decomposes into two
 regions as shown in Figure~\ref{fig:DegenPosDoms}. One is
 $\cald_3'+\cald_5'$, which is now an index one flow from $\South$ to
 $\West$ on the left; and the remaining region which we denote here
 ${\mathcal D}''$, thought of as connecting a generator to itself, and
 following the Reeb chord covering $\Zin_{i+1}$ once. Both flows
 contain Reeb orbits at $p$ in their interior.
 Let $s_1=s\circ \ev_p(\ModFlow^{\cald_3'+\cald_5'}(\South,\West))$ and 
 $s_2=s\circ\ev_p(\ModFlow^{\cald''}(x,x;\rho))$,
 where $\rho$ is the length one Reeb orbit that covers $\Zin_{i+1}$ once.
 We will choose our complex structure so that $s_1<s_2$.
 
 Consider the maps
 \begin{equation}
   \label{eq:SegmentOne}
   s\circ \ev_p\colon \ModFlow^{\cald_1+\cald_3'+\cald_5'}(Y',Y')\to
 [0,1] \qquad{\text{and}}\qquad s\circ \ev_p\colon
 \ModFlow^{\cald_3'+\cald_5'+\cald_6}(\East,\West)\to [0,1] 
\end{equation} We
claim that these two moduli spaces map degree one to $[0,s_1]$ and
$[s_1,1]$ respectively.  For example, consider the first of these maps 
$Y'=\South$.  The moduli space
$\ModFlow^{\cald_1+\cald_3'+\cald_5'}(\South,\South)$ has two ends:
one is a $\beta$-boundary degeneration, where $s\circ \ev_p$
converges to $0$; and the other is a broken flowline, juxtaposing the
flow with shadow $\cald_1$ from $\South$ to $\West$, followed by the
flowline with shadow $\cald_3'+\cald_5'$ from $\West$ to $\South$. 
Here, $s\circ \ev_p$ converges to $s_1$. 
When $Y'=\West$, the order of the two parts of the broken flowline are reversed.
It follows that the first map from Equation~\eqref{eq:SegmentOne}
has degeree one onto $[0,s_1]$. The analysis of the second map 
from Equation~\eqref{eq:SegmentOne} follows similarly (except in that
case, rather than a boundary degeneration, the
flow degenerates to an index one flowline containing 
a Reeb chord along its $\alpha$-boundary).

 By the usual stretching arguments, the degree of the first map at
 $s_2$ counts the number of representatives of
 Equation~\eqref{eq:ExcludeThis} (for either choice of $Y'$), while
 the degree of the second map counts the number of representatives of
 Equation~\eqref{eq:IncludeThis}. Thus, for our choice of $s_1<s_2$,
 the actions from Equation~\eqref{eq:ExcludeThis} are excluded
 and the one from Equation~\eqref{eq:IncludeThis} is included.

 We can now apply Lemma~\ref{lem:PosDDchar} to
 $\lsup{\cBlg_2}\DAmodExt(\Hpos{i})_{\cBlg_1} \DT~
 \lsup{\cBlg_1,\nDuAlg_1}\CanonDD$
\end{proof}

\begin{prop}
  \label{prop:PosDDMatched}
  Equation~\eqref{eq:ComputePosDD} holds when $i$ and $i+1$ are matched.
\end{prop}

\begin{proof}
  First note that 
  if $i$ and $i+1$ are matched, then the $DD$ bimodule of a positive crossing exhibited in
  Figure~\ref{fig:PosDD}, with $\alpha=i+1$ and $\beta=i$, is uniquely characterized by
  the properties listed in the statement of Lemma~\ref{lem:PosDDchar}. The proof follows exactly
  as in the proof of Lemma~\ref{lem:PosDDchar}. In the present application, note that
  we expect a curvature term of $U_i U_{i+1}\otimes 1$ in $\delta\circ \delta$.

  With this said, the proof of Proposition~\ref{prop:PosDDUnmatched} applies.
\end{proof}

\subsection{A negative crossing}
\label{subsec:Neg}

\begin{prop}
  \label{prop:ComputeNegDD}
  For the extended middle diagram for a negative crossing, 
  Equation~\eqref{eq:ComputeNegDD} holds.
\end{prop}

\begin{proof}
Observe that there is a symmetry (which is reflection through a
vertical axis in Figure~\ref{fig:PosCrossDiag}) taking the positive
crossing diagram to the negative crossing diagram (shown in
Figure~\ref{fig:NegCrossDiag}) This switches $\West$ and $\East$, $i$
and $2n-i$, and reverses the orientation of the Heegaard diagram,

\begin{figure}[h]
 \centering
 \input{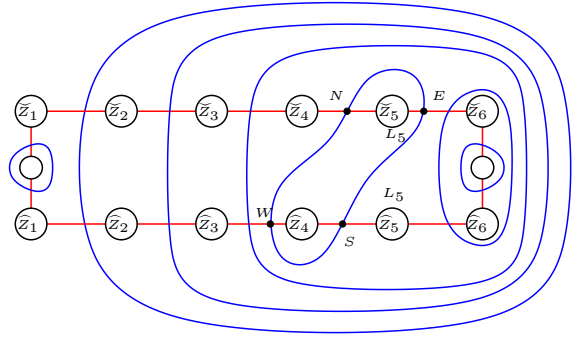}
 \caption{{\bf The (extended) negative crossing diagram.} }
 \label{fig:NegCrossDiag}
 \end{figure}

Thus, the differential in the type $DD$ bimodule associated to
$\Hneg{i}$ is computed by taking the differential the type $DD$
bimodule $\Hpos{2n-i}$, reversing the arrows, switching $\West$ and
$\East$, and switching strand $j$ with strand $2n-j$. This 
operation also transforms $\lsup{\cBlg_2,\nDuAlg_1}\Pos_i$
to the $DD$ bimodule for $\Neg_{2n-i}$. 

With these remarks, the proposition follows immediately from
Proposition~\ref{prop:PosDDUnmatched} or~\ref{prop:PosDDMatched}. 
\end{proof}

\begin{proof}[Proof of Theorem~\ref{thm:ComputeDDmods}]
  The proof follows immediately
    from Propositions~\ref{prop:MinimumComputation}, \ref{prop:PosDDUnmatched}, \ref{prop:PosDDMatched}.
    and~\ref{prop:ComputeNegDD}.
\end{proof}

\section{Computing the type $D$ structures of an upper diagram}
\label{sec:computeD}

Suppose that $\DiagUp$ is an upper knot diagram (thought of as a
diagram in an upper half space). After Reidemeister moves and
isotopies, we can arrange for all the local maxima of $\DiagUp$ to be
global maxima, and for each of the local minima and crossings to occur
at different heights.  We can construct an upper Heegaard diagram for
each upper slice of an {\em acceptable} diagram, as follows. We start from
the standard upper diagram corresponding to all the local maxima, and
then glue on the standard middle diagrams for the remaining crossings
and local mimuma. In this manner we associate the {\em canonical upper
  diagram} $\Hup(\DiagUp)$ associated to the acceptable diagram
$\DiagUp$.

Given an acceptable upper knot diagram $\DiagUp$ with $n$ strands at
the bottom and matching $\Matching$, there is an associated {\em
  algebraically defined type $D$ structure}
$\lsup{\cBlg}\DmodAlg(\DiagUp)$, where $\cBlg=\cBlg(n,\Matching)$,
obtained by starting with $\lsup{\cClg(n,\Matching)}k$ for the global
maxima, and then tensoring with the bimodules
$\lsup{\cClg_2}\Pos^i_{\cClg_1}$, $\lsup{\cClg_2}\Neg^i_{\cClg_1}$,
$\lsup{\cClg_2}\Min^c_{\cClg_1}$ as dictated by the local slices of
$\DiagUp$. (This is very closely related to the algebraic description
appearing in~\cite{Bordered2}; see the proof of
Theorem~\ref{thm:MainTheorem} below.)

\begin{thm}
  \label{thm:ComputeD}
  Let $\DiagUp$ be any acceptable knot diagram $\DiagUp$, and
  let $\Hup$ be its associated  upper Heegaard diagram.
  Then there is an identification
  \begin{equation}
    \label{eq:ComputeD}
    \lsup{\cClg(n,\Matching)}\Dmod(\Hup)\simeq~\lsup{\cClg(n,\Matching)}\DmodAlg(\DiagUp).
  \end{equation}
\end{thm}

We will prove the above theorem at the end of this section.  The
hypothesis that $\DiagUp$ is acceptable is removed in
Section~\ref{sec:Further} (see especially
Theorem~\ref{thm:ComputeD2}); though it is not logically necessary for
the computation of knot Floer homology (Theorem~\ref{thm:MainTheorem},
proved in Section~\ref{sec:Comparison}).

\begin{defn}
  \label{def:Relevant}
  For $\Clg=\Clg(n,\Matching)$,
  a type $D$ structure $\lsup{\cClg}X$ is called {\em relevant} if
  there is an $\Ainfty$ module $Z_{\nDuAlg}$ over $\nDuAlg=\nDuAlg(n,\Matching)$
  with the property that
  \[
    \lsup{\cBlg}i_{\cClg}\DT~ \lsup{\cClg}X=Z_{\nDuAlg}\DT~\lsup{\nDuAlg,\cBlg}\CanonDD.\]
\end{defn}

Consider the type $D$ structure $\lsup{\cClg}k$ from Section~\ref{subsec:ExampleTypeDs} (which is identified
in Lemma~\ref{lem:StandardTypeD} with the
type $D$ structure of the standard upper diagram).
We shall see that $\lsup{\cClg}k$ is relevant, in the sense of
Definition~\ref{def:Relevant}. Since $\lsup{\cClg}k$ is a one-dimensional vector space, so is $Z_{\nDuAlg}$. The 
$\Ainfty$-actions on this module are fairly
complicated. Letting if $\x$ be the generator,
we have that
\[ \x\cdot E_{2i-1} \cdot E_{2i} + \x \cdot E_{2i}\cdot E_{2i-1} =\x;\]
i.e. exactly one of the two following equations holds:
\[ \x\cdot E_{2i-1}\cdot E_{2i}=\x\qquad {\text{or}}\qquad
\x\cdot E_{2i}\cdot E_{2i-1}=\x.\]
Suppose that the first equation holds. Then, applying the $\Ainfty$
relation with $\x$, $E_{2i-1}$ and $E_{2i}$, we can conclude that
exactly one of 
\[
m_3(\x,E_{2i-1}, U_{2i})=\x\qquad 
{\text{or}}\qquad
m_3(\x,U_{2i-1}, E_{2i})=\x\]
holds. The $\Ainfty$ relations then imply many further actions. We
organize this construction in the following:

\begin{lemma}
  \label{lem:kIsRelevant}
  The type $D$ structure $\lsup{\cClg}k$ is relevant.
\end{lemma}

\begin{proof}
  Let $\x_1$ denote the idempotent state for $\nDuAlg$ $\x_1$ that
  consists of all odd numbers between $1$ and $2n-1$, together with
  $0$.  Let $\BigMod$ be the right $\Blg(2n,n+1)$-module consisting of $2^{n}$
  copies of $\Idemp{\x_1}\cdot (L_1 L_2\backslash \Blg(2n,n+1))$, indexed by
  subsets $S$ of $\{1,\dots,n\}$; i.e.
  \[ \BigMod = \bigoplus_{S\subset\{1,\dots,n\}} \BigMod_S,\]
  and there are preferred identifications of $\BigMod_S\cong \BigMod_T$ for all
  $S, T\subset \{1,\dots,n\}$.

  We endow $\BigMod$ with a differential, as follows.
  Given $i=1,\dots,n$, if $i\not\in S$, let 
  \[\partial_i \colon \BigMod_S \to \BigMod_{S\cup\{i\}} \]
  be multiplication by $U_{2i-1}$
  (composed with the preferred identification $\BigMod_S\cong \BigMod_{S\cup\{i\}}$) and
  \[\partial_i \colon \BigMod_{S\cup\{i\}} \to \BigMod_{S}; \]
  be multiplication by $U_{2i}$.
  Note that $\partial = \sum_{i=1}^n \partial_i \colon \BigMod \to \BigMod$
  satisfies $\partial\circ\partial =0$. (This follows from the fact
  that $\partial^2$ is multiplication by $z=\sum_{i=1}^{n} U_{2i-1}U_{2i}$,
  and $\Idemp{\x_1}\cdot z = \Idemp{\x_1}\cdot U_1 U_2 = \Idemp{\x_1}
  L_1 L_2 \cdot R_2 R_1$.

  We can extend the action on $\BigMod$ by $\Blg(2n,n+1)$ to all of
  $\nDuAlg$, as follows.
  For fixed $S$ with $i\not\in S$, let
  \[m_2(\cdot,E_{2i-1})\colon \BigMod_{S} \to
  \BigMod_{S\cup\{i\}}\qquad\text{and}\qquad
  m_2(\cdot,E_{2i})\colon \BigMod_{S \cup \{i\}} \to \BigMod_{S}
  \]
  be the zero map, while
  \[m_2(\cdot,E_{2i})\colon \BigMod_{S} \to
  \BigMod_{S\cup\{i\}}\qquad\text{and}\qquad
  m_2(\cdot,E_{2i-1})\colon \BigMod_{S \cup \{i\}} \to \BigMod_{S}
  \]
  are the standard identifications.  On each $\BigMod_{S}$ multiplication by
  exactly one of $E_{2i-1}\cdot E_{2i}$ or $E_{2i}\cdot E_{2i-1}$ is
  non-zero, and the non-zero one is the preferred identification. It
  follows that the action by $\DuAlg$ descends to an action by
  $\nDuAlg$. Moreover, if $\{j,k\}$ are unmatched, then $E_{j}\cdot
  E_{k}$ acts the same as $E_{k} E_{j}$. It follows that there is an
  induced action of $\nDuAlg$ on $\BigMod$.

  We also claim that the homology of $\BigMod$ is one-dimensional, generated
  by
  $P_{\{\}} \left(\prod_{i=1}^{n-1} R_{2i}\right)$. We see this as
  follows.
  Let $Q_j\subset \BigMod$ be the vector subspace generated by elements of
  the form $b \cdot P_S$ with the property that
  \[ b= \prod_{i=1}^{j} R_{2i}\cdot c \] with $w_i(c)=0$ for
  $i=1,\dots,2j-1$ and $\{1,\dots,j\}\cap S=\emptyset$.  In
  particular, $Q_0=\BigMod\supset Q_1\supset\dots\supset Q_{n}$, and $Q_{n}$
  is the one-dimensional vector space spanned by $P_{\{\}}\cdot
  (\prod_{i=1}^{n-1}R_{2i})$.  Observe that $Q_j$ is a subcomplex.
  Define
  \[ H(a\cdot P_S) =\left\{\begin{array}{ll}
      b\cdot (E_{2i-1} \cdot P_S)  & \text{if $a=U_{2i-1} \cdot b$
        and $w_j(b)<1$ for all $j<2i-1$} \\
      b\cdot (E_{2i} \cdot P_S)  & \text{if $a=U_{2i} \cdot b$
      and $w_j(b)<1$ for all $j<2i$} 
    \end{array}\right.\]
  We claim that $\Id + \partial\circ H + H\circ \partial$ maps $Q_j$
  to $Q_{j+1}$.  Thus, iterating the chain map $F=\Id + \partial\circ
  H + H\circ \partial$ $n$ times, we obtain a chain homotopy contraction of
  $\BigMod$ onto the one-dimensional subcomplex spanned by $P_{\{\}}\cdot
  (\prod_{i=1}^{n-1} R_{2i})$, as claimed.
  
  The homological perturbation lemma now endows the one-dimensional
  vector space $Z_{\nDuAlg}=H(\BigMod,\partial)$, with an $\Ainfty$ action
  by $\nDuAlg$.  
  We claim that $Z_{\nDuAlg}\DT
  \lsup{\nDuAlg,\cBlg}\CanonDD$ has trivial differential. 
  This follows from the $\Delta$-gradings: the algebra output in
  $\nDuAlg$ of each element is either $0$ (if it is some $E_i$) 
  $-1/2$ (if it is some $L_i$ or $R_i$). But such sequences of algebra
  elements act trivially on $Z_{\nDuAlg}$, for if
  \[ \Delta(x)=m_{\ell+1}(x,a_1,\dots,a_{\ell})=\Delta(x) + \ell-1 +
  \Delta(a_1)+\dots+\Delta(a_\ell) \geq \Delta(x)+ \frac{\ell}{2}-1
  \] so $\ell=1$. The case where $\ell=1$ is excluded since
  $m_2(\x,E_i)=0$. The case where $\ell=2$ is also excluded by
  direct consideration.
  Moreover,
  \[
  \Idemp{\{0,2,\dots,2i,\dots,2n\}}\cdot \left(Z_{\nDuAlg}\DT\lsup{\nDuAlg,\cBlg}\CanonDD\right)
  =Z_{\nDuAlg}\DT\lsup{\nDuAlg,\cBlg}\CanonDD. \]
  The lemma follows.
\end{proof}

\begin{prop}
  \label{prop:InductiveStep}
  Let $\cClg_1=\cClg(n,\Matching_1)$
  Suppose  $\lsup{\Clg_1}X$ is relevant, then
  \begin{align}
    \lsup{\cClg_2}\Min^c_{\cClg_1} \DT~ \lsup{\cClg_1}X &\simeq
    \lsup{\cClg_2}\DAmodExt(\Hmin{c})_{\cClg_1} \DT~ \lsup{\cClg_1}X \label{eq:MinRel}\\
    \lsup{\cClg_2}\Pos^i_{\cClg_1} \DT~ \lsup{\cClg_1}X &\simeq
    \lsup{\cClg_2}\DAmodExt(\Hpos{i})_{\cClg_1} \DT~ \lsup{\cClg_1}X  \label{eq:PosRel}\\
    \lsup{\cClg_2}\Neg_{\cClg_1} \DT~ \lsup{\cClg_1}X &\simeq
    \lsup{\cClg_2}\DAmodExt(\Hneg{i})_{\cClg_1} \DT~ \lsup{\cClg_1}X; \label{eq:NegRel}
    \end{align}
  and moreover all the type $D$ structures appearing on the left
  are relevant.
\end{prop}

\begin{proof}
  Consider Equation~\eqref{eq:MinRel}.
  Combining  Lemma~\ref{lem:RestrictIdempotents}, 
  the relevance of $\lsup{\Blg_1}X$, and the associativity of tensor product,
  we have that
  \begin{align*}
    \lsup{\cBlg_2}i_{\cClg_2}\DT(\lsup{\cClg_2}\Min^c_{\cClg_1}\DT
    \lsup{\cClg_1}X) &\simeq
    \lsup{\cBlg_2}\Min^c_{\cBlg_1}\DT\lsup{\cBlg_2}i_{\cClg_2}\DT
    \lsup{\cClg_1}X \\
    &\simeq
    \lsup{\cBlg_2}\Min^c_{\cBlg_1}
    \DT
    (Z_{\nDuAlg_1}\DT\lsup{\nDuAlg_1,\cBlg_1}\CanonDD)  \\
    &\simeq
    Z_{\nDuAlg_1}
    \DT
    (\lsup{\cBlg_2}\Min^c_{\cBlg_1}\DT\lsup{\cBlg_1,\nDuAlg_1}\CanonDD).
    \end{align*}
  By Proposition~\ref{prop:BimodulesOvernDuAlg}, 
  \[ 
    Z_{\nDuAlg_1}\DT
    (\lsup{\cBlg_2}\Min^c_{\cBlg_1}\DT\lsup{\nDuAlg_1,\cBlg_1}\CanonDD)  
    \simeq
    (Z_{\nDuAlg_1}\DT 
        \lsup{\nDuAlg_1}\Min^c_{\nDuAlg_2})\DT\lsup{\nDuAlg_2,\cBlg_2}\CanonDD;
    \]
    so
    $\lsup{\cClg_2}\Min^c_{\cClg_1}\DT\lsup{\cClg_1}X$ is relevant.

    Applying  Theorem~\ref{thm:ComputeDDmods} and Proposition~\ref{prop:ExtendDAPrecise}, 
    \begin{align*}
    Z_{\nDuAlg_1}\DT
    (\lsup{\cBlg_2}\Min^c_{\cBlg_1}\DT\lsup{\nDuAlg_1,\cBlg_1}\CanonDD)  
    &\simeq
    Z_{\nDuAlg_1}\DT
    (\lsup{\cBlg_2}\DAmodExt(\Hmin{c})_{\cBlg_1}
    \DT \lsup{\cBlg_1,\nDuAlg_1}\CanonDD) \\
    &\simeq
    \lsup{\cBlg_2}\DAmodExt(\Hmin{c})_{\cBlg_1}\DT
    (Z_{\nDuAlg_1}
    \DT \lsup{\nDuAlg_1,\cBlg_1}\CanonDD) \\
    &\simeq \lsup{\cBlg_2}\DAmod(\Hmin{c})_{\cBlg_1}\DT~ \lsup{\cBlg_1}i_{\cClg_1} \DT ~\lsup{\cClg_1}X \\
    &\simeq \lsup{\cBlg_1}i_{\cClg_1} \DT \lsup{\cClg_2}\DAmod(\Hmin{c})_{\cClg_1}\DT ~ \lsup{\cClg_1}X.
    \end{align*}
    It follows that
    \begin{equation}
      \label{eq:MinRel1}
      \lsup{\cBlg_2}i_{\cClg_2}\DT~ \lsup{\cClg_2}\Min^c_{\cClg_1}\DT~ \lsup{\cClg_1}X 
    \simeq \lsup{\cBlg_2}i_{\cClg_2}\DT~
    \lsup{\cClg_2}\DAmodExt(\Hmin{c})_{\cClg_1} \DT~ \lsup{\cClg_1}X.
    \end{equation}
    
    Equation~\eqref{eq:MinRel} follows from Equation~\eqref{eq:MinRel1} and
    the following observation: if $\lsup{\cClg_2}P$ and $\lsup{\cClg}Q$ are any two type $D$ structures, then
    \[ \Mor^{\cClg_2}(\lsup{\cClg_2}P,\lsup{\cClg_2}Q)=
    \Mor^{\cBlg_2}(\lsup{\cBlg_2}i_{\cClg_2}\DT\lsup{\cClg_2}P,\lsup{\cBlg_2}i_{\cClg_2}\DT\lsup{\cClg}Q),\]
    since 
    \[ \lsub{\cClg}i^{\cBlg}\DT\lsub{\cBlg}\Blg_{\cBlg}\DT \lsup{\cBlg}i_{\cClg}=
    {\iota}\cdot \Blg \cdot {\iota}=\cClg\]
    (with $\iota$ as in Equation~\eqref{eq:DefIota});
    so $\lsup{\cClg_2}P\simeq \lsup{\cClg_2}Q$ iff 
    $\lsup{\cBlg_2}i_{\cClg_2}\DT\lsup{\cClg_2}P\simeq 
    \lsup{\cBlg_2}i_{\cClg_2}\DT\lsup{\cClg_2}Q$.
    
    Equations~\eqref{eq:PosRel} and~\eqref{eq:NegRel} follows similarly.
\end{proof}

\begin{proof}[Proof of Theorem~\ref{thm:ComputeD}]
  We prove Theorem~\ref{thm:ComputeD}, together with the statement
  that $\lsup{\Clg}\Dmod(\Hup)$ is homotopy equivalent to a relevant
  type $D$ structure, by induction on the sum of the number of
  crossings and local minima. When this sum is zero, $\Hup$ is the
  standard upper diagram, and the theorem follows from
  Lemma~\ref{lem:StandardTypeD}. Moreover, this module is relevant by
  Lemma~\ref{lem:kIsRelevant}.

  The inductive step is now provided by
  Proposition~\ref{prop:InductiveStep} and
  Theorem~\ref{thm:PairDAwithD}.
\end{proof}

\begin{remark}
  With a little extra work, one can show that the type $D$ structure
  $\Dmod(\Hup)$ is independent of the choice of upper Heegaard diagram $\Hup$.
\end{remark}

\newcommand\nMinGen[1]{{\mathbf S}_{#1}}
\section{Comparing the knot invariants}
\label{sec:Comparison}

An {\em acceptable knot diagram} $\Diag$ is a diagram for
an oriented knot whose local maxima are
at the same level, all of whose other events occur at different
heights, and whose global minimum is the marked edge; thus, an
acceptable knot diagram is obtained from an acceptable upper knot
diagram in the sense of Section~\ref{sec:computeD} by adding a single
global minimum.

The methods of this paper now give the following explicit computation of knot Floer homology:

\begin{thm}
  \label{thm:ComputeHFK}
  Let $\Diag$ be an acceptable knot diagram, obtained by adding a
  global minimum to $\DiagUp$, and let $\Hup$ be upper diagram
  associated to $\DiagUp$.  Then, there is a homotopy equivalence
  \[\lsup{\Ring}\CFKsimp(\orK)\simeq\lsup{\Field[U,V]}[\Psi]_{\cClg(1)}\DT~\lsup{\cClg(1)}\DmodAlg(\Hup),\]
  where 
  $\Psi\colon \Clg(1)\to \Field[U,V]$ is the isomorphism from Equation~\eqref{eq:IsoClg1}.
\end{thm}

\begin{proof}
  To compute $\lsup{\Ring}\CFKsimp(\orK)$, 
  we use
  By the pairing theorem (Theorem~\ref{thm:PairAwithD})
  \[ \lsup{\Ring}\CFKsimp(\orK)=
  \lsup{\Ring}\Amod(\Hdown)_{\Clg(1)}\DT \lsup{\Clg(1)}\Dmod(\Hup).\]
  The result now follows from Lemma~\ref{lem:GlobalMinimum} with 
  Theorem~\ref{thm:ComputeD}.
\end{proof}

\begin{rem}
  The application of Theorem~\ref{thm:PairAwithD}
  in the above proof in fact uses the
  special case where $n=1$, where we glue to the standard lower diagram.
  This is a particularly simple special case of the pairing theorem;
  the more interesting applications of the pairing theorem
  are contained in the use of Theorem~\ref{thm:ComputeD}.
\end{rem}

The above description of knot Floer homology is not exactly the
construction formulated in~\cite{Bordered2}, but it is close enough
that the proof of Theorem~\ref{thm:MainTheorem} follows quickly:

\begin{proof}[Proof of Theorem~\ref{thm:MainTheorem}]
  In~\cite{Bordered2}, we defined $\Hwz(\orK)$ as the homology of a
  complex $\Cwz(\Diag)$ associated to a diagram.  The complex
  $\Cwz(\Diag)$ is constructed similarly to $\DmodAlg(\DiagUp)$
  described above: the diagram $\Diag$ is broken into pieces, and to
  each elementary piece we associate bimodules, and $\Cwz(\Diag)$ is
  obtained by tensoring together these pieces. Specifically, to an upper diagram
  $\DiagUp$, tensoring together local bimodules
  defines an algebraically defined type $D$ structure
  $\lsup{\Alg(n,\Matching)}\PartInv(\DiagUp)$, where $2n$ is the
  number of strands out of $\DiagUp$ and $\Matching$ is the matching
  induced by $\DiagUp$; and if $\Diag$ is obtained by adding a global
  minimum to $\DiagUp$, then
  $\Cwz(\Diag)=\lsup{\Field[U,V]}\TerMin_{\Alg(1)}\DT\lsup{\Alg(1)}\PartInv(\DiagUp)$,
  where $\lsup{\Field[U,V]}\TerMin_{\Alg(1)}$ is the a bimodule
  described in~\cite[Section~\ref{BK2:subsec:ConstructInvariant}]{Bordered2}.

  We claim that
  \begin{equation}
    \label{eq:IdentifyPartialInvariants}
    \lsup{\Alg}T_{\cBlg}\DT~ \lsup{\cBlg}i_{\cClg}\DT~ \lsup{\cClg}\DmodAlg(\DiagUp)\simeq
  \lsup{\Alg}\PartInv(\DiagUp),
  \end{equation}
  where $\Alg=\Alg(n,\Matching)$, $\Blg=\Blg(n,\Matching)$,
  $\Clg=\Clg(n,\Matching)$ for $\Matching$ as determined by $\DiagUp$.
  This is seen by induction on the number of events in $\DiagUp$. For
  the basic case, where $\DiagUp$ contains only maxima, recall
  from~\cite{Bordered2} that $\lsup{\Alg}\PartInv(\DiagUp)$ is
  generated by a single element $\x$, and $\delta^1(\x)=C\otimes \x$,
  where $C=\sum_{i=1}^{n} C_{\{2i-1,2i\}}$; i.e. there is an
  identification of $\lsup{\Alg}\PartInv(\DiagUp)$ with
  \[ \lsup{\Alg}T_{\cBlg}\DT~ \lsup{\cBlg}i_{\cClg}\DT~ \lsup{\cClg}k = 
  \lsup{\Alg}T_{\cBlg}\DT~ \lsup{\cBlg}i_{\cClg}\DT~ \lsup{\cClg}\DmodAlg(\DiagUp).\]
  For the inductive step, when we add the bimodule associated to one
  more standard piece, we use Proposition~\ref{prop:CurvedDABimodules}
  and Lemma~\ref{lem:RestrictIdempotents}.

  Next, we consider $\lsup{\Ring}\TerMin_{\Alg(1)}$, used in the definition of $\Cwz(\Diag)$.
  The bimodule $\TerMin$  has three generators $\XX$,
  $\YY$, and $\ZZ$, with
  \[ \XX \cdot \Idemp{\{0\}}=\XX, \qquad 
  \YY \cdot \Idemp{\{1\}}=\YY, \qquad 
  \ZZ \cdot \Idemp{\{2\}}=\ZZ.\]
  When $1$ is oriented upwards, $\TerMin$ is the $DA$ bimodule with $\delta^1_k=0$ for $k\neq 2$, and
  all $\delta^1_2$ are determined by
  \[ 
  \begin{array}{lll}
    \delta^1_2(\YY,L_1)= u\otimes \XX, & \delta^1_2(\XX,R_1)= u\otimes \YY,  \\
    \delta^1_2(\YY,R_2)= v\otimes \ZZ, & \delta^1_2(\ZZ,L_2)= v\otimes \YY, \\
    \delta^1_2(\XX,C_{\{1,2\}})=
    \delta^1_2(\YY,C_{\{1,2\}})=
    \delta^1_2(\ZZ,C_{\{1,2\}})=0.
  \end{array}
  \]
  (When $1$ is oriented downwards, we define the actions as above, exchanging the roles of $u$ and $v$.)
  It follows immediately from this description that
  \begin{equation}
    \label{eq:DescribeTerMin}
    \lsup{\Ring}\TerMin_{\Alg(1)}\DT~ \lsup{\Alg(1)}T_{\cBlg(1)}\DT~\lsup{\cBlg(1)}i_{\cClg(1)}=\lsup{\Ring}[\Psi]_{\cClg(1)},
    \end{equation}
  where $\Psi$ is as in Equation~\eqref{eq:IsoClg1}.

  Thus, combining the definition of $\Cwz$ with Equations~\eqref{eq:IdentifyPartialInvariants}, 
  \eqref{eq:DescribeTerMin}, and Theorem~\ref{thm:ComputeHFK}, we find that
  \begin{align*} 
    \Cwz(\Diag)&=\lsup{\Ring}\TerMin_{\Alg(1)}\DT~\lsup{\Alg(1)}\PartInv(\Hup) \\
    &= \lsup{\Ring}\TerMin_{\Alg(1)}\DT~ \lsup{\Alg(1)}T_{\cBlg(1)}\DT~ \lsup{\cBlg(1)}i_{\cClg(1)}\DT~ \lsup{\cClg(1)}\DmodAlg(\DiagUp) \\
    &= \lsup{\Ring}[\Psi]_{\cClg(1)}\DT~ \lsup{\cClg(1)}\DmodAlg(\Hup) \\
    &= \lsup{\Ring}\CFKsimp(\orK),
  \end{align*}
    as required.
\end{proof}

\newcommand\BigMin{\Theta''}
\newcommand\nBlg{\Blg^\circ}
\section{Further remarks on invariance}
\label{sec:Further}

This section is a further elaboration on Theorem~\ref{thm:ComputeD},
wherein we explore the invariance properties of the two sides of
Equation~\eqref{eq:ComputeD}. 

Observe that the object appearing on the right hand side of
Equation~\eqref{eq:ComputeD}, $\lsup{\cClg}\DmodAlg(\DiagUp)$, is
associated to an acceptable upper diagram $\DiagUp$.
We explain first how to define $\lsup{\cClg}\DmodAlg(\DiagUp)$ for a
more generic kind of knot diagram.

By analogy with~\cite[Section~\ref{BK1:sec:ConstructionAndInvariance}]{BorderedKnots},
we say that an upper diagram $\DiagUp$ is in {\em bridge position} if
it is drawn on the $y\geq 0$ plane so that:
\begin{enumerate}
\item all the critical points are local 
  minima or maxima
\item minima, maxima, and crossings all have distinct $y$ coordinates
\item $\DiagUp$ has no closed components
\end{enumerate}

To an upper diagram in bridge position, we can associate an
algebraically defined, curved type $D$ structure
$\lsup{\cClg}\DmodAlg(\DiagUp)$ in the obvious way: we tensor together
the curved DA bimodules associated to the crossings local maxima, and
local minima ($\lsup{\cBlg_2}\Pos^i_{\cBlg_1}$,
$\lsup{\cBlg_2}\Neg^i_{\cBlg_1}$, $\lsup{\cBlg_2}\Max^c_{\cBlg_1}$,
and $\lsup{\cBlg_2}\Min^c_{\cBlg_1}$ from
Proposition~\ref{prop:CurvedDABimodules}) as they appear in $\DiagUp$.

Adapting methods from~\cite[Section~\ref{BK1:sec:ConstructionAndInvariance}]{BorderedKnots}
(see also~\cite[Section~\ref{BK2:sec:Construction}]{Bordered2}), we can show that $\lsup{\cClg}\DmodAlg(\DiagUp)$
depends only on the planar isotopy class of the diagram $\Diag$.

\begin{prop}
  \label{prop:CurvedDPlanarIsotopies}
  If $\DiagUp$ and $\DiagUp'$ are two upper diagrams that differ
  by planar isotopies (fixing $y=0$), then 
  the associated curved type $D$ structures $\lsup{\cClg}\DmodAlg(\DiagUp)$
  and $\lsup{\cClg}\DmodAlg(\DiagUp')$ are homotopy equivalent 
  (as curved type $D$ structures).
\end{prop}

The above is proved in Section~\ref{subsec:AlgInvar}, after some
algebraic background is set up in Section~\ref{subsec:LocalMaximum}.

With a little more work, we can show that this curved type
$D$-structure is in fact an invariant of the tangle represented by
$\DiagUp$.

\begin{prop}
  \label{prop:TangleInvariance}
  If $\DiagUp$ and $\DiagUp'$ are two upper diagrams that represent the same tangle in $S^3$, then 
  the associated curved type $D$ structures $\lsup{\cClg}\DmodAlg(\DiagUp)$
  and $\lsup{\cClg}\DmodAlg(\DiagUp')$ are homotopy equivalent.
\end{prop}

To an upper diagram $\DiagUp$ in bridge position, there is an
associated Heegaard diagram $\HD(\DiagUp)$, obtained by stacking the
middle diagrams associated to the crossings and critical points as
they appear in $\DiagUp$.

We have the following variant of Theorem~\ref{thm:ComputeD}:

\begin{thm}
  \label{thm:ComputeD2}
  Let $\DiagUp$ be an upper diagram in bridge position, and let $\Hup$
  be its associated upper Heegaard diagram.  Then there is an
  identification
  \begin{equation}
    \label{eq:ComputeD2}
    \lsup{\cClg(n,\Matching)}\Dmod(\Hup)\simeq~\lsup{\cClg(n,\Matching)}\DmodAlg(\DiagUp).
  \end{equation}
\end{thm}

Combining Theorem~\ref{thm:ComputeD2} with
Proposition~\ref{prop:CurvedDPlanarIsotopies}, we can conclude thet
the homotopy class of the analytically defined
$\lsup{\cClg(n,\Matching)}\Dmod(\Hup)$, which appears to depend on the
choice of Heegaard diagram, in fact is an invariant for tangles.  (One
could alternatively prove invariance
$\lsup{\cClg(n,\Matching)}\Dmod(\Hup)$ using more analytical methods
by showing invariance under isotopies, handleslides, and
stablizations, in the spirit of~\cite{HolDisk}.)

Theorem~\ref{thm:ComputeD2} is proved in Section~\ref{subsec:ComputeD2}

\subsection{Local maxima}
\label{subsec:LocalMaximum}

Consider the curved DA bimodule $\lsup{\cBlg_1}\Max^{c}_{\cBlg_2}$
from Proposition~\ref{prop:CurvedDABimodules};
i.e. where 
\[\cBlg_1=\Blg(n,\Matching_1)
\qquad {\text{and}} \qquad \cBlg_1=\Blg(n,\Matching_2),\] where
$\Matching_1$ is some matching on $\{1,\dots,2n\}$, and $\Matching_2$
is the matching on $\{1,\dots,2n+2\}$ obtained by adding to
$\phi_c(\Matching_1)$, the additional pair $\{c,c+1\}$, with $\phi_c$
as in Equation~\eqref{eq:DefInsert}.
By construction, $\lsup{\cBlg_2}\Max^{c}_{\cBlg_1}\cdot \iota =
\iota \cdot \lsup{\cBlg_1}\Max^{c}_{\cBlg_2}$
(compare Lemma~\ref{lem:RestrictIdempotents}); i.e.
there is a corresponding bimodule
\[  \lsup{\cBlg_2}i_{\cClg_2} \DT~ \lsup{\cClg_2}\Max^c_{\cClg_1} \simeq~
\lsup{\cBlg_2}\Max^c_{\cBlg_1}\DT~ \lsup{\cBlg_1}i_{\cClg_1}  \]

The key algebraic step required to  adapt
the invariance proof from~\cite{Bordered2} is the existence of a bimodule 
$\lsup{\nDuAlg_1}\Max^c_{\nDuAlg_2}$ that is dual to 
$\lsup{\cBlg_2}\Max^c_{\cBlg_1}$, as follows:

\begin{prop}
  \label{prop:DualMax}
  Given $\cBlg_1$ and $\cBlg_2$ as above,
  there is a bimodule
  $\lsup{\nDuAlg_1}\Max^c_{\nDuAlg_2}$ with the property that
  \[ 
  \lsup{\cBlg_2}\Max^c_{\cBlg_1}\DT \lsup{\cBlg_1,\nDuAlg_1}\CanonDD
  \simeq
  \lsup{\nDuAlg_1}\Max^c_{\nDuAlg_2}\DT \lsup{\nDuAlg_2,\cBlg_2}\CanonDD. \]
\end{prop}

The proof occupies the rest of this subsection.

The construction of $\lsup{\nDuAlg_1}\Max^c_{\nDuAlg_2}$ is similar to
the constructions of the bimodules associated to local minimum
from~\cite{Bordered2} and~\cite{BorderedKnots}.  As
in~\cite[Section~\ref{BK2:sec:Min}]{Bordered2}, we start with the
construction of $\Max^c$ in the case where $c=1$.

A {\em preferred idempotent state} for
$\nDuAlg_2$ is an idempotent state $\x$ with
\[ \x\cap \{0,1,2\}\in\{\{0\},\{2\},\{0,2\}\}.\] We have a map $\psi$
from preferred idempotent states $\x=\{x_1,\dots,x_{k+1}\}$ of $\nDuAlg_2$ to idempotents of
$\nDuAlg_1$ defined byLet $\nDuAlg_1=\DuAlg(n+1,k+1,M_1)$ and
$\nDuAlg_2=\DuAlg(n,k,M_2)$
\[
\psi(\x) = \left\{
\begin{array}{ll}
\{0,x_3-2,\dots,x_{k+1}-2\} &{\text{if $|\x\cap\{0,1,2\}|=2$}} \\
\{x_2-2,\dots,x_{k+1}-2\} &{\text{if $|\x\cap\{0,1,2\}|=1$}} \\
\end{array}\right.
\]
Generators of $\Max^1=\lsup{\nDuAlg_1}\Max^1_{\nDuAlg_2}$ correspond to preferred idempotent
states. Letting $\nMinGen{\x}$ denote the generator corresponding to the preferred
idempotent state $\x$, the bimodule action is specified by 
\[ \Idemp{\x}\cdot \nMinGen{\x}\cdot \Idemp{\psi(\x)}=\nMinGen{\x}.\]

The module is constructed via the homological perturbation lemma, 
following~\cite[Section~\ref{BK2:subsec:AltConstr}]{Bordered2}.

Fix a matching $M_2$ on $\{1,\dots,2n\}$, and let $M_1$
be the matching $\{1,2\}\cup \phi_1(M_1)$.
Let
\begin{align*} \nBlg_1=\Blg(2n+2,n+2)\qquad{\text{and}}\qquad
\nBlg_2=\Blg(2n,n+1). 
\end{align*}

Consider the idempotent in $\nBlg_1$ given by
\[ {\mathbf I}=\sum_{\{\x\big|x_1=1\}} \Idemp{\x}.\]
Consider the right $\nBlg$-module
\[ M= (L_1 L_2 \nBlg\backslash {\mathbf I}\cdot\nBlg)\oplus (L_1 L_2
\nBlg\backslash {\mathbf I}\cdot\nBlg),\] which can also be viewed as a left module
over the subalgebra of ${\mathbf I}\cdot \nBlg\cdot {\mathbf I}$ consisting of
elements with $\weight_1=\weight_2=0$, which in turn is identified
with $\nBlg2$. Denote this identification
\[\phi\colon \nBlg_2\to {\mathbf I}\cdot \nBlg\cdot {\mathbf I}.\]

Let $\XX$ and $\YY$ be the generators of the two summands of $M$. 
Let
\[ m_{1|1|0}(b_2,\XX\cdot b_1)=\XX \cdot \phi(b_2)\cdot b_1 \qquad 
m_{1|1|0}(b_2,\YY\cdot b_1)=\YY \cdot \phi(b_2)\cdot b_1,\]
where
$b_2\in\nBlg_2$ and $b_1\in{ {\mathbf I}\cdot \nBlg}\backslash{L_1 L_2 \cdot \nBlg}$.
Equip $M$ with the differential
\[ m_{0|1|0}(\XX)= \YY\cdot U_2 \qquad m_{0|1|0}(\YY)= \XX\cdot U_1.\] 
Think of the right $\nBlg$-action as inducing further operations 
\[ m_{0|1|1}(\XX\cdot b_1,b_1')=\XX\cdot (b_1\cdot b_1') \qquad
m_{0|1|1}(\YY\cdot b_1,b_1')=\YY\cdot (b_1\cdot b_1').\] All the
operations described above give $M$ the structure of a $\nBlg_2-\nBlg$
bimodule, written $\lsub{\nBlg_2}M_{\nBlg}$.

As explained in~\cite[Lemma~\ref{BK2:lem:BigDAGens}]{Bordered2}, $M$
can be viewed as arising from a type $DA$ bimodule
$\lsup{\nBlg_2}\Theta_{\nBlg_1}$ via
$\lsub{\nBlg_2}M_{\nBlg_1}=\lsub{\nBlg_2}(\nBlg_2)_{\nBlg_2}\DT~\lsup{\nBlg_2}\Theta_{\nBlg_1}$. (The
key notational difference here is twe have restricted the incoming
algebra from the algebra $\Blg$ of
\cite[Lemma~\ref{BK2:lem:BigDAGens}]{Bordered2} to its subalgenra
$\nBlg_1$ where the idempotent states have $n+2$ elements; and
correspondingly the output algebra has a similar constraint.)

In~\cite{Bordered2}, this DA bimodule $\Theta$ is promoted to a 
bimodule over versions of $\Alg$. Instead, we promote here the
incoming algebra to $\nDuAlg_1$ (and the outgoing one to $\nDuAlg_2$).
This is done by introducing the  following actions:
\[\begin{array}{rl}
  \delta^1_2(\XX,E_1)=1\otimes \YY &
  \delta^1_2(\YY,E_2)=1\otimes \XX\\
  \delta^1_2(\XX,E_2)=0 &
  \delta^1_2(\YY,E_1)=0\\
  \delta^1_2(\XX,E_{\phi_c(i)})=E_{i}\otimes \XX & 
  \delta^1_2(\YY,E_{\phi_c(i)})=E_{i}\otimes \YY \\
  \delta^1_2(\XX,E_1 E_2)= 1\otimes \XX  &\qquad
  \delta^1_2(\YY,E_2 E_1)= 1\otimes \YY \\
  \delta^1_2(\XX,E_2 E_1)= 0& \qquad
  \delta^1_2(\YY,E_1 E_2)= 0 
  \end{array}
\]

These actions are illustrated in the following picture:
\begin{equation}
  \label{eq:ModuleVersionMin}
    \mathcenter{\begin{tikzpicture}[scale=1.5]
    \node at (0,0) (X) {$\XX$} ;
    \node at (6,0) (Y) {$\YY$} ;
    \draw[->] (X) [bend left=15] to node[above,sloped] {\tiny{$U_2 + 1 \otimes E_{1}$}}  (Y)  ;
    \draw[->] (Y) [bend left=15]to node[below,sloped]  {\tiny{$U_1+ 1\otimes E_{2}$}}  (X);
    \draw[->] (X) [loop] to node[above,sloped]{\tiny{$1\otimes E_{1} E_{2}+\sum_{i=1}^{2n} E_i\otimes E_{\phi_c(i)}$}} (X);
    \draw[->] (Y) [loop] to node[above,sloped]{\tiny{$1\otimes E_{2}E_{1}+\sum_{i=1}^{2n} E_i\otimes E_{\phi_c(i)}$}} (Y);
\end{tikzpicture}} 
\end{equation}
(Here the arrow labels with $U_2$ and $U_1$
alone represent $\delta^1_1$, actions where the outgoing algebra element
is $1$.)

We denote the result by $\lsup{\nDuAlg_2}{\BigMin}_{\nDuAlg_1}$.

\begin{lemma}
  \label{lem:BigDAGens2}
  The operations described above make
  $\lsup{\nDuAlg_2}{\BigMin}_{\nDuAlg_1}$ into a type $DA$ bimodule.
\end{lemma}

\begin{proof}
  This is straightforward.
\end{proof}

Consider the map $h^1\colon \BigMin \to \BigMin$
determined on the generating set $\XX\cdot a$ and $\YY\cdot a$ for $a\in \Gamma$) by 
\begin{align*}
h^1(\XX a)&=\left\{\begin{array}{ll}
      \YY a' &{\text{if there is a  $a'\in \Gamma$ with $a=U_1 a'$}} \\
      0 &{\text{otherwise}}
      \end{array}\right. \\
    h^1(\YY a)&=\left\{\begin{array}{ll}
      \XX a' &{\text{if there is a pure $a'\in \Gamma$ with $a=U_2 a'$}} \\
      0 &{\text{otherwise}}
      \end{array}\right. 
\end{align*}

Let $Q\subset \BigMin$ be the two-dimensional vector space $\XX L_1$
and $\YY R_2$. There are inclusions $i^1\colon Q \to \BigMin\subset \nDuAlg_2\otimes \BigMin$ and a projection $\pi^1\colon \BigMin\to Q$

\begin{lemma}
  For the above operators, we have the identities:
\[  \begin{array}{lllll}
    (\pi^1\circ i^1) =\Id_{Q}, & i^1\circ \pi^1=\Id_{\BigMin}+ dh^1, &
    h^1\circ h^1=0, & h^1\circ i^1=0, & \pi^1\circ h^1 = 0.
    \end{array}\]
\end{lemma}

The homological perturbation lemma, we can gives $\XX L_1 \oplus \YY
R_2$ the structure of a type $DA$ bimodule.

\begin{defn}
  Let $\lsup{\nDuAlg_1}\Max^c_{\nDuAlg_2}$ be the $DA$ bimodule
  structure on $\XX L_1 \oplus \YY R_2$ induced by the homological
  perturbation lemma.
\end{defn}

We can now verify duality: 
\begin{lemma}
  \label{lem:MaxDual}
  The bimodule $\lsup{\nDuAlg_1}\Max^{1}_{\nDuAlg_2}$ as defined above
  satisfies the requirements of Proposition~\ref{prop:DualMax} when $c=1$.
\end{lemma}

\begin{proof}
  This is a simple computation using the homological perturbation lemma,
  exactly as in~\cite[Lemma~\ref{BK2:lem:MinDual}]{Bordered2}.
\end{proof}
  
\begin{proof}[Proof of Proposition~\ref{prop:DualMax}]
  By analogy with~\cite[Section~\ref{BK2:subsec:GenMin}]{Bordered2},
  we can define $\lsup{\nDuAlg_1}\Max^c_{\nDuAlg_2}$ inductively in $c$,
  using the action of the bimodules of crossings; i.e.
  \[ \lsup{\nDuAlg_1}\Max^c_{\nDuAlg_2} =
  \lsup{\nDuAlg_1}\Max^{c-1}_{\nDuAlg_4}\DT~\lsup{\nDuAlg_4}\Pos^c_{\nDuAlg_3}
  ~\DT\lsup{\nDuAlg_3}\Pos^{c-1}_{\nDuAlg_2},\] for suitably chosen
  $\nDuAlg_3$ and $\nDuAlg_4$.  Property~\ref{prop:DualMax} is then
  proved by induction, with the basic case given by
  Lemma~\ref{lem:MaxDual}, and the inductive step using
  Proposition~\ref{prop:BimodulesOvernDuAlg} (for the positive
  crossing). (The steps are as in the proof
  of~\cite[Theorem~\ref{BK2:thm:MinDual}]{Bordered2}.)
\end{proof}

\subsection{Algebraic invariance}
\label{subsec:AlgInvar}

\begin{proof}[Proof of Proposition~\ref{prop:CurvedDPlanarIsotopies}]
Following terminology
from~\cite[Section~\ref{BK1:subsec:InvarainceUnderBridgeMoves}]{BorderedKnots},
Proving that the invariant $\lsup{\cClg}\DmodAlg(\DiagUp)$ is an isotopy invariant amounts to proving invariance under bridge moves,
which are:
\begin{enumerate}
\item Commutations of distant crossings
\item Trident moves
\item Critical points commute with distant crossings
\item Commuting distant critical points
\item Pair creation and annihilation.
\end{enumerate}

In~\cite{BorderedKnots,Bordered2}, these are verified
as obtained identities between corresponding $DA$
bimodules. To expedite computations, we worked in stead
with $DD$ bimodules.

In the notation of the present paper, the requisite identities for the
$DD$
module versions of commutation moves (for distant crossings and
critical points) are of the form
\[ 
X^i \DT Y^j \DT \CanonDD  \simeq Y^{j'}\DT X^{i'}\DT \CanonDD 
\]
where $X$, $Y$ can be either of $\Pos$, $\Neg$, $\Max$ or $\Min$,
$|i-j|>1$ (and the superscripts $i'$ and $j'$ have to be chosen
accordingly), curved modules over $\cBlg$ (as in
Subsection~\ref{subsec:FormalModules}), and
$\CanonDD=\lsup{\cBlg,\nDuAlg}\CanonDD$.  Trident moves come from
identities
\begin{align*}
  \Min^c\DT \Pos^{c+1}\DT \CanonDD& \simeq
  \Pos^{c+1}\DT \Min^c\DT \CanonDD \\
  \Pos^{c+1}\DT \Max^c\DT \CanonDD& \simeq
  \Min^c\DT \Pos^{c+1} \DT \CanonDD
\end{align*}
Pair creation and annihilation invariance from the identities  
\[ \Max^{c+1}\DT \Min^c\simeq \Id \simeq 
\Max^c\DT \Min^{c+1}. \]

The verifications from~\cite{Bordered2} (dropping $C_{i,j}$, and
viewing the algebras over $\Blg$ as curved) now appy verbatim to give
the stated identities.

In~\cite{Bordered2}, the type $DD$ identities implied corresponding
corresponding identities of type $DA$ modules, due to a Koszul
duality, which we have not established here. Nonethless, the above
identities do establish corresponding identities where the $DA$
modules act on any ``relevant'' type $D$ structure, in the sense of
Definition~\ref{def:Relevant}. Thus, the needed invariance follows
once we know that $\lsup{\cClg}\DmodAlg(\Diag)$ is relevant.  

Relevance
of $\lsup{\Clg}\DmodAlg(\Diag)$ is proved by induction on the number
of crossings and critical points in $\Diag$ (as in the proof of
Proposition~\ref{prop:InductiveStep}), with the inductive step
furnished by the fact that for $X$ is $\Pos$, $\Neg$, $\Max$, or
$\Min$ (curved DA bimodules over $\Blg$), there are corresponding DA
bimodules $X''$ so that
\[ \lsup{\Blg_2}X_{\Blg_1}\DT \lsup{\Blg_1,\nDuAlg_1}\CanonDD \simeq
\lsup{\nDuAlg_1}X_{\nDuAlg_2}\DT \lsup{\Blg_1,\nDuAlg_1}\CanonDD. \]
For $X\in\{\Pos,\Neg,\Min\}$, this was verified in
Proposition~\ref{prop:CurvedDABimodules}, while for $\Max$, it was
Proposition~\ref{prop:DualMax}.
\end{proof}

\begin{proof}[Proof of Proposition~\ref{prop:TangleInvariance}]
  In view of Proposition~\ref{prop:CurvedDPlanarIsotopies},
  it suffices to verify invariance of $\lsup{\cClg}\DmodAlg(\Diag)$
  under the Reidemeister moves.
  The corresponding $DD$ module identities for $\Alg$ and $\DuAlg$
  
  Reidemeister $1$ invariance follows from the identity
  $\Pos^c\DT\Max^c\DT \CanonDD \simeq \Pos^c\DT \CanonDD$, whereas the
  other Reidemeister moves follow from the braid relations
  from~\cite[Section~\ref{BK2:sec:BraidRelations}]{Bordered2}.  The
  identities are verified exactly as they are done there, appealing to
  the fact that $\lsup{\cClg}\Dmod(\Diag)$ is relevant, as shown in the proof
  of Proposition~\ref{prop:CurvedDPlanarIsotopies} above.
\end{proof}

\subsection{Computing the holomorphically defined invariant}
\label{subsec:ComputeD2}

\begin{proof}[Proof of Theorem~\ref{thm:ComputeD2}]
  The proof of Theorem~\ref{thm:ComputeD} applies, once we have shown
  the analogue of Theorem~\ref{thm:ComputeDDmods} for a local
  maximum; i.e. 
\[ 
\lsup{\cClg_2}\Max^c_{\cClg_1} \DT~ \lsup{\cClg_1,\nDuAlg_1}\CanonDD \simeq
\lsup{\cClg_2}\DAmodExt(\Hmax{c})_{\cClg_1} \DT~
\lsup{\cClg_1,\nDuAlg_1}\CanonDD. \]
This verification is very similar to the proof of Proposition~\ref{prop:MinimumComputation}, and is left to the reader.
\end{proof}

\bibliographystyle{plain}
\bibliography{biblio}

\end{document}